\documentclass[11pt,reqno]{amsart}

\usepackage{amsmath, amsthm, amssymb}
\usepackage{amsfonts}
\usepackage{mathtools}

\usepackage{hyperref}

\usepackage[ansinew]{inputenc}
\usepackage[dvips]{epsfig}
\usepackage{graphicx}
\usepackage[english]{babel}

\usepackage{thmtools}
\theoremstyle{plain}
\declaretheorem[title=Theorem, parent=section]{theorem}
\declaretheorem[title=Lemma,sibling=theorem]{lemma}
\declaretheorem[title=Proposition,sibling=theorem]{proposition}
\declaretheorem[title=Corollary,sibling=theorem]{corollary}

\theoremstyle{definition}
\declaretheorem[title=Definition,sibling=theorem]{definition}
\declaretheorem[title=Remark,sibling=theorem]{remark}
\declaretheorem[title=Remark, numbered=no]{remark*}

\declaretheorem[title=Assumption, numbered=no]{assumption*}

\numberwithin{equation}{section}

\usepackage[backgroundcolor=white, bordercolor=blue,
linecolor=blue]{todonotes}

\usepackage{dsfont}
\usepackage{bbm}

\newcommand{\N}{\mathbb{N}}
\newcommand{\R}{\mathbb{R}}

\newcommand{\cP}{\mathcal{P}}

\newcommand{\cE}{\mathcal{E}}

\newcommand{\cI}{\mathcal{I}}

\newcommand{\eps}{\varepsilon}

\newcommand{\1}{\mathbbm{1}}

\DeclareMathOperator{\dist}{dist}

\DeclareMathOperator{\diam}{diam}
\DeclareMathOperator{\supp}{supp}

\DeclareMathOperator{\dvg}{div}

\DeclareMathOperator*{\osc}{osc}

\DeclareMathOperator{\tail}{Tail}

\renewcommand{\d}{\textnormal{\,d}}

\newcommand{\average}{{\mathchoice {\kern1ex\vcenter{\hrule height.4pt
width 6pt depth0pt} \kern-9.7pt} {\kern1ex\vcenter{\hrule
height.4pt width 4.3pt depth0pt} \kern-7pt} {} {} }}
\newcommand{\dashint}{\average\int}

\begin{document}
\allowdisplaybreaks
\title{Higher regularity in nonlocal free boundary problems}
 
\author{Bego\~na Barrios}
\author{Xavier Ros-Oton}
\author{Marvin Weidner}

	\address{Departamento de An\'{a}lisis Matem\'{a}tico,
		Universidad de La Laguna e IMAULL
		\hfill \break \indent C/. Astrof\'{\i}sico Francisco S\'{a}nchez s/n, 
		38200 -- La Laguna, Spain}
	\email{bbarrios@ull.edu.es}

\address{ICREA, Pg. Llu\'is Companys 23, 08010 Barcelona, Spain \& \newline Universitat de Barcelona, Departament de Matem\`atiques i Inform\`atica, \newline Gran Via de les Corts Catalanes 585, 08007 Barcelona, Spain \& \newline Centre de Recerca Matem\`atica, Barcelona, Spain}
\email{xros@icrea.cat}
\urladdr{www.ub.edu/pde/xros}

\address{Departament de Matem\`atiques i Inform\`atica, Universitat de Barcelona, \newline Gran Via de les Corts Catalanes 585, 08007 Barcelona, Spain}
\email{mweidner@ub.edu}
\urladdr{https://sites.google.com/view/marvinweidner/}



\allowdisplaybreaks

\begin{abstract}
We study the higher regularity in nonlocal free boundary problems posed for general integro-differential operators of order $2s$.
Our main result is for the nonlocal one-phase (Bernoulli) problem, for which we establish that $C^{2,\alpha}$ free boundaries are~$C^\infty$.
This is new even for the fractional Laplacian, as it was only known in case $s=\frac12$.
We also establish a general result for overdetermined problems, showing that if the boundary condition is smooth, then so is $\partial\Omega$.
Our approach is very robust and works as well for the nonlocal obstacle problem, where it yields a new proof of the higher regularity of free boundaries, completely different from the one in \cite{AbRo20}.
In order to prove our results, we need to develop, among other tools, new integration by parts formulas and delicate boundary H\"older estimates for nonlocal equations with (local) Neumann boundary conditions that had not been studied before and are of independent interest.
\end{abstract}

\allowdisplaybreaks

\maketitle

\tableofcontents
\addtocontents{toc}{\protect\setcounter{tocdepth}{1}} 

\section{Introduction}

Free boundary problems appear in several areas of pure and applied mathematics, and have been a central line of research in elliptic PDE since the 1970s. One of the most important questions in this context is to understand the smoothness of free boundaries. Two pioneering works in this direction were those of Caffarelli \cite{Caf77} and Kinderlehrer-Nirenberg \cite{KN77}, which provided the foundations for many new ideas and techniques developed during the last fifty years; see e.g. \cite{AC81,Caf89,Sak91,CKS00,DGPT17,ESV20,DSV21}.

During the last two decades, special focus has been put on understanding \emph{nonlocal} free boundary problems \cite{CSS08,CRS10,CF13,DeSa15,CRS17,BFR18,BFR18b,EKPSS21,FeRo24c,RoWe24b,RTW25}. 
They arise as natural models whenever long-range interactions need to be taken into account e.g. in flame propagation models, probability and finance, or large systems of interacting particles; see \cite{CRS10,FeRo24,DL76} and references therein.
In this paper we study the two main problems in this context: the nonlocal one-phase free boundary problem, and the nonlocal obstacle problem.

\subsection{The one-phase free boundary problem}

Two very classical and motivating questions in the context of free boundary problems are \cite{AC81,Jer90}: 

\vspace{2mm}

\noindent \textbf{Question 1}: {\em Assume that we have a harmonic function in $\Omega\cap B_1$ satisfying both $u=0$ and $\partial_\nu u=1$ on $\partial\Omega\cap B_1$.
Can we conclude that $\partial\Omega\cap B_{1/2}$ is actually $C^\infty$?}

\vspace{2mm}

\noindent \textbf{Question 2}: {\em Assume that we have a domain $\Omega$ such that its Poisson kernel is smooth.\footnote{More precisely, given $z\in \Omega$ we assume the Poisson kernel $P_z$ equals $h|_{\partial\Omega}$ for some $h\in C^\infty(\R^n)$.}
Can we conclude that $\partial\Omega$ is actually $C^\infty$?}

\vspace{2mm}

These questions, and their relation to models in fluid mechanics and in flame propagation, motivated the development of the regularity theory for the so-called one-phase free boundary problem, which is also known as Bernoulli or Alt-Caffarelli free boundary problem; see \cite{AC81,CS05,Vel23} and the references therein.
Solutions to this problem --- i.e. $\Delta u=0$ in $\Omega\cap B_1$, with  $u=0$ and $\partial_\nu u=1$ on $\partial\Omega\cap B_1$ --- are critical points of the energy functional 
\begin{align*}
\mathcal E(u)=\int_{B_1} |\nabla u|^2 dx + \big|\{u>0\}\cap B_1\big|.
\end{align*}
The regularity theory for this problem can be essentially split into four steps:
\begin{itemize}
\item[(i)] If the free boundary is locally flat, then it is $C^{1,\alpha}$ \cite{AC81}.
\item[(ii)] The set of singular points, where the free boundary is not locally flat, is small \cite{AC81,Wei99}.
\item[(iii)] $C^{1,\alpha}$ free boundaries are actually $C^{2,\alpha}$ \cite{DFS19,LZ23}.
\item[(iv)] $C^{2,\alpha}$ free boundaries are actually $C^\infty$ \cite{KN77}.
\end{itemize}
We refer to the book \cite{Vel23} for a detailed account of these classical results.

The regularity theory for the obstacle problem follows a similar structure; see \cite{Caf77,KN77,Wei99,FRS20}.
An important difference for our purposes is that, in case of the obstacle problem, (iii) and (iv) can be done simultaneously; see \cite[Theorem~1]{KN77}. 
However, in case of the one-phase problem one usually has to do steps (iii) and (iv) separately, which requires the solution (and hence the domain) to be $C^2$; see \cite[Theorem~2]{KN77}.

\subsection{Nonlocal one-phase free boundary problems}

The nonlocal version of the one-phase problem was first studied by Caffarelli-Roquejoffre-Sire \cite{CRS10}, tackling the analog of \emph{Question 1} above in case of the fractional Laplacian operator $(-\Delta)^s$.
More generally, one considers general nonlocal operators of order~$2s$ 
\begin{equation}\label{L}
L v(x) := 2\int_{\R^n} \big( v(x) - v(x+h) \big) K(h) \d h,
\end{equation}
where $s\in(0,1)$, and the kernel $K$ satisfies the following ellipticity and homogeneity conditions
\begin{align}
\label{eq:K-comp-intro}
K(h) = |h|^{-n-2s} K(h/|h|), \qquad \lambda \le K(h/|h|) \le \Lambda, \qquad K(h)=K(-h),
\end{align}
for some positive constants $\lambda,\Lambda$.
Then, solutions to the nonlocal one-phase free boundary problem are nonnegative functions satisfying 
\begin{equation}\label{NOP}
\begin{array}{rcll}
L v &=& 0 & \quad  \text{ in } \Omega \cap B_1,\\
v &=&0 &\quad \text{ in } B_1 \setminus \Omega,\\
\displaystyle\frac{v}{d^s_{\Omega}} &=& A(\nu) &\quad \text{ on } \partial \Omega \cap B_1.
\end{array}
\end{equation}
Here, $d^s_{\Omega}$ is the distance to $\partial\Omega$,  and $A$ is given in terms of the Fourier symbol of $L$; see \eqref{eq:A-nu-def}.

The simplest and most canonical example is given by the fractional Laplacian $L=(-\Delta)^s$, corresponding to $K|_{\mathbb S^{n-1}}\equiv {c_{n,s}}$ and $A(\nu) \equiv c_s$.

Solutions to the problem \eqref{NOP} can be seen as nonnegative critical points of the energy functional
\begin{align*}
\cI(v) := \iint\limits_{(B_1^c \times B_1^c)^c} \big|v(x) - v(y)\big|^2 K(x-y) \d y \d x + |\{ v > 0 \} \cap B_1|.
\end{align*}
In the particular case of the square root of the Laplacian, i.e. $L=\sqrt{-\Delta}$, the regularity theory for this problem has been completely developed, thanks to its equivalence to a \emph{thin} one-phase problem in $\R^{n+1}_+$; see \cite{DeRo12,DeSa12,DeSa15,DeSa15b}.

Concerning the fractional Laplacian $(-\Delta)^s$ and more general operators $L$ of the form \eqref{L}-\eqref{eq:K-comp-intro}, the optimal regularity of solutions was established in \cite{CRS10} and in \cite{RoWe24c}, respectively, while the $C^{1,\alpha}$ regularity of flat free boundaries was established in \cite{DeRo12,DSS14} and in \cite{RoWe24b}, respectively.
Moreover, the smallness of the singular set was established for the fractional Laplacian in \cite{EKPSS21}.

As is mentioned in \cite[Remark 1.2]{RoWe24b}, the $C^\infty$ regularity of free boundaries remained an open problem, both for the fractional Laplacian ($s\neq\frac12$) and for more general operators $L$. This is the main question we tackle in this paper. It requires several new ideas, and in particular, the approach for the nonlocal obstacle problem from \cite{AbRo20} based on higher order boundary Harnack inequalities is not applicable here.

\subsection{Main results}

One of our main contributions is to establish the following result for the nonlocal one-phase problem.

\begin{theorem}
\label{thm:main-onephase-intro} 
Let $L$ be an operator of the form \eqref{L}-\eqref{eq:K-comp-intro} with $K|_{\mathbb{S}^{n-1}}\in C^\infty$.
Let $v \in L^\infty(\R^n)$ be any solution of \eqref{NOP} and let $\alpha > 0$.
Then we have 
\[\partial \Omega \cap B_1\in C^{2,\alpha} \qquad \Longrightarrow \qquad \partial\Omega\cap B_{1/2} \in C^\infty.\]
\end{theorem}

As explained in more detail below, proving this result requires substantial new ideas, since we cannot use any equivalence with a thin one-phase problem.

Thanks to \autoref{thm:main-obstacle-intro}, the only remaining step to complete the regularity theory for the non-local one-phase free boundary problem is to show that $C^{1,\alpha}$ free boundaries are actually $C^{2,\alpha}$.
We will do this in a future paper where, in the spirit of \cite{DeSa12}, we will modify the improvement of flatness in \cite{RoWe24b} to show that flat free boundaries are $C^{2,\alpha}$.

On the other hand, we also want to answer the nonlocal analogue of \emph{Question 2} above.
In case of the fractional Laplacian, it is known that the natural \emph{boundary}-value (not exterior-value) Dirichlet problem is given by $(-\Delta)^s u=0$ in $\Omega$, $u=0$ in $\Omega^c$, and $u/d_{\Omega}^{s-1}=\phi$ on $\partial\Omega$; see \cite{Aba15,Gru15}.
Furthermore, given $z\in \Omega$, the Poisson kernel corresponding to this problem is given by $P_z:=\frac{G_z}{d_{\Omega}^s}|_{\partial\Omega}$, where $G_z$ is the Green function, satisfying $(-\Delta)^s G_z=\delta_z$ in $\Omega$, and $G_z=0$ in $\Omega^c$.

Therefore, after translating and rescaling if necessary, the Green function $v=G_z$ will satisfy locally
\begin{equation}\label{NOP2}
\begin{array}{rcll}
L v &=& 0 & \quad  \text{ in } \Omega \cap B_1,\\
v &=&0 &\quad \text{ in } B_1 \setminus \Omega,\\
\displaystyle \frac{v}{d^s_{\Omega}} &=& h &\quad \text{ on } \partial \Omega \cap B_1,
\end{array}
\end{equation}
where $L=(-\Delta)^s$ and $h|_{\partial\Omega}$ coincides with the Poisson kernel $P_z$.

Having this in mind, we prove the following:

\begin{theorem}
\label{thm:main-onephase-intro2} 
Let $L$ be an operator of the form \eqref{L}-\eqref{eq:K-comp-intro} with $K|_{\mathbb{S}^{n-1}}\in C^\infty$, and let $h\in C^\infty(\R^n)$.
Let $0 \not= v \in L^{1}_{2s}(\R^n)\cap L^\infty(B_1)$ be any nonnegative solution of \eqref{NOP2} and let $\alpha > 0$.
Then we have 
\[\partial \Omega \cap B_1 \in C^{2,\alpha} \qquad \Longrightarrow \qquad \partial\Omega \cap B_{1/2} \in C^\infty.\]
\end{theorem}

This gives for the first time an answer to the nonlocal analogue of Question 2 above. Our theorem is new even for the fractional Laplacian. We also show a more general result in \autoref{prop:boundary-improvement2}, where we only assume that $h \in C^{k,\alpha}$ for some $k \ge 2$.

Another way to see this result is the following.
It is known that if $\partial\Omega$ is $C^\infty$, then $u/d_{\Omega}^s\in C^\infty(\overline\Omega \cap B_1)$; see \cite{Gru15,AbRo20}.
Our new \autoref{thm:main-onephase-intro2} is a ``reverse regularity'' result, in the sense that if $(u/d_{\Omega}^s)|_{\partial\Omega}$ is $C^\infty$, then so is $\partial\Omega$.

Finally, it is important to mention that the methods we develop in this paper are very robust, and allow us to treat at the same time one-phase free boundary problems and obstacle problems.
Actually, the case of the obstacle problem is simpler, and we provide a new proof of the following higher regularity result.

\begin{theorem}
\label{thm:main-obstacle-intro} 
Let $L$ be an operator of the form \eqref{L}-\eqref{eq:K-comp-intro} with $K|_{\mathbb{S}^{n-1}}\in C^\infty$.
Let $v \in L^{\infty}(\R^n)$ be a solution to the nonlocal obstacle problem $\min\{Lv,v-\varphi\}$ in $B_1$, where $\varphi\in C^\infty(\R^n)$, and denote $\Omega=\{v>\varphi\}$.
Assume that all free boundary points in $B_1$ are regular, in the sense that $u/d_{\Omega}^{1+s}\geq c>0$ in $\Omega\cap B_1$.
Let $\alpha > 0$. Then we have 
\[\partial \Omega \cap B_1 \in C^{1,\alpha} \qquad \Longrightarrow \qquad \partial\Omega \cap B_{1/2} \in C^\infty.\]
\end{theorem}

This result was first proved for the fractional Laplacian in \cite{KRS19,JN17} using the extension property of $(-\Delta)^s$ and then for general operators $L$ in \cite{AbRo20} by means of a higher-order boundary Harnack principle.
As mentioned before, this last approach does not work for the nonlocal one-phase problem, and new ideas are needed.

\subsection{Ideas of the proof}

As mentioned above, \autoref{thm:main-onephase-intro} is new even for the fractional Laplacian when $s\neq\frac12$, but it is much more general and deep than that. 
Proving this result for general nonlocal operators $L$ requires substantial new insights, since we cannot use any the extension problem for the fractional Laplacian as in \cite{DeSa12,JN17,KRS19}.

Another extra difficulty when treating general operators $L$ is that the boundary condition in the one-phase problem \eqref{NOP} (which comes from the criticality condition for the functional $\cI$) is not constant; see \cite{RoWe24b}.
This leads to \emph{oblique boundary conditions} in the auxiliary problems we consider, and fresh ideas are required in order to treat them.
Indeed, in our proof we need completely new integration by parts formulas for nonlocal equations with singular weights (see  \autoref{lemma:ibp-nonflat-trafo} and \autoref{lemma:ibp}), which lead to local boundary conditions that involve oblique derivatives whenever $L$ is not the fractional Laplacian.
We also prove boundary H\"older estimates and higher order regularity estimates for weak solutions to nonlocal equations with such singular weights that exhibit a local boundary condition (see \autoref{lemma:C-alpha-estimate} and \autoref{lemma:bdry-reg}).

Let us give a brief sketch of the strategy and some of the main ideas of the proofs of \autoref{thm:main-onephase-intro} and \autoref{thm:main-onephase-intro2}. 

We start with a solution $v$ of either \eqref{NOP} or \eqref{NOP2}, and assume $\partial\Omega\cap B_1$ to be $C^{k,\alpha}$, with $k\geq2$.
Moreover, assume that $\partial\Omega$ can be written as a graph in the $e_n$ direction, so that $\nu\cdot e_n>0$ on $\partial\Omega\cap B_1$.
Then, we take (a truncated version of) the quotient
\[ w = \frac{\partial_i v}{\partial_n v},\]
and we prove that it solves a weighted equation of the form
\begin{align}
\label{eq:weighted-eq-intro}
\int_{\Omega} \big(w(x)-w(y)\big)\partial_n v(x)\partial_n v(y)K(x-y)dy = 0 \quad \textrm{for}\quad x\in \Omega,
\end{align}
together with some \emph{local} Neumann-type boundary condition on $\partial\Omega$.
Moreover, it can be proved that $w$ is $C^{k-1,\alpha}$ in $\overline \Omega$, and $\partial_n v\sim d_{\Omega}^{s-1}$, where $d_{\Omega}$ is the distance to the boundary.
Furthermore, the $i$-th component of the normal vector $\nu^{(i)}$ can be written in terms of $w$, and therefore, if we can show that $w$ is more regular than $C^{k-1,\alpha}$, this would improve the regularity of $\partial\Omega$.

For this, the rough idea is to prove an (a priori) $C^\gamma$ estimate for $w$, for some $\gamma>0$, using the equation \eqref{eq:weighted-eq-intro} it satisfies. This regularity result can be found in \autoref{lemma:bdry-reg}.
Then, iterating the $C^\gamma$ estimate by taking incremental quotients in tangential directions, we get more tangential regularity for $w$, say $C^{k-1+\alpha+\gamma}$, which in turn yields that $\partial\Omega$ is $C^{k+\alpha+\gamma}$ (and therefore $C^\infty$).

However, there are many issues when trying to run this argument.
First, taking higher order incremental quotients for nonlocal equations is very delicate and technical.
Second, and more importantly, proving such a $C^\gamma$ estimate up to the boundary turns out to be quite difficult, as we are dealing with a nonlocal equation possessing singular weights near the boundary, which had not been studied before. Moreover, the equation exhibits a local Neumann-type boundary condition.

The aforementioned boundary conditions arise as follows: first, we show that since $\partial_iv/d^{s-1}, \partial_n v/d^{s-1}\in C^1(\overline\Omega\cap B_1)$, it holds
\[\nabla w|_{\partial\Omega}= \frac{s(\frac{v}{d^s})\partial_\tau \nu + \nu \,\partial_\tau (\frac{v}{d^s})}{s(\frac{v}{d^s})(\nu^{(n)})^2},\]
where $\tau(x):=\nu^{(n)}(x)e_i - \nu^{(i)}(x) e_n$ is tangential to $\partial\Omega$.
In particular, since $\nu\cdot \partial_\tau \nu=\partial_\tau(\frac12|\nu|^2)=0$, for problem \eqref{NOP2} it follows that 
\[ \partial_\nu w = \frac{ \partial_\tau h}{sh(\nu^{(n)})^2} \in C^{k-1,\alpha} \quad \textrm{on}\quad \partial\Omega\cap B_1.\]
For the nonlocal one-phase problem \eqref{NOP} this boundary condition is not good enough, since $h(x)=A(\nu(x))$, and so we would only get $\partial_{\nu} w \in C^{k-2,\alpha}$, which is not sufficient to improve the regularity of $w$ beyond $C^{k-1,\alpha}$.
Instead, we prove that
\[\partial_{\theta(x)} w = 0 \quad \textrm{on}\quad \partial\Omega\cap B_1, \]
for a direction $\theta(x)\, ||\, \nabla A(\nu(x))$.
That is, we get an oblique derivative condition (matching perfectly with the integration by parts formula in \autoref{lemma:ibp-nonflat}), which allows us to use our novel regularity result in \autoref{lemma:bdry-reg} to improve the boundary regularity of $w$.

For both problems \eqref{NOP} and \eqref{NOP2} we are able to prove the desired $C^\gamma$ estimate for every $\gamma\in (\max\{\frac12,2s-1\},1)$. As mentioned before, a unified version of this result is contained in \autoref{lemma:bdry-reg} and its proof turns out to require significant new ideas.

First, we need to write the equation \eqref{eq:weighted-eq-intro} in a weak formulation, since otherwise the boundary condition would not pass to the limit in any compactness argument.
For this, we find and prove new integration by parts identities (see \autoref{lemma:ibp-nonflat}), which were not known before, even for the fractional Laplacian.
In the simplest case (see \autoref{lemma:ibp}), this identity reads as
\begin{align*}
\int \hspace{-2mm}\int \left|\frac{u(x)}{(x_n)^{s-1}}-\frac{u(y)}{(y_n)^{s-1}}\right|^2 & \hspace{-2mm}(x_n)_+^{s-1}(y_n)_+^{s-1}\frac{dx\,dy}{|x-y|^{n+2s}} \\
&+ 2\int_{\{x_n>0\}} \hspace{-2mm} u\,(-\Delta)^su = c_{n,s} \int_{\{x_n=0\}} \frac{u}{(x_n)^{s-1}} \partial_\nu \left(\frac{u}{(x_n)^{s-1}}\hspace{-1mm}\right),
\end{align*}
for any $u\sim (x_n)_+^{s-1}$ in $\{x_n>0\}$, with $u=0$ in $\{x_n\leq0\}$.

Such an integration by parts formula can be seen as a generalization of the Green's identity in \cite{Gru20} (see also \cite{Aba15}), in the sense that our formula contains a nonlocal energy term (which did not appear in previous works). Moreover, from our identity we can recover the nonlocal Green's identity
\[ c_{n,s} \int_{\{x_n>0\}} u\,(-\Delta)^sv-v\,(-\Delta)^su = \int_{\{x_n=0\}} \frac{u}{(x_n)^{s-1}} \partial_\nu \left(\frac{v}{(x_n)^{s-1}}\right)-\frac{v}{(x_n)^{s-1}} \partial_\nu \left(\frac{u}{(x_n)^{s-1}}\right),\]
which was proved in \cite{Gru20}. However, for our purposes, such a formula is not enough, and we need the new identity stated above, which involves the nonlocal energy.

Thanks to our novel integration by parts formulas, by using a contradiction and compactness argument, the $C^\gamma$ estimate reduces to showing a $C^\delta$ estimate for a simpler version of equation \eqref{eq:weighted-eq-intro} 
\begin{equation}\label{sfggg}
\begin{array}{rcll}
\int_{\{y_n>0\}} \big(w(x)-w(y)\big)(x_n)^{s-1}(y_n)^{s-1}K(x-y)dy &=& 0 & \quad  \text{ in } \{x_n>0\},\\
\partial_n w &=&0 &\quad \text{ on } \{x_n=0\},
\end{array}
\end{equation}
which in the weak sense reads as 
\[ \int_{\R^n} \int_{\R^n} \big(w(x)-w(y)\big)\big(\eta(x)-\eta(y)\big)(x_n)^{s-1}(y_n)^{s-1}K(x-y)dx\,dy = \int_{\{x_n=0\}} w\, \partial_{\theta_K'}\eta, \]
for every $\eta\in C^\infty_c(\R^n)$, where $\theta_K'\in \R^{n-1}$ is a fixed direction that depends only on the kernel $K$.

To prove a $C^\delta$ estimate for weak solutions of \eqref{sfggg} (see \autoref{lemma:C-alpha-estimate}), we use a De Giorgi iteration technique for nonlocal energies, and special care is needed due to the singular weights near $\{x_n=0\}$.
Such a $C^\delta$ estimate also allows us to establish a Liouville theorem for \eqref{sfggg} (see \autoref{prop:weighted-Liouville}), which is the key point to conclude the proof of the $C^\gamma$ estimate for the more general equation \eqref{eq:weighted-eq-intro} with coefficients in \autoref{lemma:bdry-reg}.

\vspace{2mm}

We actually do \emph{all} the proofs with a general weight of the form $(x_n)^{\beta-1}$, with $\beta \in [s,1+s]$.
This requires only little extra work, however, the most difficult case turns out to be $\beta=s$.
The case $\beta=1+s$ corresponds to the obstacle problem, and this allows us to prove \autoref{thm:main-obstacle-intro}.
Moreover, the case $\beta\in(s,1+s)$ will be useful in the study of the nonlocal Alt-Phillips problem.

\subsection{Outline}
This article is structured as follows.
In Section 2 we give some preliminary results, mainly related with the incremental quotients of the kernel as well as boundary estimates of solutions.
In Section 3 we prove our new integration by parts identities, both in the half-space and for general domains.
In Section 4 we implement the De Giorgi iteration to prove the $C^\delta$ estimate and the Liouville-type theorem needed to establish the H\"older estimates up to the boundary for general equations that are in turn developed in Section 5.
Finally, in Section 6 we deduce the higher regularity of free boundaries for the problems discussed above and prove our main results \autoref{thm:main-onephase-intro}, \autoref{thm:main-onephase-intro2}, and \autoref{thm:main-obstacle-intro}.

\subsection{Acknowledgments}
BB was partially supported by the project {\it An\'alisis de Fourier y Ecuaciones no locales en Derivadas Parciales} Grant PID2023-148028NB-I00 founded by MCIN/ AEI/10.13039/ 501100011033/ FEDER, UE.
XR and MW were supported by the European Research Council under the Grant Agreement No. 101123223 (SSNSD), and by AEI project PID2021-125021NA-I00 (Spain).
Moreover, BB and XR were partially supported by the AEI grant RED2024-153842-T (Spain). 
XR was also supported by the AEI--DFG project PCI2024-155066-2 (Spain--Germany), and by the Spanish State Research Agency through the Mar\'ia de Maeztu Program for Centers and Units of Excellence in R{\&}D (CEX2020-001084-M).

\section{Preliminaries}

In this section we collect several preliminary results and introduce some basic objects that will be relevant for the result of the paper. 
In the following, let $d_{\Omega}$ be a suitable regularized distance function as in \cite{AbRo20}.
Recall that the kernel $K$ satisfies
\begin{align}
\label{eq:K-comp}
K(h) = |h|^{-n-2s} K(h/|h|), \qquad \lambda \le K(h/|h|) = K(-h/|h|) \le \Lambda,
\end{align}
for some $0 < \lambda \le \Lambda < \infty$, and $s \in (0,1)$. 
Moreover, we will often assume that $\Vert K \Vert_{C^{k,\sigma}(\mathbb{S}^{n-1})} \le \Lambda$ for some $k \in \N \cup \{ 0 \}$ and $\sigma \in (0,1)$. In particular, if $k = 0$, this property implies
\begin{align}
\label{eq:K-reg}
|K(x) - K(y)| \le C \min\{ |x|,|y|\}^{-n-2s-\sigma'} |x-y|^{\sigma'}
\end{align}
for any $\sigma' \in (0,\sigma]$, where $C > 0$ depends only on $\sigma,\Lambda$. 

We also use the following notation, given a kernel $I : \R^n \times \R^n \to \R$
\begin{align*}
\mathcal{L}\big( I(x,y)\big)(u)(x) = 2 \int_{\R^n} \big( u(x) - u(y) \big) I(x,y) \d y.
\end{align*}

\subsection{Flattening the boundary}
\label{subsec:flattening}

Let $\Omega \subset \R^n$ be such that $\partial \Omega \in C^{k,\alpha}$ in $B_2$ for some $k \in \N$ and $\alpha \in (0,1]$. We consider a diffeomorphism that flattens the boundary around $0 \in \partial \Omega$ as in \cite[Lemma A.3]{AbRo20} (see also \cite[Section 2.4]{RoWe25} for more details), i.e. let $\Phi : B_2 \to B_2$ be a diffeomorphism such that $\Phi(0) = 0$, $\Phi(B_2 \cap \{ x_n =  0 \}) = B_{2} \cap \partial \Omega$, and $(x_n)_+ = d_{\Omega}(\Phi(x))$, and moreover with $\Phi \in C^{k,\alpha}(B_2) \cap C^{\infty}_{loc}(B_2 \cap \{ x_n > 0 \})$
\begin{align*}
|D^j \Phi| \le C_j d_{\Omega}^{k + \alpha - j} ~~ \text{ in } B_2 \cap \{ x_n > 0 \} \qquad \forall j \in \N: ~~ j > k + \alpha.
\end{align*}
Note that in particular we have
\begin{align*}
\delta_{i,n} = \nabla d_{\Omega}(\Phi(x)) \partial_i \Phi(x), \qquad \nabla d_{\Omega}(\Phi(x)) = \partial_n \Phi(x).
\end{align*}
Moreover, as in \cite{AbRo20} we assume that $\Omega$ is flat outside $B_2$ so that we can extend $\Phi$ by the identity in $B_2^c$. In this way, the regularity norms remains bounded and $\Phi(B_2^c) = B_2^c$.
Note that since $\Phi$ is a diffeomorphism, also $\Phi^{-1}$ is at least Lipschitz continuous. Hence, it always holds (globally after the extension)
\begin{align}
\label{eq:Phi-comp}
C^{-1} |x-y| \le |\Phi(x) - \Phi(y)| \le C |x-y|, \qquad C^{-1} |x| \le |\Phi(x)| \le C |x| \qquad \forall x,y \in \R^n,
\end{align}
where the second property follows since $\Phi(0) = 0$. Moreover, we can assume without loss of generality that $D \Phi(0) = I$. Therefore, it holds $|D \Phi(x) - I| \le C |x|^{\alpha}$, which implies that for $|x| \le \frac{1}{2}C^{-1/\alpha}$
\begin{align}
\label{eq:DPhi-comp}
\frac{1}{2} \le 1 - C |x|^{\alpha} \le |D \Phi(x)| \le 1 + C |x|^{\alpha} \le \frac{3}{2}.
\end{align}

Moreover, we will have that $|\det D \Phi| \ge C$, at least in a small ball centered at $0$. We assume from now on that this property and \eqref{eq:DPhi-comp} hold true in $B_2$ and since the extension of $\Phi$ satisfies it globally, we have it in the full space.

\subsection{Incremental quotients}

Given $h \in \R^n$ with $h \cdot e_n = h_n = 0$, we write $h = (h',h_n) = (h',0)$. Moreover, we set for $k \in \N$:
\begin{align*}
\Delta_h w(x) := \Delta^1_h w(x) := w(x+h) - w(x), \qquad \Delta^k_h w(x) = \Delta^{k-1}_h (\Delta_h w)(x).
\end{align*}
Similarly, for a function $K : \R^n \times \R^n \to \R$, we define 
\begin{align*}
\Delta_h K(x,y) := \Delta_h^1 K(x,y) := K(x+h,y+h) - K(x,y), \qquad \Delta^k_h K(x,y) = \Delta^{k-1}_h (\Delta_h K)(x,y).
\end{align*}
Moreover, given $k \in \N$ and $\alpha \in (0,1]$, we set 
\begin{align*}
D_h^{k+\alpha}w(x) = |h|^{-k-\alpha} \Delta^{k+1}_h w(x), \qquad D_h^{k+\alpha}K(x,y) = |h|^{-k-\alpha}\Delta^{k+1}_h K(x,y).
\end{align*}
Following the previous notation we can prove the next estimates regarding with the incremental quotients of the kernel that will be used often in our arguments.
\begin{lemma}
\label{lemma:K-estimate}
Let $K$ be as in \eqref{eq:K-comp} and assume that $K \in C^{k,\alpha}(\mathbb{S}^{n-1})$ for some $k \in \N$ and $\alpha \in (0,1]$ with $\Vert K \Vert_{C^{k,\alpha}(\mathbb{S}^{n-1})} \le \Lambda$. Moreover, let $\Phi \in C^{k,\alpha}(\R^n)$ be as in Subsection \ref{subsec:flattening}. Then, it holds
\begin{align}
|D_h^{k-1+\alpha} K(\Phi(x) - \Phi(y))| + |K(\Phi(x) - \Phi(y))| &\le C |x-y|^{-n-2s},\label{Pain1}
\end{align}
If in addition we assume that $K \in C^{k+1,\alpha}(\mathbb{S}^{n-1})$, it also holds
\begin{align}
|D_h^{k-1+\alpha} [K(\Phi(x) - \Phi(z))] - D_h^{k-1+\alpha} [K(\Phi(y) - \Phi(z))]| &\le C  \frac{|x-y|}{\min\{ |x-z| , |y-z| \}^{n+2s+1}}.\label{Pain2}
\end{align}
In the previous two estimates $C$ only depends on $n,s,\Lambda$, and the constant in \eqref{eq:Phi-comp}.
\end{lemma}

\begin{proof}
It holds by the two-sided bounds for $\Phi$ in \eqref{eq:Phi-comp}
\begin{align*}
|K(\Phi(x) - \Phi(y))| \le C |\Phi(x) - \Phi(y)|^{-n-2s} \le C [\Phi^{-1}]_{C^{0,1}(\R^n)}^{-1} |x-y|^{-n-2s}.
\end{align*}
Let us first prove \eqref{Pain1}, which in case $k=1$ simply reads:
\begin{align}\label{B3}
|D_h^{\alpha} K(\Phi(x) - \Phi(y))| \le C |x-y|^{-n-2s}.
\end{align}
To see this, recall that if $K \in C^{0,1}(\mathbb{S}^{n-1})$
\begin{align}\label{toPain2}
|K(x) - K(y)| \le C \min\{|x| , |y| \}^{-n-2s-1} |x-y|. 
\end{align}
Therefore, 
\begin{align}\label{kk1}
|D_h^{\alpha} K(\Phi(x) - \Phi(y))| &\le |h|^{-\alpha} |K(\Phi(x+h) - \Phi(y+h)) - K(\Phi(x) - \Phi(y))| \nonumber\\
&\le C |h|^{-\alpha} \min\{|\Phi(x+h) - \Phi(y+h)| , |\Phi(x) - \Phi(y)| \}^{-n-2s-1} \nonumber\\
&\qquad \qquad \qquad \qquad |(\Phi(x+h) - \Phi(y+h)) - (\Phi(x) - \Phi(y))| \\
&\le C [\Phi^{-1}]_{C^{0,1}(\R^n)}^{n+2s+1} |x-y|^{-n-2s-1} |h|^{-\alpha}|(\Phi(x+h) - \Phi(y+h)) - (\Phi(x) - \Phi(y))|\nonumber.
\end{align}
Let us estimate the last factor as follows
\begin{align}
\begin{split}
\label{eq:Pain1-help0}
& |(\Phi(x+h) - \Phi(y+h)) - (\Phi(x) - \Phi(y))| \le |D \Phi(x) - D \Phi(y)||h| \\
&\qquad + |\Phi(x+h) - \Phi(x) - D \Phi(x) \cdot h| + |\Phi(y+h) - \Phi(y) - D \Phi(y) \cdot h| \\
&\le [\Phi]_{C^{1+\alpha}(\R^n)} |h| \max\{|h|^{\alpha} , |x-y|^{\alpha} \}.
\end{split}
\end{align}
Moreover, it holds
\begin{align}
\begin{split}
\label{eq:Pain1-help}
& |(\Phi(x+h) - \Phi(y+h)) - (\Phi(x) - \Phi(y))| \le |D \Phi(y+h) - D \Phi(y)||x-y| \\
&\qquad + |\Phi(x+h) - \Phi(y+h) - D \Phi(y+h) \cdot (x-y)| + |\Phi(x) - \Phi(y) - D \Phi(y) \cdot (x-y)| \\
&\le [\Phi]_{C^{1+\alpha}(\R^n)} |x-y| \max\{|h|^{\alpha} , |x-y|^{\alpha} \}.
\end{split}
\end{align}
Thus, by combination of the previous two estimates, we conclude
\begin{align*}
|D_h^{\alpha} K(\Phi(x) - \Phi(y))| &\le C |x-y|^{-n-2s-1} |h|^{-\alpha} |(\Phi(x+h) - \Phi(y+h)) - (\Phi(x) - \Phi(y))| \\
 &\le C |x-y|^{-n-2s-1} \begin{cases}
|h|^{1-\alpha} |x-y|^{\alpha}, ~~& \text{ if } |h| \le |x-y|, \\
|x-y|, ~~& \text{ if } |x-y| \le |h|
\end{cases} \\
&\le C |x-y|^{-n-2s}.
\end{align*}
This proves the claim in case $k = 1$. To prove it for $k=2$ let us assume that $\Phi,\, K \in C^{2,\alpha}$. Using that 
\begin{align}\label{grad}
|DK(x-y)| \le |x-y|^{-n-2s-1},
\end{align}
and that $|DK(x) - DK(y)| \le C \min\{|x| , |y| \}^{-n-2s-2} |x-y|$, by Taylor and a similar computation as the one done in \eqref{kk1}, we get
\begin{align*}
|D_h^{1+\alpha} K(\Phi(x) - \Phi(y))| &\le {\sup _{t\in [0,1]}|D_h^{\alpha} \left[(\partial_{(h,h)})(K(\Phi(x+th) - \Phi(y+th))\right]|}\\
& \le \sup _{t\in [0,1]} |D_h^{\alpha} [ (DK)(\Phi(x+th) - \Phi(y+th)) (\partial_h \Phi(x+th) - \partial_h \Phi(y+th)) ] | \\
&\le \sup _{t\in [0,1]} |[D_h^{\alpha} (DK)(\Phi(x+th) - \Phi(y+th)) ](\partial_h \Phi(x+th) - \partial_h \Phi(y+th)) ] | \\
&\quad + |(DK)(\Phi(x+th) - \Phi(y+th)) D_h^{\alpha}[(\partial_h \Phi(x+th) - \partial_h \Phi(y+th)) ] |\\
&\le C |x-y|^{-n-2s-1} \left(|x-y| +\sup _{t\in [0,1]} |D_h^{\alpha}[(\partial_h \Phi(x+th) - \partial_h \Phi(y+th)) ] |\right).
\end{align*}
Here, by the notation $\partial_{(h,h)} G(x,y)$ for a function $G(x,y)$ we mean the derivative of $G$ into the direction $(h,h)$.
We conclude the proof for the case $k=2$ by using that, since $\partial_h \Phi\in C^{1,\alpha}$ (see \eqref{eq:Pain1-help0} and \eqref{eq:Pain1-help}), 
\begin{align}
\label{eq:Pain1-help1}
\begin{split}
 D_h^{\alpha}[\partial_h \Phi(x) - \partial_h \Phi(y)] &= |h|^{-\alpha}|(\partial_h\Phi(x+h) - \partial_h\Phi(y+h)) - (\partial_h\Phi(x) - \partial_h\Phi(y))| \\
&\le |{D \partial_h}\Phi(y+h) -{D \partial_h}\Phi(y)||x-y| \\
&\qquad + |\partial_h\Phi(x+h) - \partial_h\Phi(y+h) - D \partial_h\Phi(y+h) \cdot (x-y)| \\
&\quad + |\partial_h\Phi(x) - \partial_h\Phi(y) - D \partial_h\Phi(y) \cdot (x-y)| \\
&\le C \begin{cases}
|h|^{1-\alpha} |x-y|^{\alpha}, ~~& \text{ if } |h| \le |x-y|, \\
|x-y|, ~~& \text{ if } |x-y| \le |h|
\end{cases}\\
&\le C|x-y|.
\end{split}
\end{align} 

We can now prove the result for general $k\in\N$ by proceeding in an analogous way as before. Indeed, if $\Phi,\, K \in C^{k,\alpha}$, we get
\begin{align*}
&|D_h^{k-1+\alpha} K(\Phi(x)- \Phi(y))|\le \sup _{t\in [0,1]}\left|D_h^{\alpha} \left[(\partial^{k-1}_{(h,h)})\big(K(\Phi(x+th) - \Phi(y+th))\big)\right]\right|\\
& \le \sup _{t\in [0,1]}\left|D_h^{\alpha} \left[  \sum_{j=1}^{k-1} C_{k-1,j} D^j K(\Phi(x+th) - \Phi(y+th))  
\sum_{\substack{\ell_1 + \cdots + \ell_j = k-1 \\ \ell_i \geq 1}}
\prod_{i=1}^{j} \left( \partial^{\ell_i}_{h} \Phi(x+th) - \partial^{\ell_i}_{h} \Phi(y+th) \right)
\right] \right| \\
& \le \sup _{t\in [0,1]} \left(\sum_{j=1}^{k-1} C_{k-1,j}\left| D_h^{\alpha} \left[ D^j K(\Phi(x+th) - \Phi(y+th)) \right]\right| 
\sum_{\substack{\ell_1 + \cdots + \ell_j = k-1 \\ \ell_i \geq 1}}
\prod_{i=1}^{j} \left| \partial^{\ell_i}_{h} \Phi(x+th) - \partial^{\ell_i}_{h} \Phi(y+th) \right|\right.  \\
&\quad +\left. \sum_{j=1}^{k-1} C_{k-1,j}\left| D^j K(\Phi(x+th) - \Phi(y+th)) \right|
\sum_{\substack{\ell_1 + \cdots + \ell_j = k-1 \\ \ell_i \geq 1}}\left| D_h^{\alpha} \left[
\prod_{i=1}^{j} \left(\partial^{\ell_i}_{h} \Phi(x+th) - \partial^{\ell_i}_{h} \Phi(y+th) \right)\right] \right|\right)\\
&\leq C|x-y|^{-n-2s}.
\end{align*}
To get the last estimate we have used that, for every $j \in \{ 1,\,\ldots, k-1 \}$,
\begin{align*}
|D^j K(x) - D^j K(y)| \le C \min\{|x| , |y| \}^{-n-2s-(j+1)} |x-y|, \qquad |D^jK(x-y)| \le |x-y|^{-n-2s-j},
\end{align*}
which in particular implies $D_h^{\alpha} \left[ D^j K(\Phi(x) - \Phi(y)) \right] \leq C|x-y|^{-n-2s-j}$ (see \eqref{kk1}). Moreover the regularity of $\Phi \in C^{k,\alpha}$ allows us to deduce
\begin{align*}
D_h^{\alpha} \left[
\prod_{i=1}^{j} \left( \partial^{\ell_i}_{h} \Phi(x+th) - \partial^{\ell_i}_{h} \Phi(y+th) \right) \right]\leq C|x-y|^{j},
\end{align*}
by using the product rule for the incremental quotients, and the fact that
\begin{align*}
|D_h^{\alpha}(\partial^{\ell_i}_{h} \Phi(x+th) - \partial^{\ell_i}_{h} \Phi(y+th))|\leq C|x-y|, \quad \ell_i\leq k-1.
\end{align*}
and similar computations as the ones done in \eqref{eq:Pain1-help1}.

Let us now prove \eqref{Pain2}. 
We have by the mean value formula and the product rule 
\begin{align*}
&|D_h^{k-1+\alpha} [K(\Phi(x) - \Phi(z))] - D_h^{k-1+\alpha} [K(\Phi(y) - \Phi(z))]|\\
&=\left|D_h^{k-1+\alpha} \int_{0}^{1}\frac{d}{dt}K\left(\Phi(tx+(1-t)y)-\Phi(z)\right) \d t \right|\\
&=\left|D_h^{k-1+\alpha} \int_{0}^{1}DK \left(\Phi(tx+(1-t)y)-\Phi(z)\right)  \cdot (D\Phi)(tx+(1-t)y)(x-y) \d t \right|
\end{align*}
Taking into account that $D_{h}^{r}G(x,y):=D_{h}^{r}(x-y)=0$ unless $r=0$, applying the product rule for the incremental quotients we get
\begin{align*} 
&|D_h^{k-1+\alpha} [K(\Phi(x) - \Phi(z))] - D_h^{k-1+\alpha} [K(\Phi(y) - \Phi(z))]|\\
&\le |x-y|\int_0^1  \left|D_h^{k-1+\alpha} \left[DK \left(\Phi(tx+(1-t)y)-\Phi(z)\right)(D\Phi)(tx+(1-t)y)\right]\right| \d t \\
&\le C |x-y| \sum_{l = 1}^{k} \int_0^1  \left|D_h^{l-1+\alpha} \left(DK \left(\Phi(tx+(1-t)y)-\Phi(z)\right)\right)\right| |D_h^{k-l}(D\Phi)(tx+(1-t)y+(l-1)h)| \d t \\
&\quad + C |x-y|\int_{0}^{1}\left|DK \left(\Phi(tx+(1-t)y)-\Phi(z)\right)\right| |D_h^{k-1+\alpha} (D\Phi)(tx+(1-t)y| \d t.
\end{align*}
Using now that $\Phi\in C^{k,\alpha}(\R^n)$ (which in particular implies $|D_h^{k-1+\alpha} (D\Phi)|\leq C$ with $C$ independent of $h$) and that $DK\in C^{k,\alpha}(\mathbb{S}^{n-1})$, by \eqref{grad} and \eqref{Pain1} applied to a derivative of $K$, we conclude
\begin{align*}
&|D_h^{k-1+\alpha} [K(\Phi(x) - \Phi(z))] - D_h^{k-1+\alpha} [K(\Phi(y) - \Phi(z))]|\\
&\qquad \leq C|x-y|\left(\int_{0}^{1} |(x-y)t+(y-z)|^{-n-2s-1} \, dt\right) \leq \frac{C|x-y|}{\min\{|x-z| , |y-z|  \}^{n+2s+1}}.
\end{align*}
\end{proof}

\subsection{Fine boundary asymptotic for solutions to nonlocal equations}

We introduce the following H\"older spaces with respect to a domain $\Omega \subset \R^n$ for which boundary decay estimates hold true localized in a ball $B \subset \R^n$. That is, given $k,l \in \N$, $\alpha,\beta \in (0,1]$ with $l + \beta \le k + \alpha$, we define
\begin{align*}
C^{k+\alpha}_{l+\beta}(\Omega | B) = \Big\{ v \in C^{k+\alpha}_{loc}(\Omega) \cap C^{l+\beta}(\overline{\Omega}) ~:~ & \Vert v \Vert_{C^{k+\alpha}_{l+\beta}(\Omega | B)} < \infty \Big\},
\end{align*}
where
\begin{align*}
\Vert v \Vert_{C^{k+\alpha}_{l+\beta}(\Omega | B)} &:= \Vert v \Vert_{C^{l+\beta}(\overline{\Omega} \cap B)} + [v]_{C^{k+\alpha}_{l+\beta}(\Omega | B)},\\
[v]_{C^{k+\alpha}_{l+\beta}(\Omega | B)} &:= \sup \Big\{r^{-(l+\beta) + (m+\gamma)} [v]_{C^{m+\gamma}(B_{r/2}(x_0))} :  m+\gamma \in [l+\beta,k+\alpha],\\
& \qquad \qquad \qquad\qquad \qquad \qquad \qquad \qquad \qquad~~ x_0 \in \Omega, ~ r = d_{\Omega}(x_0),\, B_r(x_0) \subset \Omega \cap B \Big\}.
\end{align*}

These function spaces occur frequently in our analysis. They take into account the discrepancy between interior and boundary regularity, but keep track of the correct scaling of the interior estimates at the boundary. Note that $C^{k+\alpha}_{k+\alpha}(\Omega | B) = C^{k+\alpha}(\overline{\Omega} \cap B)$.

The following lemma generalizes the boundary regularity results from \cite{AbRo20} and \cite{RoWe24a}. Its potential suboptimality with respect to the regularity of $K$ is due to \cite[Theorem 2.2]{AbRo20}.

\begin{lemma}
\label{lemma:boundary-reg}
Let $\Omega \subset \R^n$ be a bounded domain with $\partial \Omega \in C^{k,\alpha}$ for some $k \in \N$ and $\alpha \in (0,1)$. Let $K$ be as in \eqref{eq:K-comp} with $K\in C^{2(k+\alpha) + 1}(\mathbb{S}^{n-1})$, $\Vert K \Vert_{C^{2(k+\alpha) + 1}(\mathbb{S}^{n-1})} \le \Lambda$, and  $\alpha \not \in\{s,1-s\}$, and let $u \in L^{\infty}(\R^n)$ be a solution to
\begin{align*}
\begin{cases}
L u &= f ~~ \text{ in } \Omega \cap B_1,\\
u &= 0 ~~ \text{ in } B_1 \setminus \Omega.
\end{cases}
\end{align*}
If $f \in C^{k+\alpha}(\overline{\Omega} \cap B_1)$, then, it holds $u/d_{\Omega}^s \in C^{k+\alpha+2s}_{k-1,\alpha}( \Omega | B_{1/2})$ and
\begin{align*}
\left\Vert \frac{u}{d_{\Omega}^s} \right\Vert_{C^{k+\alpha+2s}_{k-1,\alpha}(\Omega | B_{1/2})} \le C \left(  \Vert u \Vert_{L^{\infty}(\R^n)} + \Vert f \Vert_{C^{k+\alpha}(\overline{\Omega} \cap B_1)} \right).
\end{align*}
Moreover, if $f \in C^{k+\alpha+1}(\Omega \cap B_1)$, then it holds $\nabla u/d_{\Omega}^{s-1} \in C^{k+\alpha+2s}_{k-1,\alpha}(\overline{\Omega} | B_{1/2})$ and
\begin{align*}
\left\Vert \frac{\nabla u}{d_{\Omega}^{s-1}} \right\Vert_{C^{k+\alpha+2s}_{k-1,\alpha}(\Omega | B_{1/2})} \le C \left(  \Vert u \Vert_{L^{\infty}(\R^n)} + \Vert f \Vert_{C^{k+\alpha+1}(\overline{\Omega} \cap B_1)} \right).
\end{align*}
The constant $C$ depends only on $n,s,\lambda,\Lambda,k,\alpha$, and the $C^{k,\alpha}$ radius of $\Omega$.
\end{lemma}

\begin{proof}
After a normalization we can assume that 
\begin{align*}
\Vert u \Vert_{L^{\infty}(\R^n)} + \Vert f \Vert_{C^{k-1+\alpha-s}(\Omega \cap B_1)} = 1.
\end{align*}
Let us start by proving the first claim.
Note that $u/d_{\Omega}^s \in C^{k-1,\alpha}(\overline{\Omega} \cap B_{1/2})$ together with the corresponding estimate follows from \cite[Theorem 1.4]{AbRo20} in case $k > 1$ since $\alpha+1\notin\mathbb{N}$, $\alpha+1 \mp s\notin\mathbb{N}$, and from \cite{RoSe17,FeRo24} in case $k = 1$. A crucial step in these articles is the following estimate:  for any $z \in \partial \Omega \cap B_{1/2}$ there exists a polynomial $Q_z \in \cP_{k-1}$ such that for any $x \in \Omega \cap B_{1/2}(z)$ it holds
\begin{align*}
|u(x) - Q_z(x) d_{\Omega}^s(x)| \le C |x-z|^{k-1+\alpha+s} \le C |x-z|^{k+\alpha+2s} d_{\Omega}^{-1-s}(x).
\end{align*}
In particular, we get that for any $x_0 \in \Omega \cap B_{1/2}(z)$, denoting $r = d_{\Omega}(x_0)$,
\begin{align}
\label{eq:aux-reg-est-1}
\Vert u - Q_z d^s_{\Omega} \Vert_{L^{\infty}(B_{r/2}(x_0))} \le c r^{k+\alpha + s - 1}, \qquad [ u - Q_z d_{\Omega}^s]_{C^{k+\alpha+2s}(B_{r/2}(x_0))} \le c r^{-s-1}.
\end{align}
Indeed, while the first estimate is immediate from the expansion, the second result follows by denoting
\begin{align*}
v_r(x) = r^{-k-\alpha-2s} \big( u(x_0 + rx) - Q_z(x_0 + rx) d^s_{\Omega}(x_0 + rx) \big),
\end{align*}
which is a solution to
\begin{align*}
L v_r = \widetilde{f} := r^{-k-\alpha} \big( f(x_0 + r \cdot) - L (Q_z d^s_{\Omega}(x_0 + r \cdot) \big),
\end{align*}
and observing that by the first estimate, and the assumptions on $f$ we have
\begin{align*}
\Vert v_r \Vert_{L^{\infty}(B_R)} &\le C(1 + r^{-1-s}) ~~ \forall R > 0, \\
[\widetilde{f}]_{C^{k+\alpha}(B_1)} &\le [f]_{C^{k+\alpha}(\overline{\Omega} \cap B_1)} + [L (Q_z d_{\Omega}^s)]_{C^{k+\alpha}(\overline{\Omega} \cap B_1)} \le C r^{-1-s},
\end{align*}
where we used \cite[Corollary 2.3]{AbRo20} (or rather \cite[Corollary 3.9]{Kuk21}) to deduce the estimate for $L (Q_z d_{\Omega}^s)$ (see also \cite[Proof of Lemma 2.4]{RoWe24a}). By applying interior Schauder estimates of order $C^{k+\alpha+2s}$ to $v_r$, and using the two estimates from above, we deduce the second estimate in \eqref{eq:aux-reg-est-1}, exactly as in the proof of \cite[Theorem 1.4]{AbRo20}.

Hence, by H\"older interpolation we deduce that for any $\delta \in (0,k+\alpha+2s]$
\begin{align}
\label{eq:expansion-Holder-interpol}
[ u - Q_z d_{\Omega}^s]_{C^{\delta}(B_{r/2}(x_0))} \le c r^{k +\alpha + s-1 - \delta}.
\end{align}

To establish the first claim, we now follow the proof of \cite[Proof of Theorem 1.4]{AbRo20}. Indeed, given now any $j \in \N \cup \{ 0 \}$ and $\beta \in (0,1]$ with $k-1+\alpha \le j + \beta \le k+\alpha+2s$, we deduce from the fact that for any multi-index $\xi \in (\N \cup \{0\})^n$
\begin{align*}
\Vert \partial^{\xi} d_{\Omega}^{-s}\Vert_{L^{\infty}(B_{r/2}(x_0))} \le C r^{-s - |\xi|} , \quad [ \partial^{\xi} d_{\Omega}^{-s}]_{C^{\beta}(B_{r/2}(x_0))} \le C r^{-s - |\xi| - \beta}, \quad \Vert d^{s}_{\Omega} \Vert_{L^{\infty}(B_{r/2}(x_0))} \le C r^s,
\end{align*}
which follows from \cite[Lemma A.2]{AbRo20}, by \eqref{eq:expansion-Holder-interpol} and by the product rule:
\begin{align*}
[ud^{-s}_{\Omega}]_{C^{j+\beta}(B_{r/2}(x_0))} &= [ud^{-s}_{\Omega} - Q_z]_{C^{j+\beta}(B_{r/2}(x_0))} \\
&= \big[D^j \big( d^{-s}_{\Omega} (u - Q_z d^s_{\Omega}) \big) \big]_{C^{\beta}(B_{r/2}(x_0))} \\
&\le \sum_{|\gamma| = j} \sum_{|\xi| \le \gamma} \left[ (\partial^{\xi} d_{\Omega}^{-s}) \big( \partial^{\gamma - |\xi|} (u - Q_z d^s_{\Omega}) \big) \right]_{C^{\beta}(B_{r/2}(x_0))} \\
&\le C\sum_{|\gamma| = j} \sum_{|\xi| \le \gamma}\Bigg(  \Vert \partial^{\xi} d^{-s}_{\Omega}\Vert_{L^{\infty}(B_{r/2}(x_0))} \left[ \partial^{\gamma - |\xi|} (u - Q_z d^s_{\Omega}) \right]_{C^{\beta}(B_{r/2}(x_0))} \\
&\qquad\qquad\qquad + [ \partial^{\xi} d^{-s}_{\Omega}]_{C^{\beta}(B_{r/2}(x_0))} \left\Vert \partial^{\gamma - |\xi|} (u - Q_z d^s_{\Omega}) \right\Vert_{L^{\infty}(B_{r/2}(x_0))} \bigg) \\
&\le C \sum_{|\gamma| = j} \sum_{|\xi| \le \gamma} r^{-s-|\xi|} r^{k+\alpha+s-1 - |\gamma| + |\xi| - \beta} + r^{-s-|\xi|-\beta} r^{k+\alpha+s-1-|\gamma|+|\xi|} \\
&\le C r^{k+\alpha-1 - (j + \beta)},
\end{align*}
where we also used that $Q_z \in \cP_{k-1}$. Since by the first estimate in \eqref{eq:aux-reg-est-1} we have in particular
\begin{align}
\label{eq:aux-reg-est-2}
\begin{split}
\Vert ud^{-s}_{\Omega} \Vert_{L^{\infty}(B_{r/2}(x_0))} &\le r^{-s} \Vert u \Vert_{L^{\infty}(B_{r/2}(x_0))} \\
&\le r^{-s} \Vert u - Q_z d^{s}_{\Omega} \Vert_{L^{\infty}(B_{r/2}(x_0))} + C \Vert Q_z \Vert_{L^{\infty}(B_{r/2}(x_0))} \le C,
\end{split}
\end{align}
we can immediately deduce the first claim by H\"older interpolation.

The proof of the second claim follows by a very similar argument, resembling the one in the proof of \cite[Theorem 1.4]{RoWe24a}. 
First, note that by the slightly stronger assumptions on $f$, we now have that \eqref{eq:expansion-Holder-interpol} holds true for any $k-1+\alpha \le j+\beta \le k + \alpha + 1 + 2s$. Moreover, since for any multi-index $\xi \in (\N \cup \{0\})^n$,
\begin{align*}
\Vert \partial^{\xi} d^{1-s}_{\Omega}\Vert_{L^{\infty}(B_{r/2}(x_0))} \le C r^{1-s - |\xi|} , \quad [ \partial^{\xi} d^{1-s}_{\Omega}]_{C^{\beta}(B_{r/2}(x_0))} \le C r^{1-s - |\xi| - \beta}, \quad \Vert d^{1-s}_{\Omega} \Vert_{L^{\infty}(B_{r/2}(x_0))} \le C r^{1-s},
\end{align*}
which follows by the properties of $d_{\Omega}$ from \cite[Lemma A.2]{AbRo20} and H\"older interpolation, we deduce by the product rule and \eqref{eq:expansion-Holder-interpol} that for any $k-1+\alpha \le j+\beta \le k + \alpha + 1 + 2s$:
\begin{align*}
[ud^{1-s}_{\Omega} - Q_z d_{\Omega}]_{C^{j+\beta}(B_{r/2}(x_0))}  &= \big[D^j \big( d^{1-s}_{\Omega} (u - Q_z d^s_{\Omega}) \big) \big]_{C^{\beta}(B_{r/2}(x_0))} \\
&\le \sum_{|\gamma| = j} \sum_{|\xi| \le \gamma} \left[ (\partial^{\xi} d^{1-s}_{\Omega}) \big( \partial^{\gamma - \xi} (u - Q_z d^s_{\Omega}) \big) \right]_{C^{\beta}(B_{r/2}(x_0))} \le C r^{k+\alpha - (j + \beta)}.
\end{align*}
Since $Q_z \in \cP_{k-1}$ we have by the product rule and \cite[Lemma A.2]{AbRo20},
\begin{align*}
[Q_z d_{\Omega}]_{C^{j+\beta}(B_{r/2}(x_0))} = \Vert Q_z \Vert_{L^{\infty}(B_{r/2}(x_0))} [d_{\Omega}]_{C^{j+\beta}(B_{r/2}(x_0))} \le C r^{k+\alpha - (j+\beta)},
\end{align*}
and therefore for any $k-1+\alpha \le j+\beta \le k + \alpha + 2s$
\begin{align*}
[\nabla(ud^{1-s}_{\Omega})]_{C^{j+\beta}(B_{r/2}(x_0))} = [ud_{\Omega}^{1-s}]_{C^{j+\beta+1}(B_{r/2}(x_0))} \le C r^{k+\alpha - 1 - (j + \beta)}.
\end{align*}
Since \eqref{eq:aux-reg-est-2} and the first claim in particular imply (using H\"older interpolation again),
\begin{align*}
\Vert \nabla(ud^{1-s}_{\Omega}) \Vert_{L^{\infty}(B_{r/2}(x_0))} &= [u d^{1-s}_{\Omega}]_{C^{0,1}(B_{r/2}(x_0))} \\
&\le \Vert d_{\Omega} \Vert_{L^{\infty}(B_{r/2}(x_0))} [ud^{-s}_{\Omega}]_{C^{0,1}(B_{r/2}(x_0))} \\
&\quad + \Vert ud^{-s}_{\Omega} \Vert_{L^{\infty}(B_{r/2}(x_0))} [d_{\Omega}]_{C^{0,1}(B_{r/2}(x_0))} \le C ,
\end{align*}
we can also deduce the second claim immediately by H\"older interpolation. The proof is complete.
\end{proof}

Moreover, we need the following generalization of the Hopf lemma.

\begin{lemma}
\label{lemma:Hopf}
Let $\Omega \subset \R^n$ be a bounded domain with $\partial \Omega \in C^{1,\alpha}$ for some $\alpha \in (0,1)$ with $\alpha \not\in\{s,1-s\}$. Let $K$ be as in \eqref{eq:K-comp} with $K \in C^{3 + 2\alpha}(\mathbb{S}^{n-1})$, and $u \in L^{\infty}(\R^n)$ be a solution to
\begin{align*}
\begin{cases}
L u &= f \ge 0 ~~ \text{ in } \Omega \cap B_1,\\
u &\ge 0 ~~ \text{ in }  \R^n \setminus (\Omega \cap B_1).
\end{cases}
\end{align*}
If $f \in L^{\infty}(\Omega)$ and $u \not= 0$, then,
\begin{align*}
u \ge c_0 d^{s}_{\Omega} ~~ \text{ in } \Omega \cap B_{1/2}.
\end{align*}
Moreover, if $\partial_n (d_{\Omega}) \ge \delta$ in $\Omega \cap B_1$ for some $\delta > 0$, $K\in C^{3+2\alpha}(\mathbb{S}^{n-1})$, and $f \in C^{1+\alpha}(\overline{\Omega} \cap B_1)$ , then there exists $\eps > 0$ such that
\begin{align*}
\partial_n u \ge c_1 d^{s-1}_{\Omega} ~~ \text{ in } \Omega \cap B_{\eps}.
\end{align*}
Finally, if in addition we have $\partial_n u \ge \delta$ in $\Omega \cap B_1$, then it holds
\begin{align*}
\partial_n u \ge c_2 d^{s-1}_{\Omega} ~~ \text{ in } \Omega \cap B_{1/2}.
\end{align*}
\end{lemma}

\begin{proof}
The first claim is an immediate consequence of the nonlocal Hopf lemma (see \cite[Proposition 2.6.6]{FeRo24}). To prove the second claim, we observe that $u \in C^{1+\alpha+2s}_{loc}(\Omega)$ by interior estimates, and $u/d_{\Omega}^s \in C^{1+\alpha+2s}_{0,\alpha}(\Omega | B_{1/2})$ by \autoref{lemma:boundary-reg}. Therefore we can write
\begin{align*}
\frac{\partial_n u}{d^{s-1}_{\Omega}} = \partial_n \left( \frac{u}{d_{\Omega}^s} d_{\Omega}^s \right) d_{\Omega}^{1-s} = \partial_n \left( \frac{u}{d_{\Omega}^s} \right) d_{\Omega} + s \left(\frac{u}{d_{\Omega}^{s}}\right) \partial_n (d_{\Omega}) ~~ \text{ in } \Omega \cap B_{1/2}.
\end{align*}
Note that since $\partial_n (d_{\Omega}) \ge \delta$ in $\Omega \cap B_{1}$ by assumption, we can estimate for any $\alpha' < \alpha$, using also the first claim:
\begin{align*}
\frac{\partial_n u}{d^{s-1}_{\Omega}}  \ge - \left| \partial_n \left( \frac{u}{d_{\Omega}^s} \right) \right| d_{\Omega} + s \left( \frac{u}{d_{\Omega}^{s}} \right) \partial_n (d_{\Omega}) \ge - C d_{\Omega}^{\alpha'} + \delta s c_0 ~~ \text{ in  } \Omega \cap B_{1/2}.
\end{align*}
Hence, we can find $\eps \in (0,\delta)$ small enough such that
\begin{align*}
\frac{\partial_n u}{d_{\Omega}^{s-1}} \ge \frac{\delta s c_0}{2} ~~ \text{ in } \Omega \cap B_{1/2} \cap \{d_{\Omega} \le \eps \},
\end{align*}
which proves the second claim with $c_1 = \frac{\delta s c_0}{2}$. \\
To prove the third claim, we observe that since now it also holds $\partial_n u \ge \delta$ in $\Omega \cap B_{1/2}$ by assumption, we immediately have
\begin{align*}
\frac{\partial_n u}{d_{\Omega}^{s-1}} \ge \delta d_{\Omega}^{1-s} \ge \delta \eps^{1-s} ~~ \text{ in } \Omega \cap B_{1/2} \cap \{ d_{\Omega} \ge \eps \}.
\end{align*}
Altogether, we conclude the proof of the third claim by setting $c_2 = \min \{ c_1 , \delta \eps^{s-1} \}$.
\end{proof}

\subsection{Previous results on the nonlocal one-phase problem}

Given an open, bounded domain $\Omega \subset \R^n$, we introduce the following function spaces,
\begin{align*}
V^s(\Omega | \Omega') &:= \left \{ u \hspace{-0.1cm}\mid_{\Omega} \hspace{0.1cm} \in L^2(\Omega) : [u]^2_{V^s(\Omega | \Omega')} := \int_{\Omega} \int_{\Omega'} \frac{(u(x) - u(y))^2}{|x-y|^{n+2s}} \d y \d x < \infty \right \}, ~~ \Omega \Subset \Omega',\\
L^1_{2s}(\R^n) &:= \left\{ u : \R^n \to \R : \Vert u \Vert_{L^1_{2s}(\R^n)} := \int_{\R^n} |u(y)| (1 + |y|)^{-n-2s} \d y < \infty \right\}.
\end{align*}
The space $V^s(\Omega | \Omega')$ is equipped with the following norm:
\begin{align*}
\Vert u \Vert_{V^s(\Omega | \Omega')} := \Vert u \Vert_{L^2(\Omega)} + [u]_{V^s(\Omega | \Omega')}.
\end{align*}

Moreover, given $K$ satisfying \eqref{eq:K-comp} and $\Omega \subset \R^n$, we denote
\begin{align}
\label{eq:one-phase}
\cI_{\Omega}(u) := \cE_{(\Omega^c \times \Omega^c)^c}(u,u) + |\{ u > 0 \} \cap \Omega|,
\end{align}
whenever this expression is finite.

We recall the following definition of minimizers to the nonlocal one-phase problem \eqref{eq:one-phase} from \cite{RoWe24b}.

\begin{definition}[minimizers]
\label{def:minimizer}
Let $K$ satisfy \eqref{eq:K-comp}. Let $\Omega \Subset \Omega' \subset \R^n$ be an open, bounded domain. We say that $u \in V^s(\Omega | \Omega') \cap L^1_{2s}(\R^n)$ with $u \ge 0$ in $\R^n$ is a (local) minimizer of $\cI_{\Omega}$ (in $\Omega$) if for any $v \in V^s(\Omega | \Omega') \cap L^1_{2s}(\R^n)$ with $u = v$ in $\R^n \setminus \Omega$, it holds
\begin{align*}
\iint\limits_{(\Omega^c \times \Omega^c)^c} \Big[(u(x) - u(y))^2 - (v(x) - v(y))^2 \Big] K(x-y) \d y \d x + \Big[ | \{u > 0 \} \cap \Omega | - | \{v > 0 \} \cap \Omega | \Big] \le 0.
\end{align*}
\end{definition}

Let us summarize some of the basic results from \cite{RoWe24a,RoWe24b}, which will be needed later.

\begin{lemma}
\label{lemma:one-phase-collection}
Let $K$ satisfy \eqref{eq:K-comp} and $u$ be a minimizer of $\cI_{\Omega}$. Then, the following hold true:
\begin{itemize}
\item[(i)] $u \ge 0$ in $\Omega$,
\item[(ii)] $Lu = 0$ in $\Omega \cap \{ u > 0 \}$ in the weak sense.
\item[(iii)] $u \in C^s_{loc}(\Omega)$ and if $B_2 \subset \Omega$,
\begin{align*}
\Vert u \Vert_{C^s(B_R)} \le C R^{-s} (1 + R^{-n} \Vert u \Vert_{L^1(B_{2R})}) ~~ \forall R \in (0,1], 
\end{align*}
\item[(iv)] If $B_2 \subset \Omega$, then for some $C = C(n,s,\lambda,\Lambda) > 0$,
\begin{align*}
u(x) \ge c \dist(x , \partial\{ u > 0 \})^s ~~ \forall x \in B_1,
\end{align*} 
\item[(v)] If $\partial \{ u > 0 \}$ is in $C^{1,\alpha}$ in $B_1$ for some $\alpha \in (0,1)$, then $c = c(n,s,\lambda,\Lambda) > 0$,
\begin{align*}
\frac{u(x)}{\dist^s(x, \partial\{ u > 0 \})} = A(\nu_x) ~~ \forall x \in \partial \{ u > 0 \} \cap B_1,
\end{align*}
where for $\nu \in \mathbb{S}^{n-1}$, the value $A(\nu)$ is given by
\begin{align}
\label{eq:A-nu-def}
A(\nu):=c_{n,s}\left(\int_{\mathbb{S}^{n-1}}K(\theta)|\theta\cdot\nu|^{2s}\, d\theta\right)^{-1/2},\qquad \text{ where } c_{n,s}>0.
\end{align}
\end{itemize}
\end{lemma}

\begin{proof}
Properties (i) and (ii) where established in \cite[Lemma 4.1]{RoWe24a}. Properties (iii) and (iv) are shown in \cite[Lemma 2.3 and Lemma 2.4]{RoWe24b}, while property (v) follows from \cite[Proposition 3.1]{RoWe24b}.
\end{proof}

\subsection{Previous results on the nonlocal obstacle problem}

The nonlocal obstacle problem is given as follows
\begin{align*}
\min\{ Lu , u - \varphi \} = 0 ~~ \text{ in } \Omega \subset \R^n,
\end{align*}
where $\Omega \subset \R^n$, and $\varphi : \Omega \to \R$ is the obstacle. We will assume throughout this paper that $\varphi \in C^{1+2s + \eps}(\Omega)$ for some $\eps > 0$, since then, we can consider $v := u - \varphi$, which solves
\begin{align}
\label{eq:obstacle}
\begin{cases}
L v &= f ~~ \text{ in } \Omega \cap \{ v > 0 \},\\
v &= 0 ~~ \text{ in } \Omega \setminus \{ v  > 0 \},
\end{cases}
\end{align}
where $f = - L \varphi \in C^{1+\eps}(\Omega)$.

We recall the following results from \cite{CRS17,RTW25}:

\begin{lemma}
\label{lemma:obstacle-collection}
Let $v$ be a solution to the obstacle problem \eqref{eq:obstacle} with $\varphi \in C^{1+2s+\eps}(\Omega)$ for some $\eps > 0$ and $K \in C^1(\mathbb{S}^{n-1})$. Then, the following hold true if $B_2 \subset \Omega$:
\begin{itemize}
\item $v \in C^{1+s}_{loc}(\Omega)$, and
\begin{align*}
\Vert v \Vert_{C^{1+s}(B_1)} \le C \left( \Vert \varphi \Vert_{C^{1+2s+\eps}(B_2)} + \Vert v \Vert_{L^{\infty}(\R^n)} \right).
\end{align*}
\item For any free boundary point $x_0 \in \partial \{ v> 0 \} \cap B_1$ there is $c_0 \ge 0$ and $e \in \mathbb{S}^{n-1}$ such that for any $x \in B_1(x_0)$:
\begin{align*}
|v(x) - c_0 ((x-x_0) \cdot e)_+^{1+s}| \le C \left( \Vert \varphi \Vert_{C^{1+2s+\eps}(B_2)} + \Vert v \Vert_{L^{\infty}(\R^n)} \right) |x-x_0|^{1+2s+\alpha}
\end{align*}
for any $\alpha \in (0,s) \cap (0,1-s)$. Moreover, if $c_0 > 0$, then $x_0$ is called a regular free boundary point and for any $\gamma \in (0,s)$ the free boundary is $C^{1,\gamma}$ in $B_{\rho}(x_0)$ for some $\rho > 0$.
\end{itemize}
The constant $C> 0$ depends only on $n,s,\lambda,\Lambda,\alpha,\eps$.
\end{lemma}

\begin{proof}
The proofs follow directly from the main results in \cite{CRS17,RTW25}. Note that the $C^{1,\alpha}$ regularity of the free boundary that is proved in \cite{RTW25} can be upgraded to $C^{1,\gamma}$ regularity for any $\gamma \in (0,s)$ by applying the boundary Harnack inequality from \cite{RoSe17} in the same way as in \cite[Theorem 1.1]{CRS17}.
\end{proof}

\section{Local boundary conditions and integration by parts formulas}

The goal of this section is to establish nonlocal integration by parts formulas for energies with weights of the form $d_{\Omega}^{\beta-1}(x)d_{\Omega}^{\beta-1}(y)$ that either degenerate or explode at $\partial\Omega$, depending on $\beta \in [s,1+s]$. These formulas will be crucial to our analysis since they allow us to analyze the equations that are satisfied for $w = {\partial_i u}/{\partial_n u}$, where $u$ is a solution to the nonlocal one-phase problem ($\beta = s$) or the nonlocal obstacle problem ($\beta = 1+s$). 

The case $\beta = s$ is arguably the most complicated one, since it produces integrals supported on $\partial \Omega$ which involve classical derivatives of first order.

Note that integration by parts formulas for similar types of nonlocal energies with weights have already been established in \cite{Aba15,Gru20}, however their formulas do not involve nonlocal energies. Moreover, we refer to \cite{ScDu24,ScDu25}, where related formulas for nonlocal energies with additive instead of multiplicative weights are considered.

\subsection{Integration by parts formula in the half-space}

Let $K$ be a translation invariant kernel satisfying \eqref{eq:K-comp}. We define
\begin{align}
\label{eq:theta-K-def}
\theta_K := 2c_s \int_{\R^{n-1}} (h',1) K((h',1)) \d h' \in \R^n, \qquad c_s :=  \int_0^{1} r^{s-1} \int_{1}^{\infty} t^{s-1} |r-t|^{-2s} \d t \d r < \infty.
\end{align}
where, $c_s < \infty$ by \autoref{lemma:int-finite}. We have the following integration by parts formula in the half-space

\begin{lemma}
\label{lemma:ibp}
Let $K$ satisfy \eqref{eq:K-comp} and assume that $K \in C^{5+2\alpha}(\mathbb{S}^{n-1})$ for some $\alpha \in (0,s)$. Let $v \in C^{\gamma}(\R^n)$ for some $\gamma \in (\max\{0,2s-1\}, s+1)$ be such that $|v(x)| \le C (1 + |x|^{\gamma})$. Then, for every $\eta \in C^{\infty}_c(\R^n)$,
\begin{align}
\label{eq:ibp-energy}
\begin{split}
\int_{\R^n} \int_{\R^n} & (v(x) - v(y)) (\eta(x) - \eta(y)) (x_n)_+^{s-1} (y_n)_+^{s-1} K(x-y) \d y \d x \\
&= \int_{\{ x_n > 0 \}} v (x_n)_+^{s-1} L((x_n)_+^{s-1}\eta) \d x -  \int_{\{ x_n = 0 \}}{(\theta_K \cdot \nabla \eta)} v \d x'.
\end{split}
\end{align}
\end{lemma}

\begin{remark}
Note that in case $K(h) = |h|^{-n-2s}$, i.e. when $K$ is the kernel of the fractional Laplacian, we have a cancellation in the integral $\int_{\R^{n-1}} h'K((h',1)) \d h' = 0$ and therefore
\begin{align*}
\theta = C_s e_n, ~~ \text{ for some } C_s > 0.
\end{align*}
Hence, the boundary integral does not contain an oblique derivative the normal derivative $\partial_n v$.
\end{remark}

\begin{proof}[Proof of \autoref{lemma:ibp}]
Let $\eta \in C^{\infty}_c(\R^n)$ with $\supp(\eta) \subset B$ for some ball $B \subset \R^n$. First, we claim that
\begin{align*}
\int_{\{ x_n < \eps \} \cap B} \int_{\R^n} (v(x) - v(y)) (\eta(x) - \eta(y)) (x_n)_+^{s-1} (y_n)_+^{s-1} K(x-y) \d y \d x \to 0
\end{align*}
as $\eps \to 0$.
This follows, since $v \in C^{\gamma}(\R^n)$ and hence by \autoref{lemma:bd-int} and since $\gamma > 2s - 1$
\begin{align*}
\int_{\{ x_n < \eps \} \cap B} & \int_{2B} (v(x) - v(y)) (\eta(x) - \eta(y)) (x_n)_+^{s-1} (y_n)_+^{s-1} K(x-y) \d y \d x \\
&\le C [v]_{C^{\gamma}(\R^n)} \int_{\{ x_n < \eps \} \cap B} (x_n)_+^{s-1} \left( \int_{2B} (y_n)_+^{s-1} |x-y|^{-n-2s + 1 + \gamma} \d y \right) \d x \\
&\le C \int_{\{ x_n < \eps \} \cap B} (x_n)_+^{s-1} + (x_n)_+^{\gamma - 1} \d x \le C \eps^{s} \to 0,
\end{align*}
where $C$ depends on $\Vert \eta \Vert_{C^{0,1}(\R^n)}$ and $[v]_{C^{\gamma}(\R^n)}$.
Moreover, since for $x \in B$ and $y \in (2B)^c$ it holds $|x-y| \ge C |y|$, we deduce by the growth assumption on $v$ and by \autoref{lemma:int-polar-coord},
\begin{align*}
\int_{\{ x_n < \eps \} \cap B} & \int_{(2B)^c} (v(x) - v(y)) (\eta(x) - \eta(y)) (x_n)_+^{s-1} (y_n)_+^{s-1} K(x-y) \d y \d x \\
&\le C \int_{\{ x_n < \eps \} \cap B} (x_n)_+^{s-1} \left( \int_{(2B)^c} (y_n)_+^{s-1} (|x|^{\gamma} + |y|^{\gamma}) |y|^{-n-2s} \d y \right) \d x \\
&\le C \int_{\{ x_n < \eps \} \cap B} (x_n)_+^{s-1} \d x \le C \eps^{s} \to 0,
\end{align*}
where $C$ depends on $\Vert \eta \Vert_{C^{0,1}(\R^n)}$ and $[v]_{C^{\gamma}(\R^n)}$.
Hence, it remains to prove that as $\eps \searrow 0$
\begin{align*}
\iint_{(\{ x_n < \eps \} \times \{ y_n < \eps \})^c} & (v(x) - v(y)) (\eta(x) - \eta(y)) (x_n)_+^{s-1} (y_n)_+^{s-1} K(x-y) \d y \d x \\
&\to \int_{\{ x_n > 0 \}} (x_n)_+^{s-1} v L((x_n)_+^{s-1} \eta) \d x - \int_{\{ x_n = 0 \}} (\theta_K \cdot \nabla \eta) v \d x'.
\end{align*}

By the algebraic identity \eqref{eq:algebra-quotient-PDE} and the symmetry of $K$, we get
\begin{align*}
& \iint_{(\{ x_n < \eps \} \times \{ y_n < \eps \})^c} (v(x) - v(y)) (\eta(x) - \eta(y)) (x_n)_+^{s-1} (y_n)_+^{s-1} K(x-y) \d y \d x \\
&\quad = \iint_{(\{ x_n < \eps \} \times \{ y_n < \eps \})^c} \Big[ (v(x)(x_n)_+^{s-1} - v(y)(y_n)_+^{s-1}) (\eta(x)(x_n)_+^{s-1} - \eta(y)(y_n)_+^{s-1}) \\
&\quad \quad - ((x_n)_+^{s-1} - (y_n)_+^{s-1}) (v(x)\eta(x)(x_n)_+^{s-1} - v(y)\eta(y)(y_n)_+^{s-1}) \Big] K(x-y) \d y \d x \\
&\quad = 2 \iint_{(\{ x_n < \eps \} \times \{ y_n < \eps \})^c} \Big[ v(x)(x_n)_+^{s-1} (\eta(x)(x_n)_+^{s-1} - \eta(y)(y_n)_+^{s-1}) \\
&\quad \quad - ((x_n)_+^{s-1} - (y_n)_+^{s-1}) v(x)\eta(x)(x_n)_+^{s-1} \Big] K(x-y) \d y \d x \\
&\quad =  \int_{ \{ x_n > \eps \} } v(x)(x_n)_+^{s-1} L[\eta(x_n)_+^{s-1}](x) \d x \\
&\quad \quad -  \int_{ \{ x_n > \eps \} } \eta(x) v(x)(x_n)_+^{s-1} L[(x_n)_+^{s-1}](x) \d x \\
&\quad \quad + 2\int_{\{ x_n < \eps \}}  v(x) (x_n)_+^{s-1} \left(\int_{\{ y_n > \eps \}} (\eta(x) - \eta(y)) (y_n)_+^{s-1} K(x-y) \d y \right) \d x \\
&=: I_{\eps}^{(1)} + I_{\eps}^{(2)} + I_{\eps}^{(3)}.
\end{align*}
We treat the terms $I_{\eps}^{(1)},I_{\eps}^{(2)},I_{\eps}^{(3)}$ separately. For $I_1^{(\eps)}$, we observe that
\begin{align*}
|L[\eta(x_n)_+^{s-1}]| \le 
\begin{cases} 
C (x_n)_+^{\alpha-s},\quad &x \in 2B,\\
C |x|^{-n-2s},\quad &x \in (2B)^c.
\end{cases}
\end{align*}
The first estimate was proved in \cite[Corollary 2.5]{RoWe24a}, using that $K \in C^{5+2\alpha}(\mathbb{S}^{n-1})$, and the second one follows from a straightforward computation.
Therefore, we get by the dominated convergence theorem, using also that $\gamma < 1+s$,
\begin{align*}
I_{\eps}^{(1)} \to  \int_{\{ x_n > 0 \}} v(x)(x_n)_+^{s-1} L( \eta (x_n)_+^{s-1}) \d x.
\end{align*}
For $I_{\eps}^{(2)}$, we recall that $L(x_n)_+^{s-1} = 0$ in $\{ x_n > 0 \}$ and therefore, it holds $I_{\eps}^{(2)} = 0$.
To estimate $I_{\eps}^{(3)}$ we split the integral as follows.
\begin{align*}
I_{\eps}^{(3)} &= 2\int_{\{ x_n < \eps \} \cap 2B}  v(x) (x_n)_+^{s-1} \nabla \eta(x)\left(\int_{\{ y_n > \eps \}}  (x-y) (y_n)_+^{s-1} K(x-y) \d y \right) \d x + J_{\eps}.
\end{align*}
Since $\eta \in C^{1,\delta}_c(\R^n)$ with $\delta \in (0,s)$, and $|x-y| \ge C |x|$ for $x \in (2B)^c$, $y \in B$, by \autoref{lemma:n-1-dim-integral} we get
\begin{align*}
|J_{\eps}| &\le C \int_{\{ x_n < \eps \} \cap 2B} (x_n)_+^{s-1} \left(\int_{\{ y_n > \eps \}} (y_n)_+^{s-1} |x-y|^{-n-2s+1+\delta} \d y \right) \d x \\
&\quad + C \int_{\{ x_n < \eps \} \setminus 2B} (x_n)_+^{s-1} |x|^{\gamma} \left(\int_{\{ y_n > \eps \} \cap B} (y_n)_+^{s-1} |x-y|^{-n-2s} \d y \right) \d x\\
&\le C \int_{0}^{\eps} (x_n)_+^{s-1} \int_{\eps}^{\infty} (y_n)_+^{s-1}  |x_n - y_n|^{-2s+\delta} \d y_n \d x_n + C \int_{\{ x_n < \eps \} \setminus 2B} (x_n)_+^{s-1} |x|^{-n-2s+\gamma} \d x\\
&\le C\eps^{\delta} \int_0^1 r^{s-1} \int_1^{\infty}  t^{s-1} |t-r|^{-2s+\delta} \d t \d r + C \int_{\{ x_n < \eps \} \setminus 2B} (x_n)_+^{s-1} |x|^{-n-2s+\gamma} \d x.
\end{align*}
By \autoref{lemma:int-finite}, \autoref{lemma:int-polar-coord} and the dominated convergence theorem, we obtain 
\begin{align*}
|J_{\eps}| \to 0 ~~ \text{ as } \eps \searrow 0.
\end{align*}
To estimate the first summand of $I_{\eps}^{(3)}$, we observe that by a modification of the computation in \cite[Lemma 5.4]{RoWe24c}  for any $h_n := x_n - y_n < 0$, by \eqref{eq:K-comp} it holds
\begin{align}
\label{eq:integrate-out-kernel}
\begin{split}
\int_{\R^{n-1}} & (h',h_n) K(h',h_n) \d h' = \int_{\R^{n-1}} (h',h_n) \big( |h'|^2 + |h_n|^2 \big)^{-\frac{n+2s}{2}} K \left( \frac{(h',h_n)}{|(h',h_n)|} \right) \d h' \\
&= |h_n|^{-n-2s+1} \int_{\R^{n-1}} \left(\frac{h'}{|h_n|} , -1 \right) \left( \frac{|h'|^2}{|h_n|^2} + 1 \right)^{-\frac{n+2s}{2}} K \left( \frac{(h'/|h_n|,-1)}{[|h'|^2/|h_n|^2 + 1]^{1/2}} \right) \d h' \\
&= |h_n|^{-2s} \int_{\R^{n-1}} (h',-1) (|h'|^2 + 1)^{-\frac{n+2s}{2}} K\left( \frac{(h',-1)}{|(h',-1)|} \right) \d h' \\
&= |h_n|^{-2s} \int_{\R^{n-1}} (h',-1) K((h',-1)) \d h' = - \frac{c_s^{-1}}{2} \theta_K |h_n|^{-2s},
\end{split}
\end{align}
where $\theta_K \in \R^n$ is defined as in \eqref{eq:theta-K-def}. 
We deduce
\begin{align}
\label{eq:extra}
\begin{split}
2 &\int_{\{ x_n < \eps \} \cap 2B} v(x) (x_n)_+^{s-1}  \nabla \eta(x)\left(\int_{\{ y_n > \eps \}}  (x-y) (y_n)_+^{s-1} K(x-y) \d y \right) \d x \\
&= -c_s^{-1} \int_{\{ x_n < \eps \} \cap 2B} [v(x) (\theta_K \cdot \nabla \eta)(x) - v(x',0) (\theta_K \cdot \nabla \eta)(x',0)](x_n)_+^{s-1}  \int_{\eps}^{\infty} (y_n)_+^{s-1} |x_n-y_n|^{-2s} \d y_n \d x \\
&\quad -c_s^{-1} \int_{\{ x_n < \eps \} \cap 2B} v(x',0) (\theta_K \cdot \nabla \eta)(x',0)(x_n)_+^{s-1} \int_{\eps}^{\infty} (y_n)_+^{s-1} |x_n-y_n|^{-2s} \d y_n \d x .
\end{split}
\end{align}
We notice that, since $v \nabla \eta \in C^{\gamma}(\R^n)$ and $0<\gamma\leq 1+s$, by \autoref{lemma:int-finite}
\begin{align}
\label{eq:extra1}
\begin{split}
\Big| \int_{\{ x_n < \eps \} } & [v(x) (\theta_K \cdot \nabla \eta)(x) - v(x',0) (\theta_K \cdot \nabla \eta) (x',0)](x_n)_+^{s-1} \int_{\eps}^{\infty} (y_n)_+^{s-1} |x_n-y_n|^{-2s} \d y_n \d x \Big| \\
&\le C\int_0^{\eps} r^{s-1+\gamma} \int_{\eps}^{\infty} t^{s-1} (t-r)^{-2s} \d t \d r \le C \eps^{\gamma} \to 0.
\end{split}
\end{align}
Thus by using
\begin{align}
\label{eq:theta-constant}
\lim_{\eps \to 0} \int_0^{\eps} r^{s-1} \int_{\eps}^{\infty} t^{s-1} (t - r)^{-2s} \d t \d r = \int_0^1 \int_1^{\infty} r^{s-1} t^{s-1} (t - r)^{-2s} \d t \d r = c_s > 0,
\end{align}
from \eqref{eq:extra} and \eqref{eq:extra1}, we get 
\begin{align*}
2 &\int_{\{ x_n < \eps \} \cap 2B} v(x) (x_n)_+^{s-1}  \nabla \eta(x)\left(\int_{\{ y_n > \eps \}}  (x-y) (y_n)_+^{s-1} K(x-y) \d y \right) \d x\\
&\to - \int_{\R^{n-1}} v(x',0) (\theta_K \cdot \nabla \eta)(x',0) \d x',
\end{align*}
as desired.
\end{proof}

\begin{remark}
We highlight that, since 
\begin{align}\label{equal}
(x_{n})_{+}^{\beta-1}L((x_n)_+^{\beta-1} \eta) - (x_n)_+^{\beta-1} \eta L((x_n)_+^{\beta-1}) =\mathcal{L}\big((x_n)_+^{\beta-1}(y_n)_+^{\beta-1}K(x-y) \big) (\eta),
\end{align}
the integration by parts formula given in the previous lemma can be written as
\begin{align*}
\begin{split}
\int_{\R^n} \int_{\R^n} & (v(x) - v(y)) (\eta(x) - \eta(y)) (x_n)_+^{s-1} (y_n)_+^{s-1} K(x-y) \d y \d x \\
&= \int_{\{ x_n > 0 \}} v \mathcal{L}\big((x_n)_+^{s-1}(y_n)_+^{s-1}K(x-y) \big)(\eta) -  \int_{\{ x_n = 0 \}}{(\theta_K \cdot \nabla \eta)} v \d x',
\end{split}
\end{align*}
where we have used that $L((x_n)_+^{s-1})=0$. 
\end{remark}

We point out that the integration by parts formula \autoref{lemma:ibp} yields boundary terms only in case $\beta = s$. Indeed, we have the following integration by parts formula for general $\beta \in (s,1+s]$:

\begin{lemma}
\label{lemma:ibp-beta}
Let $K$ satisfy \eqref{eq:K-comp}. Let $v \in C^{\gamma}(\R^n)$ for some $\gamma \in (\max\{0,2s-1\}, 1+2s-\beta)$ be such that $|v(x)| \le C (1 + |x|^{\gamma})$. Let $\beta \in (s,1+s]$. Then, it holds for any $\eta \in C^{\infty}_c(\R^n)$
\begin{align}
\label{eq:ibp-energy-beta}
\begin{split}
\int_{\R^n} \int_{\R^n} & (v(x) - v(y)) (\eta(x) - \eta(y)) (x_n)_+^{\beta-1} (y_n)_+^{\beta-1} K(x-y) \d y \d x \\
&= \int_{\{ x_n > 0 \}} v \mathcal{L}\big((x_n)_+^{\beta-1}(y_n)_+^{\beta-1}K(x-y)\big) (\eta) \d x,
\end{split}
\end{align}
if $v\mathcal{L}\big((x_n)_+^{\beta-1}(y_n)_+^{\beta-1}K(x-y)\big)(\eta) \in L^1(\{ x_n > 0 \})$.
Moreover in case $\beta = 1+s$, we have for any $\eta \in C^{\infty}_c(\{  x_n > 0 \})$,
\begin{align}
\label{eq:ibp-energy-obstacle}
\begin{split}
\int_{\R^n} \int_{\R^n} & (v(x) - v(y)) (\eta(x) - \eta(y)) (x_n)_+^{s} (y_n)_+^{s} K(x-y) \d y \d x \\
&= \int_{\{ x_n > 0 \}} (x_n)_+^{s} v L((x_n)_+^{s} \eta) \d x.
\end{split}
\end{align}
Finally, for $\beta \in (s,1+s)$ it holds that for any $\eta \in C_c^{\infty}(\{ x_n > 0 \})$ and some constant $C_{\beta,K} \in \R$
\begin{align}
\label{eq:ibp-general-beta}
\begin{split}
\int_{\R^n} \int_{\R^n} & (v(x) - v(y)) (\eta(x) - \eta(y)) (x_n)_+^{\beta-1} (y_n)_+^{\beta-1} K(x-y) \d y \d x \\
&=  \int_{\{ x_n > 0 \}} (x_n)_+^{\beta-1} v L((x_n)_+^{\beta-1} \eta) \d x - C_{\beta,K} \int_{\{ x_n > 0 \}} \eta v (x_n)_+^{2\beta-2-2s}  \d x.
\end{split}
\end{align}
\end{lemma}

\begin{remark}
In case $\beta = 1+s$, the assumption $v\mathcal{L}\big((x_n)_+^{\beta-1}(y_n)_+^{\beta-1}K(x-y)\big)(\eta) \in L^1(\{ x_n > 0 \})$ is always satisfied. Indeed, if $\supp(\eta) \subset B$, then
\begin{align}
\label{eq:L-beta-finiteB}
\begin{split}
|\mathcal{L} \big((x_n)_+^{s} (y_n)_+^{s} K(x-y)\big)(\eta)(x)| &\le C (x_n)_+^{s} | L ((x_n)_+^s \eta)(x) | + C (x_n)_+^{s} | L((x_n)_+^s) (x) |\\
&\le C
\begin{cases} 
 (x_n)_+^{s}+(x_n)_+^{\alpha},\quad &x \in 2B,\\
 |x|^{-n-s},\quad &x \in (2B)^c.
\end{cases}
\end{split}
\end{align}
Here, we have used that $L((x_n)_+^s)=0$. Then, the estimate for $x \in 2B$ follows from \autoref{lemma:Ldist-est} in the flat case. The estimate for $x \in (2B)^c$ is straightforward.\\
For $\beta \in (s+\frac{1}{2},s+1)$ the assumption can be verified by a similar argument. For general $\beta \in (s,s + \frac{1}{2}]$, some additional assumptions on the boundary behavior of $\eta$ seem to be unavoidable.
\end{remark}

\begin{proof}
The proof goes exactly as in \autoref{lemma:ibp}. However, since $\beta \in (s,1+s]$, we have
\begin{align*}
I_{\eps}^{(3)} := \int_{\{ x_n < \eps \}}  v(x) (x_n)_+^{\beta-1} \left(\int_{\{ y_n > \eps \}} (\eta(x) - \eta(y)) (y_n)_+^{\beta-1} K(x-y) \d y \right) \d x \to 0,\quad \eps \searrow 0,
\end{align*}
To see this, let us take $\delta \in (\max\{2s-1 , 2s - 2\beta +1 \} ,\min\{1, 2s-\beta+1\} )$.\\
It is crucial to observe that this condition is always non-empty since $\beta \in (s,1+s]$. Note that, if $\beta = s$, then we cannot satisfy at the same time $\delta > 2s-2\beta + 1 = 1$ and $\delta \le 1$. This is precisely the reason, why we obtain a boundary integral in case $\beta = s$.\\
Then, since $\eta \in C^{0,\delta}(\R^n)$, and $|x-y| \ge C |x|$ for $x \in (2B)^c$ and $y \in B$, we have
\begin{align*}
|I_{\eps}^{(3)}| &\le C \int_{\{ x_n < \eps \} \cap 2B} (x_n)_+^{\beta-1} \left(\int_{\{ y_n > \eps \}} (y_n)_+^{\beta-1} |x-y|^{-n-2s+\delta} \d y \right) \d x \\
&\quad + C \int_{\{ x_n < \eps \} \setminus 2B} (x_n)_+^{\beta-1} |x|^{\gamma} \left(\int_{\{ y_n > \eps \} \cap B} (y_n)_+^{\beta-1} |x-y|^{-n-2s} \d y \right) \d x\\
&\le C \int_{0}^{\eps} (x_n)_+^{\beta-1} \int_{\eps}^{\infty} (y_n)_+^{\beta-1}  |x_n - y_n|^{-1-2s + \delta} \d y_n \d x_n  + C \int_{\{ x_n < \eps \} \setminus 2B} (x_n)_+^{\beta-1} |x|^{-n-2s+\gamma} \d x \to 0.
\end{align*}
Here, for the first integral, we used that $|\eta(x) - \eta(y)| \le C |x-y|^{\delta}$, since $\delta \le 1$, and applied \autoref{lemma:n-1-dim-integral}, which is possible since $\delta < 2s+1 - \beta < 2s+1$. Moreover by \autoref{lemma:int-finite}, which is applicable since $\delta \in (2s-1,2s+1-\beta)$ and therefore $\delta-1 \in (2s-2,2s-\beta)$,
\begin{align*}
\int_{0}^{\eps}  & (x_n)_+^{\beta-1} \int_{\eps}^{\infty} (y_n)_+^{\beta-1}  |x_n - y_n|^{-2s -1 + \delta} \d y_n \d x_n  \\
&= \eps^{2(\beta -s) -1 + \delta } \int_{0}^{1} r^{\beta-1} \int_{1}^{\infty} t^{\beta-1}  |r-t|^{-2s - 1 + \delta} \d r \d t \le C  \eps^{2(\beta -s) -1 + \delta } \to 0,\quad \eps \searrow 0.
\end{align*}
The convergence to zero holds true since we also assumed $\delta > 2s-2\beta + 1$.
For the second integral in the estimate for $I_{\eps}^{(3)}$ we used \autoref{lemma:int-polar-coord}, which is applicable since $\gamma < 2s -\beta + 1$, and dominated convergence.

By using \eqref{equal}, the second claim follows immediately because $L((x_n)_+^{s}) = 0$, and the third claim follows since $L((x_n)_+^{\beta-1}) = C_{\beta,K}(x_n)_+^{\beta-1-2s}$ and
\begin{align*}
\mathcal{L}\big((x_n)_+^{\beta-1}(y_n)_+^{\beta-1}K(x-y)\big) (\eta) = (x_n)_+^{\beta-1} L((x_n)_+^{\beta-1} \eta) - \eta (x_n)_+^{\beta-1} L((x_n)_+^{\beta-1}).
\end{align*}
\end{proof}

\subsection{Integration by parts in general domains}

We need to integrate by parts in general domains in order to show that the solution to the one-phase problem satisfies the Neumann boundary condition.

\begin{lemma}
\label{lemma:ibp-nonflat}
Assume that $K$ satisfies \eqref{eq:K-comp} and $K \in C^{5+2\alpha}(\mathbb{S}^{n-1})$ for some  $\alpha \in (0,s)$. Let $w \in C^{1+\alpha+2s}_{1+\alpha}(\Omega | B_2)\cap L^{\infty}(\R^n)$, $d_{\Omega}^{1-s} \nabla u \in C^{1,\delta}_c(\overline{\Omega \cap B_2})$ for some $\delta \in (\max\{0,2s-1\},s)$. Then, for every $\eta \in C^{\infty}_c(B_1)$ it holds
\begin{align}
\label{eq:ibp-nonflat}
\begin{split}
\int_{\Omega} & \int_{\Omega} (w(x) - w(y)) (\eta(x) - \eta(y)) \partial_n u(x) \partial_n u(y) K(x-y) \d y \d x \\
&=  \int_{\Omega} \eta(x) \mathcal{L}\big(\partial_n u(x) \partial_n u(y) K(x-y) \big)(w)(x) \d x - \int_{\partial \Omega \cap B_1} \eta(x) \Theta_{K,\Omega}(x) \cdot \nabla w(x) \d x,
\end{split}
\end{align}
where, for any $x\in \partial \Omega \cap B_1$, $\Theta_{K,\Omega}$ is given by
\begin{align}
\label{eq:Theta-formula-domains}
\Theta_{K,\Omega}(x) = 2c_s \left(\frac{\partial_n u}{d_{\Omega}^{s-1}}\right)^2 \int_{\mathbb{S}^{n-1} \cap \{ \nu(x) \cdot \omega > 0 \}} (\nu(x) \cdot \omega)^{2s-1} \omega K(\omega) \d \omega.
\end{align}
\end{lemma}

To prove \autoref{lemma:ibp-nonflat} we flatten the boundary and we prove an analog of \autoref{lemma:ibp} for non-translation invariant kernels (and under slightly different assumptions on the functions involved). 

\begin{lemma}
\label{lemma:ibp-nonflat-trafo}
Assume that $K$ satisfies \eqref{eq:K-comp} and $K \in C^{5+2\alpha}(\mathbb{S}^{n-1})$ for some  $\alpha \in (0,s)$.
Let  $w \in C^{1+\alpha+2s}_{1+\alpha}(\{ x_n > 0 \} | B_2) \cap L^{\infty}(\R^n)$.
Moreover, assume that $\Phi\in C^{1,\delta}(\R^n)$ is a diffeomorphism as in Subsection \ref{subsec:flattening} and that $J \in C^{1,\delta}_c(\overline{\{ x_n > 0 \} \cap B_2})$ for some $\delta \in ( \max\{0,2s-1\} , s)$. Then, for every $\eta \in C^{\infty}_c(B_1)$, it holds
\begin{align}
\label{eq:ibp-nonflat-cov}
\begin{split}
\int_{\R^n} &\int_{\R^n} (w(x) - w(y)) (\eta(x) - \eta(y)) (x_n)_+^{s-1}(y_n)_+^{s-1} J(x) J(y) K(\Phi(x) - \Phi(y)) \d y \d x \\
&= \int_{\{ x_n > 0\} \cap B_1} \eta(x) \mathcal{L}\big((x_n)_+^{s-1}(y_n)_+^{s-1} J(x) J(y) K(\Phi(x) - \Phi(y)) \big)(w)(x) \d x \\
&\quad - \int_{\{ x_n = 0 \} \cap B_1} \eta(x',0) (\theta_{K,J,\Phi}(x') \cdot \nabla w)(x',0) \d x',
\end{split}
\end{align}
with
\begin{align}
\label{eq:theta-K-def-flat}
\theta_{K,J,\Phi}(x')=\theta_{K,J,\Phi}(x',0) = 2c_s J^2(x',0) \int_{\R^{n-1}} (h',1) K(D\Phi(x',0)(h',1)) \d h',
\end{align}
where $c_s > 0$ is the constant from \eqref{eq:theta-constant}.
\end{lemma}

\begin{remark}
Note that by the construction of $\Phi$ it holds $\partial_n \Phi(x',0) = e_n$ and $\partial_i \Phi(x',0) \perp e_n$. Hence, we can write $D \Phi(x',0)(h',1) = (D(x')h',1)$ where $D(x') \in \R^{n-1,n-1}$ for any $x' \in \R^{n-1}$. \\
Hence, if in addition $K(h) = c|h|^{-n-2s}$, i.e. if $K$ is the kernel of the fractional Laplacian, we have
\begin{align*}
\theta_{K,J,\Phi} (x',0)&= c_s J^2(x',0) \int_{\R^{n-1}} (h',1) K((D(x')h',1)) \d h' \\
&= c_s J^2(x',0) \int_{\R^{n-1}} (h',1) (|D(x')h'|^2 + 1)^{-\frac{n+2s}{2}} \d h' = C_s J^2(x',0) e_n
\end{align*}
for some $C_s > 0$, where we used that $|D(x')h'| = |D(x')(-h')|$.
Hence, the boundary integral does not contain an oblique derivative in this case.
\end{remark}

\begin{proof}[Proof of \autoref{lemma:ibp-nonflat-trafo}]
The idea of the proof is very similar to the one of \autoref{lemma:ibp}. Let us recall that by \autoref{lemma:K-estimate} we have
\begin{align}
\label{eq:B2-repeat}
|K(\Phi(x) - \Phi(y))| \le C |x-y|^{-n-2s}.
\end{align}
Moreover, since $K \in C^{0,1}(\mathbb{S}^{n-1})$, by \eqref{eq:K-reg} we have 
\begin{align}
\label{eq:K-difference-estimate-ibp}
\begin{split}
|K(\Phi(x) - \Phi(y)) & - K(D \Phi(x) (x-y))| \\
&\le C \min\{|D \Phi(x) (x-y)| , |x-y| \}^{-n-2s-1} |D \Phi(x) (x-y) - \Phi(x) + \Phi(y)| \\
&\le C |x-y|^{-n-2s + \delta},
\end{split}
\end{align}
where we also used that $|D \Phi(x) (x-y)| \ge C [\Phi^{-1}]_{C^{0,1}(\R^n)}^{-1}|x-y|$ and that $\Phi \in C^{1,\delta}(\R^n)$.

Next, we observe that 
\begin{align*}
\int_{\{ x_n < \eps \} \cap B_1} \int_{\R^n} (w(x) - w(y)) (\eta(x) - \eta(y)) (x_n)_+^{s-1} (y_n)_+^{s-1} J(x) J(y) K(\Phi(x) - \Phi(y)) \d y \d x \to 0.
\end{align*}
This follows as in the proof of \autoref{lemma:ibp} by \autoref{lemma:bd-int}, but the argument is even slightly shorter, since $\supp(J) \subset B_2$, $w \in C^{1+\alpha}(B_2)$, and by \eqref{eq:B2-repeat}:
\begin{align*}
\Big| \int_{\{ x_n < \eps \} \cap B_1} & \int_{\R^n} (w(x) - w(y)) (\eta(x) - \eta(y)) (x_n)_+^{s-1} (y_n)_+^{s-1} J(x) J(y) K(\Phi(x) - \Phi(y)) \d y \d x \Big| \\
&\le C \int_{\{ x_n < \eps \} \cap B_1} (x_n)_+^{s-1} \left( \int_{B_2} (y_n)_+^{s-1} |x-y|^{-n-2s+2} \right) \d x \\
&\le C \int_{\{ x_n < \eps \} \cap B_1} (1 + (x_n)_+^{s-1}) \d x \le C (\eps + \eps^s) \to 0.
\end{align*}

Hence, it remains to prove that as $\eps \searrow 0$:
\begin{align*}
\iint_{(\{ x_n < \eps \} \times \{ y_n < \eps \})^c} & (w(x) - w(y)) (\eta(x) - \eta(y)) (x_n)_+^{s-1} (y_n)_+^{s-1} J(x) J(y) K(\Phi(x) - \Phi(y)) \d y \d x \\
&\to \int_{\{ x_n > 0 \}} \eta(x) \mathcal{L}\big((x_n)_+^{s-1} (y_n)_+^{s-1} J(x)J(y) K(x-y) \big)(w)(x) \d x \\
&\quad - \int_{\{ x_n = 0 \} \cap B_1} \eta(x',0) \theta_{K,J,\Phi}(x') \cdot \nabla w(x',0) \d x'.
\end{align*}
By the symmetry of $K$, we have
\begin{align*}
& \iint_{(\{ x_n < \eps \} \times \{ y_n < \eps \})^c} (w(x) - w(y)) (\eta(x) - \eta(y)) (x_n)_+^{s-1} (y_n)_+^{s-1} J(x) J(y) K(\Phi(x) - \Phi(y)) \d y \d x \\
&\qquad = 2 \iint_{(\{ x_n < \eps \} \times \{ y_n < \eps \})^c} \eta(x) (w(x) - w(y)) (x_n)_+^{s-1} (y_n)_+^{s-1} J(x) J(y) K(\Phi(x) - \Phi(y)) \d y \d x  \\
&\qquad =  \int_{\{ x_n > \eps \}} \eta(x) \mathcal{L} \big((x_n)_+^{s-1} (y_n)_+^{s-1} J(x) J(y) K(\Phi(x) - \Phi(y)) \big)(w)(x) \d x \\
&\qquad \quad + 2 \int_{\{ x_n < \eps \}} \eta(x) (x_n)_+^{s-1} J(x) \left( \int_{\{ y_n > \eps \}} (w(x) - w(y)) (y_n)_+^{s-1} J(y) K(\Phi(x) - \Phi(y)) \d y \right) \d x \\
&= J_{\eps}^{(1)} + J_{\eps}^{(2)}.
\end{align*}
On one hand, since $J \in C_{c}^{1,\delta}(\overline{\{ x_n > 0 \} \cap B_2})$, $\Phi \in C^{1,\delta}(\R^n)$, and $w \in C^{1+\alpha+2s}_{1+\alpha}( \{ x_n > 0 \} | B_2)$, by \eqref{eq:K-difference-estimate-ibp}, we notice that for any $x \in \{ x_n > 0 \}$,
\begin{align}
\label{eq:L-estimate-coeff2}
\begin{split}
|\mathcal{L} & \big((x_n)_+^{s-1} (y_n)_+^{s-1} J(x) J(y) K(\Phi(x) - \Phi(y)) \big)(w)(x)| \\
&\le \mathcal{L} \big((x_n)_+^{s-1} (y_n)_+^{s-1} J(x) J(y) K(D \Phi(x) (x-y)) \big)(w)(x)| \\
&\quad + C (x_n)_+^{s-1} \int_{\{ x_n > 0 \} \cap B_2} (y_n)_+^{s-1}  |w(x)- w(y)| |x-y|^{-n-2s+\delta} \d y \\
&\le C (x_n)_+^{s-1} | \mathcal{L}\big(K(D \Phi(x) (x-y)) \big) ((x_n)_+^{s-1} J w)(x) | \\
&\quad + C (x_n)_+^{s-1} | \mathcal{L}\big(K(D \Phi(x) (x-y)) \big) ((x_n)_+^{s-1})(x) | \\
&\quad + C (x_n)_+^{s-1} \int_{\{ x_n > 0 \} \cap B_2} (y_n)_+^{s-1} |x-y|^{-n-2s+\delta + 1} \d y.
\end{split}
\end{align}
Note that since $K$ is homogeneous, also $h \mapsto K(D \Phi(x) (h))$ is homogeneous for any $x$, and therefore,
\begin{align*}
\mathcal{L}\big(K(D \Phi(x) (x-y))\big) ((x_n)_+^{s-1}) = 0.
\end{align*}
Moreover, since $J \in C^{1,\delta}(\overline{\{ x_n > 0 \} \cap B_2})$  and $w \in C^{1,\alpha}(\overline{\{ x_n > 0 \} \cap B_2})$, by the regularity hypothesis on the kernel, we can apply \cite[Corollary 2.5]{RoWe24a} to get
\begin{align*}
|\mathcal{L}\big(K(D \Phi(x) (x-y))\big) ((x_n)_+^{s-1}Jw)| \le C (x_n)_+^{\min\{\alpha,\delta\} - s}.
\end{align*}
Then, from \eqref{eq:L-estimate-coeff2} it follows that
\begin{align}
\label{eq:L-estimate-coeff}
|\mathcal{L}  \big((x_n)_+^{s-1} (y_n)_+^{s-1} J(x) J(y) K(\Phi(x) - \Phi(y)) \big)(w)(x)| \le C (x_n)_+^{s-1} + C (x_n)_+^{\min\{\alpha,\delta\} - 1},
\end{align}
where we have also applied \autoref{lemma:bd-int}, using $\delta > \max\{0,2s-1\}$. Hence, by dominated convergence, 
\begin{align*}
J_{\eps}^{(1)} \to \int_{\{ x_n > 0 \} \cap B_1} \eta(x) \mathcal{L}\big((x_n)_+^{s-1} (y_n)_+^{s-1} J(x) J(y) K(\Phi(x) - \Phi(y))\big)(w)(x) \d x.
\end{align*}

On the other hand we observe that $J_{\eps}^{(2)}$ is of the same type as the term $I_{\eps}^{(3)}$ in the proof of \autoref{lemma:ibp}. Thus, we proceed in the same way as before by adding and subtracting $\nabla w(x)$, and observing that since $w \in C^{1,\delta}(B_2)$ for some $\delta \in (0,s)$, it holds
\begin{align*}
J_{\eps}^{(2)} = 2 & \int_{\{ x_n < \eps \} \cap B_1} \eta(x) (x_n)_+^{s-1} J(x) \nabla w(x)\left(\int_{\{ y_n > \eps \}}  (x-y) J(y)(y_n)_+^{s-1} K(\Phi(x) - \Phi(y)) \d y \right) \d x + J_{\eps},
\end{align*}
where by the same arguments as in the proof of \autoref{lemma:ibp} by \autoref{lemma:int-finite} we get,
\begin{align*}
|J_{\eps}| &\le C\int_{0}^{\eps} (x_n)_+^{s-1} \int_{\eps}^{\infty} (y_n)_+^{s-1}  |x_n - y_n|^{-2s+\delta} \d y_n \d x_n \le C \eps^{\delta} \to 0.
\end{align*}
Hence, it remains to prove that
\begin{align}
\label{eq:ibp-nonflat-main}
2 & \int_{\{ x_n < \eps \} \cap B_1} \eta(x) (x_n)_+^{s-1} J(x) \nabla w(x)M_{x'}(\eps,x_n) \ dx\to -\int_{\{ x_n = 0\}}  \eta(x',0) \theta_{K,J,\Phi}(x') \cdot\nabla w(x',0) \d x'.
\end{align}
where
\begin{align}
\label{eq:M-asymp}
M_{x'}(\eps,x_n) &= \int_{\eps}^{\infty} (y_n)_+^{s-1} k_{x'}(x_n,y_n) \d y_n, \\
\label{eq:k-asymp}
k_{x'}(x_n,y_n) &= \int_{\R^{n-1}} (x-y) J(y',y_n) K(\Phi(x',x_n) - \Phi(y',y_n)) \d y' .
\end{align}

To prove \eqref{eq:ibp-nonflat-main} let us first decompose $k_{x'}$ as follows
\begin{align*}
k_{x'}(x_n,y_n) &= \int_{\R^{n-1}} J(y) (x-y) K(\Phi(x) - \Phi(y)) \d y' \\
&= J(x)  \int_{\R^{n-1}} (x-y) K(D\Phi(x)(x-y)) \d y' \\
&\quad + \int_{\R^{n-1}} (J(y) - J(x)) (x-y) K(D\Phi(x)(x-y)) \d y' \\
&\quad + \int_{\R^{n-1}} J(y) (x-y) [K(\Phi(x) - \Phi(y)) - K(D\Phi(x)(x-y))] \d y'.
\end{align*}
For the second summand we have by the $C^{\delta}$ regularity of $J$, the boundedness of $|D \Phi(x)|$ in $B_2$, and \autoref{lemma:n-1-dim-integral}, using that $\delta < 2s$
\begin{align*}
\Big| \int_{\R^{n-1}} (J(x) - J(y))(x-y) K(D\Phi(x)(x-y)) \d y' \Big| &\le C \int_{\R^{n-1}} |x-y|^{-n-2s + \delta + 1} \d y' \le C |x_n - y_n|^{-2s+\delta}.
\end{align*}

For the third summand we estimate, using the boundedness of $J$, as well as \eqref{eq:K-difference-estimate-ibp} and \autoref{lemma:n-1-dim-integral}, which is applicable after choosing $\delta < 2s$:
\begin{align*}
\Big| \int_{\R^{n-1}} & J(y) (x-y) [K(\Phi(x) - \Phi(y)) - K(D\Phi(x)(x-y))] \d y' \Big| \\
&\le C \int_{\R^{n-1}}|x-y|^{-n-2s+\delta + 1} \d y' \le C |x_n - y_n|^{-2s+\delta}.
\end{align*}
Next, we have by the exact same procedure as in \eqref{eq:integrate-out-kernel}, where we set $h_n = x_n - y_n<0$
\begin{align*}
\int_{\R^{n-1}} &J(x) (x-y) K(D\Phi(x)(x-y)) \d y' = \int_{\R^{n-1}}J(x) (h',h_n) K(D\Phi(x)(h',h_n)) \d h' \\
&= J(x)\int_{\R^{n-1}} (h',h_n) |D \Phi(x) (h',h_n)|^{-n-2s} K \left( \frac{D \Phi(x)(h',h_n)}{|D \Phi(x)(h',h_n)|} \right) \d h' \\
&= |h_n|^{-n-2s+1} J(x) \int_{\R^{n-1}} \left( \frac{h'}{|h_n|} , -1 \right) \left|D \Phi(x) \left(\frac{h'}{|h_n|},-1\right)\right|^{-n-2s}  K \left( \frac{D \Phi(x)(h'/|h_n|,-1)}{|D \Phi(x)(h'/|h_n|,-1)|} \right) \d h' \\
&= |h_n|^{-2s} J(x) \int_{\R^{n-1}} (h',-1) |D\Phi(x)(h',-1)|^{-n-2s} K\left( \frac{D\Phi(x)(h',-1)}{|D\Phi(x)(h',-1)|} \right) \d h' \\
&= -|h_n|^{-2s} J(x) \int_{\R^{n-1}} (h',1) K\left( D\Phi(x)(h',1)\right) \d h' =: - \widetilde{\theta}_{K,J,\Phi}(x) |h_n|^{-2s}  .
\end{align*}
Here, $\widetilde{\theta}_{K,J,\Phi}\in C^{\alpha}(B_2)$ by the regularity of $J,K,\Phi$.
Thus, 
\begin{align*}
k_{x'}(x_n,y_n) = -\widetilde{\theta}(x) |x_n-y_n|^{-2s} + O(|x_n - y_n|^{-2s+\delta}).
\end{align*}
Since by \autoref{lemma:int-finite}, which is applicable because $2s-2<0<\delta < s$,
\begin{align*}
& \Big| \int_0^{\eps} (x_n)_+^{s-1} \int_{\eps}^{\infty} (y_n)_+^{s-1} |x_n - y_n|^{-2s+\delta\sigma} \d y_n \d x_n \Big| C \eps^{\delta\sigma} \to 0,
\end{align*}
we have
\begin{align*}
2\lim_{\eps \to 0} \int_0^{\eps} & \eta(x) J(x) \nabla w(x)  (x_n)_+^{s-1}  M_{x'}(\eps,x_n) \d x_n \\
&= -2\lim_{\eps \to 0}\int_0^{\eps} \eta(x) J(x) \widetilde{\theta}(x) \cdot \nabla w(x) (x_n)_+^{s-1} \int_{\eps}^{\infty} (y_n)_+^{s-1} |x_n - y_n|^{-2s} \d y_n \d x_n.
\end{align*}
Moreover, since $\eta J \widetilde{\theta}(x) \cdot \nabla w\in C^{\min\{\alpha , \delta\}}(B_2)$, by using one more time \autoref{lemma:int-finite} (see also \eqref{eq:theta-K-def}),
\begin{align*}
\Big| & \int_{\{ x_n < \eps \} }  [\eta(x) J(x)\widetilde{\theta}(x) \cdot \nabla w(x)  - \eta(x',0) J(x',0)\widetilde{\theta}(x',0) \cdot \nabla w(x',0)](x_n)_+^{s-1} \int_{\eps}^{\infty} (y_n)_+^{s-1} |x_n - y_n|^{-2s} \d y_n \d x \Big| \\
&  \le C\eps^{\min\{\alpha , \delta\}}\int_0^{1} r^{s-1} \int_{1}^{\infty} t^{s-1} |r-t|^{-2s} \d t \d r\le C c_s \eps^{\min\{\alpha , \delta \}}\to 0.
\end{align*}
Therefore, using again \eqref{eq:theta-constant},
\begin{align*}
2\lim_{\eps \to 0} \int_0^{\eps} & \eta(x) J(x) \nabla w(x)  (x_n)_+^{s-1}  M_{x'}(\eps,x_n) \d x_n \\
&= -2\eta(x',0) J(x',0) \widetilde{\theta}(x',0) \cdot \nabla w(x',0) \lim_{\eps \to 0} \int_0^{\eps} (x_n)_+^{s-1} \int_{\eps}^{\infty} (y_n)_+^{s-1} |x_n - y_n|^{-2s} \d y_n \d x_n \\
&= - 2c_s \eta(x',0) J(x',0) \widetilde{\theta}(x',0) \cdot \nabla w(x',0),
\end{align*}
This concludes the proof of \eqref{eq:ibp-nonflat-main} by the dominated convergence theorem and recalling the definition $\theta_{K,J,\Phi}(x') = 2 c_s J(x',0) \widetilde{\theta}(x',0)$.
\end{proof}

We are now in a position to prove the integration by parts formula in non-flat domains.

\begin{proof}[Proof of \autoref{lemma:ibp-nonflat}]
Let $\Phi$ be a flattening diffeomorphism as in Subsection \ref{subsec:flattening}. By changing variables from $x,y$ to $\Phi(x),\Phi(y)$, we obtain for $\widetilde{w}(x) = w \circ \Phi(x)$ and $\widetilde{\eta}(x) = \eta \circ \Phi(x)$,
\begin{align*}
\int_{\R^n} & \int_{\R^n} (w(x) - w(y)) (\eta(x) - \eta(y)) \partial_n u(x) \partial_n u(y) K(x-y) \d y \d x  \\
&= \int_{\R^n} \int_{\R^n} (\widetilde{w}(x) - \widetilde{w}(y)) (\widetilde{\eta}(x) - \widetilde{\eta}(y)) J(x)J(y) (x_n)_+^{s-1} (y_n)_+^{s-1} K(\Phi(x)-\Phi(y)) \d y \d x ,
\end{align*}
where $J(x) = |\det D \Phi(x)|\phi_n(\Phi(x))$ {and $\phi:= d_{\Omega}^{1-s} \nabla u$}. By construction, $J,\Phi,\widetilde{w},\widetilde{\eta}$ satisfy the assumptions of \autoref{lemma:ibp-nonflat-trafo}. Hence, we deduce
\begin{align*}
\int_{\R^n} & \int_{\R^n} (w(x) - w(y)) (\eta(x) - \eta(y)) \partial_n u(x) \partial_n u(y) K(x-y) \d y \d x \\
&=  \int_{\R^n} \widetilde{\eta}(x) \mathcal{L}\big((x_n)_+^{s-1}(y_n)_+^{s-1} J(x) J(y) K(\Phi(x) - \Phi(y))\big)(\widetilde{w})(x) \d x \\
&\quad - \int_{\{ x_n = 0 \} \cap B_1} \widetilde{\eta}(x',0) \theta_{K,J,\Phi}(x') \cdot \nabla \widetilde{w}(x',0) \d x',
\end{align*}
where
\begin{align*}
\theta_{K,J,\Phi}(x') = 2c_s J^2(x',0) \int_{\R^{n-1}} (h',1) K(D\Phi(x',0)(h',1)) \d h'.
\end{align*}
Transforming back to the original variables, we obtain
\begin{align*}
\int_{\R^n} & \widetilde{\eta}(x) \mathcal{L}\big((x_n)_+^{s-1}(y_n)_+^{s-1} J(x) J(y) K(\Phi(x) - \Phi(y))\big)(\widetilde{w})(x) \d x \\
&= \int_{\R^n} {\eta}(x) \mathcal{L}\big(\partial_n u(x) \partial_n u(y)  K(x-y)\big)(w)(x) \d x,
\end{align*}
and
\begin{align*}
& \int_{\{ x_n = 0 \} \cap B_1} \widetilde{\eta}(x',0) \theta_{K,J,\Phi}(x') \cdot \nabla \widetilde{w}(x',0) \d x' \\
&\qquad = \int_{\{ x_n = 0 \} \cap B_1} \eta(\Phi(x',0)) \big( (D\Phi(x',0) \theta_{K,J,\Phi}(x')) \cdot \nabla w(\Phi(x',0)) \big) \d x'\\
&\qquad = \int_{\partial \Omega \cap B_1} \eta(x) \Theta_{K,\Omega}(x) \cdot \nabla w(x) \d x,
\end{align*}
where
\begin{align*}
\Theta_{K,\Omega}(x) = |\det D \Phi^{-1}(x)| (D \Phi)(\Phi^{-1}(x)) \cdot \theta_{K,J,\Phi}(\Phi^{-1}(x)).
\end{align*}

To conclude the proof, we have to show that $\Theta_{K,\Omega}(x)$ can be expressed as in \eqref{eq:Theta-formula-domains}. To prove it, we show first that for a kernel $K$ satisfying \eqref{eq:K-comp} it holds for some $c_0 > 0$,
\begin{align}\label{LaK}
\int_{\R^{n-1}} (h',1) K(h',1) \d h' = c_0 \int_{\mathbb{S}^{n-1} \cap \{ \omega_n > 0 \}} \omega \omega_n^{2s-1} K(\omega) \d \omega.
\end{align}
Indeed, let us consider the projection of $(h',1)$ to $\mathbb{S}^{n-1} \cap \{\omega_n > 0 \}$ given by
\begin{align*}
(|h'|^2 + 1)^{-1/2} (h',1) := \omega \in \mathbb{S}^{n-1} \cap \{ \omega_n > 0 \},
\end{align*}
and after changing variables from $h'$ to $\omega$, whose Jacobian has determinant $\omega_n^{-n}$, and using the homogeneity of $K$, we deduce
\begin{align*}
\int_{\R^{n-1}} (h',1) K(h',1) \d h' = \int_{\R^{n-1}} \omega \omega_n^{n+2s-1} K(\omega) \d h' = c_0 \int_{\mathbb{S}^{n-1} \cap \{ \omega_n > 0 \}} \omega \omega_n^{2s-1} K(\omega) \d \omega,
\end{align*}
as claimed. Next, we denote $D(x) = D\Phi(\Phi^{-1}(x))$ and note that $D(x)^{-1} = D\Phi^{-1}(x)$. Observing that $\widetilde{K}(h) := K(D(x) (h',1))$ also satisfies \eqref{eq:K-comp}, from \eqref{LaK}, we obtain
\begin{align*}
\int_{\R^{n-1}} & D \Phi(\Phi^{-1}(x)) \cdot (h',1) K(D\Phi(\Phi^{-1}(x))\cdot (h',1)) \d h' = D(x) \cdot \int_{\R^{n-1}} (h',1) \widetilde{K}((h',1)) \d h' \\
&= \int_{\mathbb{S}^{n-1} \cap \{ \omega_n > 0 \}} D(x) \omega \omega_n^{2s-1} K(D(x) \omega) \d \omega \\
&= \int_{\mathbb{S}^{n-1} \cap \{ \omega_n > 0 \}} \frac{D(x) \omega }{|D(x) \omega |} \left(\frac{(D(x)^{-1} D(x)\omega)_n}{|D(x)\omega|} \right)^{2s-1} |D(x) \omega|^{-n} K\left(\frac{D(x) \omega}{|D(x) \omega|} \right) \d \omega  \\
&=  |\det D(x)|^{-1} \int_{\mathbb{S}^{n-1} \cap \{ (D(x)^{-1} \omega)_n > 0 \}} \omega(D(x)^{-1} \omega)_n^{2s-1} K(\omega) \d \omega \\
&=  |\det D(x)|^{-1} \int_{\mathbb{S}^{n-1} \cap \{ \nu(x) \cdot \omega > 0 \}} \omega(\nu(x) \cdot \omega)^{2s-1} K(\omega) \d \omega.
\end{align*}
Here, we employed the change of variables $\omega \mapsto \frac{D(x) \omega}{|D(x)\omega|}$, whose determinant is $\frac{|\det D(x)|}{|D(x)\omega|^n}$, and used that by the construction of $\Phi$ we have $(D(x)^{-1} \omega)_n = \nu(x) \cdot \omega$.

In fact, the relationship between the original surface element $d \sigma$ on $\mathbb{S}^{n-1}$ and the transformed surface element $d \sigma_A$ on the ellipsoid defined by the equation $z^T (D(x)^{-1})^T D(x)^{-1} z = 1$ under the linear map $z = D(x)\omega$ is given by
\begin{align*}
\d \sigma(\omega) = |\det D(x)| |(D(x)^{-1})^T \omega| \d \sigma_A(z),
\end{align*}
where we used that $\omega$ is the unit normal vector to the sphere $\mathbb{S}^{n-1}$. Indeed, this expression equals the determinant of the pushforward map on tangent vectors, namely
\begin{align*}
|\det D(x)| |(D(x)^{-1})^T \omega| = |\det (P_{\omega} D(x)^T D(x) P_{\omega})|^{1/2},
\end{align*}
where $P_{\omega} = I  - \omega \omega^T$ denotes the orthogonal projection matrix onto the tangent space of $\mathbb{S}^{n-1}$ at $\omega \in \mathbb{S}^{n-1}$. Next, projecting from the point $z = D(x) \omega$ on the ellipsoid back to the unit sphere $v = z/|z|$, we obtain
\begin{align*}
\d \sigma(v) = \frac{\cos \alpha}{|z|^{n-1}} \d \sigma_A(z),
\end{align*}
where $\alpha$ denotes the angle between $z$ and the outward unit normal to the ellipsoid $n_z = \frac{(D(x)^{-1})^T D(x)^{-1}z}{|(D(x)^{-1})^T D(x)^{-1}z|}$. Hence, we have
\begin{align*}
\cos \alpha = \frac{z^T (D(x)^{-1})^T D(x)^{-1}z}{|z| |(D(x)^{-1})^T D(x)^{-1}z|} = \frac{1}{|z| |(D(x)^{-1})^T D(x)^{-1}z|},
\end{align*}
which yields, as claimed,
\begin{align*}
\d \sigma(\omega) &= |\det D(x)| |(D(x)^{-1})^T \omega| \d \sigma_A \\
&= |\det D(x)| |(D(x)^{-1})^T \omega| \frac{\cos \alpha}{|D(x)\omega|^{n-1}} \d \sigma(v) = \frac{|\det D(x)|}{|D(x) \omega|^n} \d \sigma(v).
\end{align*}
This confirms the aforementioned computation and  yields
\begin{align*}
\Theta_{K,\Omega}(x) = 2c_s |\det D \Phi^{-1}(x)|^2 J^2(\Phi^{-1}(x)) \int_{\mathbb{S}^{n-1} \cap \{ \nu(x) \cdot \omega > 0 \}} \omega(\nu(x) \cdot \omega)^{2s-1} K(\omega) \d \omega.
\end{align*}
This concludes the proof after recalling the definitions of $\phi_n$ and $J$, and using that $|\det D \Phi^{-1}(x)| = |\det D \Phi( \Phi^{-1}(x))|^{-1}$.
\end{proof}

As in the flat case, the integration by parts formula \autoref{lemma:ibp-nonflat-trafo} yields boundary terms only in case $\beta = s$. Indeed, we have the following integration by parts formula for general $\beta \in (s,1+s]$. Note that we do not need to assume $w\in C_{loc}^{1,\gamma}(\overline{\Omega} \cap B_2)$ for some $\gamma > 0$.

\begin{lemma}
\label{lemma:ibp-nonflat-beta}
Let $\beta \in (s,1+s]$. Assume that $K$ satisfies \eqref{eq:K-comp} and $K \in C^{0,1}(\mathbb{S}^{n-1})$. Let $w \in C^{2s+\eps}_{0,\alpha}(\Omega | B_2) \cap L^{\infty}(\R^n)$, $d_{\Omega}^{1-\beta} \nabla u \in C^{\delta}_c(\overline{\Omega \cap B_2})$ for some $\eps > 0$, some $\alpha \in ( \max\{0, 2s-\beta , 1 - \beta , 1 + 2s - 2 \beta \} , 1)$, and some $\delta \in (\max\{0,2s-1\},s)$. Then, for every $\eta \in C^{\infty}_c(B_1)$ such that $\eta \mathcal{L}\big(\partial_n u(x) \partial_n u(y) K(x-y) \big)(w) \in L^1(B_1)$, it holds
\begin{align}
\label{eq:ibp-nonflat-beta}
\begin{split}
\int_{\R^n} & \int_{\R^n} (w(x) - w(y)) (\eta(x) - \eta(y)) \partial_nu(x) \partial_n u(y) K(x-y) \d y \d x \\
&=  \int_{B_1} \eta(x) \mathcal{L}\big(\partial_n u(x) \partial_n u(y) K(x-y) \big)(w)(x) \d x.
\end{split}
\end{align}
\end{lemma}

\begin{remark}
If $\partial\Omega\in C^{1,\alpha}$ the assumption $\eta \mathcal{L}\big(\partial_n u(x) \partial_n u(y) K(x-y)\big)(w) \in L^1(B_1)$ is always satisfied for $\beta = 1+s$ by a similar argument as in \eqref{eq:L-beta-finiteB} because $z := d^{-s}_{\Omega} \partial_n u \in C^{2s+\eps}_{0,\delta}(\overline{\Omega \cap B_2})$ and ${w} \in C^{2s+\eps}_{0,\alpha}(\Omega | B_{2})$. Indeed, by \autoref{lemma:Ldist-est} we obtain
\begin{align}
\label{eq:L-beta-finiteBB}
\begin{split}
|\mathcal{L}\big(\partial_nu(x)\partial_nu(y)K(x-y)\big)(w)| &= |\mathcal{L} \big(d_{\Omega}^{s}(x) d_{\Omega}^{s}(y) z(x) z(y) K(x-y)\big)(w)(x)| \\
&\le C d_{\Omega}^{s}(x) | L (d_{\Omega}^{s} z {w})(x) | + C d_{\Omega}^{s}(x) | L (z\, d_{\Omega}^{s})(x) |\\
& \le C d^s_{\Omega}(x) + Cd_{\Omega}^{\min\{\alpha,\delta \}}(x).
\end{split}
\end{align}
For $\beta \in (s+\frac{1}{2},s+1)$ it can be verified by a similar argument. For general $\beta \in (s,s + \frac{1}{2}]$, some additional assumptions on the boundary behavior of $\eta$ seem to be unavoidable.
\end{remark}

\begin{proof}[Proof of \autoref{lemma:ibp-nonflat-beta}]
The proof is similar to that of \autoref{lemma:ibp-nonflat-trafo}. The main difference is that 
\begin{align*}
J_{\eps}^{(2)} &:= \int_{\{ x_n < \eps \}} \eta(x) (x_n)_+^{\beta-1} J(x) \left( \int_{\{ y_n > \eps \}} (w(x) - w(y)) (y_n)_+^{\beta-1} J(y) K(\Phi(x) - \Phi(y)) \d y \right) \d x\to 0,
\end{align*}
which can be proved by similar arguments as in the proof of \autoref{lemma:ibp-beta}, applied with $v := \eta$ and $\eta := w$,  since by assumption, $w \in C^{0,\gamma}(\{ x_n \ge 0 \} \cap B_2)$ for some $\gamma \in (\max\{ 0 , 2s-1 , 2s -2\beta + 1\} , \min\{ 1 , 2s-\beta + 1 \})$. In this way, since $\supp(\eta) \subset B_1$, \autoref{lemma:n-1-dim-integral} and \autoref{lemma:int-finite} yield,
\begin{align*}
|J_{\eps}^{(2)}| \le C \int_{\{ x_n < \eps \} \cap B_1}  (x_n)_+^{\beta-1} \left( \int_{\{ y_n > \eps \}} (y_n)_+^{\beta-1}|x-y|^{-n-2s+\gamma} \d y \right) \d x \to 0.
\end{align*}

Moreover, note that the first computations done in the proof of \autoref{lemma:ibp-nonflat-trafo} can be adapted even if $w$ has less regularity, in order to get the convergence of $J_{\eps}^{(1)}$ to the same expression. In fact, since $\supp(J) \subset B_2$, by \autoref{lemma:bd-int} and \eqref{eq:B2-repeat}, we get
\begin{align*}
\Big| \int_{\{ x_n < \eps \} \cap B_1} & \int_{\R^n} (w(x) - w(y)) (\eta(x) - \eta(y)) (x_n)_+^{\beta-1} (y_n)_+^{\beta-1} J(x) J(y) K(\Phi(x) - \Phi(y)) \d y \d x \Big| \\
&\le C \int_{\{ x_n < \eps \} \cap B_1} (x_n)_+^{\beta-1} \left( \int_{B_2} (y_n)_+^{\beta-1} |x-y|^{-n-2s+1+\gamma} \right) \d x \\
&\le C \int_{\{ x_n < \eps \} \cap B_1} (x_n)_+^{\beta-1}(1 + (x_n)_+^{\beta-2s+\gamma}) \d x \le C (\eps^{\beta} + \eps^{2\beta-2s+\gamma}) \to 0,
\end{align*}
using that $\gamma > 2s-1$ to guarantee the convergence of the inner integral, and $\beta \ge s$ to guarantee the convergence of the outer integral, as well as the convergence to zero of the last expression. Finally, the convergence of $J_{\eps}^{(1)}$ follows immediately by dominated convergence, since by hypothesis it holds $\eta\mathcal{L} \big((x_n)_+^{s-1} (y_n)_+^{s-1} J(x) J(y) K(\Phi(x) - \Phi(y)) \big)(w)\in L^{1}(B_1)$. 
\end{proof}

\subsection{Weak solutions to weighted nonlocal equations with Neumann boundary conditions} 

The integration by parts formulas from \autoref{lemma:ibp-nonflat-trafo} and \autoref{lemma:ibp-nonflat-beta} motivate the following weak solution concept.

\begin{definition}
\label{def:weak-sol-trafo}
Let $\beta \in [s,1+s]$ and {$\gamma \in (\max\{0,2s-1\}, 2s-\beta+1)$}. Let $K$ be as in \eqref{eq:K-comp}. Moreover, let $J \in L^{\infty}(\{ x_n > 0 \})$, $\Phi \in C^{1,\alpha}(\R^n)$ be as in Subsection \ref{subsec:flattening}, $g \in L^{\infty}(B_1)$. We say that $u \in C^{\gamma}(\R^n)$ is a weak solution to
\begin{align*}
\begin{cases}
\mathcal{L} \big(J(x)J(y)(x_n)_+^{\beta-1}(y_n)_+^{\beta-1} K(\Phi(x) - \Phi(y))\big)(u) &= g  ~~ \text{ in } \{ x_n > 0 \} \cap B_{1},\\
\partial_n u &= b ~~ \text{ on } \partial \{ x_n > 0 \} \cap B_1,
\end{cases}
\end{align*}
if for any $\eta \in C^{\infty}_c(B_1)$ it holds
\begin{align*}
\int_{\R^n} &\int_{\R^n} (u(x) - u(y)) (\eta(x) - \eta(y)) (x_n)_+^{\beta-1}(y_n)_+^{\beta-1} J(x) J(y) K(\Phi(x) - \Phi(y)) \d y \d x \\
&=  \int_{B_1} g(x) \eta(x) \d x - \1_{\{\beta = s \}} \left( \int _{\{x_n=0\}\cap B_1} b(x',0) (\theta_{K,J,\Phi})_n(x') \eta(x',0) \d x' \right)\\
&\quad + \1_{\{\beta = s \}} \left( \int_{\{ x_n = 0 \}\cap B_1}   u(x',0) [(\theta_{K,J,\Phi}' (x')\cdot \nabla_{x'} \eta(x',0)) +  \dvg \theta_{K,J,\Phi}'(x') \eta(x',0)] \d x' \right),
\end{align*}
where we denoted $\theta_{K,J,\Phi}=(\theta'_{K,J,\Phi}, (\theta_{K,J,\Phi})_n) $, and define
\begin{align*}
\theta_{K,J,\Phi}(x',0)=\theta_{K,J,\Phi}(x') = 2c_s J^2(x',0) \int_{\R^{n-1}} (h',1) K(D\Phi(x',0)(h',1)) \d h' \in C^{0,1}(\{ x_n = 0 \} \cap B_1).
\end{align*}
\end{definition}

In particular, in the translation invariant case, we have the following natural weak solution concept

\begin{definition}
\label{def:weak-sol-ti}
Let $\beta \in [s,1+s]$. Let $K$ be as in \eqref{eq:K-comp}. We say that $u \in C^{\gamma}(\R^n)$ with $|u(x)| \le C (1 + |x|^{\gamma})$ for some $\gamma \in (\max\{0,2s-1\}, 2s-\beta+1)$ is a weak subsolution to 
\begin{align*}
\begin{cases}
\mathcal{L}\big((x_n)_+^{\beta-1}(y_n)_+^{\beta-1}K(x-y)\big)(v) &= 0 ~~ \text{ in } \{ x_n > 0 \},\\
\partial_n v &= 0 ~~ \text{ on } \{ x_n = 0 \},
\end{cases}
\end{align*}
if for any $\eta \in C_c^{\infty}(\R^n)$ with $\eta \ge 0$ it holds
\begin{align*}
\int_{\R^n} & \int_{\R^n} (v(x) - v(y))(\eta(x) - \eta(y)) (x_n)_+^{\beta-1}(y_n)_+^{\beta-1}K(x-y) \d y \d x \\
&\le\1_{\{\beta = s\}} \int_{\{ x_n = 0 \}} v(x',0)(\theta'_K \cdot \nabla_{x'}\eta)(x',0) \d x',
\end{align*}
where $\theta_K =(\theta'_K ,(\theta_{K})_n)\in \R^n$ is defined in \autoref{lemma:ibp}.\\
We say that $u$ is a weak supersolution if $-u$ is a weak subsolution. Moreover, we say that $u$ is a weak solution, if it is a subsolution and a supersolution. 
\end{definition}

\subsection{Transforming the boundary condition}

Our next objective will be to transform solutions $v \in C_{loc}^{\gamma}(\R^n)$ to the following nonlocal problem
\begin{align}\label{eq:oblique-PDE}
\int_{\R^n} (x_n)_+^{s-1} v L_K((x_n)_+^{s-1} \eta)  \d x - \int_{\{ x_n = 0 \}} v \vartheta \cdot \nabla \eta  \d x = 0, \quad \eta \in C^{\infty}_c(\R^n),
\end{align} 
where $\vartheta=(\vartheta',\vartheta_n)\in \R^n$, $\vartheta_n \not= 0$, into a solution $u \in C_{loc}^{\gamma}(\R^n)$ to
\begin{align}\label{eq:oblique-PDE-2}
\int_{\R^n} (x_n)_+^{s-1} u L_{\widetilde{K}}((x_n)_+^{s-1} \eta)  \d x - \int_{\{ x_n = 0 \}} u \theta \cdot \nabla \eta  \d x = 0, \quad  \eta \in C^{\infty}_c(\R^n),
\end{align}
with a different oblique boundary integral where $\theta=(\theta',\theta_n)\in \R^n$, $\theta_n \not= 0$.

The following lemma will be crucial since it allows us to restrict ourselves to proving a Liouville theorem for weak solutions as in \autoref{def:weak-sol-ti} without having to consider any boundary integrals, even in case $\beta = s$.

Note that by the same proof, one can obtain an analogous result for solutions to \eqref{eq:oblique-PDE} and \eqref{eq:oblique-PDE-2} in bounded domains.

\begin{lemma}
\label{lemma:oblique-trafo}
Assume that $K$ satisfies \eqref{eq:K-comp}. Let $\vartheta, \theta \in \R^n$ with $\vartheta_n , \theta_n \not= 0$ and $v \in C^{\gamma}_{loc}(\R^n)$ for some $\gamma \in (\max\{0,2s-1\},1+s)$ be a solution to \eqref{eq:oblique-PDE}. 
Then
\begin{align}
\label{eq:K-u-trafo-def}
u(x):= v\left( x' + \Big( \frac{\vartheta'}{\vartheta_n} - \frac{\theta'}{\theta_n} \Big) x_n , x_n \right),
\end{align}
is a solution to \eqref{eq:oblique-PDE-2} with
\begin{align}\label{eq:K-u-trafo-def-2}
\widetilde{K}(h):= \frac{\theta_n}{\vartheta_n} K \left( h' - \Big( \frac{\vartheta'}{\vartheta_n} - \frac{\theta'}{\theta_n} \Big) h_n, h_n \right). 
\end{align}
In particular, if $\vartheta_n  = (\theta_K)_n$ where $\theta_K$ was given in \eqref{eq:theta-K-def}, and if we choose
\begin{align}
\label{eq:theta-choice}
\theta = \left(\frac{1}{2(\theta_K)_n} \big( (\theta_K)' + \vartheta' \big) , 1\right),
\end{align}
then $\theta = \theta_{\widetilde{K}}$, and therefore, 
\begin{align*}
\int_{\R^n} \int_{\R^n} (u(x) - u(y)) (\eta(x) - \eta(y)) \widetilde{K}(x-y) (x_n)_+^{s-1} (y_n)_+^{s-1} \d y \d x = 0,
\end{align*}
for every $\eta \in C^{\infty}_c(\R^n)$.
\end{lemma}

\begin{proof}
For any $x \in \R^n$ let us define $\widetilde{x}:= \left( x' + \Big( \frac{\vartheta'}{\vartheta_n} - \frac{\theta'}{\theta_n} \Big) x_n , x_n \right)$. Let $\eta \in C^{\infty}_c(\R^n)$ and define $\widetilde{\eta}(x):=  \eta(\widetilde{x})$.
Then, by a change of variables
\begin{align}\label{laec}
L_K((x_n)_+^{s-1} \eta)(\widetilde{x}) = \frac{\vartheta_n}{\theta_n} L_{\widetilde{K}}((x_n)_+^{s-1} \widetilde{\eta})(x).
\end{align}
Moreover it is easy to check that
\begin{align*}
\theta \cdot \nabla \widetilde{\eta}(x) &= \theta' \cdot \nabla_{x'} \eta(\widetilde{x}) + \theta_n \Big( \frac{\vartheta'}{\vartheta_n} - \frac{\theta'}{\theta_n} \Big)\nabla_{x'} \eta(\widetilde{x}) + \theta_n \partial_n \eta(\widetilde{x}) \\
&= \frac{\theta_n}{\vartheta_n} \left( \vartheta' \cdot \nabla_{x'} \eta(\widetilde{x}) + \vartheta_n \partial_n \eta(\widetilde{x}) \right) = \frac{\theta_n}{\vartheta_n} \vartheta \cdot \nabla \eta(\widetilde{x}).
\end{align*}
Therefore, since $v$ is a solution to \eqref{eq:oblique-PDE}, by another change of variables, the fact that $x = \widetilde{x}$ on $\{ x_n = 0 \}$, and \eqref{laec} we get,
\begin{align*}
0 &= \frac{\theta_n}{\vartheta_n} \left(\int_{\R^n} (x_n)_+^{s-1} v L_K((x_n)_+^{s-1} \eta) \d x - \int_{\{ x_n = 0 \}} v \vartheta \cdot \nabla \eta \d x \right) \\
 &= \frac{\theta_n}{\vartheta_n} \left( \int_{\R^n} (x_n)_+^{s-1} v(\widetilde{x}) L_K((x_n)_+^{s-1} \eta)(\widetilde{x}) \d x - \frac{\vartheta_n}{\theta_n} \int_{\{ x_n = 0 \}} v(\widetilde{x}) \theta \cdot \nabla \widetilde{\eta}(\widetilde{x}) \d x  \right)\\
 & =  \int_{\R^n} (x_n)_+^{s-1} u L_{\widetilde{K}}((x_n)_+^{s-1} \widetilde{\eta})  \d x - \int_{\{ x_n = 0 \}} u \theta \cdot \nabla \widetilde{\eta} \d x.
\end{align*}
Let us now consider $\vartheta_n = (\theta_K)_n$. By \eqref{eq:theta-K-def} it holds
\begin{align*}
\theta_{\widetilde{K}} &= 2 c_s \int_{\R^{n-1}} (h',1) \widetilde{K}((h',1)) \d h' \\
&= 2 c_s \frac{\theta_n}{\vartheta_n} \int_{\R^{n-1}} \left( h'+ \Big( \frac{\vartheta'}{\vartheta_n} - \frac{\theta'}{\theta_n} \Big) , 1 \right) K(h',1) \d h' \\
&= \frac{\theta_n}{\vartheta_n} \left( (\theta_K)' +  \Big( \frac{\vartheta'}{\vartheta_n} - \frac{\theta'}{\theta_n} \Big) (\theta_K)_n , (\theta_K)_n \right)\\
&= \left( \frac{\theta_n}{(\theta_K)_n} (\theta_K)' +  \theta_n\Big( \frac{\vartheta'}{(\theta_K)_n} - \frac{\theta'}{\theta_n} \Big) , \theta_n \right), \qquad \theta\in\R^n.
\end{align*}
If we choose $\theta = \left(\frac{1}{2 (\theta_K)_n} \big( (\theta_K)' + \vartheta' \big) , 1\right)$ in the previous expression, then we obtain, as desired,
\begin{align*}
\theta_{\widetilde{K}}=\left(\frac{1}{2 (\theta_K)_n} \big( (\theta_K)' + \vartheta' \big) , 1\right)=\theta.
\end{align*}
\end{proof}

\section{Weighted Liouville theorems}

The goal of this section is to prove a weighted Liouville theorem for weak solutions defined in \autoref{def:weak-sol-ti}. The following is our main result:

\begin{proposition}
\label{prop:weighted-Liouville}
Let $\beta \in [s,1+s]$. Let $K$ be as in \eqref{eq:K-comp}, and assume that $K \in C^{5+\alpha}(\mathbb{S}^{n-1})$ for some $\alpha \in (0,2s)$ in case $\beta =s$. Let $v \in C^{\gamma}_{loc}(\R^n)$ for some $\gamma \in (\max\{0,2s-1, s-\beta+\frac{1}{2}\}, 2s-\beta+1)$, be such that $|v(x)| \le C (1 + |x|^{\gamma})$ and assume that $v$ is a weak solution to
\begin{align*}
\begin{cases}
\mathcal{L}\big((x_n)_+^{\beta-1}(y_n)_+^{\beta-1}K(x-y)\big)(v) &= 0 ~~ \text{ in } \{ x_n > 0 \},\\
\partial_n v &= 0 ~~ \text{ on } \{ x_n = 0 \}
\end{cases}
\end{align*}
in the sense of \autoref{def:weak-sol-ti}. Then, $v$ is constant.
\end{proposition}

\begin{remark}
\label{rem:gamma}
Note that the restriction $\gamma \in (\max\{0,2s-1,s-\beta+\frac{1}{2}\},2s-\beta+1)$ is required in order for the weak solution concept in \eqref{eq:weighted-sol} to be well-defined. The condition $\gamma > s-\beta+\frac{1}{2}$ is needed in order to allow testing with the solution itself (see \autoref{elrem}), which is used in the proof of  \autoref{lemma:C-alpha-estimate}.
\end{remark}

To prove the result, we distinguish between the cases $\beta = s$ and $\beta \in (s,1+s]$. This is due to the appearance of boundary terms in the integration by parts formula \autoref{lemma:ibp}.

\subsection{A H\"older estimate up to the boundary}

As a first important step, in this subsection we establish a $C^{\alpha}$ boundary estimate that has its own interest for all $\beta \in [s,1+s]$. We notice that, since the proof of \autoref{prop:weighted-Liouville} for $\beta = 1+s$ is immediate by \cite[Theorem 2.7.2]{FeRo24} and \eqref{eq:ibp-energy-obstacle}, this boundary result is not strictly necessary in the case $\beta = 1+s$. However, since we aim to provide a unified proof of \autoref{prop:weighted-Liouville} for all $\beta \in (s,1+s]$, we include the value $\beta=1+s$ in the entire subsection.

Throughout this section we assume that $2s < n$. Note that this is no restriction, since we will only apply \autoref{lemma:C-alpha-estimate} in case $n \ge 2$.

\begin{proposition}
\label{lemma:C-alpha-estimate}
Let $\beta \in [s,1+s]$ and $u \in C^{\gamma}_{loc}(\R^n)$ for some $\gamma \in (\max\{0,2s-1, s-\beta+\frac{1}{2}\}, 2s-\beta+1)$ be such that $|u(x)| \le C (1 + |x|^{\gamma})$ and
\begin{align}
\label{eq:weighted-sol}
\int_{\R^n} \int_{\R^n} (u(x) - u(y)) (\eta(x) - \eta(y)) K(x-y) (x_n)_+^{\beta-1} (y_n)_+^{\beta-1} \d y \d x = 0
\end{align}
for every $\eta \in C^{\infty}_c(\R^n)$. Then, there exists $\alpha \in (0,\gamma]$ such that for any $R > 0$,
\begin{align*}
[u]_{C^{\alpha}(\{ x_n \ge 0 \} \cap B_R)} \le C R^{-\alpha} & \Big( R^{-n-1+\beta}\Vert u (x_n)_+^{\beta-1} \Vert_{L^1(\{ x_n \ge 0 \} \cap B_{R})}\\
&\qquad  + R^{-\beta+1+2s} \Vert (x_n)_+^{\beta-1} |x|^{-n-2s} u \Vert_{L^1(\{ x_n \ge 0 \}  \setminus B_{R})} \Big),
\end{align*}
where $C$ and $\alpha$ only depend on $n,s,\lambda,\Lambda,\beta$.
\end{proposition}

Let us also refer to the recent article \cite{ChKi25}, where $C^{\alpha}$ boundary regularity has been investigated for a class of regional nonlocal problems in generals domains, allowing for weights that degenerate of explode at the boundary. However, note that the class of weights that we consider in this article are of a slightly different form compared to the ones in \cite{ChKi25} as they might violate for instance the assumption \cite[(2.5)]{ChKi25}.

\begin{remark}\label{elrem}
It is important to bear in mind that, for the proof of \autoref{lemma:C-alpha-estimate} to work (see the test functions used in proof of \autoref{lemma:Caccioppoli} and \autoref{lemma:log-estimate}), given a weak solution \eqref{eq:weighted-sol} in the sense of \autoref{def:weak-sol-ti}, it suffices to have that for any $\tau \in C^{\infty}_c(\R^n)$,  $k\in\R$, and $d > 0$,
\begin{align*}
\left| \int_{\R^n} \int_{\R^n} (u(x) - u(y))(\tau^2 (u-k)_{+}(x) - \tau^2 (u-k)_{+}(y)) (x_n)_+^{\beta-1}(y_n)_+^{\beta-1} K(x-y) \d y \d x \right| &< \infty,\\
\left| \int_{\R^n} \int_{\R^n} (u(x) - u(y))(\tau^2 (u+d)^{-1}_{+}(x) - \tau^2 (u+d)^{-1}_{+}(y)) (x_n)_+^{\beta-1}(y_n)_+^{\beta-1} K(x-y) \d y \d x \right| &< \infty.
\end{align*}
This is guaranteed already if $u \in C^{\gamma}_{loc}(\R^n)$ for some $\gamma \in (s - \beta + \frac{1}{2} , 2s-\beta+1)$.\\
In fact, since by the integration by parts given in \autoref{lemma:ibp}, weak solutions of \eqref{eq:weighted-sol} are distributional, we can apply the interior estimates obtained in \cite{FeRo24} to get $u \in C^{2s-\eps}_{loc}(\{ x_n > 0 \})$. Thus,
\begin{align*}
\int_{\supp(\tau)} \int_{B_{(x_n)_+/2}(x)} |u(x) - u(y)||\tau^2 u (x) - \tau^2 u(y)| (x_n)_+^{\beta-1}(y_n)_+^{\beta-1} K(x-y) \d y \d x < \infty.
\end{align*}
Hence, it is suffices to check
\begin{align}\label{xxt}
\int_{\supp(\tau)} \int_{\R^n \setminus B_{(x_n)_+/2}(x)} |u(x) - u(y)||\tau^2 u (x) - \tau^2 u(y)| (x_n)_+^{\beta-1}(y_n)_+^{\beta-1} K(x-y) \d y \d x < \infty.
\end{align}
To prove it we notice that
\begin{align*}
\int_{\supp(\tau)} &\int_{\R^n \setminus 2 \supp(\tau)} |u(x) - u(y)||\tau^2 u (x) - \tau^2 u(y)| (x_n)_+^{\beta-1}(y_n)_+^{\beta-1} K(x-y) \d y \d x \\
&\le \int_{\supp(\tau)} |u(x)|^2 (x_n)_+^{\beta-1} \left( \int_{\R^n \setminus 2 \supp(\tau)} |x-y|^{-n-2s}  (y_n)_+^{\beta-1} \d y \right) \d x \\
&\quad + \int_{\supp(\tau)} |u(x)| (x_n)_+^{\beta-1} \left( \int_{\R^n \setminus 2 \supp(\tau)} |u(y)|  |x-y|^{-n-2s}  (y_n)_+^{\beta-1} \d y \right) \d x \\
&\le C \int_{\supp(\tau)} (x_n)_+^{\beta-1} \d x + C \int_{\supp(\tau)} (x_n)_+^{\beta-1} \left( \int_{\R^n \setminus 2 \supp(\tau)} (y_n)_+^{\beta-1} |y|^{-n-2s+\gamma} \d y \right) \d x < \infty, 
\end{align*}
where we used that $|x-y| \ge c |y|$,  \autoref{lemma:int-polar-coord} and $-2s+\gamma+\beta-1 < 0$. Moreover, if $\gamma\neq s - \frac{\beta}{2} + \frac{1}{2}$ by  \cite[Lemma B.2.4]{FeRo24}, we have
\begin{align*}
\int_{\supp(\tau)} &\int_{2 \supp(\tau) \setminus B_{(x_n)_+/2}(x)} |u(x) - u(y)|^2 (x_n)_+^{\beta-1}(y_n)_+^{\beta-1} K(x-y) \d y \d x \\
&\le C \int_{\supp(\tau)}  (x_n)_+^{\beta-1} \left( \int_{2 \supp(\tau) \setminus B_{(x_n)_+/2}(x)} (y_n)_+^{\beta-1} |x-y|^{-n-2s+2\gamma} \d y \right) \d x \\
&\le C \int_{\supp(\tau)}  (x_n)_+^{\beta-1} (1 + (x_n)_+^{\beta-1-2s+2\gamma}) \d x < \infty,
\end{align*}
since $2\beta-1-2s+2\gamma > 0$. In case $\gamma= s - \frac{\beta}{2} + \frac{1}{2}$, since $\beta>0$, we get
\begin{align*}
\int_{\supp(\tau)} &\int_{2 \supp(\tau) \setminus B_{(x_n)_+/2}(x)} |u(x) - u(y)|^2 (x_n)_+^{\beta-1}(y_n)_+^{\beta-1} K(x-y) \d y \d x \\
&\le C \int_{\supp(\tau)}  (x_n)_+^{\beta-1} (1 + |\log (x_n)_+|) \d x < \infty.
\end{align*}
By a combination of the last three estimates we deduce \eqref{xxt}, as desired.
\end{remark}

We prove \autoref{lemma:C-alpha-estimate} by a De Giorgi iteration for nonlocal energies.
Special care is required due to the weights within the energy form. Moreover, we define the following quantity, which takes care of the long-range interactions in the equation \eqref{eq:weighted-sol}, given a function $v : \R^n \to \R$, and $r > 0$:
\begin{align*}
\tail(v;r) = r^{-\beta + 1 + 2s} \int_{\R^n \setminus B_{r}} |v(y)| (y_n)_+^{\beta-1} |y|^{-n-2s} \d y.
\end{align*}

The following lemma will be crucial when estimating the action of cut-off functions on the weighted energy. We will postpone its proof to the appendix.

\begin{lemma}
\label{lemma:cutoff-estimate}
Let $\beta \in [s,1+s]$ and $r > 0$. Then, for any $x \in B_r$ it holds
\begin{align*}
 \int_{\R^n} (y_n)_+^{\beta-1} \min\{1 , r^{-2} |x-y|^2 \} |x-y|^{-n-2s} \d y \le C  r^{\beta-1-2s},
\end{align*}
where $C$ only depends on $n,s,\beta$.
\end{lemma}

Moreover, we have the following Poincar\'e- and Poincar\'e-Sobolev inequalities with weights.

\begin{lemma}
Let $\beta \in [s,1+s]$ and $r > 0$. Then, for any $u \in L^2(B_r)$ it holds
\begin{align}
\begin{split}
\label{eq:Poincare}
&\int_{B_r} \big( u(x) - (u)_{B_r} \big)^2 (x_n)_+^{\beta-1} \d x \\
&\qquad \le C r^{1-\beta+2s} \int_{B_r} \int_{B_r} ( u(x) - u(y) )^2 (x_n)_+^{\beta-1} (y_n)_+^{\beta-1} |x-y|^{-n-2s} \d y \d x,
\end{split}
\end{align}
where $C$ only depends on $n,s,\beta$.
Moreover, for some $\kappa \in (1,2)$:
\begin{align}
\begin{split}
\label{eq:Sobolev}
&\left( r^{-(n-1+\beta)} \int_{B_r}  \big| u(x) - (u)_{B_r} \big|^{2 \kappa} (x_n)_+^{\beta-1} \d x \right)^{1/\kappa} \\
&\qquad \le C r^{-(n-1+\beta)} r^{1-\beta+2s} \int_{B_r} \int_{B_r} ( u(x) - u(y) )^2 (x_n)_+^{\beta-1} (y_n)_+^{\beta-1} |x-y|^{-n-2s} \d y \d x.
\end{split}
\end{align}
Here, $C$ only depends on $n,s,\beta$, and
\begin{align*}
\kappa = 
\begin{cases}
\frac{n}{n-s}, \qquad ~~~ \text{ if } \beta \in [s,1],\\
\frac{n}{n-\frac{1-\beta+2s}{2}}, ~~ \text{ if } \beta \in (1,1+s],
\end{cases}
\end{align*}
and we denote
\begin{align*}
(u)_{B_r} = \left(\int_{B_r} (x_n)_+^{\beta-1} \d x \right)^{-1} \int_{B_r} u(x) (x_n)_+^{\beta-1} \d x = C r^{-(n-1+\beta)} \int_{B_r} u(x) (x_n)_+^{\beta-1} \d x.
\end{align*}
\end{lemma}

\begin{proof}
Since $(x_n)_+^{\beta-1}, (x_n)_+^{1-\beta} \in L^1(B_r)$, we can apply \cite[Theorem 2.10]{BDOR24} with $p=2$ and $w(x) = (x_n)_+^{\beta-1}$, and obtain
\begin{align}
\label{eq:Poinc-BDOR}
\begin{split}
&\int_{B_r} \big( u(x) - (u)_{B_r} \big)^2 (x_n)_+^{\beta-1} \d x \\
&\qquad \le C r^{2t} \int_{B_r} \int_{B_r} ( u(x) - u(y) )^2 (x_n)_+^{\beta-1} (y_n)_+^{\beta-1} w(B_{x,y})^{-1} |x-y|^{-n-2t} \d y \d x,
\end{split}
\end{align}
for any $t \in (0,1)$, where
\begin{align*}
w(B_{x,y}) = \dashint_{B_{\frac{1}{2}|x-y|}\big(\frac{x+y}{2}\big)} (z_n)_+^{\beta - 1} \d z.
\end{align*}
Let us first prove the desired result in case that $\beta \in [s,1]$. Then, we claim that for any $x,y \in B_r$,
\begin{align}
\label{eq:weight-claim}
w(B_{x,y}) \ge C r^{\beta - 1},
\end{align}
and observe that once the claim is proved, we immediately deduce the Poincar\'e inequality \eqref{eq:Poincare} from \eqref{eq:Poinc-BDOR} applied with $t = s$.

To prove \eqref{eq:weight-claim}, we observe that we can find a ball $B \subset B_{\frac{1}{2}|x-y|}\big(\frac{x+y}{2}\big)$ with $\mathrm{diam}(B) = \frac{|x-y|}{8}$ and $r \ge \dist(B , \{ x_n = 0 \}) \ge \frac{|x-y|}{16}$. This immediately allows us to deduce
\begin{align*}
w(B_{x,y}) \ge C \inf_{z \in B}(z_n)_+^{\beta-1} \ge C r^{\beta-1},
\end{align*}
as desired. 
Let us now assume that $\beta \in (1,1+s]$. In this case, we claim that for any $x,y \in B_r$ it holds
\begin{align}
\label{eq:weight-claim-2}
w(B_{x,y}) \ge C |x-y|^{\beta - 1}.
\end{align}
Note that once \eqref{eq:weight-claim-2} is proved, the Poincar\'e inequality follows by applying \eqref{eq:Poinc-BDOR} with $t = \frac{1-\beta+2s}{2} \in (0,1)$.
To prove \eqref{eq:weight-claim-2}, we use the same ball $B$ as above and deduce
\begin{align*}
w(B_{x,y}) \ge C \inf_{z \in B} (z_n)_+^{\beta-1} \ge C |x-y|^{\beta-1}.
\end{align*}

The proof of the Poincar\'e-Sobolev inequality \eqref{eq:Sobolev} goes in the same way, but by applying instead \cite[Theorem 2.19]{BDOR24}, which yields
\begin{align}
\label{eq:Sob-BDOR}
\begin{split}
&\left( r^{-(n-1+\beta)} \int_{B_r}  \big| u(x) - (u)_{B_r} \big|^{\frac{2n}{n-s}} (x_n)_+^{\beta-1} \d x \right)^{\frac{n-t}{n}} \\
&\qquad \le C r^{-(n-1+\beta)} r^{2t} \int_{B_r} \int_{B_r} ( u(x) - u(y) )^2 (x_n)_+^{\beta-1} (y_n)_+^{\beta-1} w(B_{x,y})^{-1} |x-y|^{-n-2t} \d y \d x
\end{split}
\end{align}
for any $t \in (0,1)$, where we computed
\begin{align*}
w(B_r) = \int_{B_r} (x_n)_+^{\beta - 1} \d x = C r^{n-1+\beta}.
\end{align*}
In case $\beta \in [s,1]$, the proof follows, exactly as for the Poincar\'e inequality, by using \eqref{eq:weight-claim} and applying \eqref{eq:Sob-BDOR} with $t=s$. In case $\beta \in (1,1+s]$, a combination of \eqref{eq:weight-claim-2} and \eqref{eq:Sob-BDOR} with $t = \frac{1-\beta+2s}{2}$  yields the desired result.
\end{proof}

The last preparatory lemma is the following Caccioppoli inequality for nonlocal weighted energies. 
\begin{lemma}
\label{lemma:Caccioppoli}
Let $\beta \in [s,1+s]$. Let $u$ be a weak subsolution to \eqref{eq:weighted-sol}. Then, for any $k \in \R$, $r > 0$ and $\tau \in C^{\infty}_c(B_r)$ with $0 \le \tau \le 1$ we have
\begin{align*}
&\int_{B_r} \int_{B_r} | u_k(x) - u_k(y) |^2 \min\{\tau^2(x) , \tau^2(y)\} (x_n)_+^{\beta-1} (y_n)_+^{\beta-1} |x-y|^{-n-2s} \d y \d x \\
&\qquad \le C \int_{B_r} u_k^2(x) (x_n)_+^{\beta-1} \left( \int_{B_r} | \tau(x) - \tau(y) |^2  (y_n)_+^{\beta-1} |x-y|^{-n-2s} \d y \right) \d x \\
&\qquad\quad + C \left(\int_{B_r} u_k(x) \tau^2(x) (x_n)_+^{\beta-1} \d x \right) \left( \sup_{x \in \supp(\tau)} \int_{\R^n \setminus B_r} u_k(y) (y_n)_+^{\beta-1} |x-y|^{-n-2s} \d y \right),
\end{align*}
where $u_k := (u-k)_+$, and $C$ only depends on $\lambda,\Lambda$.
\end{lemma}

The proof is based on testing \autoref{def:weak-sol-ti} with the admissible function $\tau^2u_k$ (see \autoref{elrem}), exactly as in the non-weighted case. Thus we skip it and refer the interested reader to \cite{DKP16}.

We introduce the measure $\mu = (x_n)_+^{\beta-1} \d x$ and observe that for any $R > 0$ it holds $\mu(B_R) = C R^{n-1+\beta}$.

Having at hand the Poincar\'e-Sobolev inequality in \eqref{eq:Sobolev},  the Caccioppoli inequality in \autoref{lemma:Caccioppoli}, and the integral estimate in \autoref{lemma:cutoff-estimate}, we are in a position to prove that solutions to \eqref{eq:weighted-sol} are locally bounded at $\{ x_n = 0 \}$. The proof in the unweighted case is by now standard. For weighted nonlocal energies, no substantial adaptations have to be made, and we refer to \cite[Proof of Theorem 3.4]{BDOR24} for a detailed proof. Due to the slightly different shape of the weights in our framework, we cannot directly apply their result, however, all modifications are rather straightforward. 

\begin{lemma}
\label{lemma:loc-bd}
Let $\beta \in [s,1+s]$. Let $u$ be a weak subsolution to \eqref{eq:weighted-sol}. Then, for any $r > 0$ and $\delta \in (0,1)$, 
\begin{align*}
\sup_{B_{r/2}} u_+ \le C(\delta)\left( r^{-(n-1+\beta)} \int_{B_r} u_+^2 (x_n)_+^{\beta-1} \d x \right)^{1/2} + \delta \tail(u_+;r/2),
\end{align*}
where $C(\delta)$ only depends on $n,s,\beta,\delta,\lambda,\Lambda$. 
In particular, if $u$ is a weak solution to \eqref{eq:weighted-sol}, then for any $r > 0$ and $\delta \in (0,1)$, 
\begin{align*}
\sup_{B_{r/2}} |u| \le C(\delta) r^{-(n-1+\beta)} \int_{B_r} |u| (x_n)_+^{\beta-1} \d x  + \delta \tail(u;r/2),
\end{align*}
where $C(\delta)$ only depends on $n,s,\beta,\delta,\lambda,\Lambda$. 
\end{lemma}

\begin{proof}
We start by proving the first claim.
Let $k > 0$ to be determined at the end of the proof. For any $j \in N$, we introduce
\begin{align*}
r_j = (1 + 2^{-j})\frac{r}{2}, \qquad \widetilde{r}_j = \frac{r_j + r_{j+1}}{2}, \qquad B_j = B_{r_j}, \qquad \widetilde{B}_j = B_{\widetilde{r}_j}, \\
k_j = (1 - 2^{-j}) k, \qquad \widetilde{k}_j = \frac{k_j + k_{j+1}}{2}, \qquad u_j = (u-k_j)_+, \qquad \widetilde{u}_j = (u - \widetilde{k}_j)_+.
\end{align*}
We take a cut-off function $\tau_j \in C^{\infty}_c(\widetilde{B}_j)$ with $0 \le \tau_j \le 1$, and $\tau_j \equiv 1$ in $B_{j+1}$ and $|\nabla \tau_j| \le C 2^j r^{-1}$. 
Hence, by the Caccioppoli inequality in \autoref{lemma:Caccioppoli}, we deduce
\begin{align*}
& \dashint_{B_{j+1}}\int_{B_{j+1}} |\widetilde{u}_j(x) - \widetilde{u}_j(y)|^2 |x-y|^{-n-2s} \d\mu(y) \d\mu(x) \\
&\qquad \le C \dashint_{B_{j+1}} \widetilde{u}_j^2(x) \left( \int_{B_{j+1}} | \tau_j(x) - \tau_j(y) |^2 |x-y|^{-n-2s} \d \mu(y) \right) \d\mu(x) \\
&\qquad\quad + C \left(\dashint_{B_{j+1}} \widetilde{u}_j(x) \tau_j^2(x)  \d \mu(x) \right) \left( \sup_{x \in \supp(\tau_j)} \int_{\R^n \setminus B_{j+1}} \widetilde{u}_j(y)  |x-y|^{-n-2s} \d\mu(y) \right) \\
&\qquad = C I_1 + C I_2 I_3.
\end{align*}
For $I_1$, we use that $|\tau_j(x) - \tau_j(y)|^2 \le C 2^{2j} r^{-2}|x-y|^2 $ and \autoref{lemma:cutoff-estimate} to deduce
\begin{align*}
I_1 \le C 2^{2j} r^{\beta-1-2s} \dashint_{B_j} \widetilde{u}_j^2 \d\mu.
\end{align*}
For $I_2$ we proceed exactly as in \cite{BDOR24} to deduce
\begin{align*}
I_2 \le C 2^{j+2} k^{-1} \dashint_{B_j} u_j^2 \d\mu,
\end{align*}
and for $I_3$, we immediately obtain from the inequality $|x| \le C 2^j |x-y|$,
\begin{align*}
I_3 \le C 2^{(n+2s)j} r^{\beta-1-2s} \tail(u_+;r/2).
\end{align*}
Combining the previous four estimates with the Poincar\'e-Sobolev inequality from \eqref{eq:Sobolev}, we deduce
\begin{align*}
&\left( \dashint_{B_{j+1}} \widetilde{u}_j^{2\kappa} \d \mu \right)^{\frac{1}{\kappa}} \le C\left( \dashint_{B_{j+1}} |\widetilde{u}_j - (\widetilde{u}_j)_{B_{j+1}}|^{2\kappa} \d \mu \right)^{\frac{1}{\kappa}} + C \left( \dashint_{B_{j+1}} \widetilde{u}_j\d \mu \right)^{\frac{1}{2}} \\
&\qquad \le C r^{1-\beta+2s} \dashint_{B_{j+1}} \int_{B_{j+1}} |\widetilde{u}_j(x) - \widetilde{u}_j(y)|^2 |x-y|^{-n-2s} \d \mu(x) \d \mu(y) + C \dashint_{B_{j+1}} \widetilde{u}_j^2 \d \mu \\
&\qquad \le C \big(2^{2j} + 2^{(n+2s+1)j} k^{-1} \tail(u_+;r/2) \big) \dashint_{B_{j+1}} \widetilde{u}_j^2 \d \mu,
\end{align*}
where $\kappa \in (1,2)$ is the parameter from \eqref{eq:Sobolev}. Note that this estimate is the exact analog of \cite[(3.17)]{BDOR24}. From here, we can proceed in the exact same way as in \cite[Proof of Theorem 3.4]{BDOR24} to deduce the first claim, upon choosing $k$ suitably.

The second claim follows by applying the first claim to $u$ and $-u$, which are both subsolutions and using a standard covering and absorption argument to reduce the exponent in the $L^2(B_r)$ norm to an $L^1(B_r)$ norm (see \cite[Theorem 6.9]{Coz17}, or \cite[Theorem 6.2]{KaWe24}).
\end{proof}

The next lemma contains a logarithmic estimate. Its proof is standard (see for instance \cite{DKP16}) and the only deviation from the proof of \cite{DKP16} is due to the appearance of the weight $(x_n)_+^{\beta-1}$, which is taken care of by \autoref{lemma:cutoff-estimate}. We closely follow the arguments in \cite[Proposition 3.10]{BDOR24} and only provide a sketch of the proof.

\begin{lemma}
\label{lemma:log-estimate}
Let $\beta \in [s,1+s]$. Let $u$ be a weak supersolution to \eqref{eq:weighted-sol} with $u \ge 0 $ in $B_R \cap \{ x_n > 0 \}$. Then, for any $d > 0$ and $0 < r \le R/2$ it holds
\begin{align*}
& r^{-(n-1+\beta)} \int_{B_r} \int_{B_r} |\log(u(x) + d) - \log(u(y) +d)|^2 (x_n)_+^{\beta-1} (y_n)_+^{\beta-1} |x-y|^{-n-2s} \d y \d x \\
&\qquad \le C r^{\beta-1-2s}  \left( 1 + d^{-1} \left( \frac{r}{R} \right)^{2s+1-\beta} \tail(u_-;R) \right),
\end{align*}
where $C$ only depends on $n,s,\beta,\lambda,\Lambda$.
\end{lemma}

\begin{proof}
As in \cite[Proposition 3.10]{BDOR24}, we write $v = u + d$, take a cut-off function $\tau \in C^{\infty}_c(B_{3r/2})$ with $0 \le \tau \le 1$, $|\nabla \tau| \le 4/r$ and $\tau \equiv 1$ in $B_r$, and use $v^{-1} \tau^2$ as a test function to deduce
\begin{align*}
0 &\le C \int_{B_{2r}}\int_{B_{2r}} (v(x) - v(y)) (v^{-1}(x) \tau^2(x) - v^{-1}(y) \tau^2(y)) |x-y|^{-n-2s} \d \mu(x) \d \mu(y) \\
&\quad + 2 \int_{\R^n \setminus B_{2r}} \int_{B_{2r}} (u(x) - u(y)) v^{-1}(x) \tau^2(x) K(x-y) \d \mu(x) \d \mu(y) \\
&= I_1 + 2 I_2.
\end{align*}
By the same arguments as in \cite[(3.24)]{BDOR24}, and using also \autoref{lemma:cutoff-estimate} to estimate
\begin{align*}
r^{-2} \int_{B_{2r}} \int_{B_{4r}(x)} |x-y|^{-n+ 2-2s} \d \mu(x) \d \mu(y) &\le C \int_{B_{2r}} (x_n)_+^{\beta-1} r^{\beta-1-2s} \d x \le C r^{n -2 + 2\beta-2s},
\end{align*}
we deduce for $I_1$,
\begin{align*}
I_1 &\le -C \int_{B_r} \int_{B_r} |\log v(x) - \log v(y)|^2 |x-y|^{-n-2s} \d \mu(x) \d \mu(y) + C r^{n-2 + 2\beta-2s}.
\end{align*}
For $I_2$, we have, again by \autoref{lemma:cutoff-estimate},
\begin{align*}
I_2 &\le C \int_{B_{3r/2}} (x_n)_+^{\beta-1} \left( \int_{\R^n \setminus B_{r/2}(x)} (y_n)_+^{\beta-1} |x-y|^{-n-2s} \d y \right) \d x \\
&\quad + C d^{-1} \int_{B_{3r/2}} (x_n)_+^{\beta-1}  \left( \int_{\R^n \setminus B_R}  u_-(y) (y_n)_+^{\beta-1} |x-y|^{-n-2s} \d y \right) \d x \\
&\le C r^{\beta-1-2s} \int_{B_{3r/2}} (x_n)_+^{\beta-1} \d x + C d^{-1} R^{\beta-2s-1} \tail(u_-,R) \left( \int_{B_{3r/2}} (x_n)_+^{\beta-1} \d x\right) \\
&\le C r^{n-2 + 2\beta - 2s} \left( 1 + d^{-1} \left( \frac{r}{R} \right)^{2s+1-\beta} \tail(u_-;R) \right).
\end{align*}
Combining the estimates for $I_1$ and $I_2$, we conclude the proof.
\end{proof}

Having at hand the logarithmic estimate,  we can show the following growth lemma, in parallel to \cite[Lemma 3.11]{BDOR24}.

\begin{lemma}
\label{lemma:growth}
Let $\beta \in [s,1+s]$. Let $u$ be a weak supersolution to \eqref{eq:weighted-sol} with $u \ge 0$ in $B_R \cap \{ x_n > 0 \}$. Let $r \le R \le 4r$ and $\nu > 0$. Then, for any $\sigma \in (0,1]$ there is $\eps \in (0,\frac{1}{2})$, depending only on $n,s,\beta,\lambda,\Lambda,\sigma$, such that if
\begin{align*}
\mu( B_{2r} \cap \{ u \ge \nu\}) \ge \sigma \mu(B_{2r}), \qquad d:= \left( \frac{r}{R} \right)^{2s+1-\beta} \tail(u_- ; R) \le \eps \nu,
\end{align*}
then it holds
\begin{align*}
\inf_{B_r} u \ge \eps \nu.
\end{align*}
\end{lemma}

\begin{proof}
First, we claim that for any $\eps \in (0,\frac{1}{4})$ it holds
\begin{align}
\label{eq:growth-lemma-claim-1}
\mu(B_{2r} \cap \{ u \le 2 \eps \nu \}) \le \frac{C}{\sigma \log(1/4\eps)}. 
\end{align}
Setting 
\begin{align*}
h(x) = \min \left\{ \left[ \frac{\nu + d}{u(x) + d} \right]_+ , -\log(3\eps) \right\},
\end{align*}
we obtain by the same arguments as in \cite{BDOR24}, and using the Poincar\'e inequality in \eqref{eq:Poincare}, as well as the logarithmic estimate in \autoref{lemma:log-estimate}, recalling that $\mu(B_r) = C r^{n-1+\beta}$,
\begin{align*}
& \dashint_{B_{2r}} |h - (h)_{B_2r}|^2 \d \mu \\
&\qquad \le C r^{1-\beta+2s} \dashint_{B_{2r}} \int_{B_{2r}} |\log(u(x) + d) - \log(u(y)+d)|^2 |x-y|^{-n-2s} \d \mu(x) \d \mu(y)\\
&\qquad \le C  \left( 1 + d^{-1} \left( \frac{r}{R} \right)^{2s+1-\beta} \tail(u_-;R) \right) \le  C.
\end{align*}
From this estimate, we deduce \eqref{eq:growth-lemma-claim-1} by following \cite[Step 1 in proof of Lemma 3.11]{BDOR24}.

Using \eqref{eq:growth-lemma-claim-1}, we can deduce the desired result in the exact same way as in \cite[Step 2 in proof of Lemma 3.11]{BDOR24} by a De Giorgi iteration argument over the measures of level sets 
\begin{align*}
A_j = \frac{\mu(B_{(1+2^{-j})r} \cap \{ u < (1 + 2^{-j})\eps\nu \})}{\mu(B_{(1+2^{-j})r})},
\end{align*}
based on the Poincar\'e-Sobolev inequality from \eqref{eq:Sobolev}. Here, \eqref{eq:growth-lemma-claim-1} serves as a means to choose $A_0$ small enough for the iterative scheme to converge. Since this procedure is standard in the nonlocal case, and no additional idea (compared to \autoref{lemma:loc-bd} and \cite{BDOR24}) is required due to the appearance of the weights, we skip all the details.
\end{proof}

We are now in a position to prove the boundary $C^{\alpha}$ estimate in \autoref{lemma:C-alpha-estimate}.

\begin{proof}[Proof of \autoref{lemma:C-alpha-estimate}]
Given $r > 0$, let us define
\begin{align*}
\nu_0 = C  r^{-(n-1+\beta)} \int_{B_r^{+}} |u| (x_n)_+^{\beta-1} \d x + C \tail(u;r/2),
\end{align*}
where $C/2$ is the constant from \autoref{lemma:loc-bd} with $\delta = 1$.
We claim that there exist $\alpha \in (0,1)$ and $\eta \in (0,1)$ such that for any $j \in \N \cup \{0\}$ it holds
\begin{align}
\label{eq:osc-claim}
\osc_{B_{\eta^jr/2}} u \le \eta^{\alpha j} \nu_0.
\end{align}
To prove it, let us define
\begin{align*}
r_j = \eta^jr/2, \qquad B_j = B_{r_j}, \qquad \nu_j = \eta^{\alpha j} \nu_0, \qquad M_j = \sup_{B_j} u, \qquad m_j = \inf_{B_j} u.
\end{align*}
We assume $\tau \in (0,1/4)$ and $\alpha < \frac{2s+1-\beta}{2} =: \alpha_0$, but will choose them even smaller at a later point of the proof, if necessary. We prove \eqref{eq:osc-claim} by induction. For $j=0$, it follows by \autoref{lemma:loc-bd}. We suppose now that \eqref{eq:osc-claim} holds true for all $j \le i$ and prove it for $j = i+1$. We distinguish between the cases
\begin{align*}
\mu(2B_{i+1} \cap \{ u - m_i \ge \nu_i/2 \}) \ge \frac{1}{2} \mu(2B_{i+1}), \qquad \mu(2B_{i+1} \cap \{ \nu_i - (u - m_i) \ge \nu_i/2 \}) \ge \frac{1}{2} \mu(2B_{i+1}),
\end{align*}
and define
\begin{align*}
u_i = u - m_i ~~ \text{ in Case 1}, \qquad u_i = \nu_i - (u - m_i) ~~ \text{ in Case 2} , \qquad d_i = \eta^{2s+1-\beta} \tail((u_i)_- ; r_i).
\end{align*}
In both cases, $u_i$ is a weak solution to \eqref{eq:weighted-sol}, which is nonnegative in $B_i$ and satisfies $|u_i| \le 2 \nu_j$ in $B_j$ for $j \le i$ and $|u_i| \le |u| + 2 \nu_0$ in $\R^n \setminus B_0$. Hence, we have by \autoref{lemma:int-polar-coord}
\begin{align*}
r_i^{\beta-1-2s} \tail((u_i)_- ; r_i) &\le \sum_{j = i}^i \int_{B_{j-1} \setminus B_j} (2 \nu_{j-1}) |x|^{-n-2s} \d \mu(x) + \int_{\R^n \setminus B_0} |u(x) + 2 \nu_0| |x|^{-n-2s} \d \mu(x) \\
&\le C \sum_{j=1}^i \nu_{j-1} r_j^{\beta-1-2s} + C r_0^{\beta-1-2s} \big(\tail(u_0;r_0) + \nu_0 \big) \le C \sum_{j=1}^i \nu_{j-1} r_j^{\beta-1-2s}.
\end{align*}
Hence, from the definitions of $r_i$ and $\nu_i$, we deduce
\begin{align*}
d_i \le C \nu_i \sum_{j=1}^i \eta^{(1+i-j)(2s+1-\beta-\alpha)} \le C \nu_i \sum_{j = 1}^i \eta^{j\alpha_0} \le C \nu_i  \frac{\eta^{\alpha_0}}{1 - \eta^{\alpha_0}}.
\end{align*}
By choosing $\eta$ so small that $d \le \frac{\eps \nu_i}{2}$, where $\eps$ is the constant in \autoref{lemma:growth} for $\sigma = \frac{1}{2}$, we can apply \autoref{lemma:growth} with $\nu := \frac{\nu_i}{2}$, $\sigma = \frac{1}{2}$, and $B_r = B_{i+1}$, $B_R = B_{i}$, and deduce $m_{i+1} = \inf_{B_{i+1}} u_{i+1} \ge \eps \nu_i/2$, which implies that 
\begin{align*}
M_{i+1} - m_{i+1} \le \left(1 - \frac{\eps}{2} \right) \nu_i = \left( 1 - \frac{\eps}{2}\right) \eta^{-\alpha} \nu_{i+1},
\end{align*}
and therefore \eqref{eq:osc-claim} follows by choosing $\alpha$ so small that $1-\frac{\eps}{2} \le \eta^{\alpha}$.

Having proved \eqref{eq:osc-claim}, we are now in a position to conclude the proof. Indeed, let $x,y \in B_{r/2}$ and define $k_0 \in \N$ as
\begin{align*}
k_0 = \inf \{k \in \N  : |x-y| \ge \eta^{-k} r/2 \}.
\end{align*}
Then, $|x-y| \le \eta^{-k_0 + 1}r/2$ and by \eqref{eq:osc-claim}, we deduce
\begin{align*}
\frac{|u(x) - u(y)|}{|x-y|^{\alpha}} &\le \eta^{k_0 \alpha} (r/2)^{\alpha}  \osc_{B_{\eta^{-k_0+1}r/2}} u \le C r^{-\alpha} \nu_0.
\end{align*}
Hence, we conclude the proof by recalling the definition of $\nu_0$.
\end{proof}

\subsection{Reduction to a 1D problem}

Thanks to the $C^{\alpha}$ boundary estimate from \autoref{lemma:C-alpha-estimate}, we get the following result. The proof is quite standard but we include it for completeness. 

\begin{lemma}
\label{lemma:reduction-to-1D}
Let $\beta \in [s,1+s]$ and $u \in C^{\gamma}_{loc}(\R^n)$ for some $\gamma \in (\max\{0,2s-1, s-\beta+\frac{1}{2}\}, 2s-\beta+1)$ be such that $|u(x)| \le C (1 + |x|^{\gamma})$. Moreover, assume that $u$ solves \eqref{eq:weighted-sol}. Then, $u(x) = U(x_n)$, where $U$ satisfies for any $\eta \in C^{\infty}_c(\R)$
\begin{align}
\label{eq:U-1D-PDE}
\int_{\R} \int_{\R} (U(x) - U(y)) (\eta(x) - \eta(y)) |x-y|^{-1-2s} x_+^{\beta-1} y_+^{\beta-1} \d y \d x = 0,
\end{align}
\end{lemma}

\begin{proof}
First, we notice that, once we prove $u(x) = U(x_n)$, then, by a computation similar to the one in \eqref{eq:integrate-out-kernel}, we can immediately deduce \eqref{eq:U-1D-PDE}. Indeed by considering $\eta(x)=\widetilde{\eta}(x')\varphi(x_n)\in C^{\infty}_c(\R^n)$, $\varphi\in C^{\infty}_c(\R)$, from \eqref{eq:weighted-sol} we get
\begin{align*}
    0=& \int_{\R^n} \int_{\R^n} (U(x_n) - U(y_n)) (\widetilde{\eta}(x')\varphi(x_n) - \widetilde{\eta}(y')\varphi(y_n)) K(x-y) (x_n)_+^{\beta-1} (y_n)_+^{\beta-1} \d y \d x\\
    &= \int_{\{x_n>0\}} \int_{\{y_n>0\}} F(x_n, y_n)\left(\int_{\R^{n-1}} \int_{\R^{n-1}} (\widetilde{\eta}(x') - \widetilde{\eta}(y')) K(x-y)\d y' \d x'\right) \d y_n \d x_n\\
    &+\int_{\{x_n>0\}} \int_{\{y_n>0\}} G(x_n, y_n)\left(\int_{\R^{n-1}} \int_{\R^{n-1}} \widetilde{\eta}(y') K(x-y)\d y' \d x'\right) \d y_n \d x_n:= I_1+I_2,
\end{align*}
where
\begin{align*}
F(x_n, y_n) &:=\varphi(x_n) (x_n)_+^{\beta-1} (y_n)_+^{\beta-1}(U(x_n) - U(y_n)),\\
G(x_n,y_n) &:=(x_n)_+^{\beta-1} (y_n)_+^{\beta-1}(U(x_n) - U(y_n)) (\varphi(x_n) - \varphi(y_n)).
\end{align*}
On the one hand, by symmetry, it is clear that $I_1=0$. On the other hand, by  the scaling properties of the kernel given in \eqref{eq:K-comp} and the computation in \eqref{eq:integrate-out-kernel}, it is easy to check that
\begin{align*}
\int_{\R^{n-1}}  K(x-y) dx'= C|x_n-y_n|^{1-2s}, \qquad C:= \int_{\R^{n-1}}K(w',1) \d w',
\end{align*}
and therefore
\begin{align*}
0 = I_2 = C\left(\int_{\R^{n-1}} \widetilde{\eta}(y') \d y'\right)\int_{\{x_n>0\}} \int_{\{y_n>0\}} G(x_n, y_n)\frac{\d y_n \ dx_n}{|x_n-y_n|^{1+2s}},
\end{align*}
which implies \eqref{eq:U-1D-PDE}, as desired.

In order to prove $u(x) = U(x_n)$, we observe that by  \autoref{lemma:C-alpha-estimate}, the growth of $u$, and \autoref{lemma:int-polar-coord},
\begin{align*}
[u]_{C^{\alpha}(\{ x_n \ge 0 \} \cap B_R)} &\le C R^{-\alpha}  \big ( R^{-n-1+\beta}\|u (x_n)_+^{\beta-1}\|_{L^{1}(\{ x_n \ge 0 \} \cap B_R))} \\
&\qquad +  R^{-\beta+1+2s}\| |x|^{-(n+2s)}(x_n)_+^{\beta-1} u\|_{L^{1}(\{ x_n \ge 0 \}  \setminus B_R)} \big) \le C R^{\gamma-\alpha}.
\end{align*}
We consider now the incremental quotient 
\begin{align*}
D_{h,1}^{\alpha} u(x):=\frac{u(x+he)-u(x)}{|h|^{\alpha}}, \qquad h=(h',0),
\end{align*}
which satisfies 
\begin{align*}
\int_{\R^n} \int_{\R^n} (D_{h,1}^{\alpha} u(x) - D_{h,1}^{\alpha} u(y)) (\eta(x) - \eta(y)) K(x-y) (x_n)_+^{\beta-1} (y_n)_+^{\beta-1} \d y \d x= 0, ~~ &\forall \eta \in C^{\infty}_c(\R^n),\\
 \Vert D_{h,1}^{\alpha} u \Vert_{L^{\infty}(\{ x_n \ge 0 \} \cap B_R)} \le C R^{\gamma-\alpha}, ~~ &\forall R > 1.
\end{align*}
That is, $D_{h,1}^{\alpha} u$ solves the same equation like $u$ with an improved growth condition. By considering
\begin{align*}
D_{h,k}^{\alpha} u(x):=\frac{D_{h,k-1}^{\alpha}u(x+he)-D_{h,k-1}^{\alpha}u(x)}{|h|^{\alpha}}, \qquad  h = (h',0),
\end{align*}
we can iterate the procedure and, by taking $R\to \infty$, we conclude that $D_{h,m+1}^{\alpha}u\equiv 0$ in $\R^{n} \cap \{ x_n > 0 \}$, where $m:=\lceil\gamma/\alpha\rceil-1$. That is, $u(\cdot,x_n)$ is a polynomial in $x'$ of degree $m$ for any $x_n > 0$. However, by the growth condition on $u$, and the fact that $\gamma<1$, this yields that $u$ is constant in $x'$, i.e. there exists $U : \R \to \R$ such that $U(x_n) = u(x',x_n)$, as desired.
\end{proof}

\subsection{Proof in case $\beta = s$}

We are now ready to prove \autoref{prop:weighted-Liouville} in case $\beta = s$.

\begin{proof}[Proof of \autoref{prop:weighted-Liouville} in case $\beta = s$]

Let us first assume that $v$ satisfies \eqref{eq:weighted-sol}. Then, by \autoref{lemma:reduction-to-1D}, there exists $U$ solving \eqref{eq:U-1D-PDE} such that $v(x',x_n) = U(x_n)$. By the integration by parts formula from \autoref{lemma:ibp}, applied only with $\eta \in C_{c}^{\infty}(\{ x_n > 0 \})$ we deduce that $V = x_+^{s-1} U$ satisfies $(-\Delta)^s_{\R} U = 0$ in $(0,\infty)$ in the distributional sense, and therefore
\begin{align*}
\begin{cases}
(-\Delta)^s_{\R}V &= 0 ~~ \text{ in } (0,\infty),\\
V &= 0 ~~ \text{ in } (-\infty,0),
\end{cases}
\end{align*}
and $|V(x)| \le C(x^{s-1} + x^{s-1+\gamma})$.
Hence, by the 1D Liouville theorem for the fractional Laplacian from \cite[Lemma 6.2]{RoSe16}, since $s-1+\gamma<2s$ and $s>0$, we deduce that $V(x) = a x_+^s + b x_+^{s-1}$ for some $a,b \in \R$. However, by the definition of $U$ this yields $U(x) = a x_+ + b$. By the growth $|U(x)| \le C (1 + |x|^{\gamma})$ for $\gamma < 1$, this yields $U \equiv b$, as desired.

Now let us prove the result for a general weak solution $v$ in the sense of \autoref{def:weak-sol-ti}. By the integration by parts formula from \autoref{lemma:ibp} we deduce that for any $\eta \in C^{\infty}_c(\R^n)$ it holds
\begin{align*}
    & \int_{\{ x_n = 0 \}} v(x',0) (\theta'_K \cdot \nabla_{x'}\eta)(x',0) \d x' \\
    &\quad = \int_{\R^n} \int_{\R^n} (v(x) - v(y))(\eta(x) - \eta(y)) (x_n)_+^{s-1}(y_n)_+^{s-1}K(x-y) \d y \d x \\
    & \quad = \int_{\R^n} (x_n)_+^{s-1} v L_K((x_n)_+^{s-1} \eta)  \d x - \int_{\{ x_n = 0 \}} v(x',0) (\theta_K \cdot \nabla\eta)(x',0)   \d x.
\end{align*}
Hence, we deduce
\begin{align*}
    \int_{\R^n} (x_n)_+^{s-1} v L_K((x_n)_+^{s-1} \eta)  \d x - \int_{\{ x_n = 0 \}} v(x',0) \vartheta \cdot \nabla \eta(x',0)  \d x = 0,
\end{align*}
where $\vartheta = \theta_K + (\theta_K',0)\in \R^n$. Since $\vartheta_n = (\theta_K)_n$, we can apply \autoref{lemma:oblique-trafo}, and by defining $\theta$ as in \eqref{eq:theta-choice}, we deduce that
\begin{align*}
    \int_{\R^n} \int_{\R^n} (u(x) - u(y)) (\eta(x) - \eta(y)) \widetilde{K}(x-y) (x_n)_+^{s-1} (y_n)_+^{s-1} \d y \d x = 0,
\end{align*}
for some $\widetilde{K}$ and $u$ defined in \eqref{eq:K-u-trafo-def} and \eqref{eq:K-u-trafo-def-2}, namely
\begin{align}
    u(x) = v(x' + \omega x_n , x_n)
\end{align}
for some $\omega \in \R^{n-1}$. Hence, by the first part of the proof, $u$ must be constant, and therefore, by construction, also $v$ is constant, as desired.
\end{proof}

\subsection{Proof in case $\beta \in (s,1+s]$}

In order to prove \autoref{prop:weighted-Liouville} in case $\beta \in (s,1+s]$, we need the following Liouville theorem in 1D.

\begin{lemma}
\label{lemma:1D-Liouville-beta}
    Let $\beta \in (s,1+s]$ and $x_+^{\beta-1}v$ be a distributional solution to 
    \begin{align}
    \label{eq:Liou-problem}
        (-\Delta)^s_{\R}(x_+^{\beta-1}v) &= f(\beta-1) x_+^{\beta-1-2s} v ~~ \text{ in } (0,\infty),
    \end{align}
    where $f(\beta-1) \in \R$ is such that $(-\Delta)^s_{\R}(x_+^{\beta-1}) = f(\beta-1) x_+^{\beta-1-2s}$, such that $|v(x)| \le C(1 + x^{\gamma})$ for some $\gamma \in (\max\{0,2s-1,s-\beta+\frac{1}{2}\},2s-\beta+1)$. Then, $v$ is of the form
    \begin{align*}
    v(x) = a +\1_{\{\beta \in (s,s+\frac{1}{2})\}} b x_+^{2s-2\beta+1}, \qquad a,b \in \R.
    \end{align*}
\end{lemma}

\begin{proof}
First, we need to prove that all solutions $v$ to \eqref{eq:Liou-problem} are linear combinations of homogeneous solutions. A homogeneous solution must be of the form $v(x) = a_{\alpha} x^{\alpha}$ for some $\alpha \in [0,2s-\beta+1)$ and $a_{\alpha} \in \R$. This can be achieved as in \cite{FaRo22}. Indeed, let us define $w(x) = x_+^{\beta-1} v(x)$ and observe that $|w(x)| \le C(1 + x^{2s-\eps})$ for some $\eps > 0$, and 
\begin{align*}
(-\Delta)^s_{\R} w(x) = f(\beta-1) x_+^{-2s} w(x).
\end{align*}
Hence, by analyzing the extended problem as in \cite[(2.3)]{FaRo22} (with $\bar{\kappa}_s$ replaced by $f(\beta-1)$), and observing that the proof of \cite[Lemma 2.2, Lemma 2.4]{FaRo22} work in the exact same way, in case $\bar{\kappa}_s$ is replaced by $f(\beta-1)$, we deduce that $w$ must be a linear combination of homogeneous solutions, and therefore, the same must be true for $v$.

Using now that  $(-\Delta)_{\R}^s x_+^{\alpha} = f(\alpha) x_+^{\alpha-2s}$ for any $\alpha \in (-1,2s)$, where
\begin{align*}
f(\alpha) = \frac{\Gamma(\alpha+1)}{\Gamma(\alpha-2s+1)} \frac{\sin(\pi(\alpha-s))}{\sin(\pi(\alpha-2s))},
\end{align*}
(see \cite[(2.8)]{FaRo22}) we obtain that $v(x) = x_+^{\alpha}$ is a solution of \eqref{eq:Liou-problem} if and only if
\begin{align}
\label{eq:f-fixed}
f(\beta-1 + \alpha) = f(\beta-1).
\end{align}

Clearly, this property is satisfied in case $\alpha = 0$. In order to find further solutions, a careful analysis of $f$ is required. It is easy to verify that $f$ restricted to the interval $[s-1,2s)$ has a global maximum at $a = s - \frac{1}{2}$, and zeros at $a = s$ and $a = s-1$. Moreover, $f$ is symmetric in the interval $[s-1,s]$, i.e.
\begin{align*}
f(a) = f(2s-1-a), ~~ a \in [s-1,s].
\end{align*}
Moreover, $f$ is strictly decreasing in $[s,2s)$ and $\lim_{a \to 2s} f(a) = -\infty$. Therefore, we obtain that all solutions to \eqref{eq:f-fixed} in the interval $\alpha \in [0,2s-\beta+1)$ are given by
\begin{align*}
\begin{cases}
\alpha \in \{ 0 , 2s - 2\beta + 1 \}, ~~ \text{ if } \beta \in \left( s , s + \frac{1}{2} \right),\\
\alpha = 0, ~~\qquad\qquad\qquad ~ \text{ if } \beta \in \left[ s + \frac{1}{2}, s + 1 \right]. 
\end{cases}
\end{align*}
This concludes the proof.
\end{proof}

\begin{proof}[Proof of \autoref{prop:weighted-Liouville} in case $\beta \not= s$]
    We proceed as in the first part of the proof for $\beta = s$. Indeed, by \autoref{lemma:reduction-to-1D} there exists $U$ solving \eqref{eq:U-1D-PDE} such that $v(x',x_n) = U(x_n)$. By the integration by parts formula from \autoref{lemma:ibp-beta} (see \eqref{eq:ibp-general-beta}), since $(-\Delta)^s_{\R} (x_+^{\beta-1}) = f(\beta-1) x_+^{\beta-1-2s}$, we deduce that $U$ satisfies for any $\eta \in C^{\infty}_c((0,\infty))$
    \begin{align*}
        0 &= \int_{\R} x_+^{\beta-1} U (-\Delta)^s_{\R}(x_+^{\beta-1}\eta) \d x -f(\beta-1) \int_{\R} (x_+^{\beta-1}\eta) U x_+^{\beta-1-2s}  \d x.
    \end{align*}
That is, that $x_+^{\beta-1} U$ satisfies the assumptions of \autoref{lemma:1D-Liouville-beta}, and therefore, 
  \begin{align*}
  U(x) = a + \1_{\{ \beta \in (s , s + \frac{1}{2})\}} b x_+^{2s-2\beta  +1} 
  \end{align*}
Note that we must have $b = 0$, since $x_+^{2s-2\beta+1}$ is not a weak solution in the sense of \autoref{def:weak-sol-ti}. In fact, taking $\eta \in C_c(\R)$ with $\eta \ge 0$, $\eta(0) = 1$, and $\eta' \le 0$ in $(0,\infty)$, we deduce
\begin{align*}
\int_{\R} & \int_{\R} (x_+^{2s-2\beta+1} - y_+^{2s-2\beta+1})(\eta(x) - \eta(y)) K(x-y) \d y \d x < 0,
\end{align*}
since for any $x,y \in (0,\infty)$ it holds $(\eta(x) - \eta(y))(x_+^{2s-2\beta+1} - y_+^{2s-2\beta+1}) \le 0$. Hence, $U$ is constant, and the proof is complete.
\end{proof}

\section{An a priori boundary H\"older estimate}\label{section:holder}

The following a priori $C^{\gamma}$ estimate is crucial in the application to the free boundary problems.

\begin{proposition}
\label{lemma:bdry-reg}
Let $\beta \in [s,1+s]$. Let $K$ be as in \eqref{eq:K-comp} and assume that $K \in C^{\sigma}(\mathbb{S}^{n-1})$ for some $\sigma > 0$. Moreover, assume that $K \in C^{5+\sigma}(\mathbb{S}^{n-1})$ if $\beta =s$. \\
Let $J \in C^{\alpha}(\{x_n \ge 0 \})$ for some $\alpha > 0$, and $I: \R^n \times \R^n \to \R$ be such that
\begin{align*}
J \ge \Lambda^{-1} \text{ in } \{ x_n > 0\} \cap B_{1/2}, \qquad &\Vert J \Vert_{C^{\alpha}(\{x_n \ge 0 \})} + \Vert K \Vert_{C^{\sigma}(\mathbb{S}^{n-1})} \le \Lambda, \qquad |I(x,y)| \le \Lambda |x-y|^{-n-2s},
\end{align*}
Moreover let $\Phi \in C^{1,\alpha}(\R^n)$ be as in Subsection \ref{subsec:flattening}, and let $u \in C_{loc}^{\gamma}(\R^n)$
satisfy for any $\eta \in C^{\infty}_c(B_1)$,
\begin{align*}
\int_{\R^n} &\int_{\R^n} (u(x) - u(y)) (\eta(x) - \eta(y)) (x_n)_+^{\beta-1}(y_n)_+^{\beta - 1} J(x) J(y) K(\Phi(x) - \Phi(y)) \d y \d x \\
&=\int_{\R^n} \eta(x) g(x) \d x + \int_{\R^n} \int_{\R^n} (f(x) - f(y)) (\eta(x) - \eta(y)) (x_n)_+^{\beta-1}(y_n)_+^{\beta-1} I(x,y) \d y \d x\\
&\quad + \1_{\{ \beta = s\}} \Bigg[ \int_{\{ x_n = 0 \}}  b(x')(\theta_{K})_n (x') \eta(x',0) \d x' \\
&\qquad\qquad\qquad\quad - \int_{\{ x_n = 0 \}} u(x',0) [(\theta_{K}' (x')\cdot \nabla_{x'} \eta(x',0)) +  \dvg \theta_{K}'(x') \eta(x',0)] \d x'\\
&\qquad\qquad\qquad\quad + \int_{\{ x_n = 0 \}} A(x') \cdot \vartheta(x') \eta(x',0)  \d x'\\
&\qquad\qquad\qquad\quad - \int_{\{ x_n = 0 \}}  a(x') [(\nu'(x')\cdot \nabla_{x'} \eta(x',0)) +  {\dvg \nu'(x') \eta(x',0)]} \d x'  \Bigg],
\end{align*}

where $\gamma \in (\max\{0,2s-1,s-\beta+\frac{1}{2}\},\min\{1,2s-\beta+1\})$ for some $\gamma \not= 2s$, and we assume $(x_n)_+^{1-\beta} g \in L^{\infty}(\{ x_n \ge 0 \} \cap B_1)$, and $f \in C^{\gamma}(\{ x_n \ge 0 \})$.
In case $\beta = s$, we also assume $b \in L^{\infty}(\{x_n = 0 \} \cap B_1)$, $A \in L^{\infty}(\{x_n = 0 \} \cap B_1)$, $a \in C^{\gamma}(\{ x_n = 0 \} \cap B_1)$, and 
\begin{align*}
\Vert \theta_{K} \Vert_{C^{0,1}(\{ x_n = 0 \} \cap B_1)} + \Vert \vartheta \Vert_{L^{\infty}(\{ x_n = 0 \} \cap B_1)} +  \Vert \nu \Vert_{C^{0,1}(\{ x_n = 0 \} \cap B_1)} \le \Lambda.
\end{align*}
Moreover, if $\gamma > 2s$, we assume
\begin{align}
\label{eq:kernel-int-reg-ass-main}
|I(x,z) - I(y,z)| \le \Lambda |x-y|^{\gamma-2s}  \min\{ & |z-x| , |z-y| \}^{-n-\gamma} ~~ \forall x,y \in \{ x_n > 0 \} \cap B_1, ~ z \in \R^n.
\end{align}
Then, it holds
\begin{align*}
[u]_{C^{\gamma}(\{ x_n \ge 0 \} \cap B_{1/2})} &\le C \Big( \Vert u \Vert_{L^{\infty}(\{ x_n \ge 0 \} \cap B_1)} + \Vert (x_n)_+^{\beta-1} u \Vert_{L^1_{2s}(\{ x_n > 0 \} \setminus B_1)} \\
&\qquad \qquad \quad  + \Vert (x_n)_+^{1-\beta} g \Vert_{L^{\infty}(\{ x_n \ge 0 \} \cap B_1) } +  \Vert f \Vert_{C^{\gamma}(\{ x_n \ge 0\})} \\
&\qquad \qquad \quad + \1_{\{ \beta = s\}} \big[  \Vert b \Vert_{L^{\infty}(\{x_n = 0 \} \cap B_1)} + \Vert A \Vert_{L^{\infty}(\{x_n = 0 \} \cap B_1)} + \Vert a \Vert_{C^{\gamma}(\{x_n = 0 \} \cap B_1)} \big] \Big),
\end{align*}
where $C > 0$ only depends on $n,s,\beta,\gamma,\alpha,\sigma,\lambda,\Lambda$, and the constant in \eqref{eq:Phi-comp}. 
\end{proposition}

For a discussion on the range of $\gamma$ in \autoref{lemma:bdry-reg}, let us refer to \autoref{rem:gamma}. Note that with a little more work, it is possible to replace the assumption $\gamma > \max\{0, 2s-1 \}$ by the weaker assumption $\gamma > 0$. This can be achieved by proving a regularity estimate in terms of the weighted space $C^{2s+\alpha}_{\gamma}(\{ x_n \ge 0 \} | B_{1/2})$. We do not pursue this direction here.

\begin{remark}
Note that the previous lemma includes in particular weak solutions to
\begin{align*}
\begin{cases}
\mathcal{L}\big(J(x)J(y)(x_n)_+^{\beta-1}(y_n)_+^{\beta-1} K(\Phi(x) - \Phi(y))\big)(u) &= g  ~~ \text{ in } \{ x_n > 0 \} \cap B_{1},\\
\partial_n u &= b ~~ \text{ on } \partial \{ x_n > 0 \} \cap B_1,
\end{cases}
\end{align*}
in the sense of \autoref{def:weak-sol-trafo}, which corresponds to the case $A=a=f=0$. The interest to include $A,a,f$ different from zero becomes apparent from the application of \autoref{lemma:bdry-reg} to the equation in \autoref{lemma:PDE-onephase-flat}(iii), which arises if source terms of the form
\begin{align*}
\mathcal{L}\big((x_n)_+^{\beta-1}(y_n)_+^{\beta-1} I(x,y)\big)(f)
\end{align*}
are integrated by parts, using \autoref{lemma:ibp-nonflat-trafo}.
\end{remark}

First, we prove an interior regularity estimate for the equation under consideration. We need two different results, depending on whether $\gamma < 2s$ or $\gamma > 2s$. While in case $\gamma <2s$, we give a proof using a blow-up argument, the case $\gamma > 2s$ follows from an application of \cite{FeRo24b}.

\begin{lemma}
\label{lemma:interior}
Let $\beta \in [s,1+s]$. Assume that $K \in C^{\sigma}(\mathbb{S}^{n-1})$ for some $\sigma > 0$ satisfies \eqref{eq:K-comp} and let $J \in C^{\alpha}(\{x_n \ge 0 \})$ for some $\alpha > 0$, and $I: \R^n \times \R^n \to \R$ be such that
\begin{align*}
J \ge \Lambda^{-1} \text{ in } \{ x_n > 0\} \cap B_{1}(e_n), \qquad \Vert J \Vert_{C^{\alpha}(\{x_n \ge 0 \})} + \Vert K \Vert_{C^{\sigma}(\mathbb{S}^{n-1})} \le \Lambda, \qquad |I(x,y)| \le \Lambda |x-y|^{-n-2s}.
\end{align*}
Moreover let $\Phi \in C^{1,\alpha}(\{x_n \ge 0 \} \cap B_{1}(e_n))$ be as in Subsection \ref{subsec:flattening}, and let $u \in C^{\gamma}_{loc}(\R^n)$ be a weak solution to 
\begin{align*}
\mathcal{L} & \big(J(x)J(y)(x_n)_+^{\beta-1}(y_n)_+^{\beta-1} K(\Phi(x) - \Phi(y))\big)(u) \\
&\qquad \qquad = g + \mathcal{L}\big((x_n)_+^{\beta-1}(y_n)_+^{\beta-1} I(x,y)\big)(f) ~~ \text{ in } B_{1}(e_n),
\end{align*}
where $\gamma \in (\max\{0,2s-1\} , \min\{1 , 2s, 2s-\beta+1\})$, and $(x_n)_+^{1-\beta} g \in L^{\infty}(B_{1}(e_n))$, and $f \in C^{\gamma}(\{ x_n \ge 0 \})$.
Then, the following holds true
\begin{align*}
[ u ]_{C^{\gamma}(B_{1/2}(e_n))} &\le C \Big( \Vert u \Vert_{L^{\infty}(B_{1}(e_n))} + \Vert (x_n)_+^{\beta-1} u |x-e_n |^{-n-2s}\Vert_{L^1(B_1(e_n)^c)}  \\
&\qquad \qquad \qquad + \Vert (x_n)_+^{1-\beta} g \Vert_{L^{\infty}(\{ x_n \ge 0 \}\cap B_{1}(e_n))} +  [f]_{C^{\gamma}(\{ x_n \ge 0\})} \Big),
\end{align*}
where $C > 0$ only depends on $n,s,\beta,\gamma,\alpha,\sigma,\lambda,\Lambda$, and the constant in \eqref{eq:Phi-comp}. 
\end{lemma}

By scaling it is easy to see that if $u$ solves a corresponding PDE in a ball $B_r(x_0)$ with $(x_0)_n = r$, then we have the following estimate (define $\widetilde{u}(x) = u((x_0)' + rx)$ and $\widetilde{f}, \widetilde{g}$ etc. analogously and apply the above result on scale one)
\begin{align}
\label{eq:interior-scaled}
\begin{split}
[ u ]_{C^{\gamma}(B_{r/2}(x_0))} &\le C r^{-\gamma} \Big( \Vert u \Vert_{L^{\infty}(B_r(x_0))} + r^{2s + 1 - \beta} \Vert (x_n)_+^{\beta-1} u |x-x_0|^{-n-2s} \Vert_{L^1(B_r(x_0)^c)} \\
&\qquad \qquad \qquad \qquad\qquad + r^{2s + 1 - \beta} \Vert (x_n)_+^{1-\beta} g \Vert_{L^{\infty}(B_r(x_0)) } + r^{\gamma} [f]_{C^{\gamma}(\{ x_n \ge 0\})}  \Big).
\end{split}
\end{align}
To prove \autoref{lemma:interior} we employ a blow-up argument in the spirit of \cite[Proposition 4.4]{FeRo24b}.

\begin{proof}
First of all we observe that it is enough to prove that for any $\delta \in (0,1)$ there exists $C_{\delta} > 0$ such that
\begin{align}
\label{eq:almost-reg-est-interior}
[u]_{C^{\gamma}(B_{1/4}(e_n))} \le \delta [u]_{C^{\gamma}(\{ x_n \ge 0\} )} + C_{\delta} \Big( \Vert u \Vert_{L^{\infty}(B_{1}(e_n))} + \Vert (x_n)_+^{1-\beta} g_k \Vert_{L^{\infty}(B_{1}(e_n)) } + [f]_{C^{\gamma}(\{ x_n \ge 0\})} \Big).
\end{align}
From here, the desired result follows by a standard covering argument (see \cite[Theorem 2.4.1]{FeRo24}), applying \eqref{eq:almost-reg-est-interior} to $\xi u$, where $\xi \in C^{\infty}_c(B_2(e_n))$ is a cut-off function satisfying $0 \le \xi \le 1$ and $\xi \equiv 1$ in $B_{1}(e_n)$, and using that
\begin{align*}
& \Vert (x_n)_+^{1-\beta}|\mathcal{L}\big(J(x)J(y) (x_n)_+^{\beta-1}(y_n)_+^{\beta-1} K(\Phi(x)-\Phi(y))\big)((1-\xi)u)\Vert_{L^{\infty}(\{ x_n > 0 \} \cap B_1(e_n))} \\
&\qquad \le C \int_{B_1(e_n)^c} (y_n)_+^{\beta-1} |u(y)| |y - e_n|^{-n-2s} \d y.
\end{align*}
The proof of \eqref{eq:almost-reg-est-interior} is split into several steps.

\textbf{Step 1:}
To prove \eqref{eq:almost-reg-est-interior}, we use \cite[Lemma 2.4.12]{FeRo24}, which yields that if \eqref{eq:almost-reg-est-interior} does not hold true, then there exist sequences $(u_k) \subset C^{\gamma}(\R^n)$, $(\Phi_k)$ as in Subsection  \ref{subsec:flattening}, kernels $(K_k)$ satisfying \eqref{eq:K-comp}, and sequences $(f_k)$, $(g_k)$, $(J_k)$, $(I_k)$ with
\begin{align}\label{B4}
\Vert J_k \Vert_{C^{\alpha}(\{x_n \ge 0 \})} + \Vert K_k \Vert_{C^{\sigma}(\mathbb{S}^{n-1})} \le \Lambda,
\end{align}
and 
\begin{align}
\label{eq:blow-up-ass-interior}
\frac{ \Vert u_k \Vert_{L^{\infty}(B_{1}(e_n))} + \Vert (x_n)_+^{1-\beta} g_k \Vert_{L^{\infty}(\{x_n>0\}\cap B_{1}(e_n))} +  [f_k]_{C^{\gamma}(\{ x_n \ge 0\})}}{[u_k]_{C^{\gamma}(\{ x_n > 0 \})}} \to 0,
\end{align}
and $r_k \searrow 0$, $(x_k) \subset B_{1/4}(e_n)$, such that
\begin{align*}
\mathcal{L} & \big(J_k(x)J_k(y)(x_n)_+^{\beta-1}(y_n)_+^{\beta-1} K_k(\Phi_k(x) - \Phi_k(y))\big)(u_k) \\
&\qquad\qquad = g_k + \mathcal{L}\big((x_n)_+^{\beta-1}(y_n)_+^{\beta-1} I_k(x,y)\big)(f_k) ~~ \text{ in } B_{1}(e_n).
\end{align*}
Moreover, the functions
\begin{align}\label{B5}
v_k(x) = \frac{u_k(x_k + r_k x) - u(x_k)}{r_k^{\gamma} [u_k]_{C^{\gamma}(\R^n)}}
\end{align}
satisfy
\begin{align}
\label{eq:blowup-contradiction-ass-interior}
\Vert v_k \Vert_{L^{\infty}( B_{1/2}(e_n))} > \frac{\delta}{2}, ~~ \text{ as } k \to \infty, \qquad v_k(0) = 0.
\end{align}
It is also clear that
\begin{align}\label{Lav}
[v_k]_{C^{\gamma}(\R^n)} = 1, \qquad |v_k(x)| \le |x|^{\gamma} + |v_k(0)| = |x|^{\gamma} ~~ \forall x \in \R^n.
\end{align}
Our objective is to obtain a contradiction with \eqref{eq:blowup-contradiction-ass-interior} by investigating the PDE satisfied by $v_k$, namely
\begin{align}
\label{B6}
\begin{split}
\mathcal{L} & \big(\widetilde{J}_k(x)\widetilde{J}_k(y)((x_k)_n + r_k x_n)_+^{\beta-1}((x_k)_n + r_k y_n)_+^{\beta-1} r_k^{n+2s} K_k(\widetilde{\Phi}_k(x) - \widetilde{\Phi}_k(y))\big)(v_k) \\
& = \widetilde{g_k} + \mathcal{L}\big(((x_k)_n + r_k x_n)_+^{\beta-1}((x_k)_n + r_k y_n)_+^{\beta-1} \widetilde{I}_k(x,y)\big)(\widetilde{f}_k)\qquad \text{ in } B_{r_k^{-1}}(r_k^{-1}(e_n - x_k)),
\end{split}
\end{align}
where 
\begin{align}
\label{B7}
\begin{split}
\widetilde{J}_k(x) &= J_k(x_k + r_k x), \qquad \widetilde{I}_k(x,y) = r_k^{n+2s} I_k(x_k + r_k x , x_k + r_k y), \qquad \widetilde{\Phi}_k(x) = \Phi_k(x_k + r_k x), \\
\widetilde{g}_k(x) &= r_k^{2s}\frac{g_k(x_k + r_k x)}{r_k^{\gamma} [u_k]_{C^{\gamma}(\R^n)}}, \qquad \widetilde{f}_k(x) = \frac{f_k(x_k + r_k x)}{r_k^{\gamma} [u_k]_{C^{\gamma}(\R^n)}}.
\end{split}
\end{align}

\textbf{Step 2:}
Let us now fix a ball $B \subset \R^n$ such that $B \subset B_{r_k^{-1}/8}(0) \subset B_{r_k^{-1}/2}(r_k^{-1}(e_n - x_k) )$ for all $k \in \N$ and $\eta \in C^{\infty}_c(B)$. Then, it holds
\begin{align}\label{B8}
|\eta(x) - \eta(y)| &\le 2 \Vert \eta \Vert_{C^1(\R^n)} \min\{ 1 , |x-y| \}.
\end{align}

The goal of this step is to prove that
\begin{align}
\label{eq:limit-ti-interior}
\begin{split}
 \int_{\R^n} & \int_{\R^n} \widetilde{J}_k(0) \widetilde{J}_k(0) ((x_k)_n+ r_k x_n)_+^{\beta-1} ((x_k)_n+ r_k y_n)_+^{\beta-1} K_k((D \Phi_k)(x_k) \cdot (x-y)) \cdot\\
& \qquad \cdot (v_k(x) - v_k(y))(\eta(x) - \eta(y))  \d y \d x \to 0.
\end{split}
\end{align}

To prove it, we write
\begin{align}\label{B88}
& \widetilde{J}_k(x)\widetilde{J}_k(y) r_k^{n+2s}K_k(\widetilde{\Phi}_k(x) - \widetilde{\Phi}_k(y)) \nonumber \\
& = \widetilde{J}_k(0)\widetilde{J}_k(0) K_k((D \Phi_k)(x_k) \cdot (x-y)) \nonumber\\
&\quad + \widetilde{J}_k(0)\widetilde{J}_k(0) \left( r_k^{n+2s}K_k(\widetilde{\Phi}_k(x) - \widetilde{\Phi}_k(y)) - K_k((D \Phi_k)(x_k) \cdot (x-y)) \right) \nonumber\\
&\quad + \left( \widetilde{J}_k(x)\widetilde{J}_k(y) - \widetilde{J}_k(0)\widetilde{J}_k(0) \right) r_k^{n+2s} K_k(\widetilde{\Phi}_k(x) - \widetilde{\Phi}_k(y)) \nonumber\\
&=: K_k^{(1)}(x,y) + K_k^{(2)}(x,y) + K_k^{(3)}(x,y),
\end{align}
and observe that it suffices to show
\begin{align}
\label{eq:conv-K2-interior}
\int_{\R^n}\int_{\R^n} ((x_k)_n+ r_k x_n)_+^{\beta-1}((x_k)_n+ r_k y_n)_+^{\beta-1}(v_k(x) - v_k(y))(\eta(x) - \eta(y)) K_k^{(2)}(x,y) \d y \d x \to 0,\\
\label{eq:conv-K3-interior}
\int_{\R^n}\int_{\R^n} ((x_k)_n+ r_k x_n)_+^{\beta-1}((x_k)_n+ r_k y_n)_+^{\beta-1}(v_k(x) - v_k(y))(\eta(x) - \eta(y)) K_k^{(3)}(x,y) \d y \d x \to 0,\\
\label{eq:conv-f-interior}
\int_{\R^n}\int_{\R^n} ((x_k)_n+ r_k x_n)_+^{\beta-1} ((x_k)_n+ r_k y_n)_+^{\beta-1} (\widetilde{f}_k(x) - \widetilde{f}_k(y))(\eta(x) - \eta(y)) \widetilde{I}_k(x,y) \d y \d x \to 0,\\
\label{eq:conv-g-interior}
 \int_{\R^n} \widetilde{g}_k(x) \eta(x) \d x \to 0.
\end{align}

First, note that since $\gamma < 2s$ and $(x_n)_{+}^{\beta-1}\in L^{1}(\{x_n>0\}\cap B_1)$, and \eqref{eq:blow-up-ass-interior}, it holds
\begin{align}\label{B9}
\left| \int_{\R^n} \widetilde{g}_k(x) \eta(x) \d x \right| \le C \Vert \eta \Vert_{L^{\infty}} |B| r_k^{2s-\gamma} \frac{\Vert (x_n)_+^{1-\beta} g_k \Vert_{L^{\infty}(B_{1/2}(e_n))} }{[u_k]_{C^{\gamma}(\R^n)}}  \to 0,
\end{align}
which proves \eqref{eq:conv-g-interior}. Moreover, it holds
\begin{align}\label{R}
|\widetilde{I}_k(x,y)| = |r_k^{n+2s} I_k(x_k + r_k x , x_k + r_k y)| \le \Lambda |x-y|^{-n-2s},
\end{align}
and
\begin{align*}
[ \widetilde{f}_k ]_{C^{\gamma}(\{ x_n \ge -r_k^{-1} (x_k)_n \}
 )} = \frac{[ f_k ]_{C^{\gamma}(\{ x_n \ge 0 \})}}{[u_k]_{C^{\gamma}(\R^n)}} \to 0.
\end{align*}
Therefore, using also that in $B_{r_k^{-1}/4}$ it holds
\begin{align}
\label{eq:32B-no-weight}
\frac{1}{4} \le - r_k|[r_k^{-1} (e_n - x_k)_n - x_n]| + 1 \le (x_k)_n + r_k x_n \le |r_k[r_k^{-1}(e_n - x_k)_n - x_n]| + 1 \le \frac{7}{4} ,
\end{align}
we deduce since $\gamma > 2s - 1$ and by \eqref{B8},
\begin{align}
\label{B10}
\begin{split}
& \left| \int_{2B}\int_{2B} ((x_k)_n + r_k x_n)_+^{\beta-1} ((x_k)_n + r_k y_n)_+^{\beta-1} \widetilde{I}_k(x,y) (\widetilde{f}_k(x) - \widetilde{f}_k(y)) (\eta(x) - \eta(y)) \d y \d x \right| \\
&\le C \frac{[ f_k ]_{C^{\gamma}(\{ x_n \ge 0 \} \cap B_1(e_n))}}{[u_k]_{C^{\gamma}(\R^n)}} \int_{2B}\int_{2B} |x-y|^{-n-2s+1+\gamma} \d y \d x \\
&\le C |B| \frac{[ f_k ]_{C^{\gamma}(\{ x_n \ge 0 \} \cap B_1(e_n))}}{[u_k]_{C^{\gamma}(\R^n)}} \to 0.
\end{split}
\end{align}
Moreover, since $|x-y| \ge c_0 := \diam(B)$ for $x \in B$ and $y \in B_{r_k^{-1}/4 }\setminus 2B$ and $\gamma<2s$,
\begin{align}
\label{B10-2}
\begin{split}
& \left| \int_{B}\int_{B_{r_k^{-1}/4} \setminus 2B} ((x_k)_n+ r_k x_n)_+^{\beta-1} ((x_k)_n+ r_k y_n)_+^{\beta-1} \widetilde{I}_k(x,y) (\widetilde{f}_k(x) - \widetilde{f}_k(y)) (\eta(x) - \eta(y)) \d y \d x \right| \\
&\le C \frac{[ f_k ]_{C^{\gamma}(\{ x_n \ge 0 \})}}{[u_k]_{C^{\gamma}(\R^n)}}  \int_{B}\int_{B_{r_k^{-1}/4} \setminus 2B} |x-y|^{-n-2s+\gamma} \d y \d x \\
&\le C \frac{[ f_k ]_{C^{\gamma}(\{ x_n \ge 0 \})}}{[u_k]_{C^{\gamma}(\R^n)}} \int_{B}\int_{B_{c_0}(x)^c} |x-y|^{-n-2s+\gamma} \d y \d x \\
&\le C \frac{[ f_k ]_{C^{\gamma}(\{ x_n \ge 0 \})}}{[u_k]_{C^{\gamma}(\R^n)}} \to 0.
\end{split}
\end{align}
Finally, we have for large enough $k \in \N$
\begin{align}\label{B11}
& \left| \int_{B}\int_{\{ y_n \ge - r_k^{-1} (x_k)_n \} \setminus B_{r_k^{-1}/4}} ((x_k)_n+ r_k x_n)_+^{\beta-1} ((x_k)_n+ r_k y_n)_+^{\beta-1} \widetilde{I}_k(x,y) (\widetilde{f}_k(x) - \widetilde{f}_k(y)) (\eta(x) - \eta(y)) \d y \d x \right| \nonumber\\
&\le C \frac{[ f_k ]_{C^{\gamma}(\{ x_n \ge 0 \})}}{[u_k]_{C^{\gamma}(\R^n)}} \int_{B} \int_{\{ y_n \ge - r_k^{-1} (x_k)_n \} \setminus B_{r_k^{-1}/4} }  ((x_k)_n+ r_k y_n)_+^{\beta-1} |x-y|^{-n-2s+\gamma} \d y \d x \nonumber\\
&\le  C \frac{[ f_k ]_{C^{\gamma}(\{ x_n \ge 0 \})}}{[u_k]_{C^{\gamma}(\R^n)}} |B| r_k^{2s-\gamma} \int_{\{ w_n \ge 0 \} \setminus B_{1/4}(x_k)} (w_n)_+^{\beta - 1} |w-x_k|^{-n-2s+\gamma} \d w \nonumber\\
&\le  C \frac{[ f_k ]_{C^{\gamma}(\{ x_n \ge 0 \})}}{[u_k]_{C^{\gamma}(\R^n)}} r_k^{2s - \gamma} \to 0,
\end{align}
where we have done the change of variables $w=x_k+r_ky$ and used the fact that $|x_k|\leq C$ in order to apply \autoref{lemma:int-polar-coord}. Observe that the integral converges if $\gamma + \beta - 1 < 2s$, which is satisfied by assumption.
Therefore \eqref{eq:conv-f-interior} follows by \eqref{B9}--\eqref{B11}.

Let us highlight now that for any $\nu \in (0,\alpha]$ and $\mu \in (0,1+\alpha]$
\begin{align}
[\widetilde{J}_k]_{C^{\nu}(\{ x_n \ge - r_k^{-1}(x_k)_n\} )} &= r_k^{\nu} [J_k]_{C^{\nu}(\{ x_n \ge 0 \} )} \le \Lambda r_k^{\nu}, \label{P1}\\
 \Vert \widetilde{J}_k \Vert_{L^{\infty}( \{ x_n \ge - r_k^{-1}(x_k)_n \})} &= \Vert J_k \Vert_{L^{\infty}( \{ x_n \ge 0 \})} \le \Lambda, \label{P2}\\
[\widetilde{\Phi}_k]_{C^{\mu}(\{ x_n \ge - r_k^{-1}(x_k)_n\})} &= r_k^{\mu} [\Phi_k]_{C^{\mu}(\{ x_n \ge 0 \})} \le \Lambda r_k^{\mu},  \label{P3}\\
[\widetilde{\Phi}_k^{-1}]_{C^{0,1}(\widetilde{\Phi}_k(\{ x_n \ge - r_k^{-1}(x_k)_n\}))} &= r_k^{-1} [\Phi_k^{-1}]_{C^{0,1}(\Phi_k(\{ x_n \ge 0 \}))} \le \Lambda r_k^{-1}, \label{P4}
\end{align}

and that we have by \eqref{eq:DPhi-comp}
\begin{align}
\label{eq:trafoed-kernel}
\begin{split}
|r_k^{n+2s}K_k(\widetilde{\Phi}_k(x) - \widetilde{\Phi}_k(y))| &\asymp r_k^{n+2s} |\widetilde{\Phi}_k(x) - \widetilde{\Phi}_k(y)|^{-n-2s} \\
&\asymp  r_k^{n+2s} [\widetilde{\Phi}_k^{-1}]_{C^{0,1}(\R^n)}^{n+2s} |x-y|^{-n-2s} \asymp |x-y|^{-n-2s}.
\end{split}
\end{align}

Moreover,
\begin{align}
\label{eq:J-product-Holder}
\begin{split}
\Big| \widetilde{J}_k(x)\widetilde{J}_k(y) - \widetilde{J}_k(0) \widetilde{J}_k(0) \Big| &\le |\widetilde{J}_k(x)||\widetilde{J}_k(y) - \widetilde{J}_k(0)| + |\widetilde{J}_k(0)| |\widetilde{J}_k(x) - \widetilde{J}_k(0)| \\
& \le C\Vert \widetilde{J}_k \Vert_{L^{\infty}(\R^n)} [\widetilde{J}_k]_{C^{\nu}(\R^n)} (|x|^{\nu} + |x-y|^{\nu}).
\end{split}
\end{align}

Thus, by \eqref{eq:32B-no-weight}, \eqref{eq:J-product-Holder}, and \eqref{eq:trafoed-kernel}, we get
\begin{align*}
\Big| \int_{2B}\int_{2B} & [\widetilde{J}_k(x)\widetilde{J}_k(y) - \widetilde{J}_k(0)\widetilde{J}_k(0)] ((x_k)_n+ r_k x)_+^{\beta-1} ((x_k)_n+ r_k y)_+^{\beta-1} \cdot \\
& \qquad \cdot (v_k(x) - v_k(y))(\eta(x) - \eta(y)) r_k^{n+2s} K_k( \widetilde{\Phi}_k(x) - \widetilde{\Phi}_k(y) ) \d y \d x \Big| \\
&\le C [v_k]_{C^{\gamma}(\R^n)} r_k^{\alpha} \int_{2B}\int_{2B} (|x| + |x-y|)^{\alpha} |x-y|^{-n-2s+1+\gamma} \d y \d x \le C r_k^{\alpha} \to 0,
\end{align*}
where the integral converges since $1+\gamma > 2s$ by assumption.
Moreover, we have
\begin{align*}
\Big| \int_{B}\int_{B_{r_k^{-1}/4} \setminus 2B} & [\widetilde{J}_k(x)\widetilde{J}_k(y) - \widetilde{J}_k(0)\widetilde{J}_k(0)] ((x_k)_n+ r_k x)_+^{\beta-1} ((x_k)_n+ r_k y)_+^{\beta-1} \cdot \\
& \qquad \cdot (v_k(x) - v_k(y))(\eta(x) - \eta(y)) r_k^{n+2s} K_k( \widetilde{\Phi}_k(x) - \widetilde{\Phi}_k(y) ) \d y \d x \Big| \\
&\le C [v_k]_{C^{\gamma}(\R^n)} r_k^{\nu} \int_{B}\int_{B_{r_k^{-1}/4} \setminus 2B} (|x| + |x-y|)^{\nu} |x-y|^{-n-2s+\gamma} \d y \d x \\
&\le C r_k^{\nu} \int_B \int_{B_{c_0}(x)^c}  |x-y|^{-n-2s +\gamma + \nu} \d y  \d x \le C r_k^{\nu} \to 0,
\end{align*}
where, since $\gamma < 2s$, we can assume that $\gamma + \nu < 2s$ upon choosing $\nu$ small enough. Finally, it holds
\begin{align*}
\Big| \int_{B} & \int_{\{ y_n \ge -(x_k)_n r_k^{-1} \} \setminus B_{r_k^{-1}/4} } [\widetilde{J}_k(x)\widetilde{J}_k(y) - \widetilde{J}_k(0)\widetilde{J}_k(0)] ((x_k)_n+ r_k x_n)_+^{\beta-1} ((x_k)_n+ r_k y_n)_+^{\beta-1} \cdot \\
& \qquad \cdot (v_k(x) - v_k(y))(\eta(x) - \eta(y)) r_k^{n+2s} K_k( \widetilde{\Phi}_k(x) - \widetilde{\Phi}_k(y) ) \d y \d x \Big| \\
&\le C r_k^{\nu} [v_k]_{C^{\gamma}} \int_{B}\int_{\{ y_n \ge -(x_k)_n r_k^{-1} \} \setminus B_{r_k^{-1}/4} } ((x_k)_n+ r_k y_n)_+^{\beta-1} (|x| + |x-y|)^{\nu} |x-y|^{-n-2s+\gamma} \d y \d x \\
&\le C |B| r_k^{2s-\gamma - \nu} \int_{\{w_n \ge 0 \} \setminus B_{1/4}(x_k)} (w_n)_+^{\beta -1} |w-x_k|^{-n-2s+\gamma + \nu} \d w,
\end{align*}
where the integrals converge by \autoref{lemma:int-polar-coord} if $\beta - 1 -2s + \gamma + \nu < 0$. Upon choosing $\nu$ small enough, this condition becomes equivalent to $\beta + \gamma < 1+ 2s$, which holds true by assumption. These three estimates imply \eqref{eq:conv-K3-interior}.

We notice now that, since $K_k \in C^{\sigma}(\mathbb{S}^{n-1})$, in particular by \eqref{eq:K-reg} it holds 
\begin{align*}
|K_k(x) - K_k(y)| \le C \min\{ |x|,|y|\}^{-n-2s-\sigma'} |x-y|^{\sigma'}
\end{align*}
for any $\sigma' \in (0,\sigma]$. Thus, and since we have $D \widetilde{\Phi}_k(x) = r_k (D \Phi_k)(r_k x)$ and $D \Phi_k(0) = I$, as well as by the homogeneity of $K_k$, it holds for any $\mu' \in (1,1+\alpha]$
\begin{align}
\label{eq:K_k-reg-interior}
\begin{split}
& | r_k^{n+2s} K_k(\widetilde{\Phi}_k(x) - \widetilde{\Phi}_k(y)) - K_k((D \Phi_k)(x_k) \cdot (x-y))| \\
&\quad = r_k^{n+2s}|K_k(\widetilde{\Phi}_k(x) - \widetilde{\Phi}_k(y)) - K_k( D \widetilde{\Phi}_k(0) \cdot (x-y))| \\
&\quad \le C r_k^{n+2s} \left( |\widetilde{\Phi}_k(x) - \widetilde{\Phi}_k(y)|^{-n-2s-\sigma'} + |D \widetilde{\Phi}_k(0) \cdot (x-y)|^{-n-2s-\sigma'} \right) |\widetilde{\Phi}_k(x) - \widetilde{\Phi}_k(y) - D \widetilde{\Phi}_k(0) \cdot (x-y)|^{\sigma'} \\
&\quad \le C r_k^{n+2s} \left( [\widetilde{\Phi}_k^{-1}]_{C^{0,1}(\R^n)}^{-1} + r_k \right)^{-n-2s-\sigma'} |x-y|^{-n-2s-\sigma'} \cdot \\
&\quad \qquad\qquad\qquad \cdot (|\widetilde{\Phi}_k(x) - \widetilde{\Phi}_k(0) - D \widetilde{\Phi}_k(0) \cdot x| +|\widetilde{\Phi}_k(y) - \widetilde{\Phi}_k(0) - D \widetilde{\Phi}_k(0) \cdot y|)^{\sigma'} \\
&\quad \le C r_k^{-\sigma'} |x-y|^{-n-2s-\sigma'} [\widetilde{\Phi}_k]_{C^{\mu}(\R^n)}^{\sigma'}(|x|^{\mu'} + |y|^{\mu'})^{\sigma'} \\
&\quad \le C r_k^{(\mu'-1) \sigma'} |x-y|^{-n-2s - \sigma'} (|x| + |x-y|)^{\mu' \sigma'}.
\end{split}
\end{align}
Thus, by analogous computations, using \eqref{eq:K_k-reg-interior} instead of \eqref{eq:J-product-Holder}, we deduce \eqref{eq:conv-K2-interior} by taking $\sigma'$ small enough. Altogether, we have shown \eqref{eq:limit-ti-interior}, as desired.

\textbf{Step 3:} Our goal is to prove that there exist $v \in C^{\gamma}(\R^n)$ and a kernel $K$ satisfying \eqref{eq:K-comp}
such that,
\begin{align}
\label{eq:global-sol-interior}
|v(x)| \le |x|^{\gamma}, \qquad \mathcal{L}\big(K(x-y)\big)(v) = 0 ~~ \text{ in } \R^n.
\end{align}

The existence of $v$ with $v_k \to v$ in $C^{\gamma}_{loc}(\R^n)$ satisfying the desired bound already follows from the Arzel\`a-Ascoli theorem. Moreover, by \eqref{eq:limit-ti-interior} it suffices to show
\begin{align}
\label{Imp}
\begin{split}
\int_{\R^n} & \int_{\R^n} \widetilde{J}_k(0) \widetilde{J}_k(0) ((x_k)_n+ r_k x_n)_+^{\beta-1}((x_k)_n+ r_k y_n)_+^{\beta-1} \\
&\qquad K_k((D \Phi_k)(x_k) \cdot (x-y))(v_k(x) - v_k(y))(\eta(x) - \eta(y)) \d y \d x \\
 & \to J_0^2 \int_{\R^n} \int_{\R^n} K(x-y)(v(x) - v(y))(\eta(x) - \eta(y)) \d y \d x
 \end{split}
\end{align}
for some kernel $K$ satisfying \eqref{eq:K-comp} and some $J_0 > 0$, for any $\eta \in C_c^{\infty}(\R^n)$.

Note that the translation invariant kernels $K_k((D \Phi_k)(x_k) \cdot (x-y))$ are homogeneous and satisfy the ellipticity assumption \eqref{eq:K-comp}. Indeed, by \eqref{eq:DPhi-comp}, since $|x_k|$ is bounded we have 
\begin{align*}
K_k((D \Phi_k)(x_k) \cdot (x-y)) \asymp |(D \Phi_k)(x_k)|^{-n-2s} |x-y|^{-n-2s} \asymp |x-y|^{-n-2s}.
\end{align*}

Note that for any ball $B$ centered around the origin and $m > 0$ such that $mB \subset B_{r_k^{-1}/8}(0) \subset B_{r_k^{-1}/2}(r_k^{-1}(e_n - x_k))$ and $\eta \in C_c^{\infty}(B)$ it holds 
\begin{align}
\label{B12}
\begin{split}
\Big| \int_{m B} \int_{m B} & \widetilde{J}_k(0) \widetilde{J}_k(0) ((x_k)_n + r_k x_n)_+^{\beta-1} ((x_k)_n + r_k y_n)_+^{\beta-1} \cdot \\
&\quad \cdot K_k((D \Phi_k)(x_k) \cdot (x-y))((v_k - v)(x) - (v_k - v)(y))(\eta(x) - \eta(y))  \d y \d x \Big|  \\
&\le C [v_k - v]_{C^{\gamma}(m B)} \int_{\frac{3}{2}B} \int_{\frac{3}{2}B} |x-y|^{-n-2s+1+\gamma} \d y \d x \\
&\le C [v_k - v]_{C^{\gamma}(m B)} \to 0,
\end{split}
\end{align}
since $((x_k)_n + r_k x_n)_+^{\beta-1} \asymp C$ in $mB$ by \eqref{eq:32B-no-weight}, and where the integrals converge since $\gamma > 2s - 1$.

Moreover, it holds since $|v_k(y) - v(y)|\le C |y|^{\gamma}$ and since $|x-y| \ge (m-1)\diam(B) = (m-1)c_0$ for $x\in B$ and $y \in B_{r_k^{-1}/4} \setminus m B$:
\begin{align}
\label{B13}
\begin{split}
\Big| \int_{B} & \int_{B_{r_k^{-1}/4} \setminus m B} \widetilde{J}_k(0) \widetilde{J}_k(0) ((x_k)_n + r_k x_n)_+^{\beta-1} ((x_k)_n + r_k y_n)_+^{\beta-1} \cdot \\
&\quad \cdot K_k((D \Phi_k)(x_k) \cdot (x-y))((v_k - v)(x) - (v_k - v)(y))(\eta(x) - \eta(y))  \d y \d x \Big|  \\
&\le C \int_{B} \int_{B_{r_k^{-1}/4} \setminus m B}(|x|^{\gamma} + |x - y|^{\gamma}) |x-y|^{-n-2s} \d y \d x \\
&\le C \int_{B} \int_{B_{(m-1)c_0}(x)^c} |x-y|^{-n-2s} \d y \d x + C \int_{B} \int_{B_{(m-1)c_0}(x)^c} |x-y|^{-n-2s + \gamma} \d y \d x \\
&\le C |B| [(m-1)c_0]^{\gamma - 2s},
\end{split}
\end{align}
where we used that $\gamma < 2s$. Finally, we have
\begin{align}
\label{B14}
\begin{split}
& \Big| \int_{B} \int_{(B_{r_k^{-1}/4})^c} \widetilde{J}_k(0) \widetilde{J}_k(0) ((x_k)_n + r_k x_n)_+^{\beta-1} ((x_k)_n + r_k y_n)_+^{\beta-1} \cdot \\
&\qquad\qquad \cdot K_k((D \Phi_k)(x_k) \cdot (x-y))((v_k - v)(x) - (v_k - v)(y))(\eta(x) - \eta(y))  \d y \d x \Big| \\
&\le C \int_{B} |v_k(x) - v(x)| \left( \int_{(B_{r_k^{-1}/4})^c} ((x_k)_n + r_k y_n)_+^{\beta-1}  |x-y|^{-n-2s} \d y \right) \d x \\
&\quad + C \int_{B} \left( \int_{(B_{r_k^{-1}/4})^c} ((x_k)_n + r_k y_n)_+^{\beta-1}  |v_k(y) - v(y)| |x-y|^{-n-2s} \d y \right) \d x  \\
&\le C \Vert v_k - v \Vert_{L^{\infty}(\supp(\eta))} \int_B \int_{(B_{r_k^{-1}/4})^c} ((x_k)_n + r_k y_n)_+^{\beta-1}  |x-y|^{-n-2s} \d y \d x  \\
&\quad + C \int_B \int_{(B_{r_k^{-1}/4})^c} ((x_k)_n + r_k y_n)_+^{\beta-1} |y|^{\gamma} |x-y|^{-n-2s} \d y \d x \\
&\le C r_k^{2s} \Vert v_k - v \Vert_{L^{\infty}(\supp(\eta))} |B| \int_{\{ w_n \ge 0\} \setminus B_{1/4}(x_k)} (w_n)_+^{\beta-1} |w-x_k|^{-n-2s} \d w  \\
&\quad + C r_k^{2s-\gamma} |B| \int_{\{ y_n \ge 0\} \setminus B_{1/4}} (y_n)_+^{\beta-1} |y|^{-n-2s + \gamma} \d y \\
&\le C r_k^{2s} \Vert v_k - v \Vert_{L^{\infty}(B)} + C r_k^{2s-\gamma} \to 0,
\end{split}
\end{align}
where we used \autoref{lemma:int-polar-coord}.
By combination of the previous three estimates, we deduce that
\begin{align*}
\Big| \int_{\R^n} \int_{\R^n} & \widetilde{J}_k(0) \widetilde{J}_k(0) ((x_k)_n + r_k x_n)_+^{\beta-1} ((x_k)_n + r_k y_n)_+^{\beta-1} \\
& K_k((D \Phi_k)(x_k) \cdot (x-y))((v_k - v)(x) - (v_k - v)(y))(\eta(x) - \eta(y))  \d y \d x \Big| \to 0.
\end{align*}
Indeed, we choose first $m$ so large that \eqref{B13} becomes small, and then we choose $k$ so large that the quantities in \eqref{B12} and \eqref{B14} are also small.

Hence, in order to prove that \eqref{Imp} holds true in the weak sense it remains to prove
\begin{align*}
\Big| \int_{\R^n} \int_{\R^n} & ((x_k)_n + r_k x_n)_+^{\beta-1} ((x_k)_n + r_k y_n)_+^{\beta-1} \cdot &\\
& \cdot [\widetilde{J}_k(0) \widetilde{J}_k(0) K_k((D \Phi_k)(x_k) \cdot (x-y)) - J_0^2K(x-y)](v(x) - v(y))(\eta(x) - \eta(y))  \d y \d x \Big| \to 0.
\end{align*}
Note that there is $0<J_0 = \lim_{k \to \infty} \widetilde{J}_k(0)$ by the uniform boundedness of $\widetilde{J}_k(0)$, and $x_0 = \lim_{k \to \infty} (x_k)_n > 0$, and also a kernel $K$ satisfying \eqref{eq:K-comp} such that
\begin{align*}
((x_k)_n + r_k x_n)_+^{\beta-1} ((x_k)_n + r_k y_n)_+^{\beta-1} (1 \wedge |x-y|^2)K_k((D \Phi_k)(x_k) \cdot (x-y)) \to x_0^{2\beta-2} (1 \wedge |x-y|^2)K(x-y)
\end{align*}
weakly in the sense of measures in $(B^c \times B^c)^c$ for any ball $B \subset \R^n$ centered at zero (see \cite[Proposition 2.2.36]{FeRo24}).
This implies the above convergence claim and yields \eqref{eq:global-sol-interior}, as desired.

\textbf{Step 4:} The Liouville theorem (see \cite[Theorem 2.9]{FeRo24}) implies that $v \equiv 0$, which is a contradiction to \eqref{eq:blowup-contradiction-ass-interior}. Hence, the proof is complete.
\end{proof}

We proceed with an interior estimate in case $\gamma = 2s+\alpha > 2s$ for some $\alpha > 0$. 

\begin{lemma}
\label{lemma:interior-2}
Let $\beta \in [s,1+s]$. Assume that $K \in C^{\alpha}(\mathbb{S}^{n-1})$ for some $\alpha > 0$ with $\alpha + 2s < 1$ satisfies \eqref{eq:K-comp} and let $J \in C^{\alpha}(\{x_n \ge 0 \})$, and $I: \R^n \times \R^n \to \R$ be such that
\begin{align}
J \ge \Lambda^{-1} \text{ in } \{ x_n > 0\} \cap B_{1}(e_n), \quad \Vert J \Vert_{C^{\alpha}(\{x_n \ge 0 \})} &+ \Vert K \Vert_{C^{\alpha}(\mathbb{S}^{n-1})} \le \Lambda, \quad  |I(x,y)| \le \Lambda |x-y|^{-n-2s} \nonumber\\
\label{eq:kernel-int-reg-ass}
|I(x,z) - I(y,z)| \le \Lambda |x-y|^{\alpha} |z-e_n|^{-n-2s}& ~~ \forall x,y \in B_{3/4}(e_n), ~  z \in B_{7/8}(e_n)^c. 
\end{align}
Moreover let $\Phi \in C^{1,\alpha}(\{x_n \ge 0 \} \cap B_{1}(e_n))$ be as in Subsection \ref{subsec:flattening}, and let $u \in C^{2s+\alpha}_{loc}(\R^n)$ be a weak solution to 
\begin{align*}
\mathcal{L} & \big(J(x)J(y)(x_n)_+^{\beta-1}(y_n)_+^{\beta-1} K(\Phi(x) - \Phi(y))\big)(u) \\
&\qquad \qquad = g + \mathcal{L}\big((x_n)_+^{\beta-1}(y_n)_+^{\beta-1} I(x,y)\big)(f) ~~ \text{ in } B_{1}(e_n),
\end{align*}
where $(x_n)_+^{1-\beta} g \in C^{\alpha}(\{ x_n \ge 0 \} \cap B_{1}(e_n))$, and $f \in C^{2s+\alpha}(\{ x_n \ge 0 \})$.
Then, the following holds true
\begin{align*}
\Vert u \Vert_{C^{2s+\alpha}(B_{1/2}(e_n))} &\le C \Big( \Vert u \Vert_{L^{\infty}(B_{1}(e_n))} + \Vert (x_n)_+^{\beta-1} u |x-e_n |^{-n-2s}\Vert_{L^1(B_1(e_n)^c)}  \\
&\qquad \qquad \qquad + \Vert (x_n)_+^{1-\beta} g \Vert_{C^{\alpha}(\{ x_n \ge 0 \}\cap B_{1}(e_n))} +  [ f ]_{C^{2s+\alpha}(\{ x_n \ge 0\})} \Big),
\end{align*}
where $C > 0$ only depends on $n,s,\beta,\alpha,\lambda,\Lambda$, and the constant in \eqref{eq:Phi-comp}. 
\end{lemma}

Note that we also have a rescaled version of \autoref{lemma:interior-2} in alignment with \eqref{eq:interior-scaled}.

\begin{proof}
The proof follows immediately by rewriting the equation for $u$ as follows:
\begin{align*}
\mathcal{L}\big( \widetilde{K}(x,y) \big)(u) = h ~~ \text{ in } B_{3/4}(e_n),
\end{align*}
where we let $\eta \in C^{\infty}(\R^n \times \R^n)$ be a smooth cut-off function satisfying $\eta \equiv 1$ in $B_{7/8}(e_n) \times B_{7/8}(e_n)$, $0 \le \eta \le 1$, and $\supp(\eta) \subset B_{15/16}(e_n) \times B_{15/16}(e_n)$, and
\begin{align*}
\widetilde{K}(x,y) &= \eta(x,y) J(x)J(y)(x_n)_+^{\beta-1}(y_n)_+^{\beta-1} K(\Phi(x) - \Phi(y)), \\
h &= \mathcal{L} \big((1 - \eta(x,y))J(x)J(y)(x_n)_+^{\beta-1}(y_n)_+^{\beta-1} K(\Phi(x) - \Phi(y))\big)(u) \\
&\quad + g + \mathcal{L}\big((x_n)_+^{\beta-1}(y_n)_+^{\beta-1} I(x,y)\big)(f).
\end{align*}
Then, it remains to prove that $\widetilde{K}$ and $h$ satisfy the assumptions of \cite[Theorem 1.4]{FeRo24b}. Note that for $\widetilde{K}$ this is easy to verify, as it follows immediately from the assumptions on $K,J,\Phi$, and the fact that the action of the weight $(x_n)_+^{\beta-1} (y_n)_+^{\beta-1}$ near $\{ x_n = 0\}$ is suppressed by the support of $\eta$. Note that the compact support of $\widetilde{K}$ is not relevant for the proof in \cite{FeRo24b}. 

Hence, we are left to prove that $h \in C^{\alpha}(B_{3/4}(e_n))$.
To prove it, let us first show the following estimate 
\begin{align}
\label{eq:L-f-estimate}
[\mathcal{L}\big((x_n)_+^{\beta-1}(y_n)_+^{\beta-1} I(x,y)\big)(f)]_{C^{\alpha}(B_{3/4}(e_n))} \le C \Vert f \Vert_{C^{2s+\alpha}(\{ x_n > 0 \})}.
\end{align}

Note that we can always replace $\Vert f \Vert_{C^{2s+\alpha}(\{ x_n > 0 \})}$ by $[ f ]_{C^{2s+\alpha}(\{ x_n > 0 \})}$ in \eqref{eq:L-f-estimate}, since the left-hand side of \eqref{eq:L-f-estimate} is invariant under the addition of constants to $f$.

To prove it, we consider $\chi \in C^{\infty}_c( B_{15/16}(e_n))$ with $\chi \equiv 1$ in $B_{7/8}(e_n)$ and define $f = \chi f + (1-\chi)f = f_1 + f_2$. For $f_2$ and $x,y \in B_{3/4}(e_n)$ we estimate, using that $f_2(x) = f_2(y) = 0$,
\begin{align}
\label{eq:L-f-estimate-help}
\begin{split}
& |\mathcal{L}\big((x_n)_+^{\beta-1}(z_n)_+^{\beta-1} I(x,z)\big)(f_2)(x) - \mathcal{L}\big((y_n)_+^{\beta-1}(z_n)_+^{\beta-1} I(y,z)\big)(f_2)(y)| \\
&\quad \le \int_{B_{7/8}(e_n)^c}  (z_n)_+^{\beta-1} |f(z)| |(x_n)_+^{\beta-1} I(x,z) - (y_n)_+^{\beta-1} I(y,z)| \d z \\
&\quad \le C |(x_n)_+^{\beta-1} - (y_n)_+^{\beta-1}| \int_{B_{7/8}(e_n)^c} (z_n)_+^{\beta-1} |f(z)| |I(x,z)| \d z \\
&\quad \quad + C (y_n)_+^{\beta-1} \int_{B_{7/8}(e_n)^c} (z_n)_+^{\beta-1} |f(z)| |I(x,z) - I(y,z)|  \d z \\
&\le C |x-y|^{\alpha} \int_{B_{7/8}(e_n)^c} (z_n)_+^{\beta-1} |f(z)| |z-e_n|^{-n-2s} \d z \le C |x-y|^{\alpha} \Vert f \Vert_{L^{\infty}(\{ x_n > 0 \})},
\end{split}
\end{align}
where we used the assumption on $I$, as well as
\begin{align*}
|(x_n)_+^{\beta-1} - (y_n)_+^{\beta-1}| \le C |x-y|^{\alpha}, \qquad (y_n)_+^{\beta-1} \le C.
\end{align*}

For $f_1$, we introduce another cut-off function $\tau \in C^{\infty}_c(B_{31/32}(e_n) \times B_{31/32}(e_n))$ with $\tau \equiv 1$ in $B_{15/16}(e_n) \times B_{15/16}(e_n)$ and $0 \le \tau \le 1$. We write
\begin{align*}
& |\mathcal{L}\big((x_n)_+^{\beta-1}(z_n)_+^{\beta-1} I(x,z)\big)(f_1)(x) - \mathcal{L}\big((y_n)_+^{\beta-1}(z_n)_+^{\beta-1} I(y,z)\big)(f_1)(y)| \\
&\le |\mathcal{L}\big((x_n)_+^{\beta-1}(z_n)_+^{\beta-1} \tau(x,z) I(x,z)\big)(f_1)(x) - \mathcal{L}\big((y_n)_+^{\beta-1}(z_n)_+^{\beta-1} \tau(y,z) I(y,z)\big)(f_1)(y)| \\
&\quad +  |\mathcal{L}\big((x_n)_+^{\beta-1}(z_n)_+^{\beta-1} [1-\tau(x,z)] I(x,z)\big)(f_1)(x) - \mathcal{L}\big((y_n)_+^{\beta-1}(z_n)_+^{\beta-1} [1-\tau(y,z)] I(y,z)\big)(f_1)(y)| \\
&= |\mathcal{L}_1(f_1)(x) - \mathcal{L}_1(f_1)(y)| + I_2.
\end{align*}
The estimate of $I_2$ is immediate from the regularity assumptions on $f$ and $I$, as well as the constructions of $f_1$ and $\tau$. In fact,
\begin{align*}
I_2 &\le \int_{B_{15/16}(e_n)^c} \big| f_1(x) (x_n)_+^{\beta-1}(z_n)_+^{\beta-1} [1-\tau(x,z)]I(x,z) - f_1(y) (y_n)_+^{\beta-1}(z_n)_+^{\beta-1} [1-\tau(y,z)]I(y,z)  \big| \d z\\
&\le |f_1(x) - f_2(y)| \int_{B_{15/16}(e_n)^c} (x_n)_+^{\beta-1}(z_n)_+^{\beta-1} [1-\tau(x,z)]I(x,z) \d z \\
&\quad + |f(x)| \int_{B_{15/16}(e_n)^c} |(x_n)_+^{\beta-1} [1-\tau(x,z)]I(x,z) - (y_n)_+^{\beta-1} [1-\tau(y,z)]I(y,z) |(z_n)_+^{\beta-1} \d z \\
&\le C \Vert f \Vert_{C^{\alpha}(B_{3/4}(e_n))} |x-y|^{\alpha}  \int_{B_{15/16}(e_n)^c} (z_n)_+^{\beta-1} |z - e_n|^{-n-2s} \d z \le C \Vert f \Vert_{C^{\alpha}(B_{3/4}(e_n))} |x-y|^{\alpha} .
\end{align*}
Moreover, the estimate for $\mathcal{L}_1(f_1)$ follows from \cite{FeRo24b} since the action of the weight $(z_n)_+^{\beta-1}$ near $\{ x_n = 0 \}$ is suppressed by the support of $\tau$, making $(x_n)_+^{\beta-1}(y_n)_+^{\beta-1} \tau(x,y) I(x,y)$ a sufficiently smooth kernel. In fact, if we let $\mathcal{L}^o$ and $\mathcal{L}^e$, to be the odd and even part of $\mathcal{L}_1$, and $\mathcal{L}^e_0$ be the translation invariant operator associated with $\mathcal{L}^e$, we have by \cite[(4.14) and (4.15)]{FeRo24b}, using in particular the regularity assumption on $I$ in the interior, as well as \cite[Lemma 2.2.6(ii)]{FeRo24}
\begin{align*}
[\mathcal{L}_1(f_1)]_{C^{\alpha}(B_{3/4}(e_n))} &\le [\mathcal{L}^e(f_1) - \mathcal{L}^e_0(f_1)]_{C^{\alpha}(B_{3/4}(e_n))} + [\mathcal{L}^o(f_1)]_{C^{\alpha}(B_{3/4}(e_n))} + [\mathcal{L}^e_0(f_1)]_{C^{\alpha}(B_{3/4}(e_n))} \\
&\le C \Vert f \Vert_{C^{2s+\alpha}(\{ x_n \ge 0 \})}.
\end{align*}
This proves \eqref{eq:L-f-estimate}, as claimed.

Finally, we estimate
\begin{align}
\label{eq:L-u-estimate}
\begin{split}
&[\mathcal{L} \big((1 - \eta(x,y))J(x)J(y)(x_n)_+^{\beta-1}(y_n)_+^{\beta-1} K(\Phi(x) - \Phi(y))\big)(u)]_{C^{\alpha}(B_{3/4}(e_n))} \\
&\qquad \qquad \le C \Vert u \Vert_{C^{\alpha}(B_{3/4}(e_n))} + C \Vert (x_n)_+^{\beta-1} u |x-e_n |^{-n-2s}\Vert_{L^1(B_{3/4}(e_n)^c)} .
\end{split}
\end{align}

Note that this estimate implies the desired result after combination with \eqref{eq:L-f-estimate} and using the assumption on $g$. In fact, by H\"older interpolation, we can estimate for any $\delta > 0$
\begin{align*}
\Vert u \Vert_{C^{\alpha}(B_{3/4}(e_n))} \le C(\delta) \Vert u \Vert_{L^{\infty}(B_{3/4}(e_n))} + \delta [ u ]_{C^{2s+\alpha}(B_{3/4}(e_n))},
\end{align*}
and use a standard covering and absorption argument (see for instance \cite[Proof of Theorem 2.4.1]{FeRo24}) to conclude the proof.

We now turn to the proof of \eqref{eq:L-u-estimate}. To prove it, we denote
\begin{align*}
J(x,z) = (1 - \eta(x,z))J(x)J(z) K(\Phi(x) - \Phi(z))
\end{align*}
and observe that by the assumptions on $J,\Phi,K$ and the construction of $\eta$, $J$ satisfies \eqref{eq:kernel-int-reg-ass}. 
For $x,y \in B_{3/4}(e_n)$,
\begin{align*}
& |\mathcal{L} \big((x_n)_+^{\beta-1}(z_n)_+^{\beta-1} J(x,z)\big)(u)(x) - \mathcal{L} \big((y_n)_+^{\beta-1}(z_n)_+^{\beta-1} J(y,z) \big)(u)(y)| \\
&\quad \le \int_{B_{7/8}(e_n)^c} |(u(x) - u(z))(x_n)_+^{\beta-1}(z_n)_+^{\beta-1} J(x,z) - (u(y) - u(z))(y_n)_+^{\beta-1}(z_n)_+^{\beta-1} J(y,z)| \d z \\
&\quad \le |u(x) - u(y)|\int_{B_{7/8}(e_n)^c} (x_n)_+^{\beta-1} |J(x,z)| (z_n)_+^{\beta-1} \d z \\
&\quad \quad + |u(x)| \int_{B_{7/8}(e_n)^c} |(x_n)_+^{\beta-1} J(x,z) - (y_n)_+^{\beta-1} J(y,z)| (z_n)_+^{\beta-1} \d z \\
&\quad \quad + \int_{B_{7/8}(e_n)^c} |u(z)| |(x_n)_+^{\beta-1} J(x,z) - (y_n)_+^{\beta-1} J(y,z)| (z_n)_+^{\beta-1} \d z \\
&\le C\Vert u \Vert_{C^{\alpha}(B_{3/4}(e_n))} + C \Vert (x_n)_+^{\beta-1} u |x-e_n |^{-n-2s}\Vert_{L^1(B_{3/4}(e_n)^c)}.
\end{align*}
In the last step, we used that $J$ satisfies \eqref{eq:kernel-int-reg-ass} and proceeded as in \eqref{eq:L-f-estimate-help} to estimate the last term.
\end{proof}

The next step is to show a regularity estimate up to the boundary. First, we prove the following result.

\begin{lemma}
\label{lemma:almost-reg-bdry}
Let everything be as in \autoref{lemma:bdry-reg}, assuming in particular that $\alpha > 0$ and $\gamma \in (\max\{0,2s-1,s-\beta+\frac{1}{2}\},\min\{1,2s-\beta+1\})$. 
Then, the following holds true: For any $\delta > 0$, it holds
\begin{align*}
\sup_{x \in \{x_n = 0\} \cap B_{1/2}} & \sup_{ y \in \{ x_n \ge 0 \}\cap B_{1/2}} \frac{|u(x) - u(y)|}{|x-y|^{\gamma}} \\
&\le \delta [u]_{C^{\gamma}(\R^n)} + C_{\delta} \Big( \Vert u \Vert_{L^{\infty}(\{ x_n \ge 0 \} \cap B_1)} + \Vert (x_n)_+^{1-\beta} g \Vert_{L^{\infty}(\{ x_n \ge 0 \} \cap B_1) } +  [f]_{C^{\gamma}(\{ x_n \ge 0\})} \\
&\quad\qquad \qquad \qquad + \1_{\{ \beta = s\}} \big[  \Vert b \Vert_{L^{\infty}(\{x_n = 0 \} \cap B_1)} + \Vert A \Vert_{L^{\infty}(\{x_n = 0 \} \cap B_1)} + \Vert a \Vert_{C^{\gamma}(\{x_n = 0 \} \cap B_1)} \big] \Big),
\end{align*}
where $C_{\delta} > 0$ only depends on $\delta$, and $n,s,\beta,\gamma,\alpha,\sigma,\lambda,\Lambda$ and the constant in \eqref{eq:Phi-comp}. 
\end{lemma}

Note that the assumption \eqref{eq:kernel-int-reg-ass-main} is not needed for this result to hold, regardless of the value of $\gamma$.

\begin{proof}
\textbf{Step 1:}
Following the proof of \cite[Lemma 2.4.12]{FeRo24}, we can show that if the claim of the lemma does not hold true, then there exist sequences $(u_k) \subset C^{\gamma}(\R^n)$, $(\Phi_k)$ as in Subsection   \ref{subsec:flattening}, kernels $(K_k)$ satisfying \eqref{eq:K-comp}, and kernels $(I_k)$ satisfying $|I_k(x,y)| \le \Lambda |x-y|^{-n-2s}$, and sequences $(f_k)$, $(g_k)$, $(J_k)$, $(b_k)$, $(a_k)$, $(\nu_{k})$, $(A_k)$, and $(\vartheta_{k})$, such that 
\begin{align}
\label{eq:B4-2}
\begin{split}
& \Vert J_k \Vert_{C^{\alpha}(\{x_n \ge 0 \})} +\Vert K_k \Vert_{C^{\sigma}(\mathbb{S}^{n-1})} \\
&\quad + \1_{\{ \beta = s\}} \big[ \Vert \theta_{k} \Vert_{C^{0,1}(\{ x_n = 0 \} \cap B_1)}  + \Vert \nu_{k} \Vert_{C^{0,1}(\{ x_n = 0 \} \cap B_1)} + \Vert \vartheta_{k} \Vert_{L^{\infty}(\{ x_n = 0 \} \cap B_1)} \big] \le \Lambda,
\end{split}
\end{align}
where we denoted $\theta_{k}:=\theta_{K_k}$, and 
\begin{align}
\label{eq:blow-up-ass-diffquot}
\begin{split}
\frac{ \Vert u \Vert_{L^{\infty}(\{ x_n \ge 0\} \cap B_{1})} + \Vert (x_n)_+^{1-\beta} g_k \Vert_{L^{\infty}(\{ x_n \ge 0 \} \cap B_1) }  } { [u_k]_{C^{\gamma}(\R^n)}} &\to 0,\\
\frac{[f_k]_{C^{\gamma}(\{ x_n \ge 0\})}} {[u_k]_{C^{\gamma}(\R^n)}} &\to 0,\\
\1_{\{ \beta = s\}} \Bigg[ \frac{\Vert b_k \Vert_{L^{\infty}(\{x_n = 0 \} \cap B_1)} + \Vert a_k \Vert_{C^{\gamma}(\{x_n = 0 \} \cap B_1)} + \Vert A_k \Vert_{L^{\infty}(\{x_n = 0 \} \cap B_1)}} {[u_k]_{C^{\gamma}(\R^n)}} \Bigg] &\to 0,
\end{split}
\end{align}
and $r_k \searrow 0$, $(x_k) \subset B_{1/2} \cap \{ x_n = 0\}$, such that  for any $\eta \in C^{\infty}_c(B_1)$
\begin{align*}
\int_{\R^n} &\int_{\R^n} (u_k(x) - u_k(y)) (\eta(x) - \eta(y)) (x_n)_+^{\beta-1}(y_n)_+^{\beta-1} J_k(x) J_k(y) K_k(\Phi_k(x) - \Phi_k(y)) \d y \d x \\
&=\int_{\R^n} \eta(x) g_k(x) \d x + \int_{\R^n} \int_{\R^n} (f_k(x) - f_k(y)) (\eta(x) - \eta(y)) (x_n)_+^{\beta-1}(y_n)_+^{\beta-1} I_k(x,y) \d y \d x\\
&\quad + \1_{\{ \beta = s\}} \Bigg[ \int_{\{ x_n = 0 \}}  b_k(x')(\theta_{k})_n (x') \eta(x',0) \d x' \\
&\qquad\qquad\qquad\quad - \int_{\{ x_n = 0 \}} u_k(x',0) [(\theta_{k}' (x')\cdot \nabla_{x'} \eta(x',0)) +  \dvg \theta_{k}'(x') \eta(x',0)] \d x'\\
&\qquad\qquad\qquad\quad + \int_{\{ x_n = 0 \}} A_k(x') \cdot \vartheta_{k}(x') \eta(x',0)  \d x'\\
&\qquad\qquad\qquad\quad - \int_{\{ x_n = 0 \}}  a_k(x') [(\nu_{k}'(x')\cdot \nabla_{x'} \eta(x',0)) +  {\dvg \nu_{k}'(x') \eta(x',0)]} \d x'  \Bigg].
\end{align*}

Moreover, the functions $v_k$ defined in \eqref{B5} also satisfy
\begin{align}
\label{eq:blowup-contradiction-ass-diffquot}
\Vert v_k \Vert_{L^{\infty}(\{ x_n \ge 0 \} \cap B_{1/2})} > \frac{\delta}{2}, ~~ \text{ as } k \to \infty, \qquad v_k(0) = 0.
\end{align}
As we did in the proof of \autoref{lemma:interior}, we want to obtain a contradiction with \eqref{eq:blowup-contradiction-ass-diffquot} by applying a Liouville type result. Let us define $\widetilde{J}_k,\, \widetilde{I}_k,\, \widetilde{\Phi}_k,\, \widetilde{f}_k$ as in \eqref{B7}, and 
\begin{align*}
\widetilde{g}_k(x) &= r_k^{2(s + 1 - \beta)}\frac{g_k(x_k + r_k x)}{r_k^{\gamma} [u_k]_{C^{\gamma}(\R^n)}}, \qquad \widetilde{b}_k(x) = r_k \frac{b_k(x_k + r_kx)}{r_k^{\gamma}[u_k]_{C^{\gamma}(\R^n)}}, \\
\widetilde{a}_k(x) &= \frac{a_k(x_k + r_k x)}{r_k^{\gamma} [u_k]_{C^{\gamma}(\R^n)}}, \qquad\qquad\quad ~~ \widetilde{A}_k(x) = r_k \frac{A_k(x_k + r_kx)}{r_k^{\gamma}[u_k]_{C^{\gamma}(\R^n)}},
\end{align*}
as well as
\begin{align*}
\widetilde{\theta}_k({x'}) = \theta_k((x_k + r_kx)'), \qquad \widetilde{\vartheta}_k({x'}) = \vartheta_k((x_k + r_kx)'), \qquad  \widetilde{\nu}_k({x'}) = \nu_k((x_k + r_kx)').
\end{align*}
Then, for any $\eta \in C^{\infty}_c(B_{r_k^{-1}}(-r_k^{-1}x_k ))$, we have
\begin{align*}
\int_{\R^n} & \int_{\R^n} (v_k(x)-v_k(y))(\eta(x) - \eta(y)) \widetilde{J}_k(x)\widetilde{J}_k(y)(x_n)_+^{\beta-1}(y_n)_+^{\beta-1} r_k^{n+2s} K(\widetilde{\Phi}_k(x) - \widetilde{\Phi}_k(y)) \d y \d x \\
&= \int_{\R^n} \eta(x) \widetilde{g}_k(x) \d x \\
&\quad + \int_{\R^n}\int_{\R^n} (\widetilde{f}_k(x) - \widetilde{f}_k(y))(\eta(x) - \eta(y)) (x_n)_+^{\beta-1}(y_n)_+^{\beta-1} \widetilde{I}_k(x,y) \d y \d x \\
&\quad + \1_{\{ \beta = s\}} \Bigg[ \int_{\{ x_n = 0 \}}  \widetilde{b}_k(x')(\theta_{k})_n (x') \eta(x',0) \d x' \\
&\qquad\qquad\qquad\quad - \int_{\{ x_n = 0 \}} v_k(x',0) [(\widetilde{\theta}_{k}' (x')\cdot \nabla_{x'} \eta(x',0)) +  \dvg \widetilde{\theta}_{k}'(x') \eta(x',0)] \d x'\\
&\qquad\qquad\qquad\quad + \int_{\{ x_n = 0 \}} \widetilde{A}_k(x') \cdot \widetilde{\vartheta}_{k}(x') \eta(x',0)  \d x'\\
&\qquad\qquad\qquad\quad - \int_{\{ x_n = 0 \}}  \widetilde{a}_k(x') [(\widetilde{\nu}_{k}'(x')\cdot \nabla_{x'} \eta(x',0)) +  {\dvg \widetilde{\nu}_{k}'(x') \eta(x',0)]} \d x'  \Bigg].
\end{align*}

\textbf{Step 2:} Let us now fix a ball $B \subset B_{r_k^{-1}/2}(-r_k^{-1}x_k )$ centered around the origin and $\eta \in C^{\infty}_c(B)$  satisfying \eqref{B8}. The goal of this step is to prove
\begin{align}
\label{eq:limit-ti-diffquot}
\begin{split}
\int_{\R^n} & \int_{\R^n} \widetilde{J}_k(0) \widetilde{J}_k(0) (x_n)_+^{\beta-1} (y_n)_+^{\beta-1} K_k((D \Phi_k)(x_k) \cdot (x-y))(v_k(x) - v_k(y))(\eta(x) - \eta(y))  \d y \d x \\
& - \1_{\{ \beta = s \}} \int_{\{ x_n = 0 \}} v_k(x',0) (\widetilde{\theta}_k'(0) \cdot \nabla_{x'} \eta(x',0)) \d x' \to 0.
\end{split}
\end{align}

To show it, taking into account \eqref{B88}, it is sufficient to prove
\begin{align}
\label{eq:conv-K2-diffquot}
\int_{\R^n}\int_{\R^n} (x_n)_+^{\beta-1}(y_n)_+^{\beta-1}(v_k(x) - v_k(y))(\eta(x) - \eta(y)) K_k^{(2)}(x,y) \d y \d x \to 0,\\
\label{eq:conv-K3-diffquot}
\int_{\R^n}\int_{\R^n} (x_n)_+^{\beta-1}(y_n)_+^{\beta-1}(v_k(x) - v_k(y))(\eta(x) - \eta(y)) K_k^{(3)}(x,y) \d y \d x \to 0,\\
\label{eq:conv-f-diffquot}
\int_{\R^n}\int_{\R^n} (x_n)_+^{\beta-1} (y_n)_+^{\beta-1} (\widetilde{f}_k(x) - \widetilde{f}_k(y))(\eta(x) - \eta(y)) \widetilde{I}_k(x,y) \d y \d x \to 0,\\
\label{eq:conv-g-diffquot}
 \int_{\R^n} \widetilde{g}_k(x) \eta(x) \d x \to 0.
\end{align}
Moreover, in case $\beta = s$, we need to show in addition:
\begin{align}
\label{eq:conv-theta-5-diffquot}
\int _{\{x_n=0\}} \widetilde{b}_k(x') (\widetilde{\theta}_k)_n(x') \eta(x',0) \d x' \to 0,\\
\label{eq:conv-theta-1-diffquot}
\int_{\{ x_n = 0 \}} v_k(x',0) [\widetilde{\theta}_k'(0) - \widetilde{\theta}_k'(x')] \cdot \nabla_{x'} \eta(x',0) \d x' &\to 0,\\
\label{eq:conv-theta-2-diffquot}
\int_{\{ x_n = 0 \} } \dvg \widetilde{\theta}_k'(x') v_k(x',0) \eta(x',0) \d x' &\to 0,\\
\label{eq:conv-theta-6-diffquot}
\int _{\{x_n=0\}} \widetilde{A}_k(x') (\widetilde{\vartheta}_{k})(x') \eta(x',0) \d x'\to 0,\\
\label{eq:conv-theta-3-diffquot}
{\int _{\{x_n=0\}} \widetilde{a}_k(x') (\widetilde{\nu}_{k}'(x') \cdot \nabla_{x'} \eta(x',0)) \d x' \to 0,}\\
\label{eq:conv-theta-4-diffquot}
\int_{\{ x_n = 0 \} } a_k(x') \dvg \widetilde{\nu}'_{k}(x') \eta(x',0) \d x' \to 0.
\end{align}

Let us first prove the claims \eqref{eq:conv-theta-5-diffquot} -- \eqref{eq:conv-theta-4-diffquot}, which are only relevant in case $\beta = s$.
We start by observing that due to \eqref{eq:B4-2} it holds for any ball $B \subset B_{r_k^{-1}}(-r_k^{-1}x_k)$,
\begin{align}
\label{eq:theta-k-tilde-bound}
\Vert \widetilde{\theta}_k \Vert_{L^{\infty}(\{ x_n = 0 \} \cap B)} + r_k^{-1} [\widetilde{\theta}_k]_{C^{0,1}(\{ x_n = 0 \} \cap B)} \le \Vert \theta_k \Vert_{C^{0,1}(\{ x_n = 0 \} \cap B_1)} \le \Lambda,\\
\label{eq:nu-tilde-bound}
\Vert \widetilde{\nu}_{k} \Vert_{L^{\infty}(\{ x_n = 0 \} \cap B)} + r_k^{-1} [\widetilde{\nu}_{k}]_{C^{0,1}(\{ x_n = 0 \} \cap B)} \le \Vert \nu_{k} \Vert_{C^{0,1}(\{ x_n = 0 \} \cap B_1)} \le \Lambda,\\
\label{eq:vartheta-tilde-bound}
\Vert \widetilde{\vartheta}_{k} \Vert_{L^{\infty}(\{ x_n = 0 \} \cap B)} \le \Vert \vartheta_{k} \Vert_{L^{\infty}(\{ x_n = 0 \} \cap B_1)} \le \Lambda.
\end{align}
Hence, the norms are all uniformly bounded in $k$. By \eqref{eq:theta-k-tilde-bound} and \eqref{eq:nu-tilde-bound} we have for $x \in \{ x_n = 0 \} \cap B$
\begin{align*}
|\widetilde{\theta}_k'(0) - \widetilde{\theta}_k'(x')| = |\theta_k'(x_k') -\theta_k'((x_k + r_kx)')| \le C r_k |x| \le C r_k, \\
|\dvg \widetilde{\theta}_k'(x')| = r_k|\dvg(\theta_k')((x_k + r_kx)')| \le C r_k,\\
|\dvg \widetilde{\nu}_{k}'(x')| = r_k|\dvg(\nu_{k}')((x_k + r_kx)')| \le C r_k.
\end{align*}

Moreover, due to \eqref{eq:blow-up-ass-diffquot} and since $\gamma < 1$, it holds
\begin{align*}
\Vert \widetilde{b}_k \Vert_{L^{\infty}(\{ x_n = 0 \} \cap B)} \le C r_k^{1-\gamma} \frac{\Vert b_k \Vert_{L^{\infty}( \{x_n = 0 \}\cap B_1)}} {[u_k]_{C^{\gamma}(\R^n)}} \to 0,\\
\Vert \widetilde{A}_k \Vert_{L^{\infty}(\{ x_n = 0 \} \cap B)} \le C r_k^{1-\gamma} \frac{\Vert A_k \Vert_{L^{\infty}( \{x_n = 0 \}\cap B_1)}} {[u_k]_{C^{\gamma}(\R^n)}} \to 0,\\
[ \widetilde{a}_k ]_{C^{\gamma}(\{ x_n = 0 \} \cap B)} \le C \frac{\Vert a_k \Vert_{C^{\gamma}( \{x_n = 0 \}\cap B_1)}} {[u_k]_{C^{\gamma}(\R^n)}} \to 0.
\end{align*}

Therefore, \eqref{eq:conv-theta-5-diffquot}, \eqref{eq:conv-theta-6-diffquot}, and \eqref{eq:conv-theta-4-diffquot} follow immediately, since $\eta \in C^{\infty}_c(B)$ is globally bounded. Similarly, \eqref{eq:conv-theta-1-diffquot}, \eqref{eq:conv-theta-2-diffquot} follow in a straightforward way, using also that $|v(x)| \le C$ in $B$ by the growth assumption implied by \eqref{Lav}.

To estimate \eqref{eq:conv-theta-3-diffquot} we notice that
\begin{align*}
\int _{\{x_n=0\}} \widetilde{a}_k(x') (\widetilde{\nu}'_{k}(x') \cdot \nabla_{x'} \eta(x',0)) \d x'&=\int _{\{x_n=0\}} \widetilde{a}_k(x') (\widetilde{\nu}'_{k}(x')- \widetilde{\nu}'_{k}(0)) \cdot \nabla_{x'} \eta(x',0) \d x'\\
&+\int _{\{x_n=0\}} (\widetilde{a}_k(x')-\widetilde{a}_k(0)) \widetilde{\nu}'_{k}(0) \cdot \nabla_{x'} \eta(x',0) \d x'\\
&+\int _{\{x_n=0\}} \widetilde{a}_k(0) \widetilde{\nu}'_{k}(0) \cdot \nabla_{x'} \eta(x',0) \d x'.
\end{align*}
It is clear that, by the classical integration by parts formula in $x'$, the third summand is equal to zero. Moreover, the second one goes to zero by the same argument as was used in \eqref{eq:conv-theta-1-diffquot}, indeed
\begin{align*}
& \left| \int _{\{x_n=0\}} (\widetilde{a}_k(x')-\widetilde{a}_k(0)) \widetilde{\nu}'_{k}(0) \cdot \nabla_{x'} \eta(x',0) \d x' \right| \\
&\quad \le C [\widetilde{a}_k]_{C^{\gamma}(\{ x_n = 0\} \cap B)} \Vert \widetilde{\nu}_k \Vert_{L^{\infty}(\{ x_n = 0 \} \cap B)} \left( \int_{\{ x_n = 0 \}} |x'|^{\gamma} \d x' \right) \le C [\widetilde{a}_k]_{C^{\gamma}(\{ x_n = 0\} \cap B)} \to 0.
\end{align*}
For the first summand, we use that by construction of $\widetilde{a}_k$, and since $\gamma < 1$,
\begin{align*}
& \left| \int _{\{x_n=0\}} \widetilde{a}_k(x') (\widetilde{\nu}'_{k}(x')- \widetilde{\nu}'_{k}(0)) \cdot \nabla_{x'} \eta(x',0) \d x' \right| \\
&\quad \le C \Vert \widetilde{a}_k \Vert_{L^{\infty}( \{ x_n = 0 \} \cap B ) } [\widetilde{\vartheta}_k]_{C^{0,1} (\{ x_n = 0 \} \cap B )} \left( \int_{\{ x_n = 0 \} \cap B } |x'| \d x' \right) \le C r_k \frac{\Vert a_k \Vert_{L^{\infty}( \{ x_n = 0 \} \cap B )}}{r_k^{\gamma} [u_k]_{C^{\gamma}(\R^n)}} \to 0.
\end{align*}
Altogether, this also establishes \eqref{eq:conv-theta-3-diffquot} and it remains to prove \eqref{eq:conv-K2-diffquot} -- \eqref{eq:conv-g-diffquot} for all $\beta \in [s,1+s]$.

To prove the remaining claims, we first observe that since $\gamma < 2s + 1 - \beta$,
\begin{align*}
\Vert (x_n)_+^{1 - \beta} \widetilde{g_k}\Vert_{L^{\infty}(\{ x_n \ge 0 \} \cap B_{r_k^{-1}}(-r_k^{-1}x_k) )} = r_k^{2s + 1 - \beta - \gamma} \frac{\Vert (x_n)_+^{1-\beta} g_k \Vert_{L^{\infty}(\{ x_n \ge 0 \} \cap B_1)} }{[u_k]_{C^{\gamma}(\R^n)}} \to 0.
\end{align*}
Therefore, and since $\beta < 2$
\begin{align*}
\left| \int_{\R^n} \widetilde{g}_k(x) \eta(x) \d x \right| &\le \Vert \eta \Vert_{L^{\infty}(\R^n)} \int_{\supp(\eta)} |\widetilde{g}_k(x)| \d x \\
&\le \Vert \eta \Vert_{L^{\infty}(\R^n)} \Vert (x_n)_+^{1-\beta} \widetilde{g}_k \Vert_{L^{\infty}(\{ x_n \ge 0 \} \cap B)} \int_{\supp(\eta)} (x_n)^{\beta - 1} \d x  \\
&\le C \Vert (x_n)_+^{1-\beta} \widetilde{g}_k \Vert_{L^{\infty}(\{ x_n \ge 0 \} \cap B)} \to 0 ~~ \text{ as }  k \to \infty,
\end{align*}
which proves \eqref{eq:conv-g-diffquot}. Moreover,
\begin{align*}
[ \widetilde{f}_k ]_{C^{\gamma}(\{ x_n \ge 0 \})} &= \frac{[ f_k ]_{C^{\gamma}(\{ x_n \ge 0 \})}}{[u_k]_{C^{\gamma}(\R^n)}} \to 0.
\end{align*}

Therefore, by \eqref{R} and \autoref{lemma:bd-int} we have,
\begin{align*}
& \left| \int_{2B} \int_{2B}  (x_n)_+^{\beta-1}(y_n)_+^{\beta-1} \widetilde{I}_k(x,y) (\widetilde{f}_k(x) - \widetilde{f}_k(y)) (\eta(x) - \eta(y)) \d y \d x \right| \\
&\le C [ \widetilde{f}_k ]_{C^{\gamma}(\{ x_n \ge 0 \})}  \int_{\{ x_n \ge 0 \} \cap 2B} (x_n)_+^{\beta-1} \int_{\{ x_n \le y_n \} \cap 2B}  (y_n)_+^{\beta - 1} |x-y|^{-n -2s + 1 + \gamma} \d y \d x \\
&\le C [ \widetilde{f}_k ]_{C^{\gamma}(\{ x_n \ge 0 \})} \int_{\{ x_n \ge 0 \} \cap 2B} (x_n)_+^{\beta-1}  + (x_n)_+^{2\beta-1 -2s + \gamma} \d x \\
&\le C [ \widetilde{f}_k ]_{C^{\gamma}(\{ x_n \ge 0 \})} \to 0,
\end{align*}
where the inner integral converges since $\gamma > 2s-1$ by assumption, and the outer integral converges since $\gamma > 2s - 2\beta$, which follows immediately, since $\gamma > 0$.

It remains to estimate the integral over $B \times (2B)^c$. Note that in the corresponding set it holds by construction $|x-y| \ge |y| - |x| \ge |y|/2$ and thus
\begin{align*}
 &  \left| \int_{B} \int_{(2B)^c}  (x_n)_+^{\beta-1}(y_n)_+^{\beta-1} \widetilde{I}_k(x,y) (\widetilde{f}_k(x) - \widetilde{f}_k(y)) (\eta(x) - \eta(y)) \d y \d x \right| \\
&\le C [ \widetilde{f}_k ]_{C^{\gamma}(\{ x_n \ge 0 \})} \int_{\{ x_n \ge 0 \} \cap B} \int_{\{ y_n \ge 0 \} \cap (2B)^c} (x_n)_+^{\beta-1}(y_n)_+^{\beta-1} |x-y|^{-n-2s + \gamma} \d y \d x \\
&\le C [ \widetilde{f}_k ]_{C^{\gamma}(\{ x_n \ge 0 \})} \left(\int_{\{ x_n \ge 0 \} \cap B} (x_n)_+^{\beta-1} \d x \right) \left( \int_{\{ y_n \ge 0 \} \cap (2B)^c} (y_n)_+^{\beta-1} |y|^{-n-2s+\gamma} \d y \right) \\
&\le C [ \widetilde{f}_k ]_{C^{\gamma}(\{ x_n \ge 0 \})} \to 0.
\end{align*}
Here, the first integral converges since $\beta > 0$ and the second integral converges by \autoref{lemma:int-polar-coord}, using that $-2s+\gamma+\beta - 1 < 0$.
The previous two estimates prove \eqref{eq:conv-f-diffquot}.

To see \eqref{eq:conv-K3-diffquot}, note that, by \eqref{P1} -- \eqref{eq:J-product-Holder}, and \eqref{Lav}, we have for any $\nu \in (0,\alpha]$,
\begin{align*}
& \left| \int_{2B} \int_{2B}  (x_n)_+^{\beta-1}(y_n)_+^{\beta-1}  (v_k(x) - v_k(y)) (\eta(x) - \eta(y)) K_k^{(3)}(x,y) \d y \d x \right| \\
&\le C \Vert \eta \Vert_{C^1(\R^n)} \Vert \widetilde{J}_k \Vert_{L^{\infty}(\R^n)} [\widetilde{J}_k]_{C^{\nu}(\{ x_n \ge 0 \})} [v_k]_{C^{\gamma}(\{ x_n \ge 0 \})} \cdot \\
&\qquad \qquad \cdot \int_{\{x_n \ge 0\} \cap 2B} \int_{ \{y_n \ge 0\} \cap 2B} (x_n)_+^{\beta-1} (y_n)_+^{\beta-1} (|x|^{\nu} + |x-y|^{\nu})  |x-y|^{-n-2s+1+\gamma} \d y \d x \\
&\le C r_k^{\nu}\int_{\{x_n \ge 0\} \cap 2B} (x_n)_+^{\beta-1} \int_{ \{y_n \ge 0\} \cap 2B} (y_n)_+^{\beta-1} |x-y|^{-n-2s+1+\gamma} \d y \d x \\
&\le C r_k^{\nu} \int_{\{x_n \ge 0\} \cap 2B} (x_n)_+^{\beta-1} (1 + (x_n)_+^{\beta - 2s + \gamma}) \d x \\
&\le C r_k^{\nu} \to 0,
\end{align*}
where, since $\beta>0$ and $\gamma>2s-1$ the inner integral converges by \autoref{lemma:bd-int}. The outer integral converges because $2 \beta + \gamma > 2s$, which holds since $\gamma > 0$. Moreover, we estimate for any $\nu \in (0,\alpha]$,
\begin{align*}
& \left| \int_{B} \int_{(2B)^c}  (x_n)_+^{\beta-1}(y_n)_+^{\beta-1}  (v_k(x) - v_k(y)) (\eta(x) - \eta(y)) K_k^{(3)}(x,y) \d y \d x \right| \\ 
&\le C \Vert \eta \Vert_{L^{\infty}(\R^n)} \Vert \widetilde{J}_k \Vert_{L^{\infty}(\R^n)} [\widetilde{J}_k]_{C^{\nu}(\{ x_n \ge 0 \} \cap B)} [v_k]_{C^{\gamma}(\{ x_n \ge 0 \})} \cdot \\
&\qquad \qquad \cdot \int_{\{x_n \ge 0\} \cap B} \int_{ \{y_n \ge 0\} \cap (2B)^c} (x_n)_+^{\beta-1} (y_n)_+^{\beta-1} (|x|^{\nu} + |x-y|^{\nu})  |x-y|^{-n-2s+\gamma} \d y \d x \\
&\le C r_k^{\nu} \int_{ \{x_n \ge 0\} \cap B}  (x_n)_+^{\beta-1} \int_{ \{y_n \ge 0\} \cap (2B)^c}  (y_n)_+^{\beta-1} |x-y|^{-n-2s+\gamma + \nu} \d y \d x \\
&\le C r_k^{\nu} \left(\int_{\{x_n \ge 0\} \cap B}  (x_n)_+^{\beta-1}\d x \right) \left( \int_{ \{y_n \ge 0\} \cap (2B)^c} (y_n)_{+}^{\beta-1}|y|^{-n-2s+\gamma+ \nu} \d y \right)\\
& \le C r_k^{\nu} \to 0,
\end{align*}
where the outer integral trivially converges and the inner integral converges because we can take $\nu>0$ small enough in order to guarantee $\beta - 1 - 2s + \gamma + \nu < 0$.

A combination of the previous two integral estimates implies \eqref{eq:conv-K3-diffquot}. Finally, to obtain \eqref{eq:conv-K2-diffquot} we proceed in a similar way by using \eqref{eq:K_k-reg-interior} instead of \eqref{eq:J-product-Holder}. Then, for any $\sigma' \in (0,\sigma]$ and $\mu' \in (0,1+\alpha]$,
\begin{align*}
& \left| \int_{2B} \int_{2B}  (x_n)_+^{\beta-1}(y_n)_+^{\beta-1}  (v_k(x) - v_k(y)) (\eta(x) - \eta(y)) K_k^{(2)}(x,y) \d y \d x \right|\\
&\le C r_k^{(\mu' - 1)\sigma'} \int_{\{x_n \ge 0\} \cap 2B} (x_n)_+^{\beta-1} (1 + (x_n)_+^{\beta - 2s - \sigma' + \gamma}) \d x\\
&\le C r_k^{(\mu' - 1)\sigma'} \to 0,
\end{align*}
where we used \autoref{lemma:bd-int} and that $\gamma> 2s - 1 + \sigma'$ by taking $\sigma'$ small enough. We also highlight that the integral over $\{x_n \ge 0\} \cap 2B$ converges, since $2 \beta + \gamma > 2s+\sigma'$.

It remains to estimate, using \eqref{eq:K_k-reg-interior}, for any $\mu'\in (0,1+\alpha]$ and $\sigma'\in (0,\sigma]$,
\begin{align*}
& \left| \int_{B} \int_{(2B)^c}  (x_n)_+^{\beta-1}(y_n)_+^{\beta-1}  (v_k(x) - v_k(y)) (\eta(x) - \eta(y)) K_k^{(2)}(x,y) \d y \d x \right|\\ &\le C r_k^{(\mu' - 1)\sigma'} \Vert \eta \Vert_{L^{\infty}(\R^n)} \Vert \widetilde{J}_k \Vert_{L^{\infty}(\R^n)}^2 [v_k]_{C^{\gamma}(\{ x_n \ge 0 \})} \cdot\\
&\qquad \qquad \cdot \int_{\{x_n \ge 0\} \cap B} \int_{ \{x_n \ge 0\} \cap (2B)^c} (x_n)_+^{\beta-1} (y_n)_+^{\beta-1} (|x|^{\mu' \sigma'} + |x-y|^{\mu' \sigma'})  |x-y|^{-n-2s+\gamma - \sigma'} \d y \d x \\
&\le C r_k^{(\mu' - 1)\sigma'} \int_{\{x_n \ge 0\} \cap B}  (x_n)_+^{\beta-1} \int_{ \{y_n \ge 0\} \cap (2B)^c}  (y_n)_+^{\beta-1} |x-y|^{-n-2s +\gamma + (\mu' - 1) \sigma'} \d y \d x \\
&\le C r_k^{(\mu' - 1)\sigma'} \left( \int_{\{x_n \ge 0\} \cap B}  (x_n)_+^{\beta-1} \d x \right) \left( \int_{ \{y_n \ge 0\} \cap (2B)^c} (y_n)_{+}^{\beta-1}|y|^{-n-2s+\gamma+(\mu' - 1)\sigma' } \d y \right) \\
& \le C r_k^{(\mu' - 1)\sigma'} \to 0,
\end{align*}
where the integrals converge by considering $(\mu',\sigma')$ such that $(\mu'-1)\sigma'> 0$ and $(\mu' - 1) \sigma' + \gamma < 2s + \beta -1$. This can be achieved by fixing $\gamma < 2s + \beta-1$ and $\mu' > 1$ first, and then taking $\sigma'$ small enough.

\textbf{Step 3:} The goal of this step is to prove that there exist $v \in C^{\gamma}(\R^n)$ with
\begin{align}
\label{eq:v-bound-bdry}
|v(x)| \le |x|^{\gamma},
\end{align}
and a kernel $K$ satisfying \eqref{eq:K-comp}
such that,
\begin{align*}
\begin{cases}
\mathcal{L}\big((x_n)_+^{\beta-1}(y_n)_+^{\beta-1}K(x-y)\big)(v) &= 0 ~~ \text{ in } \{ x_n > 0 \},\\
\partial_n v &= 0 ~~ \text{ on } \{ x_n = 0 \}
\end{cases}
\end{align*}
in the following sense: For $\theta_K \in \R^n$, defined in \autoref{lemma:ibp}, and any $\eta \in C_c^{\infty}(\R^n)$ it holds
\begin{align}
\label{eq:global-sol-diffquot}
\begin{split}
\int_{\R^n} & \int_{\R^n} (v(x) - v(y))(\eta(x) - \eta(y)) (x_n)_+^{\beta-1}(y_n)_+^{\beta-1}K(x-y) \d y \d x \\
&= \1_{\{\beta = s\}} \int_{\{ x_n = 0 \}} v(x',0) \theta'_K \cdot \nabla_{x'}\eta(x',0) \d x'.
\end{split}
\end{align}

The existence of $v$ with $v_k \to v$ locally uniformly satisfying the desired bound \eqref{eq:v-bound-bdry} already follows from the Arzel\`a-Ascoli theorem. Moreover, in the light of \eqref{eq:limit-ti-diffquot} it suffices to show
\begin{align}
\label{eq:interior-PDE-convergence}
\begin{split}
 \int_{\R^n} & \int_{\R^n} \widetilde{J}_k(0) \widetilde{J}_k(0) (x_n)_+^{\beta-1} (y_n)_+^{\beta-1} K_k((D \Phi_k)(x_k) \cdot (x-y))(v_k(x) - v_k(y))(\eta(x) - \eta(y))  \d y \d x \\
 & \to J_0^2 \int_{\R^n} \int_{\R^n} (x_n)_+^{\beta-1} (y_n)_+^{\beta-1} K(x-y)(v(x) - v(y))(\eta(x) - \eta(y))  \d y \d x
\end{split}
\end{align}
for some kernel $K$ satisfying \eqref{eq:K-comp} and some $J_0 > 0$ and any $\eta \in C^{\infty}_c(\R^n)$. In case $\beta = s$, we also need 
\begin{align}
\label{eq:boundary-condition-convergence}
\int_{\{ x_n = 0 \}} v_k(x',0) \widetilde{\theta}_k'(0) \cdot \nabla_{x'} \eta(x',0) \d x' \to \int_{\{ x_n = 0 \}} v(x',0) \theta_K' \cdot \nabla_{x'} \eta(x',0) \d x'.
\end{align}

The existence of $K$ follows from \cite[Proposition 2.2.36]{FeRo24} applied to the family of homogeneous kernels $K_k((D \Phi_k)(x_k) \cdot (x-y))$.
Note that for any ball $B$ centered around the origin it holds
\begin{align*}
& \Big| \int_B \int_B \widetilde{J}_k(0) \widetilde{J}_k(0) (x_n)_+^{\beta-1} (y_n)_+^{\beta-1} K_k((D \Phi_k)(x_k) \cdot (x-y))\cdot\\
&\qquad\qquad\qquad\qquad\qquad\qquad \cdot((v_k - v)(x) - (v_k - v)(y))(\eta(x) - \eta(y))  \d y \d x \Big| \\
&\qquad\le C [v_k - v]_{C^{\gamma}(B)} \int_B \int_B (x_n)_+^{\beta-1} (y_n)_+^{\beta-1} |x-y|^{-n-2s+1+\gamma} \d y \d x\\
&\qquad\le C [v_k - v]_{C^{\gamma}(B)} \int_B (x_n)_+^{\beta-1} (1 + (x_n)_+^{\beta - 2s +\gamma}) \d x \to 0,
\end{align*}
where the integrals converge by \autoref{lemma:bd-int} since $\gamma > 2s - 1$ and $2\beta + \gamma > 2s$.

Moreover, it holds for any $R > 1$ with $\supp(\eta) \subset B_R$, using $|v_k(y) - v(y)|\le C |y|^{\gamma}$ and $|x-y|\geq|y|/2$,
\begin{align*}
& \Big| \int_{\supp(\eta)} \int_{B_{2R}^c} \widetilde{J}_k(0) \widetilde{J}_k(0) (x_n)_+^{\beta-1} (y_n)_+^{\beta-1} K_k((D \Phi_k)(x_k) \cdot (x-y)) \cdot \\
&\qquad\qquad\qquad\qquad\qquad\qquad \cdot((v_k - v)(x) - (v_k - v)(y))(\eta(x) - \eta(y))  \d y \d x \Big| \\
&\qquad\le C \int_{\supp(\eta)} |v_k(x) - v(x)| (x_n)_+^{\beta-1} \left( \int_{B_{2R}^c} (y_n)_+^{\beta - 1} |y|^{-n-2s} \d y \right) \d x \\
&\qquad\quad + C \int_{\supp(\eta)} (x_n)_+^{\beta-1}  \left( \int_{B_{2R}^c} (y_n)_+^{\beta - 1} |v_k(y) - v(y)| |y|^{-n-2s} \d y \right) \d x  \\
&\qquad\le C R^{\beta - 1 - 2s} \Vert v_k - v \Vert_{L^{\infty}(\supp(\eta))} + C \int_{\supp(\eta)} (x_n)_+^{\beta-1}  \left( \int_{B_{2R}^c} (y_n)_+^{\beta - 1} |y|^{-n-2s + \gamma} \d y \right) \d x \\
&\qquad\le C R^{\beta - 1 - 2s + \gamma},
\end{align*}
where we used \autoref{lemma:int-polar-coord}. Note that the right-hand side converges to zero uniformly in $k$, as $R \to \infty$ since $\beta -1 - 2s + \gamma < 0$.

By combination of the previous two estimates, we deduce that
\begin{align*}
\Bigg| \int_{\R^n} \int_{\R^n} \widetilde{J}_k(0) \widetilde{J}_k(0) & (x_n)_+^{\beta-1} (y_n)_+^{\beta-1} K_k((D \Phi_k)(x_k) \cdot (x-y)) \cdot \\
& \cdot ((v_k - v)(x) - (v_k - v)(y))(\eta(x) - \eta(y))  \d y \d x \Bigg| \to 0,
\end{align*}
choosing first $R$ so large that the contribution from the second estimate becomes small and then choosing $k$ large so that the contribution from the first term with $B = B_{2R}$ becomes small.

Hence, in order to prove that \eqref{eq:interior-PDE-convergence} holds true  it remains to show
\begin{align*}
\Bigg| \int_{\R^n} \int_{\R^n} & (x_n)_+^{\beta-1} (y_n)_+^{\beta-1} [\widetilde{J}_k(0) \widetilde{J}_k(0) K_k((D \Phi_k)(x_k) \cdot (x-y)) - J_0^2K(x-y)] \cdot \\
&\cdot (v(x) - v(y))(\eta(x) - \eta(y))  \d y \d x \Bigg| \to 0.
\end{align*}
Note that there is $J_0 = \lim_{k \to \infty} \widetilde{J}_k(0)$ by the uniform boundedness of $\widetilde{J}_k(0)$ and also a kernel $K$ satisfying \eqref{eq:K-comp} such that
\begin{align*}
(x_n)_+^{\beta-1} (y_n)_+^{\beta-1} (1 \wedge |x-y|^2)K_k((D \Phi_k)(x_k) \cdot (x-y)) \to (x_n)_+^{\beta-1} (y_n)_+^{\beta-1}(1 \wedge |x-y|^2)K(x-y)
\end{align*}
weakly in the sense of measures in $(B^c \times B^c)^c$ for any ball $B \subset \R^n$ centered at zero, by the same arguments as before (see also \cite{FeRo24}). This implies the above convergence and yields \eqref{eq:interior-PDE-convergence}.

Note that \eqref{eq:boundary-condition-convergence} follows immediately by convergence of $v_k\to v$, and the fact that, up to a subsequence, $\widetilde{\theta}_k'(0) \to \theta_0' \in \R^n$, due to the uniform boundedness of $\theta_k$. 
Moreover, by following the proof of \autoref{lemma:ibp-nonflat-trafo}, it becomes apparent that $\widetilde{\theta}_k(0)$ is precisely the vector appearing in the boundary integral for the energy driven by the kernel $\widetilde{J}_k(0)^2 K_k((D \Phi_k)(x_k) \cdot (x-y))$. Hence, by the convergence properties of $J_k$, $K_k$, and $\Phi_k$, as $k \to \infty$, we obtain that $\theta_0 = \theta_{K}$ from \autoref{lemma:ibp}, and therefore \eqref{eq:global-sol-diffquot} follows.
Hence, altogether, we have proved \eqref{eq:global-sol-diffquot}.

\textbf{Step 4:} By Step 3, we have that $v$ satisfies \eqref{eq:global-sol-diffquot}, i.e. it is a weak solution in the sense of \autoref{def:weak-sol-ti}. Together with the growth condition in \eqref{eq:v-bound-bdry}, we are in a position to apply the weighted Liouville theorem in \autoref{prop:weighted-Liouville} and deduce that $v$ is constant. However, by \eqref{eq:v-bound-bdry}, this implies $v = 0$. On the other hand, recalling the uniform convergence $v_k \to v$, as well as the estimate \eqref{eq:blowup-contradiction-ass-diffquot},
\begin{align*}
    \Vert v \Vert_{L^{\infty}(\{ x_n \ge 0 \} \cap B_{1/2})} > \frac{\delta}{2} > 0,
\end{align*}
a contradiction. Hence, the proof is complete.
\end{proof}

We are now in a position to establish \autoref{lemma:bdry-reg} replacing the difference quotient over $x \in \{ x_n = 0 \} \cap B_{1/2}$ and $x \in \{ x_n \ge 0 \} \cap B_{1/2}$ by the $C^{\gamma}(\{ x_n \ge 0 \} \cap B_{1/2})$ norm.

\begin{proof}[Proof of \autoref{lemma:bdry-reg}]
It is enough to prove that for any $\delta \in (0,1)$ there exists $C_{\delta} > 0$ such that
\begin{align}
\label{eq:almost-reg-est-try}
\begin{split}
[u]_{C^{\gamma}(\{ x_n \ge 0\} \cap B_{1/4})} &\le \delta [u]_{C^{\gamma}(\{ x_n \ge 0\} )} \\
&\quad + C_{\delta} \Big( \Vert u \Vert_{L^{\infty}(\{ x_n \ge 0\} \cap B_{1})} + \Vert (x_n)_+^{\beta-1} u \Vert_{L^1_{2s}(\{ x_n > 0 \} \setminus B_1)} \\
&\quad ~~ + \Vert (x_n)_+^{1-\beta} g_k \Vert_{L^{\infty}(\{ x_n \ge 0 \} \cap B_1) }  + \Vert f \Vert_{C^{\gamma}(\{ x_n \ge 0\})} \\
&\quad ~~ + \1_{\{ \beta = s\}} \big[  \Vert b \Vert_{L^{\infty}(\{x_n = 0 \} \cap B_1)} + \Vert A \Vert_{L^{\infty}(\{x_n = 0 \} \cap B_1)} + \Vert a \Vert_{C^{\gamma}(\{x_n = 0 \} \cap B_1)} \big]\Big).
\end{split}
\end{align}
From here, the desired result follows by a standard covering argument, after applying \eqref{eq:almost-reg-est-try} to $\xi u$, where $\xi \in C^{\infty}_c(B_3)$ is a cut-off function satisfying $0 \le \xi \le 1$ and $\xi \equiv 1$ in $B_{2}$, and using that
\begin{align*}
& \Vert (x_n)_+^{1-\beta}|\mathcal{L}\big(J(x)J(y) (x_n)_+^{\beta-1}(y_n)_+^{\beta-1} K(\Phi(x)-\Phi(y))\big)((1-\xi)u) | \Vert_{L^{\infty}(\{ x_n > 0 \} \cap B_1)} \\
&\qquad \le C \int_{B_1^c} (y_n)_+^{\beta-1} |u(y)| |y|^{-n-2s} \d y.
\end{align*}
To prove \eqref{eq:almost-reg-est-try} we use \autoref{lemma:interior}, \autoref{lemma:interior-2} and \autoref{lemma:almost-reg-bdry}. First, we assume without loss of generality 
\begin{align*}
\Vert u \Vert_{L^{\infty}(\{ x_n \ge 0 \} \cap B_1)} &+\Vert (x_n)_+^{\beta-1} u \Vert_{L^1_{2s}(\{ x_n > 0 \} \setminus B_1)} + \Vert (x_n)_+^{1-\beta} g \Vert_{L^{\infty}(\{ x_n \ge 0 \} \cap B_1) } +  \Vert f \Vert_{C^{\gamma}(\{ x_n \ge 0\})} \\
& + \1_{\{ \beta = s\}} \big[  \Vert b \Vert_{L^{\infty}(\{x_n = 0 \} \cap B_1)} + \Vert A \Vert_{L^{\infty}(\{x_n = 0 \} \cap B_1)} + \Vert a \Vert_{C^{\gamma}(\{x_n = 0 \} \cap B_1)} \big]\le 1.
\end{align*}
Let us take $x,y \in \{ x_n \ge 0\} \cap B_{1/4}$ and assume without loss of generality that $(x_n)_+ \le (y_n)_+$. Let us denote by $x_0 = (x',0)$ and $y_0 = (y',0)$ the projections of $x,y$ to $\{ x_n = 0\}$, i.e. it holds $x_n = |x - x_0|$ and $y_n = |y - y_0|$. By \autoref{lemma:almost-reg-bdry} we have
\begin{align}
\label{eq:almost-reg-bdry-apply}
\sup_{z \in \{ x_n \ge 0 \} \cap B_{1/2} } \frac{|u(x_0)- u(z)|}{|x_0 - z|^{\gamma}} + \sup_{z \in \{ x_n \ge 0 \} \cap B_{1/2} } \frac{|u(y_0)- u(z)|}{|y_0 - z|^{\gamma}} \le \delta [u]_{C^{\gamma}(\{ x_n \ge 0\} )} + C_{\delta}.
\end{align}

\textbf{Case 1:} Assume that $x \in B_{y_n/2}(y)$. In this case, if $\gamma < 2s$ we apply \autoref{lemma:interior} to  $u - u(x)$, and if $\gamma > 2s$, we apply \autoref{lemma:interior-2} to $u-u(x)$ with $\alpha := \gamma-2s$, and obtain
\begin{align*}
\frac{|u(x) - u(y)|}{|x-y|^{\gamma}} &\le C (y_n)^{-\gamma} \Vert u - u(x) \Vert_{L^{\infty}(B_{y_n}(y))} \\
&\quad + C(y_n)_+^{2s + 1-\beta - \gamma} \Vert (\cdot_n)_+^{\beta-1} |u-u(x)||\cdot - y|^{-n-2s} \Vert_{L^1(B_{y_n}(y)^c)} + C.
\end{align*}

Using \eqref{eq:almost-reg-bdry-apply} and $|x_0 - x| = x_n \le y_n$ we deduce:
\begin{align*}
(y_n)^{-\gamma} \Vert u - u(x) \Vert_{L^{\infty}(B_{y_n}(y))} &\le C (y_n)^{-\gamma} |u(x_0) - u(x)| + C(y_n)^{-\gamma} \Vert u - u(x_0) \Vert_{L^{\infty}(B_{y_n}(y))} \\
&\le (y_n)^{-\gamma} (x_n)^{\gamma} (\delta [u]_{C^{\gamma}(\{ x_n \ge 0\} )} + C_{\delta}) + C(\delta [u]_{C^{\gamma}(\{ x_n \ge 0\} )} + C_{\delta}) \\
& \le C \delta [u]_{C^{\gamma}(\{ x_n \ge 0\} )} + C_{\delta}.
\end{align*}
Moreover, by the same argument, for any $z \in B_{1/2} \setminus B_{y_n}(y)$, it holds
\begin{align*}
(y_n)^{-\gamma} |u(z) - u(x)| \le C \delta [u]_{C^{\gamma}(\{ x_n \ge 0\} )} + C_{\delta},
\end{align*}
and therefore,
\begin{align*}
(y_n)_+^{2s + 1-\beta - \gamma} & \Vert (\cdot_n)_+^{\beta-1} |u-u(x)||\cdot - y|^{-n-2s} \Vert_{L^1(B_{y_n}(y)^c)} \\
& \le C (y_n)_+^{2s + 1-\beta - \gamma} \int_{B_{1/2} \setminus B_{y_n}(y)} (z_n)_+^{\beta-1}|u(z) - u(x)| |z-y|^{-n-2s} \d z \\
&\quad + C (y_n)_+^{2s + 1-\beta - \gamma}\int_{B_{1/2}^c} (z_n)_+^{\beta-1}|u(z) - u(x)| |z-y|^{-n-2s} \d z \\
&\le C \delta [u]_{C^{\gamma}(\{ x_n \ge 0\} )} + C_{\delta},
\end{align*}
where we used \cite[Lemma B.2.4]{FeRo24} for the first summand and the fact that $\gamma < 2s-\beta+1$ for the second summand.

\textbf{Case 2:} Assume that $x \not\in B_{y_n/2}(y)$. In this case it is clear that
\begin{align*}
|x-y| \ge |x_0 - y_0|, \qquad |x-y| \ge y_n/2 \ge x_n/2 = |x - x_0|/2, \qquad |x-y| \ge y_n/2 = |y - y_0|/2.
\end{align*}
Hence, by \autoref{lemma:almost-reg-bdry}, we have
\begin{align*}
\frac{|u(x) - u(y)|}{|x-y|^{\gamma}} \le C \frac{|u(x) - u(x_0)|}{|x-x_0|^{\gamma}} + C \frac{|u(x_0) - u(y_0)|}{|x_0-y_0|^{\gamma}}  + C \frac{|u(y_0) - u(y)|}{|y_0-y|^{\gamma}} \le C\delta [u]_{C^{\gamma}(\{ x_n \ge 0\} )} + C_{\delta}.
\end{align*}

Altogether, we deduce \eqref{eq:almost-reg-est-try} in both cases and the proof is complete.
\end{proof}

\section{Higher regularity of free boundaries}\label{section:applications}

\subsection{One-phase problem}

The goal of this section is to prove our main result on the nonlocal one-phase free boundary problem, namely that $C^{2,\alpha}$ free boundaries are $C^{\infty}$ at regular points (see \autoref{thm:main-onephase-intro} and \autoref{thm:main-onephase-intro2}). We will prove it in several steps, using a delicate bootstrap argument.

\begin{lemma}
\label{lemma:PDE-onephase}
Let $K$ be as in \eqref{eq:K-comp} and assume that
\begin{align}
\label{eq:K-reg-ass}
K\in C^{2(k+\alpha) + 1}(\mathbb{S}^{n-1}),
\end{align} for some $k \in \N$ with $k \ge 2$ and $\alpha \in (0,1)$ with $\alpha \not\in \{s,1-s\}$.\\
Let $\partial \Omega \in C^{k,\alpha}$ in $B_4$ and $v \in L^{1}_{2s}(\R^n)$ be a nonnegative weak solution to
\begin{align*}
\begin{cases}
L v &= 0 ~~ \text{ in } \Omega \cap B_4,\\
v &=0 ~~ \text{ in } B_4 \setminus \Omega,\\
\frac{v}{d^s_{\Omega}} &= h ~~ \text{ on } B_4 \cap \partial \Omega
\end{cases}
\end{align*}
for some $h \in C^{k,\alpha}(B_4 \cap \partial \Omega)$. 
Let $0 \in \partial \Omega$, and assume that for some $\delta > 0$,
\begin{align}
\label{eq:blue-ass}
\partial_n v \ge \delta, \qquad \partial_n (d_{\Omega}) \ge \delta \qquad \text{ in } \Omega \cap B_2.
\end{align}

Let $\kappa_1, \kappa_2 \in  C^{\infty}(\R^n)$ be two cut-off functions such that $0 \le \kappa_i \le 1$ in $\R^n$ and $\kappa_1 \equiv 1$ in $B_{7/4}$, $\kappa_1 \equiv 0$ in $B_2^c$, and $\kappa_2 \equiv 1$ in $B_{4/3}$, $\kappa_2 \equiv 0$ in $B_{3/2}^c$. Then, for any $i \in \{ 1 , \dots , n-1\}$, the function
\begin{align}
\label{eq:w-def}
\widetilde{w} := \widetilde{w}_i := \kappa_2 \frac{\partial_i u}{\partial_n u}, \qquad \text{ where } \qquad u = \kappa_1 v,
\end{align} 
satisfies
\begin{itemize}
\item[(i)] $d_{\Omega}^{1-s} \nabla u \in C^{k-1,\alpha}(\overline{\Omega} \cap B_{2})$ with $\supp(u) \subset \overline{\Omega \cap B_2}$, and it holds $d_{\Omega}^{1-s} \partial_n u \ge c_0 > 0$ in $\Omega \cap B_{3/2}$ for some $c_0 > 0$, depending on $\delta$.
\item[(ii)] $\widetilde{w} \in C_{k-1+,\alpha}^{k+\alpha+2s}(\Omega | B_{3/2})$ with $\supp(\widetilde{w}) \subset \overline{\Omega \cap B_{3/2}}$.

\item[(iii)] The following equation holds true for $\widetilde{w}$ in the strong sense.
\begin{align*}
\begin{cases}
\mathcal{L}\big(\partial_n u(x)\partial_n u(y)K(x-y)\big)(\widetilde{w}) &= \widetilde{g} ~~ \text{ in } \Omega \cap B_1,\\
\partial_{\Xi} \widetilde{w} &= \widetilde{b} ~~ \text{ on } \partial \Omega \cap B_1,
\end{cases}
\end{align*}
where $\widetilde{g}$ satisfies $d_{\Omega}^{1-s} \widetilde{g} \in C^{k-1,\alpha}(\overline{\Omega} \cap B_{3/2})$, and given a function $\Xi = \Xi(x) \in \R^n$ on $\partial \Omega$, 
\begin{align}
\label{eq:tilde-b-def}
\begin{split}
\widetilde{b}(x) &:= \frac{s h(x) \Xi(x) \cdot \partial_{\tau(x)} \nu(x) + \Xi(x) \cdot \nu(x)\partial_{\tau(x)} h(x)}{s h(x) (\nu^{(n)}(x))^2} ~~ \text{ on } \partial \Omega \cap B_1,\\
\tau(x) &:= (\nu^{(n)}(x) e^{(i)} - \nu^{(i)}(x) e^{(n)} ),
\end{split}
\end{align}
where $\tau(x) \perp \nu(x)$ is tangential to $\partial \Omega$.  In particular, if \ $\Xi(x) = \nu(x)$, then 
\[\partial_\nu \tilde w = \frac{\partial_\tau h}{sh \nu^{(n)}} \in C^{k-1,\alpha}(\partial\Omega\cap B_1).\]
\end{itemize}
\end{lemma}

\begin{proof}
Due to interior estimates and by the regularity of $K$, $v$ is a classical solution and $v \in L^{\infty}(\Omega \cap B_{2})$ (see \cite[Theorem 2.4.1]{FeRo24}).
Hence, by construction $u\in L^{\infty}(\R^n)$ is a strong solution to
\begin{align}
\label{eq:u-PDE}
\begin{cases}
Lu &= g_1 ~~ \text{ in } \Omega \cap B_{3/2},\\
u&= 0 ~~ \text{ in }{\R^n \setminus (B_2 \cap \Omega)},
\end{cases}
\end{align}
where $g_1 = -L((1-\kappa_1)v) \geq0$.

By \autoref{lemma:boundary-reg} we have $d_{\Omega}^{1-s} \nabla v \in C^{k+\alpha+2s}_{k-1+\alpha}(\Omega  | B_{2})$, and by the definition of $u$, it therefore also holds $d_{\Omega}^{1-s} \nabla u \in C^{k+\alpha+2s}_{k-1+\alpha}( \Omega | B_{2})$, which in particular implies $d_{\Omega}^{1-s} \nabla u \in C^{k-1,\alpha}_c(\overline{\Omega} \cap B_{2})$. The second claim in (i) follows immediately by the third claim in \autoref{lemma:Hopf}, since by construction, $\partial_n u \ge \delta$ in $\Omega \cap B_{3/2}$.

To prove (ii), we observe first that since $d_{\Omega}^{1-s} \nabla u \in C^{k+\alpha+2s}_{k-1,\alpha}(\Omega | B_{3/2})$ and $d_{\Omega}^{1-s} \partial_n u \ge c_0 > 0$ in $\Omega \cap B_{3/2}$, by the chain rule, it holds in particular
\begin{align*}
(d_{\Omega}^{1-s} \partial_n u)^{-1} \in C^{k+\alpha+2s}_{k-1,\alpha}(\Omega | B_{3/2}).
\end{align*}
Therefore, by the product rule, we have
\begin{align*}
\frac{\partial_i u}{\partial_n u} = \frac{\partial_i u}{d^{s-1}_{\Omega}} \frac{d^{s-1}_{\Omega}}{\partial_n u} \in C^{k+\alpha+2s}_{k-1,\alpha}(\Omega | B_{3/2}),
\end{align*}
which proves (ii). 

Let us now prove (iii). 
By the regularity of $K$, it holds $g_1 \in C^{2(k+\alpha)+ 1}(B_{4/3})$. Moreover,
\begin{align*}
{\Vert g_1 \Vert_{C^{2(k+\alpha)+ 1}(B_{4/3})}\le C \Vert v \Vert_{L^{1}_{2s}(\R^n)}.}
\end{align*} 

In particular, by differentiating the equation for $u$, we deduce for any $i \in \{1,\dots,n\}$
\begin{align*}
\begin{cases}
L (\partial_i u) &= \partial_i g_1 ~~ \text{ in } \Omega \cap B_{3/2},\\
\partial_i u &= 0 ~~ \text{ in } {\R^n \setminus ({B_{2}} \cap \Omega),}
\end{cases}
\end{align*}
and moreover,
\begin{align*}
\begin{cases}
L (\kappa_2 \partial_i u) &= \partial_i f - L((1-\kappa_2)\partial_i u) ~~ \text{ in } \Omega \cap B_{1},\\
\kappa_2 \partial_i u &= 0 \qquad\qquad\qquad\qquad ~~ ~~ \text{ in } \R^n \setminus ({B_{3/2}} \cap \Omega).
\end{cases}
\end{align*}

We observe the following algebraic identity:
\begin{align}
\label{eq:algebra-quotient-PDE}
\begin{split}
& (\kappa_2(x)\partial_i u(x) \partial_n u(y) - \kappa_2(y)\partial_i u(y) \partial_n u(x))(\eta(x) - \eta(y)) \\
&\qquad\qquad\qquad = (\partial_n u(x)\eta(x) - \partial_n u(y) \eta(y))(\kappa_2(x)\partial_i u(x) - \kappa_2(y)\partial_i u(y)) \\ 
&\qquad\qquad\qquad \quad- (\kappa_2(x)\partial_i u(x) \eta(x) - \kappa_2(y)\partial_i u(y) \eta(y))(\partial_n u(x) - \partial_n u(y)),
\end{split}
\end{align}
and therefore, we obtain for any $\eta \in C_c^{\infty}(\Omega \cap B_1)$
\begin{align*}
&\int_{\R^n} \int_{\R^n} (\widetilde{w}(x) - \widetilde{w}(y)) (\eta(x) - \eta(y)) \partial_n u(x) \partial_n u(y) K(x-y) \d y \d x \\
&\quad = \int_{\R^n} \int_{\R^n} (\kappa_2(x)\partial_i u(x) \partial_n u(y) - \kappa_2(y)\partial_i u(y) \partial_n u(x)) (\eta(x) - \eta(y)) K(x-y) \d y  \d x \\
&\quad = \int_{\R^n} \int_{\R^n} (\partial_n u(x)\eta(x) - \partial_n u(y) \eta(y))(\kappa_2(x)\partial_i u(x) - \kappa_2(y)\partial_i u(y)) K(x-y) \d y  \d x \\
&\quad\quad - \int_{\R^n} \int_{\R^n} (\kappa_2(x) \partial_i u(x) \eta(x) - \kappa_2(y) \partial_i u(y) \eta(y))(\partial_n u(x) - \partial_n u(y)) K(x-y) \d y  \d x \\
&\quad = \int_{\R^n} [(\partial_n u) (\partial_i g_1 - L((1-\kappa_2)\partial_i u)) - (\kappa_2\partial_i u) (\partial_n g_1)](x)\eta(x) \d x = \int_{\R^n} \widetilde{g}(x) \eta(x) \d x.
\end{align*}
In the last step, we used the equation for $\kappa_2\partial_i u$ with $\partial_n u \eta$ as a test function and also the equation for $\partial_n u$ with $\kappa_2 \partial_i \eta$ as a test function.

Note that the integrals on both sides of the chain of equalities are well-defined. Since $\eta$ is supported away from $\partial \Omega$, this is immediate by the interior regularity of $\widetilde{w},u$ established in (i), (ii), and the regularity of $f$.
In particular, by Fubini's theorem, we have
\begin{align*}
\int_{\R^n} \eta(x)\mathcal{L}\big(\partial_nu(x)\partial_nu(y)K(x-y)\big)(\widetilde{w})  \d x = \int_{\R^n} \widetilde{g}(x) \eta(x) \d x,
\end{align*}
with $\widetilde{g}$ defined as
\begin{align}
\label{eq:g-def}
\widetilde{g} = (\partial_n u) (\partial_i g_1 - L((1-\kappa_2)\partial_i u)) - (\kappa_2\partial_i u) (\partial_n g_1) , \qquad \text{ where } \qquad g_1 = -L ((1-\kappa_1)v).
\end{align}
Therefore, we deduce the equation for $\widetilde{w}$, since $\eta$ was arbitrary. Note that the left-hand side is finite by the properties of $u,\widetilde{w}$ from (i) and (ii), namely in particular we have
${\widetilde w \in C^{1+\alpha+2s}_{1+\alpha}(\overline{\Omega} | B_{3/2})}$ and $z := d^{1-s}_{\Omega} \partial_n u \in C^{1,\alpha}(\overline{\Omega} \cap B_2)$. Hence, since $\partial\Omega\in C^{2,\alpha}$ in $B_3$, {\cite[Corollary 2.5]{RoWe24a}} yields
\begin{align*}
\begin{split}
|\mathcal{L}\big(\partial_nu(x)\partial_nu(y)K(x-y)\big)(\widetilde{w})| &= |\mathcal{L} \big(d_{\Omega}^{s-1}(x) d_{\Omega}^{s-1}(y) z(x) z(y) K(x-y)\big)(\widetilde{w})| \\
&\le C d_{\Omega}^{s-1}(x) | L (d_{\Omega}^{s-1} z \widetilde{w})(x) | + C d_{\Omega}^{s-1}(x) | L ( z d_{\Omega}^{s-1})(x)|\\&\le C (d_{\Omega}^{s-1}(x) +  d_{\Omega}^{\alpha - 1}(x)).
\end{split}
\end{align*}
The property $d_{\Omega}^{1-s} \widetilde{g} \in C^{k-1,\alpha}(\overline{\Omega} \cap B_{3/2})$ follows immediately from the definition of $\widetilde{g}$ in \eqref{eq:g-def} and the results obtained in (i). 

Let us now consider $x_0\in \partial\Omega\cap B_1$ and let us prove that $\partial_{\Xi}\widetilde{w}(x_0)=\widetilde{b}(x_0)$, where  $\widetilde{b}$ is defined as in \eqref{eq:tilde-b-def}. For that we take into account that, by assumption, 
\begin{align}\label{Ss}
\frac{v}{d^{s}_{\Omega}}(x_0) =\lim_{x\to x_0, x\in\Omega} \frac{v}{d^{s}_{\Omega}}(x) = h(x_0) \qquad \forall x_0\in\partial\Omega\cap B_2.
\end{align} 

Let us assume from now on without loss of generality that $x_0 = 0 \in\partial\Omega$. Using that $\partial\Omega\in C^{2,\alpha}$ in $B_3$, \eqref{Ss}, and \cite[Proposition 4.1]{AbRo20}, it follows
\begin{align}
\label{eq_b}
v(x)=d^{s}_{\Omega}(x) h(0)(1+b\cdot x)+\widetilde{v}(x) ~~ \forall x\in B_{1}, \quad \text{ where } \quad b:=h(0)^{-1} \nabla\left(\displaystyle\frac{v}{d^{s}_{\Omega}}\right)(0) \in \R^{n}.
\end{align}
Here, we used that $h > 0$ by the Hopf lemma (see \autoref{lemma:Hopf}).
Therefore, for every $i \in \{1, \dots, n\}$ and $x\in B_{1/2}(0)$, denoting by $b=(b^{(1)},\ldots, b^{(n)})$,
\begin{align*}
\partial_{i}v(x):=sh(0)d^{s-1}_{\Omega}(x)(1+b\cdot x)\partial_{i}d_{\Omega}(x)+d^{s}_{\Omega}(x)h(0)b^{(i)}+\partial_{i}\widetilde{v}(x).
\end{align*}
Thus, for any $x \in B_1$,
\begin{align}\label{ec_w}
\widetilde{w}=\frac{\partial_{i}v(x)}{\partial_{n}v(x)}=\frac{sh(0)(1+b\cdot x)\partial_{i}d_{\Omega}(x)+d_{\Omega}(x)h(0)b^{(i)}+\frac{\partial_{i}\widetilde{v}}{d^{s-1}_{\Omega}}(x)}{sh(0)(1+b\cdot x)\partial_{n}d_{\Omega}(x)+d_{\Omega}(x)h(0)b^{(n)}+\frac{\partial_{n}\widetilde{v}}{d^{s-1}_{\Omega}}(x)} =: \frac{\sum_{j=1}^{3}F_{i,j}(x)}{\sum_{j=1}^{3}F_{n,j}(x)}.
\end{align}
We set $F_i = \sum_{j=1}^3 F_{i,j}$. Denoting by 
\begin{equation}\label{nu}
\nu_0:=(\nu_0^{(1)},\ldots, \nu_0^{(n)})=\nabla d_{\Omega}(0)=(\partial_1 d_{\Omega}(0),\ldots,\partial_n d_{\Omega}(0)),
\end{equation}
and $\Xi:=(\Xi^{(1)},\ldots, \Xi^{(n)})$ we get
\begin{align*}
\partial_{\Xi} F_{i}(0)&= s h(0) \left((\partial_{\Xi} \partial_i d_{\Omega})(0) + \Xi^{(i)}b\cdot\nu_0\right)+ (\partial_{\Xi} d_{\Omega})(0) h(0)b^{(i)}+\partial_{\Xi}\left(\frac{\partial_{i}\widetilde{v}}{d^{s-1}_{\Omega}}\right)(0).
\end{align*}

Next, we claim that
\begin{align}
\label{eq:error-estimate}
\left|\partial_{\Xi}\left(\frac{\partial_{i}\widetilde{v}}{d^{s-1}_{\Omega}}\right) \right|\leq C d_{\Omega}^{\alpha} ~~ \text{ in } \Omega\cap B_{1}.
\end{align}
To see it, let us observe that on the one hand, we have 
\begin{align}
\label{eq:error-estimate-help-1}
\frac{\nabla \widetilde{v}}{d^{s-1}_{\Omega}} \in C^{1,\alpha}(\overline{\Omega} \cap B_1),
\end{align}
since by \autoref{lemma:boundary-reg} and $\partial \Omega \in C^{2,\alpha}$ in $B_3$ it holds $\nabla v / d^{s-1}_{\Omega} \in C^{1,\alpha}(\overline{\Omega} \cap B_1)$, and moreover, 
\begin{align*}
\frac{\nabla (d^s_{\Omega} (1 + b \cdot x))}{d^{s-1}_{\Omega}} = s (\nabla d_{\Omega})(1+ b \cdot x) + d_{\Omega} \nabla(1 + b \cdot x) \in C^{1,\alpha}(\overline{\Omega} \cap B_1).
\end{align*}
On the other hand, we have
\begin{align}
\label{eq:error-estimate-help-2}
\left| \frac{\nabla \widetilde{v}}{d^{s-1}_{\Omega}} \right| \le C d_{\Omega}^{1+\alpha} ~~ \text{ in } \Omega \cap B_1,
\end{align}
since \eqref{eq:expansion-Holder-interpol} (applied with $\delta = 1$ and $k = 2$) implies that for any $x_0 \in \Omega \cap B_1$ it holds for $r = d_{\Omega}(x_0)$:
\begin{align*}
\Vert \nabla \widetilde{v} \Vert_{L^{\infty}(B_{r/2}(x_0))} \le [ v - d^s_{\Omega} h(0) (1 + b \cdot x) ]_{C^{0,1}(B_{r/2}(x_0))} \le C r^{\alpha+s}.
\end{align*}
From here, \eqref{eq:error-estimate} follows immediately by \eqref{eq:error-estimate-help-1}, \eqref{eq:error-estimate-help-2}, and H\"older interpolation.

Thus, using \eqref{eq:error-estimate-help-2}, we deduce that $F_{i,m}(0) = 0$ for $m \in\{ 2,3\}$, whereas $F_{i,1}(0) = s h(0) \nu_{0}^{(i)}$ for any $i \in \{1, \dots, n\}$. Using also \eqref{eq:error-estimate}, this implies
\begin{align*}
 F_{j}(0)\partial_{\Xi}F_{i}(0)=s h(0) \nu_{0}^{(j)}\left[s h(0) \left((\partial_{\Xi} \partial_i d_{\Omega})(0)+  ( b \cdot \Xi ) \nu_0^{(i)}\right) + (\partial_{\Xi}d_{\Omega})(0) h(0)b^{(i)}\right],
\end{align*}
and therefore, using that \eqref{nu} implies $(\partial_{\Xi}d_{\Omega})(0)=\Xi\cdot\nu_0$, we get,
\begin{align*}
F_{n}(0)\partial_{\Xi} F_{i}(0)-F_{i}(0)\partial_{\Xi}F_{n}(0)&= s^2 h^2(0) \big( \Xi \cdot \partial_i (\nabla d_{\Omega})(0) \nu_0^{(n)} - \Xi \cdot \partial_n (\nabla d_{\Omega})(0) \nu_0^{(i)} \big) \\
&\quad + s h^2(0) (\Xi \cdot \nu_0) [b^{(i)}\nu_{0}^{(n)}-b^{(n)}\nu_{0}^{(i)}].
\end{align*}
We notice now that, since $b^{(i)}\nu_{0}^{(n)}-b^{(n)}\nu_{0}^{(i)} = b\cdot\tau$ with $\tau:=\nu_{0}^{(n)} e^{(i)}-\nu_{0}^{(i)}e^{(n)} \perp \nu_0$ being tangential to $\partial \Omega$ at $0 \in \partial \Omega$, by \eqref{eq_b} and \eqref{Ss}, it must be $b \cdot \tau = h(0)^{-1} \nabla h(0) \cdot \tau$. Therefore,
\begin{align*}
F_{n}(0)\partial_{\Xi}F_{i}(0)-F_{i}(0)\partial_{\Xi}F_{n}(0) &= s^2h^2(0) \left(\Xi \cdot \partial_{\tau} (\nabla d_{\Omega})(0)\right) + s h^2(0) (\Xi \cdot \nu_0) (b \cdot \tau) \\
&= s^2h^2(0) \Xi \cdot (\partial_{\tau} \nu_0) + sh(0) (\Xi \cdot \nu_0)\partial_{\tau} h(0).
\end{align*}
Then, we conclude from \eqref{ec_w} 
\begin{align*}
\partial_{\Xi}\widetilde{w}(0)=&\frac{F_{n,1}(0)\partial_{\Xi}F_{i}(0)-F_{i,1}(0)\partial_{\Xi}F_n(0)}{F^2_{n}(0)}=\frac{ s^2h^2(0) \Xi \cdot (\partial_{\tau} \nu_0) + sh(0)(\Xi \cdot \nu_0)\partial_{\tau} h(0)}{(sh(0)\nu_0^{(n)})^2} ,
\end{align*}
which is well defined since $\partial_{n}d_{\Omega}(0) \ge \delta$ and $h(0) > 0$, for every $\Xi\in\mathbb{S}^{n-1}$. Therefore, for a general point $x\in\partial\Omega\cap B_1$ we obtain \eqref{eq:tilde-b-def}, as desired.

Therefore, if $\Xi(x) = \nu(x)$, then
\begin{align*}
(s h(x) \nu^{(n)}(x))^2 \widetilde{b}(x) &= s^2 h^2(x) \nu(x) \cdot \partial_{\tau(x)} \nu(x) + sh(x)\nu(x) \cdot \nu(x)\partial_{\tau(x)} h(x) = sh(x)\,{\partial_{\tau(x)} h(x)},
\end{align*}
where we used that $2 \nu(x)\cdot\partial_{\tau(x)} \nu(x) = \partial_{\tau(x)} |\nu(x)|^2 = 0$ and $|\nu(x)|^2 = 1$. Hence, since $d_{\Omega} \in C^{k,\alpha}(\overline{\Omega})$ and $h \in C^{k,\alpha}(\partial \Omega \cap B_1)$, we have $\widetilde{b} \in C^{k-1,\alpha}(\partial \Omega \cap B_1)$, as desired.
\end{proof}

\subsubsection{Application to overdetermined problems}

The goal of this subsection is to show \autoref{thm:main-onephase-intro2} by using \autoref{lemma:PDE-onephase}. Another main result of this subsection is \autoref{prop:boundary-improvement2}, where we show that we can improve the regularity of a $C^{k,\alpha}$ domain to $C^{k+\alpha+1-\eps}$, whenever $h \in C^{k,\alpha}$.

First, we derive the PDE after applying the change of variables $\Phi$ from Subsection \ref{subsec:flattening}. Moreover, in (iv) we derive an equation for incremental quotients of $w$. This equation will be used to gain more regularity of $\partial \Omega$ in the proof of \autoref{thm:main-onephase-intro2} through \autoref{lemma:bdry-reg}.

\begin{corollary}
\label{lemma:PDE-onephase-flat}
Assume that we are in the situation of \autoref{lemma:PDE-onephase}. Let $\Phi \in C^{k,\alpha}$ be as in Subsection~\ref{subsec:flattening} and $w:= \widetilde{w} \circ \Phi$. Then, it holds
\begin{itemize}
\item[(i)] $w \in C_{k-1+\alpha}^{k+\alpha+2s}(\{ x_n > 0 \} | B_{3/2})$ with $\supp(w) \subset \{ x_n \ge 0 \} \cap \overline{B_{3/2}}$.

\item[(ii)] The following equation holds in the strong sense.
\begin{align*}
\begin{cases}
\mathcal{L}\big((x_n)_+^{s-1}(y_n)_+^{s-1}J(x)J(y)K(\Phi(x)-\Phi(y))\big)(w) &= g ~~ \text{ in } \{ x_n > 0 \} \cap B_1,\\
\partial_{n} w &= b ~~ \text{ on } \{ x_n = 0 \} \cap B_1,
\end{cases}
\end{align*}
where $g$ satisfies $(x_n)_+^{1-s} g \in C^{k-1,\alpha}(\{ x_n \ge 0 \}\cap B_{3/2})$ and is given by $g:= \widetilde{g} \circ \Phi$, and $J \in C^{k-1,\alpha}_c(\{ x_n \ge 0 \} \cap B_{2})$ satisfies $J \ge c_0$ in $B_{3/2}$ for some $c_0 > 0$, and is given by $J(x) = (d_{\Omega}^{1-s} \partial_n u \circ \Phi)(x)|\det D \Phi(x)|$, and $b:=\widetilde{b} \circ \Phi \in C^{k-1,\alpha}(\{x_n  = 0\} \cap B_1)$.

\item[(iii)] Given $h = (h',0) \in \R^n$ with $|h| < \frac{1}{k}$, $\bar{k} \in \{ 1, \dots, k \}$ and $\bar{\alpha} \in (0,1)$ such that $\bar{k}+\bar{\alpha} \le k+\alpha$, we define $W^{(h)} = D_{h}^{\bar{k}-1+\bar{\alpha}} w$. The following equation holds in the strong sense
\begin{align*}
\begin{cases}
& \mathcal{L} \big((x_n)_+^{s-1}(y_n)_+^{s-1} J(x)J(y) K(\Phi(x)-\Phi(y))\big)(W^{(h)}) \\
&\qquad\qquad \qquad = G^{(h)} + \sum_{j = 0}^{\bar{k}-1} \mathcal{L}\big((x_n)_+^{s-1}(y_n)_+^{s-1} I^{(h)}_{j}(x,y)\big)(F^{(h)}_{j}) ~~ \text{ in } \{ x_n > 0 \} \cap B_{1- k|h|},\\
&\quad\quad~ \partial_{n} W^{(h)} = B^{(h)} \qquad\qquad\qquad\qquad\qquad\qquad\qquad\qquad\qquad\qquad ~~~ \text{ on } \{ x_n = 0 \} \cap B_{1- k|h|},
\end{cases}
\end{align*}
where for any $j \in \{ 0 , \dots, \bar{k}-1 \}$, and $x,y,z \in \R^n$:
\begin{align*}
\Vert(x_n)_+^{1-s} G^{(h)} \Vert_{L^{\infty}(\{ x_n > 0 \} \cap B_{\frac{3}{2} - k|h|})} + \Vert F^{(h)}_{j} \Vert_{C^{k - \bar{k},\alpha}(\{ x_n \ge 0 \})} + \Vert B^{(h)} \Vert_{L^{\infty} (\{ x_n = 0 \} \cap B_{1 - k|h|})} & \le C \bar{C}, \\
|x-y|^{n+2s}|I^{(h)}_{j}(x,y)| &\le C \bar{C},\\
\left( \frac{\min\{|x-z|,|y-z|\}}{|x-y|} \right)^{\min\{ 1 , k - \bar{k} + \alpha - \bar{\alpha} \} }\min\{|x-z|,|y-z|\}^{n+2s}|I^{(h)}_{j}(x,z) - I^{(h)}_{j}(y,z)| &\le C \bar{C},
\end{align*}
where $C$ only depends on $n,s,k,\alpha,\lambda,\Lambda, c_0, \delta$, but not on $h$. We denoted
\begin{align*}
\bar{C} &:= \Vert (x_n)_+^{1-s} g \Vert_{C^{k-1,\alpha}(\{ x_n \ge 0 \}\cap B_{3/2})} + \Vert w \Vert_{C^{k+\alpha+2s}_{k-1,\alpha}(\{ x_n > 0 \} | B_{3/2})} + \Vert J \Vert_{C^{k-1,\alpha}(\{ x_n \ge 0 \})} \\
&\quad + \Vert \Phi \Vert_{C^{k,\alpha}(\R^n)} + \Vert K \Vert_{C^{2(k+\alpha)+1}(\mathbb{S}^{n-1})} + \Vert b \Vert_{C^{k-1,\alpha}(\{ x_n = 0 \} \cap B_{1 - k|h|})}.
\end{align*}

\item[(iv)] Moreover, $W^{(h)} \in C^{\gamma}(\R^n)$ for any $\gamma \in (0, 1)$, satisfies the equation in (iii) in the weak sense, i.e. (see  \autoref{def:weak-sol-trafo}) for every $\eta \in C^{\infty}_c(B_{1-k|h|})$, it holds
\begin{align*}
\int_{\R^n} &\int_{\R^n} (W^{(h)}(x) - W^{(h)}(y)) (\eta(x) - \eta(y)) (x_n)_+^{s-1}(y_n)_+^{s-1} J(x) J(y) K(\Phi(x) - \Phi(y)) \d y \d x \\
&=\int_{\R^n} \eta(x) G^{(h)}(x) \d x + \int_{\{ x_n = 0 \}\cap B_{1-k|h|}}   B^{(h)}(x')(\theta_{K,J,\Phi})_n (x') \eta(x',0) \d x' \\
&\quad - \int_{\{ x_n = 0 \}\cap B_{1-k|h|}}   W^{(h)}(x',0) [(\theta_{K,J,\Phi}' (x')\cdot \nabla_{x'} \eta(x',0)) +  \dvg \theta_{K,J,\Phi}'(x') \eta(x',0)] \d x'\\
&\quad + \sum_{j = 0}^{\bar{k}-1} \int_{\R^n} \int_{\R^n} (F^{(h)}_{j}(x) - F^{(h)}_{j}(y)) (\eta(x) - \eta(y)) (x_n)_+^{s-1}(y_n)_+^{s-1} I^{(h)}_{j}(x,y) \d y \d x\\
&\quad + \sum_{j = 1}^{\bar{k}-1} \int_{\{ x_n = 0 \}} \partial_n F^{(h)}_{j}(x',0) \eta(x',0) (\theta_{I_{j}^{(h)}})_n(x') \d x'\\
&\quad -\sum_{j = 1}^{\bar{k}-1}\int_{\{ x_n = 0 \}}  F^{(h)}_{j}(x',0) [(\theta_{I^{(h)}_{j}}(x')\cdot \nabla_{x'} \eta(x',0)) +  {\dvg \theta_{I^{(h)}_{j}}(x') \eta(x',0)]} \d x' \\
&\quad + \int_{\{ x_n = 0 \}}\eta(x',0) (\theta_{I^{(h)}_{0}}(x') \cdot \nabla F^{(h)}_{0}(x',0)) \d x',
\end{align*}
where $\theta_{K,J,\Phi}$ is as in \autoref{lemma:ibp-nonflat-trafo} and for any $j \in \{ 0 , \dots , \bar{k} -1\}$, $\theta_{I^{(h)}_{j}} \in L^{\infty}(\{ x_n = 0 \} \cap B_{1-k|h|})$,
\begin{align*}
\Vert \theta_{K,J,\Phi} \Vert_{C^{k-1,\alpha}(\{ x_n = 0 \} \cap B_1)} &\le C \bar{C},\\
\Vert \theta_{I^{(h)}_{0}} \Vert_{L^{\infty}(\{ x_n = 0 \} \cap B_{1-k|h|})} + \Vert \theta_{I^{(h)}_{j}} \Vert_{C^{0,1}(\{ x_n = 0 \} \cap B_{1-k|h|})} &\le C \bar{C}, ~~ \text{ if } j \not= 0,\\
\Vert \nabla F_{0}^{(h)} \Vert_{L^{\infty}(\{ x_n = 0 \} \cap B_{1-k|h|})} + \Vert \partial_n F_{j}^{(h)} \Vert_{L^{\infty}(\{ x_n = 0 \} \cap B_{1-k|h|})} &\le C \bar{C}.
\end{align*}
\end{itemize}
\end{corollary}

Note that some of the regularity estimates in the previous result are not sharp. We decided to state the assumptions in such a way that it is immediately visible that \autoref{lemma:bdry-reg} can be applied to $W^{(h)}$. The sharp estimates are established within the proof. 

Moreover, note that we will typically apply the previous result with $\bar{k} = k$ and $\bar{\alpha} = \alpha$. However, if $\alpha < \max\{0,2s-1,\frac{1}{2}\}$, we cannot apply \autoref{lemma:bdry-reg} to improve the regularity of the free boundary. In that case, we have to apply \autoref{lemma:PDE-onephase-flat} with $\bar{k} = k-1$ and some $\bar{\alpha} \in ( \max \{ 0 , 2s-1,\frac{1}{2} \} , 1)$.

\begin{proof}
The proofs of (i) and (ii) follow immediately from \autoref{lemma:PDE-onephase} after application of the change of variables $\Phi$ from Subsection \ref{subsec:flattening} and choosing $\Xi(x) = \nu(x)$. 
To show (iii), note that the boundary condition follows immediately from $\partial_n w = b$ and the fact that $h \perp e_n$, after defining $B^{(h)} := D^{\bar{k}-1+\bar{\alpha}}_h b = D^{\bar{k}-1+\bar{\alpha}} \partial_n w = \partial_n W^{(h)}$. Then, we have $B^{(h)} \in C^{k -\bar{k}-1+\alpha-\bar{\alpha}}(\{ x_n = 0 \} \cap B_{1 - k|h|})$ with
\begin{align}
\label{eq:Bh-est}
\Vert B^{(h)} \Vert_{C^{k -\bar{k}-1+\alpha-\bar{\alpha}} (\{ x_n = 0 \} \cap B_{1 - k|h|})} \le \Vert b \Vert_{C^{k-1,\alpha}(\{ x_n = 0 \} \cap B_{1 - k|h|})},
\end{align}
where the estimate follows immediately from the definition of $b$ as $b = \widetilde{b} \circ \Phi$, where $\widetilde{b} \in C^{k-1,\alpha}(\partial \Omega \cap B_1)$ is given in \eqref{eq:tilde-b-def}.
To derive the equation, let us introduce the following notation
\begin{align*}
\Delta_h[\mathcal{L}\big(I(x,y)\big)(w)](x) = \mathcal{L}\big(I(x+h,y+h)\big)(w(\cdot+h))(x) - \mathcal{L}\big(I(x,y)\big)(w)(x).
\end{align*}
We observe that by the product rule for the operator $\Delta^{\bar{k}}$, it holds for any $j \in \N$:
\begin{align}
\label{eq:product-rule-Delta}
\Delta_h^j [\mathcal{L}\big(I(x,y)\big)(w)](x) = \sum_{l = 0}^j \binom{j}{l} \mathcal{L}\big(\Delta^l I(x,y)\big) (\Delta^{j-l} w(\cdot + l h))(x).
\end{align}
Moreover, we set $M(x,y) := J(x)J(y) K(\Phi(x)-\Phi(y))$. Then, we observe that by the product rule for the operator $\Delta^{\bar{k}}$ and the definition of $\mathcal{L}$,
\begin{align*}
\mathcal{L} & \big((x_n)_+^{s-1}(y_n)_+^{s-1} J(x)J(y) K(\Phi(x)-\Phi(y))\big)(\Delta^{\bar{k}} w)(x) \\
& = \sum_{m = 0}^{\bar{k}} \binom{\bar{k}}{m} \Delta^{m} \big[ \mathcal{L}\big((x_n)_+^{s-1}(y_n)_+^{s-1} \Delta^{\bar{k}-m} M(x,y)\big)(w(\cdot + (\bar{k}-m)h)) \big](x).
\end{align*}
By combination with \eqref{eq:product-rule-Delta}, we obtain
\begin{align*}
\mathcal{L} & \big((x_n)_+^{s-1}(y_n)_+^{s-1} J(x)J(y) K(\Phi(x)-\Phi(y))\big)(\Delta^{\bar{k}} w)(x) \\
& = \Delta^{\bar{k}} g(x) + \sum_{m = 0}^{\bar{k}-1} \sum_{l = 0}^m \binom{\bar{k}}{m} \binom{m}{l} \mathcal{L}\big((x_n)_+^{s-1}(y_n)_+^{s-1} \Delta^{\bar{k}-m+l} M(x,y)\big)(\Delta^{m-l} w(\cdot + (\bar{k}-m + l)h))(x).
\end{align*}
Hence, we obtain
\begin{align*}
\mathcal{L} & \big((x_n)_+^{s-1}(y_n)_+^{s-1} J(x)J(y) K(\Phi(x)-\Phi(y))\big)(W^{(h)})(x)  \\
&= G^{(h)}(x) + \sum_{j = 0}^{\bar{k}-1} \mathcal{L}\big((x_n)_+^{s-1}(y_n)_+^{s-1} I^{(h)}_{j}(x,y)\big)(F^{(h)}_{j})(x),
\end{align*}
where we denoted
\begin{align*}
G^{(h)}(x) &= D^{\bar{k}-1+\bar{\alpha}}_h g(x), \quad
I^{(h)}_{j}(x,y) = \Big[ \sum_{m = j}^{\bar{k}-1} \binom{\bar{k}}{m} \binom{m}{m-j} \Big] D_h^{\bar{k}-j-1+\bar{\alpha}} M(x,y), 
\end{align*}
\begin{align*}
F^{(h)}_{j}(x) = D_h^{j} w(x + (\bar{k}-j)h)).
\end{align*}
{Note that $(x_n)_+^{1-s}G^{(h)} \in C^{k-\bar{k}+\alpha-\bar{\alpha}}(\{ x_n > 0 \} \cap B_{\frac{3}{2}-k|h|})$ follows immediately from the assumption on $g$.} Moreover, since $j \le \bar{k}-1$, we have that ${F^{(h)}_{j} \in C^{0,1}_{k-\bar{k},\alpha}(\{ x_n > 0 \} | B_{3/2 - k|h|})}$, together with a corresponding bound on the norm, which follows from the assumption $w \in C^{k+\alpha+2s}_{k-1,\alpha}( \{ x_n > 0 \} | B_{3/2})$. Since, by the compact support of $w$, it holds in particular $w \in C^{k-1,\alpha}(\{ x_n \ge 0 \})$, we deduce that $F_{j}^{(h)} \in C^{\bar{k}-k,\alpha}(\{ x_n \ge  0 \})$, also with a uniform bound on its norm.\\
To see the bound on $I_{j}^{(h)}$, we recall the definition of $M(x,y) = J(x)J(y)K(\Phi(x)-\Phi(y))$ and obtain from the product rule:
\begin{align*}
\Delta_h^{\bar{k}-j} M(x,y) = \sum_{i = 0}^{\bar{k}-j} \binom{\bar{k}-j}{i} \Delta^{\bar{k}-j-i}[J(x+ih)J(y+ih)] \Delta^{i} [K(\Phi(x) - \Phi(y))],
\end{align*}
and therefore
\begin{align*}
D_h^{\bar{k}-j-1+\bar{\alpha}} M(x,y) &= D_h^{\bar{k}-j-1+\bar{\alpha}}[J(x)J(y)] K(\Phi(x)-\Phi(y)) \\
&\quad + \sum_{i = 1}^{\bar{k}-j} \binom{\bar{k}-j}{i} D_h^{\bar{k}-j-i}[J(x+ih)J(y+ih)] D_h^{i-1+\bar{\alpha}} [K(\Phi(x) - \Phi(y))].
\end{align*}
Since $J \in C^{k-1+\alpha}(\{ x_n \ge 0 \} \cap B_2)$ and $\bar{k}-j \le \bar{k}$, proving the bound for $|I_{j}^{(h)}(x,y)|$ reduces to establishing for any $i \in \{ 1 , \dots , \bar{k} \}$,
\begin{align}\label{vac}
|\varphi(x,y)|:=|D_h^{i-1+\bar{\alpha}} [K(\Phi(x) - \Phi(y))]| &\le C |x-y|^{-n-2s},\\
\frac{|\varphi(x,z)-\varphi(y,z)|}{|x-y|} &\le C  \min\{ |z-x| , |z-y| \}^{-n-2s - 1}, 
\end{align}
which follows from \autoref{lemma:K-estimate}. Indeed, if $\min\{ |z-x| , |z-y| \} \le 10$,
\begin{align*}
&|D_h^{\bar{k}-j-1+\bar{\alpha}} M(x,z)-D_h^{\bar{k}-j-1+\bar{\alpha}} M(y,z)| \\
&\leq \left|D_h^{\bar{k}-j-1+\bar{\alpha}}[J(x)J(z)] -D_h^{\bar{k}-j-1+\bar{\alpha}}[J(y)J(z)] \right||K(\Phi(x) - \Phi(z))|\\
&+\left|D_h^{\bar{k}-j-1+\bar{\alpha}}[J(x)J(z)]\right|\left|K(\Phi(x) - \Phi(z))-K(\Phi(y) - \Phi(z))\right|\\
&+\sum_{i = 1}^{\bar{k}-j} \binom{\bar{k}-j}{i}\left|D_h^{\bar{k}-j-i}[J(x+ih)J(z+ih)] -D_h^{\bar{k}-j-i}[J(y+ih)J(z+ih)]\right||\varphi(x,z)|\\
&+\sum_{i = 1}^{\bar{k}-j} \binom{\bar{k}-j}{i}\left|D_h^{\bar{k}-j-i}[J(y+ih)J(z+ih)]\right||\varphi(x,z)-\varphi(y,z)|.
\end{align*}
By the hypothesis on $K$ and $\Phi$ (see \eqref{eq:K-reg}, \eqref{eq:Phi-comp} and \autoref{lemma:K-estimate}) we get
$$\max\{|K(\Phi(x) - \Phi(z))|,|\varphi(x,z)|\}\leq |x-z|^{-n-2s}\leq\min\{ |z-x| , |z-y| \}^{-n-2s},$$
$$\left|K(\Phi(x) - \Phi(z))-K(\Phi(y) - \Phi(z))\right|\leq |x-y|\min\{ |z-x| , |z-y| \}^{-n-2s-1}.$$
Moreover since $J \in C^{k-1+\alpha}(\{ x_n \ge 0 \} \cap B_2)$ 
\begin{align*}
D_h^{\bar{k}-j-1+\bar{\alpha}}[J(x)J(z)] &\in
 C^{k+\alpha-\bar{k}-\bar{\alpha}+j}(\{ x_n = 0 \} \cap B_{1- k|h|}), \\
D_h^{\bar{k}-j-i}[J(x+ih)J(z+ih)] &\in
 C^{k+\alpha-\bar{k}+j+i}(\{ x_n = 0 \} \cap B_{1- k|h|}).
\end{align*}
Thus by \eqref{vac},
\begin{align*}
&|D_h^{\bar{k}-j-1+\bar{\alpha}} M(x,z)-D_h^{\bar{k}-j-1+\bar{\alpha}} M(y,z)|\\
&\leq C\Big[|x-y|^{\min\{1,k+\alpha-\bar{k}-\bar{\alpha} \}}\min\{ |z-x| , |z-y| \}^{-n-2s}+|x-y|\min\{ |z-x| , |z-y| \}^{-n-2s-1}\\
&+\sum_{i = 1}^{\bar{k}-j} \binom{\bar{k}-j}{i}|x-y| \Big( \min\{ |z-x| , |z-y| \}^{-n-2s} + \min\{ |z-x| , |z-y| \}^{-n-2s - 1} \Big) \Big]\\
&\leq C\min\{ |z-x| , |z-y| \}^{-n-2s}\Big[|x-y|^{\min\{1,k+\alpha-\bar{k}-\bar{\alpha} \}}+\frac{|x-y|}{\min\{ |z-x| , |z-y| \}}\Big]\\
&\leq C\min\{ |z-x| , |z-y| \}^{-n-2s}\left(\frac{|x-y|}{\min\{ |z-x| , |z-y| \}}\right)^{\min\{1,k+\alpha-\bar{k}-\bar{\alpha} \}},
\end{align*}
which proves (iii), if $\min\{ |z-x| , |z-y| \} \le 10$. Otherwise, it must be either $x,y \not\in \supp(J)$ or $z \not\in \supp(J)$, and therefore, the estimate is trivial.

To prove (iv), since $W^{(h)} \in C^{1+\alpha+2s}_{1+\alpha}(\{ x_n > 0 \} | B_2) \cap L^{\infty}(\R^n)$ we can apply the integration by parts formula from \autoref{lemma:ibp-nonflat-trafo}, which yields for any $\eta \in C^{\infty}_c(B_{1-k|h|})$
\begin{align}
\label{Bbb}
\begin{split}
\int_{\R^n} &\int_{\R^n} (W^{(h)}(x) - W^{(h)}(y)) (\eta(x) - \eta(y)) (x_n)_+^{s-1}(y_n)_+^{s-1} J(x) J(y) K(\Phi(x) - \Phi(y)) \d y \d x \\
&= \int_{\R^n} \eta(x) \mathcal{L}\big((x_n)_+^{s-1}(y_n)_+^{s-1} J(x) J(y) K(\Phi(x) - \Phi(y))\big)(W^{(h)})(x) \d x \\
&\quad - \int_{\{ x_n = 0 \}} \eta(x',0) (\theta_{K,J,\Phi}(x') \cdot \nabla W^{(h)}(x',0)) \d x',
\end{split}
\end{align}
where
\begin{align*}
\theta_{K,J,\Phi}(x') = 2c_s J^2(x',0) \int_{\R^{n-1}} (h',1) K(D\Phi(x',0)(h',1)) \d h'.
\end{align*}
Note that $\theta_{K,J,\Phi} \in C^{k-1,\alpha}(\{ x_n = 0 \} \cap B_1)$ by the assumptions on $K,J,\Phi$.

For the boundary term in \eqref{Bbb}, we notice that by \eqref{eq:Bh-est} we have $\theta_{K,J,\Phi} \cdot \nabla W^{(h)} = \theta_{K,J,\Phi}' \cdot \nabla_{x'} W^{(h)} + (\theta_{K,J,\Phi})_n B^{(h)}$, and hence, we get by the classical integration by parts formula for first derivatives
\begin{align*}
- & \int_{\{ x_n = 0 \}} \eta(x',0) (\theta_{K,J,\Phi}(x') \cdot \nabla W^{(h)}(x',0)) \d x' \\
&= \int_{\{ x_n = 0 \}}  W^{(h)}(x',0) [(\theta_{K,J,\Phi}' (x')\cdot \nabla_{x'} \eta(x',0)) +  \dvg \theta_{K,J,\Phi}'(x') \eta(x',0)] \d x' \\
&\quad - \int_{\{ x_n = 0 \}} \eta(x',0) (\theta_{K,J,\Phi})_n(x') B^{(h)}(x',0) \d x'.
\end{align*}
By (iii) and \autoref{lemma:ibp-nonflat-trafo}, we have for the first term on the right-hand side of \eqref{Bbb},
\begin{align}
\label{eq:theta-I-ibp}
\begin{split}
 & \int_{\R^n} \eta(x) \mathcal{L}\big((x_n)_+^{s-1}(y_n)_+^{s-1} J(x) J(y) K(\Phi(x) - \Phi(y))\big)(W^{(h)})(x) \d x \\
&=\int_{\R^n} \eta(x) G^{(h)}(x) \d x + \sum_{j = 0}^{\bar{k}-1} \int_{\{ x_n = 0 \}} \eta(x',0) (\theta_{I^{(h)}_{j}}(x') \cdot \nabla F^{(h)}_{j}(x',0)) \d x'\\
&\quad + \sum_{j = 0}^{\bar{k}-1} \int_{\R^n}\int_{\R^n} (F^{(h)}_{j}(x) - F^{(h)}_{j}(y))(\eta(x) - \eta(y)) (x_n)_+^{s-1}(y_n)_+^{s-1} I^{(h)}_{j}(x,y) \d y \d x,
\end{split}
\end{align}
where
\begin{align}
\label{eq:theta-I-formula}
\theta_{I^{(h)}_{j}} = D_h^{\bar{k}-j-1+\bar{\alpha}} \theta_{K,J,\Phi}, \quad \theta_{I^{(h)}_{j}} \in  \begin{cases}
 C^{k-\bar{k}+\alpha - \bar{\alpha}+j-1,1}(\{ x_n = 0 \} \cap B_{1- k|h|}),&  \text{ if } j \ge 1,\\
 L^{\infty}(\{ x_n = 0 \} \cap B_{1- k|h|}),& \text{ if } j=0.
\end{cases}
\end{align}
Indeed, since the kernels $I_{j}^{(h)}$ are (up to a constant) given as incremental quotients of the form $D_h^{\bar{k}-j-1+\bar{\alpha}} [J(x)J(y)K(\Phi(x)-\Phi(y))]$ and $h_n = 0$, we can apply \autoref{lemma:ibp-nonflat-trafo} separately to each summand in the incremental quotient and sum up the respective contributions to obtain the identity in \eqref{eq:theta-I-ibp} with $\theta_{I_{j}^{(h)}}$ given by \eqref{eq:theta-I-formula}. The regularity of $\theta_{I_{j}^{(h)}}$ claimed in \eqref{eq:theta-I-formula} follows since $\theta_{K,J,\Phi} \in C^{k-1,\alpha}(\{ x_n = 0 \} \cap B_1)$.

Hence, it remains to treat the boundary integrals involving $F_{j}^{(h)}$ in order to conclude the proof of (iv).
Note that since $h_n = 0$ and $b \in C^{k-1,\alpha}(\{ x_n = 0 \} \cap B_1)$, it holds
\begin{align*}
\partial_n F^{(h)}_{j} = \partial_n [D_h^{j} w](\cdot + (\bar{k}-j)h) = D_h^{j} b(\cdot + (\bar{k}-j)h) \in C^{\bar{k}-1-j,\alpha}(\{ x_n = 0 \} \cap B_{1-k|h|}) 
\end{align*}
for every $j \in \{ 0,\ldots, k-1\}$. 
We have enough regularity of $\theta_{I_{j}^{(h)}}$ in case $j \ge 1$ for all summands by a similar argument as for the boundary term in \eqref{Bbb}, integrating by parts in $x'$. In case $j=0$, $\theta_{I_{0}^{(h)}}$ does not possess sufficient regularity. However, since $F_{0}^{(h)} = w(\cdot + \bar{k} h) \in C^{k-1+\alpha}(\{ x_n = 0 \} \cap B_{1- k|h|})$ and $k \ge 2$, we can interpret this boundary integral in the classical sense, i.e. without integrating by parts in $x'$. This yields
\begin{align*}
&\sum_{j = 0}^{\bar{k}-1}\int_{\{ x_n = 0 \}} \eta(x',0) (\theta_{I^{(h)}_{j}}(x') \cdot \nabla F^{(h)}_{j}(x',0)) \d x'\\
&=-\sum_{j = 1}^{\bar{k}-1} \int_{\{ x_n = 0 \}}  F^{(h)}_{j}(x',0) [(\theta_{I^{(h)}_{j}}(x')\cdot \nabla_{x'} \eta(x',0)) +  {\dvg \theta_{I^{(h)}_{j}}(x') \eta(x',0)]} \d x' \\
&\quad + \sum_{j = 1}^{\bar{k}-1} \int_{\{ x_n = 0 \}} \eta(x',0) (\theta_{I_{j}^{(h)}})_n(x') \partial_n F^{(h)}_{j}(x',0)\d x'\\
&\quad + \int_{\{ x_n = 0 \}}\eta(x',0) (\theta_{I^{(h)}_{0}}(x') \cdot \nabla F^{(h)}_{0}(x',0)) \d x'.
\end{align*}
This concludes the proof of (iv). 
\end{proof}

By combining \autoref{lemma:PDE-onephase}, \autoref{lemma:PDE-onephase-flat}, and \autoref{lemma:bdry-reg}, we can show the following improvement of the regularity of the solution domain $\partial \Omega$.

\begin{proposition}
\label{prop:bootstrap-onephase} 
Let $K$ be as in \eqref{eq:K-comp} and assume that it satisfies \eqref{eq:K-reg-ass} 
for some $k \in \N$ with $k \ge 2$ and $\alpha \in (0,1)$ with $\alpha \not\in \{s,1-s\}$.\\
Let $\partial \Omega \in C^{k,\alpha}$ in $B_4$ and $v \in L^{1}_{2s}(\R^n)$ be a nonnegative weak solution to
\begin{align*}
\begin{cases}
L v &= 0 ~~ \text{ in } \Omega \cap B_4,\\
v &=0 ~~ \text{ in } B_4 \setminus \Omega,\\
\frac{v}{d^s_{\Omega}} &= h ~~ \text{ on } B_4 \cap \partial \Omega
\end{cases}
\end{align*}
for some $h \in C^{k,\alpha}(B_4 \cap \partial \Omega)$. Assume that $0 \in \partial \Omega$. Moreover, assume \eqref{eq:blue-ass} and let $\widetilde{w}_i$ for any $i \in \{1,\dots,n-1\}$ be defined as in \eqref{eq:w-def}.\\
Then, for any $\gamma \in (\max\{\frac{1}{2}, 2s-1\} , 1)$ with $\alpha + \gamma \not= 1$ and $i \in \{1,\dots,n-1\}$ it holds 
\begin{align*}
\begin{cases}
\widetilde{w}_i \in C^{k-1+\alpha+ \min\{\gamma , 2s \} }_{loc}(\partial{\Omega} \cap B_1) \qquad\qquad\qquad ~~ \text{ if } \alpha \ge \gamma,\\
\widetilde{w}_i \in C^{k-1+\min\{\gamma,\alpha+2s\} - \eps}_{loc}(\partial{\Omega} \cap B_1) ~~ \forall \eps \in (0,1) ~~ \text{ if } \alpha < \gamma.
\end{cases}
\end{align*}
In particular, it holds for any $\eps \in (0,1)$,
\begin{align*}
\begin{cases}
\partial \Omega \in C^{k+\alpha+\min\{\alpha,2s\}}_{loc}(B_1) ~~\qquad\qquad\qquad \text{ if } \alpha > \max\{\frac{1}{2} , 2s-1\},\\
\partial \Omega \in C^{k+ \min\{ 1, \alpha + 2s \} - \eps }_{loc}(B_1) ~~ \forall \eps \in (0,1) ~~ \text{ if } \alpha \le \max\{\frac{1}{2},2s-1\}.
\end{cases}
\end{align*}
\end{proposition}

Under the assumptions of \autoref{prop:bootstrap-onephase}, we only have $\widetilde{w}_i \in C^{k-1+\alpha}(\overline{\Omega} \cap B_{3/2})$ by \autoref{lemma:PDE-onephase-flat}. Hence, \autoref{prop:bootstrap-onephase} shows a gain of regularity for $\widetilde{w}_i$ of order $\min\{\gamma,2s\}$ if $\alpha \ge \gamma$ and of order $\min\{\gamma - \alpha,2s\}$ if $\alpha < \gamma$.

\begin{proof}
Let us fix $i \in \{ 1 , \dots, n-1 \}$ and let $w = \widetilde{w}_i \circ \Phi$ be as in \autoref{lemma:PDE-onephase-flat}. By \autoref{lemma:PDE-onephase-flat}, we have $w \in C^{k+\alpha+2s}_{k-1+\alpha}(\{ x_n > 0 \} | B_{3/2} )$ and $\supp(w) \subset \{ x_n \ge 0 \} \cap \overline{B_{3/2}}$. In particular, it holds $w \in C^{k-1+\alpha}(\R^n)$.
Let us now assume $\alpha \ge \gamma$ and $\gamma < 2s$. We will provide the proof of the other cases at the end.

We fix $h = (h',0) \in \R^n$ with $|h| < \frac{1}{2k}$, and define $W^{(h)} = D_{h}^{k-1+\alpha} w$. Then, recalling that $k \ge 2$ and $\gamma \in (0,1)$, we have $W^{(h)} \in C^{0,\gamma}(\R^n)$. 
By \autoref{lemma:PDE-onephase-flat} it follows that $W^{(h)}$ satisfies all the assumptions of \autoref{lemma:bdry-reg} with $\beta = s$, $\alpha := 1$, and $u := W^{(h)}$, where we set and deduce from \autoref{lemma:PDE-onephase-flat}
\begin{align*}
g &:= G^{(h)}, \qquad (x_n)_+^{1-s} G^{(h)} \in  L^{\infty}(\{ x_n \ge 0 \} \cap B_{1/2}),\\
f_{j} &:= F_{j}^{(h)} \in C^{0,\gamma}(\{ x_n \ge 0 \}),~~ j \in \{0 , \dots, k-1\}, \\
b &:= B^{(h)} \in L^{\infty}(\{ x_n = 0 \} \cap B_{1/2}),\\
\theta_K &:= \theta_{K,J,\Phi} \in C^{0,1}(\{ x_n = 0 \} \cap B_{1/2}),\\
A_{j} &:= \partial_n F_{j}^{(h)} \in L^{\infty}(\{ x_n = 0 \} \cap B_{1/2}), ~~ j \in \{1 , \dots, k-1\},\\
A_{0} &:= \nabla F_{0}^{(h)} \in L^{\infty}(\{ x_n = 0 \} \cap B_{1/2}), \\
\vartheta_{j} &:= \theta_{I_{j}^{(h)}} \in L^{\infty}(\{ x_n = 0 \} \cap B_{1/2}),~~ j \in \{0 , \dots, k-1\},\\
a_{j} &:= F_{j}^{(h)} \in C^{0,\gamma}(\{ x_n = 0 \} \cap B_{1/2}), ~~ j \in \{1 , \dots, k-1\},\\
\nu_{j} &:= \theta_{I_{j}^{(h)}} \in C^{0,1}(\{ x_n = 0 \} \cap B_{1/2} ), ~~ j \in \{1 , \dots, k-1\}, \\
I_{j} &= I_{j}^{(h)}, ~~ j \in \{0 , \dots, k-1\}.\
\end{align*}
The norms of all of these quantities are bounded from above by a constant of the form $C\bar{C}$, where $\bar{C}$ is as in \autoref{lemma:PDE-onephase-flat} and $C$ depends only on $n,s,k,\alpha,\lambda,\Lambda,c_0,\delta$. Then, by \autoref{lemma:bdry-reg} (note that by the linearity of the PDE under consideration, it is no problem to have multiple terms involving the quantities $f_{j}$, $A_{j}$, $\vartheta_{j}$, $a_{j}$, and $\nu_{j}$), we deduce
\begin{align}
\label{eq:Wh-estimate}
[W^{(h)}]_{C^{0,\gamma}(\{ x_n \ge 0 \} \cap B_{1/4})} \le C \bar{C},
\end{align}
where we used that since $\Vert W^{(h)} \Vert_{L^{\infty}(\R^n)} \le \bar{C}$, we also have
\begin{align*}
\Vert (x_n)_+^{1-s} W^{(h)} \Vert_{L^1_{2s}(\{ x_n > 0 \} \setminus B_{1/4})} \le C \bar{C}.
\end{align*}
Since $h = (h',0)$ with $|h| \le \frac{1}{2k}$ was arbitrary, and $C,\bar{C}$ do not depend on $h$, and $\alpha + \gamma \not=1$, we deduce that (see for instance \cite[p.431]{And97})
\begin{align*}
[w]_{C^{k-1+\alpha+\gamma}(\{ x_n = 0 \} \cap B_{1/4} )} \le C \bar{C}.
\end{align*}
By recalling the definition of $w$, this implies 
$\widetilde{w}_i \in C^{k-1+\alpha+\gamma}(\partial \Omega \cap B_{1/4})$ for any $i \in \{1,\dots,n-1\}$, and in particular, recalling the definition of $\widetilde{w}$ and $v$, we deduce that
\begin{align*}
\frac{\partial_i v}{\partial_n v} \in C^{k-1+\alpha+\gamma}(\partial \Omega \cap B_{1/4}) ~~ \forall i \in \{1 ,\dots,n-1\}.
\end{align*}
This is sufficient to deduce that the normal vector $\nu$ of the free boundary satisfies $\nu \in C^{k-1+\alpha+\gamma}(\partial \Omega \cap B_{1/4})$.
Indeed, rewriting the normal vector $\nu(x)$ as follows,
\begin{align*}
\nu^{(i)}(x) = \frac{\partial_i v(x)}{|\nabla v(x)|} = \frac{\partial_i v(x) / \partial_n v(x) }{\big( \sum_{j=1}^{n-1} ( \partial_j v(x) / \partial_n v(x))^2 + 1 \big)^{1/2}}, ~~ i \in \{ 1 , \dots, n\},
\end{align*}
we deduce that $\nu \in C^{k-1+\alpha+\gamma}(\partial \Omega \cap B_{1/4})$ and therefore, we conclude $\partial \Omega \in C^{k+\alpha+\gamma}$ in $B_{1/4}$, as desired.

In case $\alpha \ge \gamma > 2s$ we cannot proceed as before, since \eqref{eq:kernel-int-reg-ass-main} would not be satisfied. We circumvent this issue by defining $W^{(h)} = D_h^{{k}-1+\bar{\alpha}} w$ for some $\bar{\alpha} \in (0,1)$ such that {$0<\alpha-\bar{\alpha}<1$} and $\gamma= 2s + \alpha - \bar{\alpha}$. Then, we still have $W^{(h)} \in C^{0,\gamma}(\R^n)$ and we can apply \autoref{lemma:bdry-reg} with the same choices as above since now \eqref{eq:kernel-int-reg-ass-main} is satisfied for all $I_{j,l}$. Hence, we obtain \eqref{eq:Wh-estimate}, which implies, proceeding as before,
\begin{align*}
[w]_{C^{k-1+\bar{\alpha}+\gamma}(\{ x_n = 0 \} \cap B_{1/4})} \le C \bar{C},
\end{align*}
and therefore, by the definition of $\bar{\alpha}$, we have $w \in C^{k+\alpha-1+2s}(\{ x_n = 0 \} \cap B_{1/4})$, and therefore $\partial \Omega \in C^{k+\alpha+2s}$ in $B_{1/4}$.

Let us now discuss the case $\alpha < \gamma$. In that case, we cannot proceed as before, since we only have $F_{j}^{(h)} \in C^{\alpha}(\{ x_n \ge 0 \} \cap B_{1/2})$ with a bound uniform in $h$, but need this function to be uniformly bounded in $C^{\gamma}$. We circumvent this issue by defining $W^{(h)} = D_h^{k-2+\bar{\alpha}} w$ for some $\bar{\alpha} \in (0,1)$, i.e. choosing $\bar{k} = k - 1$ in \autoref{lemma:PDE-onephase-flat}. Then, we still have $W^{(h)} \in C^{0,\gamma}(\R^n)$ and can apply \autoref{lemma:bdry-reg} with the same choices as above, since now $F_{j}^{(h)} \in C^{1,\alpha}(\{ x_n \ge 0 \} \cap B_{1/2}) \subset C^{\gamma}(\{ x_n \ge 0 \} \cap B_{1/2})$. Moreover, note that if $\gamma > 2s$, we need to choose $\bar{\alpha} \in (0,1)$ in such a way that $\gamma \le 2s + k + \alpha - \bar{k} - \bar{\alpha} = 2s +1 + \alpha - \bar{\alpha}$, so that \eqref{eq:kernel-int-reg-ass-main} holds true for all $I_{j}$. Hence, we obtain \eqref{eq:Wh-estimate}, which implies, proceeding as before,
\begin{align*}
[w]_{C^{k-2+\bar{\alpha}+\gamma}(\{ x_n = 0 \} \cap B_{1/4})} \le C \bar{C},
\end{align*}
and therefore $w \in C^{k-1+\gamma - \eps}(\{ x_n = 0 \} \cap B_{1/4})$ if $\gamma < 2s$ for any $\eps > 0$ (since $\bar{\alpha} \in (0,1)$ was arbitrary), which yields  $\partial \Omega \in C^{k+\gamma-\eps}$ in $B_{1/4}$. If $\gamma > 2s$, we get $w \in C^{k - 1 +\min\{\gamma, \alpha + 2s\} - \eps }(\{ x_n = 0 \} \cap B_{1/4})$, which yields  $\partial \Omega \in C^{k+\min\{\gamma, \alpha + 2s\}- \eps}$ in $B_{1/4}$.
\end{proof}

Having at hand \autoref{prop:bootstrap-onephase}, we show in \autoref{prop:boundary-improvement2} that we can improve the regularity of a $C^{k,\alpha}$ domain $\Omega$ to a maximum of $\partial \Omega \in C^{k+\alpha + 1 - \eps}$, given $h \in C^{k,\alpha}$. First, we prove another refinement of \autoref{prop:bootstrap-onephase}.

\begin{proposition}
\label{prop:boundary-improvement} 
Let $K \in C^{2(k + \max\{2\alpha, 1\}) + 1}(\mathbb{S}^{n-1})$ be as in \eqref{eq:K-comp} for some $k \in \N$ with $k \ge 2$ and $\alpha \in (0,1)$ with $\alpha \not\in \{s,1-s\}$. Let $\partial \Omega \in C^{k,\alpha}_{loc}(B_4)$ and $v \in L^{1}_{2s}(\R^n)$ be as in \autoref{prop:bootstrap-onephase} for some $h \in C^{k,\alpha}(B_4 \cap \partial \Omega)$. Then, for any $\eps \in (0,1)$ it holds $\partial \Omega \in C_{loc}^{k+\alpha+\min\{1,2s\}-\eps}(B_1)$.
\end{proposition}

\begin{proof}

If $\alpha > \max\{\frac{1}{2} , 2s-1 \}$, from \autoref{prop:bootstrap-onephase} we deduce that $\partial \Omega \in C^{k+\alpha+\min\{\alpha,2s\}}_{loc}(B_1)$ so, if $\alpha\geq 2s$ we are done. Let us suppose then that $\max\{\frac{1}{2} , 2s-1 \}<\alpha<2s$. Since $K \in C^{2(k+2\alpha)+1}(\mathbb{S}^{n-1})$, we can proceed as in the proof of \autoref{lemma:PDE-onephase-flat} and deduce that $w \in C_{k-1+2\alpha}^{k+2\alpha+2s}(\{x_n > 0 \} | B_{3/2})$. Since $2\alpha>1$, for every $j \in \{0 , \dots, k-1\}$ and any $\gamma \in (0,1)$,
\begin{align*}
f_{j}= F_{j}^{(h)}= D_h^{j} w(\cdot + (k-j)h))\in C^{0,2\alpha}(\{ x_n \ge 0 \})\subseteq C^{0,\gamma}(\{ x_n \ge 0 \}), 
\end{align*}
we can define $W^{(h)} = D_h^{k-1+\alpha} w$ and proceed as in the first case of the proof of \autoref{prop:bootstrap-onephase} by applying \autoref{lemma:bdry-reg} that yields $W^{(h)} \in C^{0,\gamma}(\{ x_n \ge 0 \} \cap B_{1/4})$ with a bound that is uniform in $h$. Thus $\widetilde{w} \in C^{k+\alpha-\eps}(\partial \Omega \cap B_{1/4})$, which yields $\partial \Omega \in C_{loc}^{k+\alpha+1-\eps}(B_{1/4})$ for any $\eps \in (0,1)$, as desired.

If $\alpha\leq \max\{\frac{1}{2} , 2s-1 \}$, from \autoref{prop:bootstrap-onephase} we deduce that $\partial \Omega \in C^{k+\min\{1,\alpha+2s\}-\eps}_{loc}(B_1)$ for every $\eps \in (0,1)$, so, if $\alpha<1-2s$ we are done. Let us suppose then that $1-2s\leq \alpha\leq \max\{\frac{1}{2} , 2s-1 \}$ in which case \autoref{prop:bootstrap-onephase}  implies $\partial \Omega \in C^{k+1-\eps}_{loc}(B_1)$ for any $\eps \in (0,1)$. Since $K \in C^{2(k+1)+1}(\mathbb{S}^{n-1})$, by \autoref{lemma:PDE-onephase-flat} we get that $w \in C_{k-\eps}^{k+1-\eps+2s}(\{x_n > 0 \} | B_{3/2})$. Since for every $j \in \{0 , \dots, k-1\}$,
\begin{align*}
f_{j}= F_{j}^{(h)}= D_h^{j} w(\cdot + (k-j)h))\in C^{1-\eps}(\{ x_n \ge 0 \}),
\end{align*}
we can define $W^{(h)} = D_h^{k-1+\alpha} w$ and proceed as before to get a uniform bound for $W^{(h)} \in C^{\gamma}(\{ x_n \ge 0\} \cap B_{1/4})$ for $\gamma = 1-\eps$. Thus $\widetilde{w} \in C^{k+\alpha-\eps}(\partial \Omega \cap B_{1/4})$, which yields $\partial \Omega \in C_{loc}^{k+\alpha+1-\eps}(B_{1/4})$ for any $\eps \in (0,1)$, as wanted.

\end{proof}

\begin{corollary}
\label{prop:boundary-improvement2} 
Let $K \in C^{2(k+\alpha+2s)+ 1}(\mathbb{S}^{n-1})$ be as in \eqref{eq:K-comp} for some $k \in \N$ with $k \ge 2$ and $\alpha \in (0,1)$ with $\alpha \not\in \{s,1-s\}$. Let $\partial \Omega \in C^{k,\alpha}_{loc}(B_4)$ and $v \in L^{1}_{2s}(\R^n)$ be as in \autoref{prop:bootstrap-onephase} for some $h \in C^{k,\alpha}(B_4 \cap \partial \Omega)$. Then, for any $\eps \in (0,1)$ it holds $\partial \Omega \in C_{loc}^{k+\alpha+1-\eps}(B_1)$.
\end{corollary}

\begin{proof}
From the previous Proposition it is clear that we only have to analyze the cases $\alpha > \max\{\frac{1}{2} , 2s \}$  and $\alpha\leq \min\left(\max\{\frac{1}{2} , 2s-1 \},1-2s\right).$

In the first situation, since we have obtained $\partial \Omega \in C^{k+\alpha+2s}_{loc}(B_1)$, using that $K \in C^{2(k+\alpha+2s)+1}(\mathbb{S}^{n-1})$ we get that  $w \in C_{k-1+\alpha+2s}^{k+\alpha+4s}(\{x_n > 0 \} | B_{3/2})$. Thus for every $\gamma\leq \alpha+2s$, since $f_{j}= D_h^{j} w(\cdot + (k-j)h))\in C^{\gamma}(\{ x_n \ge 0 \})$ we can also apply \autoref{prop:bootstrap-onephase} with $W^{(h)} = D_h^{k-1+\alpha} w$ in order to conclude that $\widetilde{w} \in C^{k-1+\alpha+\gamma}(\partial \Omega \cap B_{1/4})$. Since the previous affirmation is true for every $\alpha>2s$, we can take $\gamma=4s$, i.e. $\widetilde{w} \in C^{k-1+\alpha+4s}(\partial \Omega \cap B_{1/4})$. Iterating the application of \autoref{prop:bootstrap-onephase} with $W^{(h)} = D_h^{k-1+\alpha} w$ a finite number of time we are able to choose $\gamma=1-\eps$ and conclude as we did in \autoref{prop:boundary-improvement} when $\max\{\frac{1}{2} , 2s-1 \}<\alpha<2s$.

If the second case we notice that, $\partial \Omega \in C^{k+\alpha+2s-\eps}_{loc}(B_1)$ for every $\eps \in (0,1)$, so if $s>1/2$ we are done. Thus, let us assume that $\alpha\leq \min\left(\frac{1}{2},1-2s\right)$, $s<1/2$. Since we have enough regularity in the kernel, by \autoref{lemma:PDE-onephase-flat} we obtain that $w \in C_{k-1+\alpha+2s-\eps}^{k+\alpha+4s-\eps}(\{x_n > 0 \} | B_{3/2})$ for every $\eps \in (0,1)$. Since for every $j \in \{0 , \dots, k-1\}$ and every $\gamma\leq \alpha+2s-\eps$, $f_{j}= D_h^{j} w(\cdot + (k-j)h))\in C^{\gamma}(\{ x_n \ge 0 \})$, we can also get the uniform bounds for $W^{(h)} = D_h^{k-1+\alpha} w$ in $C^{\gamma}(\{ x_n \ge 0\} \cap B_{1/4})$ that allows to affirm that $\widetilde{w} \in C^{k-1+\alpha+\gamma}(\partial \Omega \cap B_{1/4})$. That is,  $\widetilde{w} \in C^{k-1+2\alpha+2s-\eps}(\partial \Omega \cap B_{1/4})$, $\eps \in (0,1)$. Iterating a finite number of time we can also chose $\gamma = 1-\eps$ to conclude.
\end{proof}

\begin{proof}[Proof of \autoref{thm:main-onephase-intro2}]
We can assume without loss of generality that $0 \in \partial \Omega$ with $\nu(0) = e_n$. In particular, there exists $\rho > 0$ such that $\partial_n d_{\Omega} \ge \delta$ in $B_{\rho}$ for some $\delta > 0$. Moreover, by \autoref{lemma:Hopf}, we have that $\partial_n v \ge c_1 d_{\Omega}^{s-1}$ in $B_{\rho}$ for some $\rho > 0$. Hence, we have verified \eqref{eq:blue-ass} in the small ball $B_{\rho}$. Then, the desired result follows immediately by iterating \autoref{prop:boundary-improvement}.
\end{proof}

\subsubsection{Application to the one-phase problem}

In this section we apply the ideas from the previous section to the nonlocal one-phase problem and show \autoref{thm:main-onephase-intro}. The main difference to the proof of regularity of $\partial \Omega$ in the overdetermined problem in \autoref{thm:main-onephase-intro2}  comes through the different properties of the boundary condition. While in the previous subsection, we assumed $\frac{u}{d_{\Omega}^s} = h \in C^{k,\alpha}$ if $\partial \Omega \in C^{k,\alpha}$, this is not true for the nonlocal one-phase problem. In fact, we only have $h(x) = A (\nu(x)) \in C^{k-1,\alpha}$, since $\nu \in C^{k-1,\alpha}$. In order to get enough regularity of the boundary condition, we will choose the oblique derivative $\Xi$ in a specific way, namely $\Xi = \Theta_{K,\Omega}$ (where $\Theta_{K,\Omega}$ is given in \eqref{eq:Theta-formula-domains}), making use of the precise structure of the function $A$ in \autoref{lemma:one-phase-collection}(v).

\begin{corollary}
\label{lemma:PDE-onephase-flat-2}
Assume that we are in the situation of \autoref{lemma:PDE-onephase} with $h = A \circ \nu$, where $A$ is given in \eqref{eq:A-nu-def}. Let $\Phi \in C^{k,\alpha}$ be as in Subsection \ref{subsec:flattening} and $w:= \widetilde{w} \circ \Phi$. Then, it holds
\begin{itemize}
\item[(i)] $w \in C_{k-1+\alpha}^{k+\alpha+2s}(\{ x_n > 0 \} | B_{3/2})$ with $\supp(w) \subset \{ x_n \ge 0 \} \cap \overline{B_{3/2}}$.

\item[(ii)] The following equation holds in the strong sense.
\begin{align*}
\begin{cases}
\mathcal{L}\big((x_n)_+^{s-1}(y_n)_+^{s-1}J(x)J(y)K(\Phi(x)-\Phi(y))\big)(w) &= g ~~ \text{ in } \{ x_n > 0 \} \cap B_1,\\
\partial_{\theta_{K,J,\Phi}} w &= 0 ~~ \text{ on } \{ x_n = 0 \} \cap B_1,
\end{cases}
\end{align*}
where $\theta_{K,J,\Phi}$ is given in \eqref{eq:theta-K-def-flat}, and $g$ satisfies $(x_n)_+^{1-s} g \in C^{k-1,\alpha}(\{ x_n \ge 0 \}\cap B_{3/2})$ and is given by $g:= \widetilde{g} \circ \Phi$, and $J \in C^{k-1,\alpha}_c(\{ x_n \ge 0 \} \cap B_{2})$ satisfies $J \ge c_0$ in $B_{3/2}$ for some $c_0 > 0$, and is given by $J(x) = (d_{\Omega}^{1-s} \partial_n u \circ \Phi)(x)|\det D \Phi(x)|$.

\item[(iii)] Given $h = (h',0) \in \R^n$ with $|h| < \frac{1}{k}$, $\bar{k} \in \{ 1, \dots, k \}$ and $\bar{\alpha} \in (0,1)$ such that $\bar{k}+\bar{\alpha} \le k+\alpha$, we define $W^{(h)} = D_{h}^{\bar{k}-1+\bar{\alpha}} w$. Then, $W^{(h)} \in C^{\gamma}(\R^n)$ for any $\gamma \in (0, 1)$ and for every $\eta \in C^{\infty}_c(B_{1-k|h|})$, it holds
\begin{align*}
\int_{\R^n} &\int_{\R^n} (W^{(h)}(x) - W^{(h)}(y)) (\eta(x) - \eta(y)) (x_n)_+^{s-1}(y_n)_+^{s-1} J(x) J(y) K(\Phi(x) - \Phi(y)) \d y \d x \\
&=\int_{\R^n} \eta(x) G^{(h)}(x) \d x \\
&\quad + \sum_{j = 0}^{\bar{k}-1} \int_{\R^n} \int_{\R^n} (F^{(h)}_{j}(x) - F^{(h)}_{j}(y)) (\eta(x) - \eta(y)) (x_n)_+^{s-1}(y_n)_+^{s-1} I^{(h)}_{j}(x,y) \d y \d x,
\end{align*}
where for any $j \in \{ 0 , \dots, \bar{k}-1 \}$, and $x,y,z \in \R^n$:
\begin{align*}
\Vert(x_n)_+^{1-s} G^{(h)} \Vert_{L^{\infty}(\{ x_n > 0 \} \cap B_{\frac{3}{2} - k|h|})} + \Vert F^{(h)}_{j} \Vert_{C^{k - \bar{k},\alpha}(\{ x_n \ge 0 \})} & \le C \bar{C}, \\
|x-y|^{n+2s}|I^{(h)}_{j}(x,y)| &\le C \bar{C},\\
\left( \frac{\min\{|x-z|,|y-z|\}}{|x-y|} \right)^{\min\{ 1 , k - \bar{k} + \alpha - \bar{\alpha} \} }\min\{|x-z|,|y-z|\}^{n+2s}|I^{(h)}_{j}(x,z) - I^{(h)}_{j}(y,z)| &\le C \bar{C},
\end{align*}
where $C$ only depends on $n,s,k,\alpha,\lambda,\Lambda, c_0, \delta$, but not on $h$. We denoted
\begin{align*}
\bar{C} &:= \Vert (x_n)_+^{1-s} g \Vert_{C^{k-1,\alpha}(\{ x_n \ge 0 \}\cap B_{3/2})} + \Vert w \Vert_{C^{k+\alpha+2s}_{k-1,\alpha}(\{ x_n > 0 \} | B_{3/2})} + \Vert J \Vert_{C^{k-1,\alpha}(\{ x_n \ge 0 \})} \\
&\quad + \Vert \Phi \Vert_{C^{k,\alpha}(\R^n)} + \Vert K \Vert_{C^{2(k+\alpha)+1}(\mathbb{S}^{n-1})} .
\end{align*}
\end{itemize}
\end{corollary}

\begin{proof}
The proof of (i) and (ii) is exactly the same as in \autoref{lemma:PDE-onephase-flat}. We only need to justify the boundary condition in the strong sense. To do so, let us recall $h(x) = A(\nu(x))$, where $A$ is given in \eqref{eq:A-nu-def}, and observe that it suffices to show
\begin{align}
\label{eq:bdry-condition-theta-1}
\partial_{\theta(x)} \widetilde{w}(x) = 0 ~~ \text{ on } \partial \Omega \cap B_1,
\end{align}
where 
\begin{align*}
\theta(x):= h(x)^{2 - \frac{1}{s}} \int_{\mathbb{S}^{n-1} \cap \{ \nu(x) \cdot \omega > 0 \}} \omega (\nu(x) \cdot \omega)^{2s-1} K(\omega) \d \omega.
\end{align*}
Indeed, once \eqref{eq:bdry-condition-theta-1} is satisfied, we immediately deduce that $\partial_{\Theta_{K,\Omega}(x)} \widetilde{w}(x) = 0$ on $\partial \Omega \cap B_1$ from \eqref{eq:Theta-formula-domains} and the fact that $J(x) > 0$, $h(x) > 0$, and $|\det D \Phi(x)|^{-1} > 0$ in $B_1$. From there, the boundary condition for $w$ follows immediately by a change of variables, and using that by construction,
\begin{align*}
\Theta_{K,\Omega}(x) = |\det D \Phi^{-1}(x)| (D \Phi)(\Phi^{-1}(x)) \cdot \theta_{K,J,\Phi}(\Phi^{-1}(x)).
\end{align*}

Note that by recalling \eqref{eq:tilde-b-def} namely
\begin{align*}
\partial_{\theta(x)}\widetilde{w}(x) = \frac{s^2 h^2(x) \theta(x) \cdot \partial_{\tau(x)} \nu(x) + sh(x)\theta(x) \cdot \nu(x)\partial_{\tau(x)} h(x)}{(s h(x) \nu^{(n)}(x))^2},
\end{align*}
using the positivity of $h$ and of $(\nu^{(n)})^2$, as well as the product rule, the claim in \eqref{eq:bdry-condition-theta-1} is equivalent to 
\begin{align}\label{equiv}
\theta(x) \cdot \partial_{\tau(x)} (h^{\frac{1}{s}} \nu)(x) = 0.
\end{align} 

To prove \eqref{equiv}, {we notice that by using the symmetry of $K$ we get}, 
\begin{align}
\label{aux1}
\begin{split}
\nu(x) \cdot \theta(x) &= h(x)^{2-\frac{1}{s}}  \int_{\mathbb{S}^{n-1} \cap \{ \nu(x) \cdot \omega > 0 \}} |\nu(x) \cdot \omega|^{2s} K(\omega) \d \omega = \frac{c_{n,s}^{2}}{2} h(x)^{-\frac{1}{s}},\\
\nu(x) \cdot \partial_{\tau(x)} (h^{\frac{1}{s}-2} \theta)(x) &= (2s-1) [\partial_{\tau(x)} \nu(x)] \int_{\mathbb{S}^{n-1} \cap \{ \nu(x) \cdot \omega > 0 \}} |\nu(x) \cdot \omega|^{2s-1} K(\omega) \d \omega \\
&= \frac{c^2_{n,s}(2s-1)}{4s} \partial_{\tau(x)} (h^{-2})(x),
\end{split}
\end{align}
where $c_{n,s}$ is given in \eqref{eq:A-nu-def}. Then we obtain
\begin{align*}
\theta(x) \cdot \partial_{\tau(x)} (h^{\frac{1}{s}} \nu)(x) &= \partial_{\tau(x)} (\theta \cdot \nu h^{\frac{1}{s}})(x) -  h^{\frac{1}{s}}(x)  \nu(x) \cdot \partial_{\tau(x)} \theta(x) \\
&= -c_{n,s}^{2}h^{\frac{1}{s}}(x)\left( \nu(x) \cdot \theta(x) h^{\frac{1}{s}-2} \partial_{\tau(x)} h^{2 - \frac{1}{s}}(x) + h^{2-\frac{1}{s}}(x) \nu(x) \cdot  \partial_{\tau(x)} (h^{\frac{1}{s}-2} \theta)(x) \right) \\
&=  -c_{n,s}^{2}h^{\frac{1}{s}}(x)\left(\frac{1}{2} h^{-2}(x) \partial_{\tau(x)} h^{2 - \frac{1}{s}}(x) + \frac{2s-1}{4s} h^{2 - \frac{1}{s}}(x) \partial_{\tau(x)} (h^{-2})(x)\right) \\
&= \Big(\frac{1}{2s}+ \frac{4s-2}{4s}-1 \Big) c_{n,s}^2 h^{-1}(x) \partial_{\tau(x)} h(x)  = 0,
\end{align*}
where we have used the two properties in \eqref{aux1} in the first and second step, respectively. This concludes the proof of (ii).

To prove (iii), we proceed exactly as in the proof of \autoref{lemma:PDE-onephase-flat}, however we take incremental quotients directly in the weak formulation. In fact, (ii) implies that for any $\eta \in C_c^{\infty}(B_1)$:
\begin{align}
\label{eq:w-weak-pde-one-phase}
\begin{split}
\int_{\R^n} &\int_{\R^n} (w(x) - w(y)) (\eta(x) - \eta(y)) (x_n)_+^{s-1}(y_n)_+^{s-1} J(x) J(y) K(\Phi(x) - \Phi(y)) \d y \d x \\
&= \int_{\{ x_n > 0\} \cap B_1} \eta(x) \mathcal{L}\big((x_n)_+^{s-1}(y_n)_+^{s-1} J(x) J(y) K(\Phi(x) - \Phi(y)) \big)(w)(x) \d x \\
&\quad - \int_{\{ x_n = 0 \} \cap B_1} \eta(x',0) (\theta_{K,J,\Phi}(x') \cdot \nabla w)(x',0) \d x' \\
&= \int_{\{ x_n > 0 \}} g \eta \d x.
\end{split}
\end{align}

Next, we derive the following weak solution property for the increments $\Delta^{\bar{k}} w$, using the product rule for increments that was already used in the proof of \autoref{lemma:PDE-onephase-flat}(iii), namely
\begin{align*}
& \int_{\R^n} \int_{\R^n} (\Delta^{\bar{k}}w(x) - \Delta^{\bar{k}}w(y)) (\eta(x) - \eta(y)) (x_n)_+^{s-1}(y_n)_+^{s-1} J(x) J(y) K(\Phi(x) - \Phi(y)) \d y \d x \\
&= \int_{\{ x_n > 0 \}} \Delta^{\bar{k}} g \eta \d x \\
&\quad + \sum_{m = 0}^{\bar{k}-1} \sum_{l=0}^j \binom{\bar{k}}{m} \binom{m}{l} \int_{\R^n} \int_{\R^n} \big(\Delta^{m-l}w(x + (\bar{k}-m+l)h) - \Delta^{m-l}w(y + (\bar{k}-m+l)h) \big) \cdot \\
&\qquad\qquad\qquad\qquad\qquad\qquad\qquad\qquad\qquad \cdot (\eta(x) - \eta(y)) (x_n)_+^{s-1}(y_n)_+^{s-1} \Delta^{\bar{k}-m+l}M(x,y) \d y \d x,
\end{align*}

where we set again $M(x,y) = J(x)J(y)K(\Phi(x) - \Phi(y))$. Setting now $W^{(h)}, G^{(h)}, I_{j}^{(h)}, F_{j}^{(h)}$ exactly as in the proof of \autoref{lemma:PDE-onephase-flat}, we derive
\begin{align*}
& \int_{\R^n} \int_{\R^n} (W^{(h)}(x) - W^{(h)}(y)) (\eta(x) - \eta(y)) (x_n)_+^{s-1}(y_n)_+^{s-1} J(x) J(y) K(\Phi(x) - \Phi(y)) \d y \d x \\
&= \int_{\{ x_n > 0 \}} G^{(h)}(x) \eta(x) \d x \\
&\quad + \sum_{j= 0}^{\bar{k}-1}\int_{\R^n} \int_{\R^n} (F_{j}^{(h)}(x) - F_{j}^{(h)}(y)) (\eta(x) - \eta(y)) (x_n)_+^{s-1}(y_n)_+^{s-1} I_{j}^{(h)}(x,y) \d y \d x,
\end{align*}
which yields the desired result and concludes the proof.
\end{proof}

In order to get $C^{\infty}$ regularity of the free boundary in the nonlocal one-phase problem, we need to assume that $K \in C^{\infty}(\mathbb{S}^{n-1})$ and iterate the following proposition.

\begin{proposition}
\label{prop:bootstrap-onephase-2} 
Let $K$ be as in \eqref{eq:K-comp} and assume that it satisfies \eqref{eq:K-reg-ass} 
for some $k \in \N$ with $k \ge 2$ and $\alpha \in (0,1)$ with $\alpha \not\in \{s,1-s\}$.\\
Let $v \in L^{1}_{2s}(\R^n)$ be a minimizer to the nonlocal one-phase problem in $B_4$. Assume that $0 \in \partial \{ v > 0 \}$ is a regular free boundary point with $\partial \{v > 0\} \in C^{k,\alpha}$ in $B_4$. 
Moreover, assume \eqref{eq:blue-ass} and let $\widetilde{w}_i$ for any $i \in \{1,\dots,n-1\}$ be defined as in \eqref{eq:w-def}.\\
Then, for any $\gamma \in (\max\{\frac{1}{2}, 2s-1\} , 1)$ with $\alpha + \gamma \not= 1$ and $i \in \{1,\dots,n-1\}$ it holds 
\begin{align*}
\begin{cases}
\widetilde{w}_i \in C^{k-1+\alpha+ \min\{\gamma , 2s \} }_{loc}(\partial{\Omega} \cap B_1) \qquad\qquad\qquad ~~ \text{ if } \alpha \ge \gamma,\\
\widetilde{w}_i \in C^{k-1+\min\{\gamma,\alpha+2s\} - \eps}_{loc}(\partial{\Omega} \cap B_1) ~~ \forall \eps \in (0,1) ~~ \text{ if } \alpha < \gamma.
\end{cases}
\end{align*}
In particular, it holds for any $\eps \in (0,1)$,
\begin{align*}
\begin{cases}
\partial \Omega \in C^{k+\alpha+\min\{\alpha,2s\}}_{loc}(B_1) ~~\qquad\qquad\qquad \text{ if } \alpha > \max\{\frac{1}{2} , 2s-1\},\\
\partial \Omega \in C^{k+ \min\{ 1, \alpha + 2s \} - \eps }_{loc}(B_1) ~~ \forall \eps \in (0,1) ~~ \text{ if } \alpha \le \max\{\frac{1}{2},2s-1\}.
\end{cases}
\end{align*}
\end{proposition}

\begin{proof}
The proof is exactly the same as the proof of \autoref{prop:bootstrap-onephase}, using \autoref{lemma:PDE-onephase-flat-2} instead of \autoref{lemma:PDE-onephase-flat} and setting $b = \theta_K = A_{j,l} = \vartheta_{j,l} = a_{j,l} = \nu_{j,l} = 0$.
\end{proof}

\begin{proof}[Proof of \autoref{thm:main-onephase-intro}]
We can assume without loss of generality that $\alpha \not\in \{s,1-s\}$, up to making it a little smaller.
Let us denote $\Omega = \{ v> 0 \}$. We can assume without loss of generality that $\nu(0) = e_n$. In particular, there exists $\rho > 0$ such that $\partial_n d_{\Omega} \ge \delta$ for some $\delta > 0$. Moreover, by \autoref{lemma:one-phase-collection}(iv), we have that $v \ge c_0 d_{\Omega}^s$ in $B_{2}$. Hence, by application of the second claim in \autoref{lemma:Hopf}, we deduce that $\partial_n v \ge c_1 d_{\Omega}^{s-1}$ in $B_{\rho}$, upon making $\rho$ even smaller, if necessary, where $c_1 = \frac {\delta s c_0}{2}$. Hence, we have verified \eqref{eq:blue-ass} in the small ball $B_{\rho}$. Therefore, using that by \autoref{lemma:one-phase-collection}, we have $u/d^s_{\Omega} = A(\nu) \in C^{2,\alpha}$, up to a rescaling, we are in the setting of \autoref{prop:bootstrap-onephase-2} with $k = 2$ and $\alpha \in (0,1)$. 

If $\alpha \le \max\{ \frac{1}{2} , 2s-1 \}$, an application of \autoref{prop:bootstrap-onephase-2} yields that $\partial \Omega \in C^{2+\min\{1,2s+\alpha\}-\eps}$ in $B_{\rho/2}$ for any $\eps \in (0,1)$. Similarly, if $\alpha \ge \max\{ \frac{1}{2} , 2s-1 \}$, \autoref{prop:bootstrap-onephase-2} yields that $\partial \Omega \in C^{2 + \alpha + \min\{\alpha,2s\}}$. In both cases, we have gained regularity of $\partial \Omega$ of degree at least $\min\{1-\alpha,\alpha,2s\}$. Then, by applying \autoref{prop:bootstrap-onephase-2} again, we can bootstrap the regularity of $\partial \Omega$ up to $C^{\infty}$. In particular, by repeating the aforementioned procedure, we deduce that there exists $\rho > 0$, such that $\partial \Omega \in C^{k,\gamma}$ in $B_{\rho}$, as desired. 
\end{proof}

\subsection{Obstacle problem}

In this subsection we show that free boundaries of the nonlocal obstacle problem are smooth near regular points, establishing \autoref{thm:main-obstacle-intro}. This result was already established in \cite{AbRo20} by an approach based on higher order boundary Harnack inequalities. Here we provide an independent proof that follows a strategy analogous to the one for the nonlocal one-phase problem.

\begin{lemma}
\label{lemma:PDE-obstacle}
Let $K$ be as in \eqref{eq:K-comp} and assume that \eqref{eq:K-reg-ass} is satisfied for some $k \in \N$ and $\alpha \in (0,1)$ with $\alpha \not\in \{s,1-s\}$. Let $\varphi \in C^{k+\alpha+1+2s}(B_4)$.

Let $v \in L^{\infty}(\R^n)$ be a solution to the nonlocal obstacle problem in $B_4$ and denote $\Omega = \{ v > 0 \} \cap B_4$. Let $0 \in \partial \Omega$ be a regular free boundary point with $\partial \Omega \in C^{k,\alpha}$ in $B_3$, and assume that for some $\delta > 0$,
\begin{align}
\label{eq:blue-ass-obstacle}
\partial_n v \ge \delta, \qquad \partial_n (d_{\Omega}) \ge \delta \qquad \text{ in } \Omega \cap B_2.
\end{align}

Let $\kappa_1, \kappa_2 \in  C^{\infty}(\R^n)$ be two cut-off functions such that $0 \le \kappa_i \le 1$ in $\R^n$ and $\kappa_1 \equiv 1$ in $B_{7/4}$, $\kappa_1 \equiv 0$ in $B_2^c$, and $\kappa_2 \equiv 1$ in $B_{4/3}$, $\kappa_2 \equiv 0$ in $B_{3/2}^c$. Then, for any $i \in \{ 1 , \dots , n-1\}$, the function
\begin{align}
\label{eq:w-def-obstacle}
\widetilde{w} := \widetilde{w}_i := \kappa_2 \frac{\partial_i u}{\partial_n u}, \qquad \text{ where } \qquad u = \kappa_1 v,
\end{align} 
satisfies
\begin{itemize}
\item[(i)] $d_{\Omega}^{-s} \nabla u \in C^{k-1,\alpha}_c(\overline{\Omega} \cap B_{2})$ with $\supp(u) \subset \overline{\Omega \cap B_2}$, and it holds $\partial_n u \ge c_0 d_{\Omega}^{s} > 0$ in $\Omega \cap B_{3/2}$ for some $c_0 > 0$, depending on $\delta$.
\item[(ii)] $\widetilde{w} \in C_{k-1+\alpha}^{k+\alpha+2s}(\Omega | B_{3/2})$ with $\supp(\widetilde{w}) \subset \overline{\Omega \cap B_{3/2}}$.

\item[(iii)] In the strong sense,
\begin{align*}
\mathcal{L}\big(\partial_n u(x)\partial_n u(y)K(x-y)\big)(\widetilde{w}) = \widetilde{g} ~~ \text{ in } \Omega \cap B_1,
\end{align*}
where $\widetilde{g}$ satisfies $d_{\Omega}^{-s} \widetilde{g} \in C^{k-1,\alpha}(\overline{\Omega} \cap B_{3/2})$.
\end{itemize}
\end{lemma}

\begin{proof}
Due to interior estimates and by the regularity of $K$, $v$ is a classical solution to \eqref{eq:obstacle}. Hence, by construction we have that $u\in L^{\infty}(\R^n)$ is a strong solution to \eqref{eq:u-PDE}, where $g_1 = -L((1-\kappa_1)v)+L\varphi \ge 0$. Since $v \in C^{1+s}_{loc}(B_4)$ by \autoref{lemma:obstacle-collection}(i), we have that $\nabla u$ solves in the strong sense
\begin{align*}
\begin{cases}
L (\nabla v) &= \nabla g_1 ~~ \text{ in } \Omega \cap B_{3/2},\\
\nabla v &= 0 ~~~~~ \text{ in } \R^n \setminus (B_2 \setminus \Omega).
\end{cases}
\end{align*}

Since $\nabla g_1 \in C^{k+\alpha}(B_{3/2})$ by the assumptions on $K$ and on $\varphi$, by \autoref{lemma:boundary-reg} we have $d_{\Omega}^{-s} \nabla v \in C^{k+\alpha+2s}_{k-1+\alpha}(\Omega  | B_{2})$. Hence, by the definition of $u$ it also holds $d_{\Omega}^{-s} \nabla u \in C^{k+\alpha+2s}_{k-1+\alpha}( \Omega | B_{2})$, which in particular implies $d_{\Omega}^{-s} \nabla u \in C^{k-1,\alpha}_c(\overline{\Omega} \cap B_{2})$. The second claim in (i) follows immediately from the first claim in \autoref{lemma:Hopf}, since by construction, $\partial_n u \ge \delta$ in $\Omega \cap B_{3/2}$.

To prove (ii), we observe first that since $d_{\Omega}^{-s} \nabla u \in C^{k+\alpha+2s}_{k-1,\alpha}(\Omega | B_{3/2})$ and $d_{\Omega}^{-s} \partial_n u \ge c_0 > 0$ in $\Omega \cap B_{3/2}$, by the chain rule, it holds in particular
\begin{align*}
(d_{\Omega}^{-s} \partial_n u)^{-1} \in C^{k+\alpha+2s}_{k-1,\alpha}(\Omega | B_{3/2}).
\end{align*}
Therefore, by the product rule, we have
\begin{align*}
\frac{\partial_i u}{\partial_n u} = \frac{\partial_i u}{d^{s}_{\Omega}} \frac{d^{s}_{\Omega}}{\partial_n u} \in C^{k+\alpha+2s}_{k-1,\alpha}(\Omega | B_{3/2}),
\end{align*}
which proves (ii). 

We are now in the same situation as in the proof of \autoref{lemma:PDE-onephase-flat}(iii). By the exact same proof, we can verify the equation for $\widetilde{w}$, namely
\begin{align*}
\int_{\R^n} \eta(x) \mathcal{L}\big(\partial_nu(x)\partial_nu(y)K(x-y)\big)(\widetilde{w}) \d x = \int_{\R^n} \widetilde{g}(x) \eta(x) \d x,
\end{align*}
with $\widetilde{g}$ defined as
\begin{align}
\label{eq:g-def-obstacle}
\widetilde{g} = (\partial_n u) (\partial_i g_1 - L((1-\kappa_2)\partial_i u)) - (\kappa_2\partial_i u) (\partial_n g_1) , \qquad \text{ where } \qquad g_1 = -L ((1-\kappa_1)v).
\end{align}
To prove that the integral on the left-hand side is finite we proceed as in \eqref{eq:L-beta-finiteBB} applying \autoref{lemma:Ldist-est} by using that $\widetilde{w} \in C^{2s+\alpha}_{0,\alpha}(\Omega | B_{3/2})$, $z := d^{-s}_{\Omega} \partial_n u \in C_{0,\alpha}^{1+\alpha+2s}(\Omega | B_2)$ and $\partial\Omega\in C^{1,\alpha}$.  

The property $d_{\Omega}^{-s} \widetilde{g} \in C^{k-1,\alpha}(\overline{\Omega} \cap B_{3/2})$ follows from the definition of $\widetilde{g}$ in \eqref{eq:g-def-obstacle} and (i).
\end{proof}

\begin{corollary}
\label{lemma:PDE-obstacle-flat}
Assume that we are in the situation of \autoref{lemma:PDE-obstacle}. Let $\Phi \in C^{k,\alpha}$ be as in Subsection \ref{subsec:flattening} and $w:= \widetilde{w} \circ \Phi$. Then, it holds
\begin{itemize}
\item[(i)] $w \in C_{k-1+\alpha}^{k+\alpha+2s}(\{ x_n > 0 \} | B_{3/2})$ with $\supp(w) \subset \{ x_n \ge 0 \} \cap \overline{B_{3/2}}$.

\item[(ii)] In the strong sense,
\begin{align*}
\mathcal{L}\big((x_n)_+^{s}(y_n)_+^{s}J(x)J(y)K(\Phi(x)-\Phi(y))\big)(w) = g ~~ \text{ in } \{ x_n > 0 \} \cap B_1,
\end{align*}
where $g$ satisfies $(x_n)_+^{-s} g \in C^{k-1,\alpha}(\{ x_n \ge 0 \}\cap B_{3/2})$ and is given by $g:= \widetilde{g} \circ \Phi$, and $J \in C^{k+\alpha+2s}_{k-1,\alpha}(\{ x_n > 0 \} | B_{2})$ satisfies $J \ge c_0$ in $B_{3/2}$ for some $c_0 > 0$, as well as $\supp(J) \subset \{ x_n \ge 0 \} \cap B_2$, and is given by $J(x) = (d_{\Omega}^{-s} \partial_n u \circ \Phi)(x)|\det D \Phi(x)|$.

\item[(iii)] Given $h = (h',0) \in \R^n$ with $|h| < \frac{1}{k}$, $\bar{k} \in \{ 1, \dots, k \}$ and $\bar{\alpha} \in (0,1)$ such that $\bar{k}+\bar{\alpha} \le k+\alpha$, we define $W^{(h)} = D_{h}^{\bar{k}-1+\bar{\alpha}} w$. The following equation holds in the strong sense
\begin{align*}
\begin{cases}
& \mathcal{L} \big((x_n)_+^{s}(y_n)_+^{s} J(x)J(y) K(\Phi(x)-\Phi(y))\big)(W^{(h)}) \\
&\qquad\qquad \qquad = G^{(h)} + \sum_{j = 0}^{\bar{k}-1} \mathcal{L}\big((x_n)_+^{s-1}(y_n)_+^{s-1} I^{(h)}_{j}(x,y)\big)(F^{(h)}_{j}) ~~ \text{ in } \{ x_n > 0 \} \cap B_{1- k|h|},
\end{cases}
\end{align*}
where for any $j \in \{ 0 , \dots, \bar{k}-1 \}$,
\begin{align*}
\Vert(x_n)_+^{-s} G^{(h)} \Vert_{L^{\infty}(\{ x_n > 0 \} \cap B_{\frac{3}{2} - k|h|})} + \Vert F^{(h)}_{j} \Vert_{C^{k - \bar{k},\alpha}(\{ x_n \ge 0 \})} & \le C \bar{C}\\
\forall x,y \in \R^n: \quad |x-y|^{n+2s}|I^{(h)}_{j}(x,y)| &\le C \bar{C},
\end{align*}
where $C$ only depends on $n,s,k,\alpha,\lambda,\Lambda, c_0, \delta$, but not on $h$, and we denoted
\begin{align*}
\bar{C} &:= \Vert (x_n)_+^{-s} g \Vert_{C^{k-1,\alpha}(\{ x_n \ge 0 \}\cap B_{3/2})} + \Vert w \Vert_{C^{k+\alpha+2s}_{k-1,\alpha}(\{ x_n > 0 \} | B_{3/2})} + \Vert J \Vert_{C^{k-1,\alpha}(\{ x_n \ge 0 \})} \\
&\quad + \Vert \Phi \Vert_{C^{k,\alpha}(\R^n)} + \Vert K \Vert_{C^{2(k+\alpha)+1}(\mathbb{S}^{n-1})}.
\end{align*}

\item[(iv)] Moreover, $W^{(h)} \in C^{\gamma}(\R^n)$ for any $\gamma \in (0, 1)$, satisfies the equation in (iii) in the weak sense, i.e. (see \autoref{def:weak-sol-trafo}) for every $\eta \in C^{\infty}_c(B_{1-k|h|})$, it holds
\begin{align*}
\int_{\R^n} &\int_{\R^n} (W^{(h)}(x) - W^{(h)}(y)) (\eta(x) - \eta(y)) (x_n)_+^{s}(y_n)_+^{s} J(x) J(y) K(\Phi(x) - \Phi(y)) \d y \d x \\
&=\int_{\R^n} \eta(x) G^{(h)}(x) \d x \\
&\quad + \sum_{j = 0}^{\bar{k}-1} \int_{\R^n} \int_{\R^n} (F^{(h)}_{j}(x) - F^{(h)}_{j}(y)) (\eta(x) - \eta(y)) (x_n)_+^{s}(y_n)_+^{s} I^{(h)}_{j}(x,y) \d y \d x.
\end{align*}
\end{itemize}
\end{corollary}

\begin{proof}
The proof is exactly the same as the proof of \autoref{lemma:PDE-onephase-flat}, upon replacing $(x_n)_+^{s-1}$ by $(x_n)_+^s$. However, since the integration by parts formula in \autoref{lemma:ibp-nonflat-beta} does not involve any boundary integrals in case $\beta = 1+s$, we can set $B^{(h)} = 0$ and ignore all terms involving boundary data in the proof. 
\end{proof}

\begin{proposition}
\label{prop:bootstrap-obstacle} 
Let $K$ be as in \eqref{eq:K-comp} and assume that it satisfies \eqref{eq:K-reg-ass} 
for some $k \in \N$ and $\alpha \in (0,1)$ with $\alpha \not\in \{s,1-s\}$. Moreover, if $k = 1$, we assume that $\alpha >\max\{0, 2s-1\}$.\\
Let $v \in L^{\infty}(\R^n)$ be a solution to the nonlocal obstacle problem in $B_4$ with $\varphi \in C^{k+\alpha+1+2s}(B_4)$ and denote $\Omega = \{ v > 0 \} \cap B_4$. Assume that $0 \in \partial \Omega$ is a regular free boundary point with $\partial \Omega \in C^{k,\alpha}$ in $B_3$. Moreover, assume \eqref{eq:blue-ass-obstacle} and let $\widetilde{w}_i$ for any $i \in \{1,\dots,n-1\}$ be defined as in \eqref{eq:w-def}.\\
Then, for any $\gamma \in (\max\{0, 2s-1\} , s)$ with $\alpha + \gamma \not= 1$ {such that $\gamma\leq\alpha$ in the case $k=1$,}
 it holds 
\begin{align*}
\begin{cases}
\widetilde{w}_i \in C^{k-1+\alpha+\gamma}_{loc}(\partial{\Omega} \cap B_1) ~~ \text{ if } \alpha \ge \gamma,\\
\widetilde{w}_i \in C^{k-1+\gamma - \eps}_{loc}(\partial{\Omega} \cap B_1) ~~ \forall \eps \in (0,1) ~~ \text{ if } \alpha < \gamma.
\end{cases}
\end{align*}
where $i \in \{1,\dots,n-1\}$. In particular, the free boundary satisfies for any $\eps \in (0,1)$,
\begin{align*}
\begin{cases}
\partial \Omega \in C^{k+2\alpha}_{loc}(B_1) ~~ \text{ if } \alpha \in ( \max\{ 0 , 2s-1\} , s),\\
\partial \Omega \in C^{k+\alpha + s - \eps}_{loc}(B_1) ~~ \forall \eps \in (0,1) \text{ if } \alpha \in (s, 1),\\
\partial \Omega \in C^{k+s-\eps}_{loc}(B_1) ~~ \forall \eps \in (0,1) ~~ \text{ if } \alpha \le \max\{ 0 , 2s-1\}.
\end{cases}
\end{align*}
\end{proposition}

\begin{proof}
The proof is the same as in \autoref{prop:bootstrap-onephase}. Indeed, let us fix $i \in \{ 1 , \dots, n-1 \}$ and let $w = \widetilde{w}_i \circ \Phi$ be as in \autoref{lemma:PDE-obstacle-flat}. By \autoref{lemma:PDE-obstacle-flat}, we have $w \in C^{k+\alpha+2s}_{k-1+\alpha}(\{ x_n > 0 \} | B_{3/2} )$ and $\supp(w) \subset \{ x_n \ge 0 \} \cap \overline{B_{3/2}}$. In particular, it holds $w \in C^{k-1+\alpha}(\R^n)$. 

Let us now assume that $\alpha \ge \gamma$. We will provide the proof in the other case at the end.
We fix $h = (h',0) \in \R^n$ with $|h| < \frac{1}{2k}$, and define $W^{(h)} = D_{h}^{k-1+\alpha} w$. Then, since $\alpha \ge \gamma$, we have that $W^{(h)} \in C^{0,\gamma}(\R^n)$. 
By \autoref{lemma:PDE-obstacle-flat} it follows that $W^{(h)}$ satisfies all the assumptions of \autoref{lemma:bdry-reg} with $\beta = 1+s$, $\alpha := \alpha$, and $u := W^{(h)}$, where we set and deduce from \autoref{lemma:PDE-obstacle-flat}
\begin{align*}
g &:= G^{(h)}, \qquad (x_n)_+^{-s} G^{(h)} \in  L^{\infty}(\{ x_n \ge 0 \} \cap B_{1/2}),\\
f_{j} &:= F_{j}^{(h)} \in C^{0,\gamma}(\{ x_n \ge 0 \}),~~ j \in \{0 , \dots, k-1\}, \\
b &:= \theta_K := A := \vartheta := a := \nu := 0.
\end{align*}
The norms of all of these quantities are bounded from above by a constant of the form $C\bar{C}$, where $\bar{C}$ is as in \autoref{lemma:PDE-obstacle-flat} and $C$ depends only on $n,s,k,\alpha,\lambda,\Lambda,c_0,\delta$. Then, by \autoref{lemma:bdry-reg}, we deduce
\begin{align}
\label{eq:Wh-estimate-obstacle}
[W^{(h)}]_{C^{0,\gamma}(\{ x_n \ge 0 \} \cap B_{1/4})} \le C \bar{C},
\end{align}
and therefore, since $h = (h',0)$ with $|h| \le \frac{1}{2k}$ was arbitrary, and $C,\bar{C}$ do not depend on $h$, and $\alpha + \gamma \not=1$, we deduce that (see for instance \cite[p.431]{And97})
\begin{align*}
[w]_{C^{k-1+\alpha+\gamma}(\{ x_n = 0 \} \cap B_{1/4} )} \le C \bar{C}.
\end{align*}
By recalling the definition of $w$, this implies 
$\widetilde{w}_i \in C^{k-1+\alpha+\gamma}(\partial \Omega \cap B_{1/4})$ for any $i \in \{1,\dots,n-1\}$, and in particular, proceeding as in the proof of \autoref{prop:bootstrap-onephase}, it is sufficient to deduce that the normal vector $\nu$ of the free boundary satisfies $\nu \in C^{k+\alpha+\gamma}(\partial \Omega \cap B_{1/4})$. Therefore, we conclude that $\partial \Omega \in C^{k+\alpha+\gamma}$ in $B_{1/4}$. Thus, since $\gamma < s$ and $\gamma \le \alpha$, this implies $\partial \Omega \in C^{k+2\alpha}$ if $\alpha < s$ and $\partial \Omega \in C^{k+\alpha+s-\eps}$ if $\alpha \in (s,1)$, as claimed.

Let us now discuss the case $\alpha < \gamma$. In that case, we cannot proceed as before, since we only have $F_{j}^{(h)} \in C^{\alpha}(\{ x_n \ge 0 \} \cap B_{1/2})$, but need this function to be $C^{\gamma}$.
In case $k = 1$, we have ruled out by assumption that $\alpha < \gamma$, hence we can assume that $k \ge 2$.
We proceed as in \autoref{prop:bootstrap-onephase} and define $W^{(h)} = D_h^{k-2+\bar{\alpha}} w$ for some $\bar{\alpha} \in (0,1)$, i.e. choosing $\bar{k} = k - 1$ in \autoref{lemma:PDE-obstacle-flat} in case $k \ge 2$. Then, we still have $W^{(h)} \in C^{0,\gamma}(\R^n)$ and can apply \autoref{lemma:bdry-reg} with the same choices as above, since now $F_{j}^{(h)} \in C^{1,\alpha}(\{ x_n \ge 0 \} \cap B_{1/2}) \subset C^{\gamma}(\{ x_n \ge 0 \} \cap B_{1/2})$.  Hence, we obtain \eqref{eq:Wh-estimate}, which implies,
\begin{align*}
[w]_{C^{k-2+\bar{\alpha}+\gamma}(\{ x_n = 0 \} \cap B_{1/4})} \le C \bar{C},
\end{align*}
and therefore, $\partial \Omega \in C^{k-1+\bar{\alpha}+\gamma}$ in $B_{1/4}$. Since $\bar{\alpha} \in (0,1)$ and $\gamma \in (\max\{0,2s-1\},s)$ was arbitrary, we conclude the proof.
\end{proof}

\begin{proof}[Proof of \autoref{thm:main-obstacle-intro}]
We can assume without loss of generality that $\alpha \not\in \{s,1-s\}$, up to making it a little smaller. {Moreover, by \autoref{lemma:obstacle-collection} we can also assume that $\alpha>\max\{0,2s-1\}$}. Let us denote $\Omega = \{ v> 0 \}$. We can assume without loss of generality that $\nu(0) = e_n$. In particular, there exists $\rho > 0$ such that $\partial_n d_{\Omega} \ge \delta$ for some $\delta > 0$, depending only on the $C^{1,\alpha}$ radius of $\Omega$. Moreover, by the Hopf lemma \autoref{lemma:Hopf}, we have that $v \ge c_0 d_{\Omega}^s$ in $B_{2}$. Hence, by application of the second claim in \autoref{lemma:Hopf}, we deduce that $\partial_n v \ge c_1 d_{\Omega}^{s-1}$ in $B_{\rho}$, upon making $\rho$ even smaller, if necessary. Hence, we have verified \eqref{eq:blue-ass-obstacle} in the small ball $B_{\rho}$. Therefore, up to a rescaling, we are in the setting of \autoref{prop:bootstrap-obstacle} with $k = 1$ and $\alpha \in (\max\{0,2s-1\},1)$. 

We can now bootstrap the regularity of the free boundary up to $C^{\infty}$ using \autoref{prop:bootstrap-obstacle}, using the same arguments as in the proof of \autoref{thm:main-onephase-intro}.
\end{proof}

\section{Appendix}

In this appendix we collect several integral estimates that are used frequently in the paper.

\begin{lemma}
\label{lemma:int-polar-coord}
Let $a,b \in \R$ with $a > 0$ and $a-1 < b$. Then, for any $R > 0$
\begin{align*}
\int_{\R^n \setminus B_R} (y_n)_+^{a - 1} |y|^{-n-b} \d y \le c_1 R^{a-1-b},
\end{align*}
where $c_1 > 0$ depends only on $n,a,b$. Moreover, if $d \in \R$ with $a - 1  > -d$, then for any $R > 0$
\begin{align*}
\int_{B_R} (y_n)_+^{a - 1} |y|^{-n+d} \d y \le c_2 R^{a-1 + d},
\end{align*}
where $c_2 > 0$ depends only on $n,a,d$. 
\end{lemma}

\begin{proof}
We use polar coordinates with $y_n = r \cos(\theta)$ for some $r > 0$ and $\theta \in [0,2\pi)$, and compute
\begin{align*}
\int_{\R^n \setminus B_R} (y_n)_+^{a - 1} |y|^{-n-b} \d y & \le C \int_{0}^{2\pi} \cos(\theta)_+^{a-1} \int_{R}^{\infty} r^{a-1} r^{-1-b} \d r \d \theta \\
&\le C R^{a-1-b} \int_{0}^{2\pi} \cos(\theta)_+^{a-1}  \d \theta \le C R^{a-1-b}.
\end{align*}
Moreover, analogously
\begin{align*}
\int_{B_R} (y_n)_+^{a - 1} |y|^{-n+d} \d y &\le C \int_{0}^{2\pi} \cos(\theta)_+^{a-1} \int_{0}^{R} r^{a-1} r^{-1+d} \d r \d \theta \\
&\le C R^{a-1+d} \int_{0}^{2\pi} \cos(\theta)_+^{a-1}  \d \theta \le C R^{a-1+d}.
\end{align*}
\end{proof}

\begin{lemma}
\label{lemma:bd-int}
Let $\Omega \subset \R^n$ be a bounded Lipschitz domain and $x \in \Omega$. Let $a,d > 0$, $-d\neq a-1$. Then, 
\begin{align*}
\int_{\Omega} d_{\Omega}^{a-1}(y) |x-y|^{-n+d} \d y \le C \left( 1 + d_{\Omega}^{a-1+d}(x) \right),
\end{align*} 
where $C > 0$ only depends on $n,a,d$, and $\Omega$.
\end{lemma}

Note that in case $-d=a-1$, by the same proof, we can bound the integral by $C(1+\log d_{\Omega}(x))$.

\begin{proof}
Let us denote $\rho := d_{\Omega}(x)$. We split the integration domains into several parts. First, we observe that for $y \in B_{\rho/2}(x)$ it holds $\rho/2 \le d_{\Omega}(y) \le 2 \rho$ and thus, since $d > 0$,
\begin{align*}
\int_{\Omega \cap B_{\rho/2}(x)} d_{\Omega}^{a-1}(y) |x-y|^{-n+d} \d y \le C \rho^{a-1} \int_{B_{\rho/2}(x)} |x-y|^{-n+d} \d y \le C \rho^{a-1+d}.
\end{align*}
Moreover, since $\partial \Omega$ is a Lipschitz graph, there exists $\kappa \in (0,\frac{1}{2})$ such that for any $t \in (0,\kappa)$, $A \in (0,1)$, and $x \in \Omega$ with $d_{\Omega}(x) \le \kappa$ it holds
\begin{align*}
\mathcal{H}^{n-1} \big( \{ d_{\Omega} = t \} \cap (B_{2A}(x) \setminus B_{A}(x)) \big) \le C A^{n-1}.
\end{align*}
Let us now assume $\rho \le \kappa$. Hence, we can estimate for $M = -\log_2(\rho)$ by the co-area formula
\begin{align*}
 &\int_{( \Omega \setminus B_{\rho/2}(x) ) \cap \{ d_{\Omega} \le \kappa/2 \} }  d_{\Omega}^{a-1}(y) |x-y|^{-n+d} \d y \\
 & \le C \sum_{k = 1}^{M} (2^k \rho)^{-n+d} \int_{\big( B_{2^{k+1} \rho}(x) \setminus B_{2^{k} \rho}(x) \big) \cap \{ d_{\Omega} \le \kappa/2 \} } d_{\Omega}^{a-1}(y) \d y \\
 &\quad + C \int_{\Omega \setminus B_1}  d_{\Omega}^{a-1}(y) |x-y|^{-n+d} \d y \\
&\le C \sum_{k = 1}^{M} (2^k \rho)^{-n+d} \int_0^{\kappa/2} t^{a-1} \mathcal{H}^{n-1} \left( \{ d_{\Omega} = t \} \cap (B_{2^{k+1} \rho}(x) \setminus B_{2^{k} \rho}(x)) \right) \d t + C\\
&\le C \sum_{k = 1}^{M} (2^k \rho)^{-n+d} (2^k \rho)^{n-1} + C \le C \rho^{d-1} 2^{M(d-1)}  + C \le C.
\end{align*}
On the other hand, if $\rho \ge \kappa$, then we have
\begin{align*}
\int_{( \Omega \setminus B_{\rho/2}(x) ) \cap \{ d_{\Omega} \le \kappa/2 \} }  d_{\Omega}^{a-1}(y) |x-y|^{-n+d} \d y \le C \int_{( \Omega \setminus B_{\rho/2}(x) ) \cap \{ d_{\Omega} \le \kappa/2 \} }  d_{\Omega}^{a-1}(y) \d y \le C,
\end{align*}
since $c_1 \le |x-y| \le c_2$ for some $c_1,c_2 > 0$, depending only on $\kappa,\Omega$, and since $a > 0$. Finally, we have
\begin{align*}
\int_{( \Omega \setminus B_{\rho/2}(x) ) \cap \{ d_{\Omega} \ge \kappa/2 \} }  d_{\Omega}^{a-1}(y) |x-y|^{-n+d} \d y \le C,
\end{align*}
since on this set, it holds $c_3 \le d_{\Omega}(y) \le c_4$ and $c_5 \le |x-y| \le c_6$ for some $c_3,c_4,c_5,c_6 > 0$, depending only on $\kappa,\Omega$. The proof is complete.
\end{proof}

\begin{lemma}
\label{lemma:n-1-dim-integral}
Let $\delta < 1+2s$. Then, for any $x \in \R^n$ and $y_n \in \R$,
\begin{align*}
\int_{\R^{n-1}} |x-y|^{-n-2s+\delta} \d y' \le C |x_n - y_n|^{-1-2s+\delta},
\end{align*}
where $C$ depends only on $n,s,\delta$.
\end{lemma}

\begin{proof}
We compute, setting $h_n = x_n-y_n$,
\begin{align*}
\int_{\R^{n-1}} |x-y|^{-n-2s+\delta} \d y' &= \int_{\R^{n-1}} [(|h'|^2 + |h_n|^2)^{1/2}]^{-n-2s+\delta} \d h' \\
&= |h_n|^{-n-2s+\delta} \int_{\R^{n-1}} [(|h'/|h_n||^2 + 1)^{1/2}]^{-n-2s+\delta} \d h' \\
&= |h_n|^{-1-2s+\delta} \int_{\R^{n-1}} [(|h'|^2 + 1)^{1/2}]^{-(n-1)-2s-1+\delta} \d h' = C|x_n - y_n|^{-1-2s+\delta}.
\end{align*}
Note that the condition $\delta < 1 + 2s$ was used in order to guarantee the convergence of the integral.
\end{proof}

The following lemma gives an estimate in case $\delta > 1 + 2s$.
\begin{lemma}
\label{lemma:n-1-dim-integral-2}
Let $\delta > 1+2s$. Then, for any $x \in \R^n$ and $y_n \in \R$,
\begin{align*}
\int_{\R^{n-1} \cap B_r(x)} |x-y|^{-n-2s+\delta} \d y' \le C r^{-1-2s+\delta},
\end{align*}
where $C$ depends only on $n,s,\delta$.
\end{lemma}

\begin{proof}
We proceed similarly to \autoref{lemma:n-1-dim-integral}, denote $h_n = |x_n - y_n|$ and compute
\begin{align*}
\int_{\R^{n-1} \cap B_r(x)} |x-y|^{-n-2s+\delta} \d y' &= |h_n|^{-1-2s+\delta} \int_{\R^{n-1} \cap B_{\frac{r}{|h_n|}}} [(|h'|^2 + 1)^{1/2}]^{-(n-1)-2s-1+\delta} \d h' \\
& \le C |h_n|^{-1-2s+\delta} \int_{B'_{\frac{r}{|h_n|}}} |h'|^{-(n-1)-2s-1+\delta} \d h' \le C r^{-1-2s+\delta}.
\end{align*}
\end{proof}

Moreover, we have the following one-dimensional estimate:

\begin{lemma}
\label{lemma:int-finite}
Let $\beta \in [s,1+s]$ and $\delta \in (2s-2 , 2s - \beta)$. Then, there is $C > 0$ such that
\begin{align*}
I := \int_0^1 r^{\beta-1} \int_1^{\infty} t^{\beta - 1} |t-r|^{-2s+\delta} \d t \d r \le C. 
\end{align*}
\end{lemma}

\begin{proof}
We compute for any $r \in (0,1)$:
\begin{align*}
\int_1^2 t^{\beta - 1} |t-r|^{-2s+\delta} \d t &\le C \int_1^2 |t-r|^{-2s+\delta} \d t \le C |1-r|^{1-2s+\delta}, \\
\int_2^{\infty} t^{\beta - 1} |t-r|^{-2s+\delta} \d t &\le C \int_2^{\infty} t^{\beta - 1 - 2s + \delta} \d t \le C,
\end{align*}
where we used that $\delta < 2s - \beta$ in order to have convergence of the second integral. Hence, we have
\begin{align*}
I \le C + C \int_0^1 r^{\beta - 1} |1-r|^{1-2s+\delta} \d r \le C ,
\end{align*}
where we used that $\delta > 2s-2$, as desired.
\end{proof}

We are now in a position to give a proof of \autoref{lemma:cutoff-estimate}.

\begin{proof}[Proof of \autoref{lemma:cutoff-estimate}]
We start by estimating the integral over $B_r(x)$. Let us first give a very short estimate in case $\beta \in (1,1+s]$. We compute
\begin{align*}
\int_{B_r(x)} (y_n)_+^{\beta-1} \min\{1 , r^{-2} |x-y|^2 \}  |x-y|^{-n-2s} \d y 
\le C  r^{\beta-1} \int_{\R^n} r^{-2} |x-y|^{-n-2s+2} \d y \le C r^{\beta-1-2s}.
\end{align*}

Let us now give a proof that works for all $\beta \in [s,1+s]$, i.e. in particular for $\beta \in [s,1]$. We compute, using that $x_n \le r$,
\begin{align*}
(x_n)_+^{\beta-1} r^{-2} & \int_{B_r(x)} (y_n)_+^{\beta-1} |x-y|^{-n-2s+2} \d y \\
&\le C  r^{-2} \int_0^{2r} (y_n)_+^{\beta-1} \left( \int_{\R^{n-1} \cap B_r(x)} |x-y|^{-n-2s+2} \d y' \right) \d y_n.
\end{align*}
On one hand, in case $s > \frac{1}{2}$, we apply \autoref{lemma:n-1-dim-integral} with $\delta:= 2 < 1 + 2s$ to deduce
\begin{align*}
r^{-2}  \int_{B_r(x)} (y_n)_+^{\beta-1} |x-y|^{-n-2s+2} \d y 
\le C r^{-2} \int_0^{2r} (y_n)_+^{\beta-1} |x_n - y_n|^{-2s+1} \d y_n 
\le Cr^{\beta-1-2s}.
\end{align*}
In the last estimate we have used that
\begin{align*}
\int_0^{2r} t^{\beta-1} |\eta - t|^{-2s+1} \d t \le C r^{\beta+1-2s}, ~~ \forall \eta \le r,
\end{align*}
To prove the previous affirmation, we notice that, firstly, it is easy to see that, 
\begin{align*}
\int_0^{\eta} t^{\beta-1} |\eta - t|^{-2s+1} \d t = \eta^{\beta-1-2s} \int_0^1 t^{\beta-1} |1-t|^{-2s+1} \d t =B(\beta, 2-2s) \eta^{\beta + 1-2s} \le C r^{\beta + 1-2s},
\end{align*}
where in the last estimate we have used that $\beta>2s-1 > 0$. Secondly,
\begin{align*}
\int_\eta^{2r} t^{\beta-1} |\eta - t|^{-2s+1} \d t &= \eta^{\beta + 1-2s} \int_0^{\frac{2r-\eta}{\eta}} (1+t)^{\beta-1} t^{-2s+1} \d t \\
&{\le} C \eta^{\beta + 1-2s} \Big( \left(1 + \frac{r}{\eta}\right)^{\beta}  \left( \frac{r}{\eta} \right)^{1-2s} + \left(1 + \frac{r}{\eta}\right)^{\beta-1} \left(\frac{r}{\eta} \right)^{2-2s} \Big) \\
&\le C(\eta+r)^{\beta} r^{1-2s} + C (\eta+r)^{\beta-1} r^{2-2s} \le C r^{\beta + 1-2s},\quad \eta\le r.
\end{align*}

On the other hand, in case $s < \frac{1}{2}$, we apply \autoref{lemma:n-1-dim-integral-2} to deduce
\begin{align*}
 r^{-2} \int_{B_r(x)} (y_n)_+^{\beta-1} |x-y|^{-n-2s+2} \d y 
\le  C  r^{-2} \int_0^{2r} (y_n)_+^{\beta-1} r^{-2s+1} \d y_n \le C  r^{\beta-1-2s},
\end{align*}
as desired.
In case $s=\frac{1}{2}$, we estimate $r^{-2}|x-y|^2 \le r^{-1+\alpha} |x-y|^{1+\alpha}$ for some $\alpha \in (0,1)$ and proceed along the same lines. This concludes the estimate for the integral over $B_r(x)$.

Let us now turn to the estimate for the integral over $\R^n \setminus B_r(x)$. In case $\beta \in [s,1]$, we immediately have by \autoref{lemma:int-polar-coord}
\begin{align*}
\int_{\R^n \setminus B_r(x)} (y_n)_+^{\beta-1} |x-y|^{-n-2s} \d y &=  \int_{\R^n \setminus B_r} (y_n + x_n)_+^{\beta-1} |y|^{-n-2s} \d y \\
& \le \int_{\R^n \setminus B_r} (y_n)_+^{\beta-1} |y|^{-n-2s} \d y \le C  r^{\beta-1-2s}.
\end{align*}
In case $\beta \in (1,1+s]$, we compute for $y_n \le 4r$,
\begin{align*}
 \int_{ \{ y_n \le 4r \} \setminus B_r(x)} (y_n)_+^{\beta-1} |x-y|^{-n-2s} \d y \le C r^{\beta-1} \int_{\R^n \setminus B_r(x)}|x-y|^{-n-2s} \d y \le C r^{\beta-1-2s}.
\end{align*}
Moreover, for $y_n \ge 4r$, we have since $x_n \le r$ that $y_n \le 2|x_n - y_n| \le 2|x-y|$ and since $\beta < 1 + 2s$,
\begin{align*}
 \int_{ \{ y_n \ge 4r \} \setminus B_r(x)} (y_n)_+^{\beta-1} |x-y|^{-n-2s} \d y \le C  \int_{\R^n \setminus B_r(x)}  |x-y|^{-n-2s+\beta-1} \d y \le C r^{\beta-1-2s}.
\end{align*}
\end{proof}

We conclude this section by proving the following lemma, complementing \cite[Proposition B.2.1]{FeRo24} and \cite[Corollary 2.3]{AbRo20}.

\begin{lemma}
\label{lemma:Ldist-est}
Let $\partial \Omega \in C^{1,\alpha}$ in $B_1$ for some $\alpha \in (0,1)$ and let $\eps > 0$. Then, for any $\phi \in C^{2s+\eps}_{0,\alpha}(\Omega | B_1)$,
\begin{align*}
L(d_{\Omega}^s \phi) \le C \Vert \phi \Vert_{C^{2s+\eps}_{0,\alpha}(\Omega | B_1)} (d_{\Omega}^{\alpha - s} + 1) ~~ \text{ in } B_{1/2},
\end{align*}
where $C$ depends only on $n,s,\lambda,\Lambda, \alpha, \eps$, and $\Omega$.
\end{lemma}

\begin{proof}
Up to a normalization, we can assume that $\Vert \phi \Vert_{C^{2s+\eps}_{0,\alpha}(\Omega | B_1)} \le 1$.
Let $x_0 \in B_{1/2}$ and set $\rho_0 := d_{\Omega}(x_0)$, as well as $\bar{\phi} = \phi - \phi(x_0)$. Then, it holds for $y \in B_{\rho_0/2}$,
\begin{align*}
|d^s_{\Omega}(x_0 + y) \bar{\phi}(x_0+y) & + d^s_{\Omega}(x_0 - y) \bar{\phi}(x_0-y) - 2 d^s_{\Omega}(x_0) \bar{\phi}(x_0)| \\
&\le C|y|^{2s+\eps} [d^s_{\Omega} \bar{\phi}]_{C^{2s+\eps}(B_{\rho_0/2}(x_0))}  \le C |y|^{2s+\eps} \rho_0^{-s - \eps + \alpha}.
\end{align*}
In the last inequality we have used the product rule, the fact that $[d_{\Omega}^s]_{C^{\gamma}(B_{\rho_0/2}(x_0))} \le C \rho_0^{s-\gamma}$ for any $ \gamma > 0$ (see \cite[Lemma A.2]{AbRo20}) and that $[\bar{\phi}]_{C^{\gamma}(B_{\rho_0/2}(x_0))} \le C \rho_0^{\alpha-\gamma}$ for all $\gamma \in (0,2s+\eps]$. The latter property follows by construction of $\bar{\phi}$. Indeed, since on the one hand it holds $\Vert \bar{\phi} \Vert_{L^{\infty}(B_{\rho_0/2}(x_0))} \le C \rho_0^{\alpha}$ since $\phi \in C^{\alpha}(\overline{\Omega} \cap B_1)$, and on the other hand $[\bar{\phi}]_{C^{\gamma}(B_{\rho_0/2}(x_0))} \le C \rho_0^{\alpha-\gamma}$ for any $\gamma \in [\alpha,2s+\eps]$ since $\bar\phi\in C^{2s+\eps}_{0,\alpha}(B_{\rho_0/2}(x_0))$, the desired estimate follows by H\"older interpolation.

Moreover, for $y \in B_1 \setminus B_{\rho_0/2}$,
\begin{align*}
|d^s_{\Omega}(x_0 + y) \bar{\phi}(x_0+y) - d^s_{\Omega}(x_0) \bar{\phi}(x_0) | &= d^s_{\Omega}(x_0 + y) |\phi(x_0+y) - \phi(x_0)| \\
&\le C d^s_{\Omega}(x_0 + y) |y|^{\min\{2s+\eps , 1 \}} d_{\Omega}^{\alpha - \min\{2s + \eps , 1\}}(x_0+y),
\end{align*}
and for $y \in \R^n \setminus B_1$, by the boundedness of $\phi$,
\begin{align*}
|d^s_{\Omega}(x_0 + y) \bar{\phi}(x_0+y)| \le C |y|^s.
\end{align*}
Therefore, using \cite[Lemma B.2.4]{FeRo24}, we deduce
\begin{align*}
|L(d_{\Omega}^s \bar{\phi})(x_0)| &\le C \rho_0^{-s-\eps+\alpha} \int_{B_{\rho_0/2}} |y|^{-n+\eps} \d y \\
&\quad + C \int_{B_1 \setminus B_{\rho_0/2}} d^{s + \alpha - \min\{2s + \eps , 1\}}_{\Omega}(x_0 + y) |y|^{-n-2s + \min\{2s+\eps , 1\}} \d y +  C \int_{\R^n \setminus B_1} |y|^{-n-s} \d y \\
&\le  C \rho_0^{\alpha-s} + C \le C (d_{\Omega}^{\alpha-s}(x_0) + 1)
\end{align*}
Altogether, since $|L(d^s_{\Omega})| \le C d_{\Omega}^{\alpha-s}$ by \cite[Proposition B.2.1]{FeRo24}, we obtain 
\begin{align*}
|L(d^s_{\Omega} \phi)(x_0)| \le |\phi(x_0)| |L(d^s_{\Omega})(x_0)| + |L(d^s_{\Omega} \bar{\phi})(x_0)| \le C (d_{\Omega}^{\alpha-s}(x_0) + 1),
\end{align*}
as desired.
\end{proof}

\end{document}